\documentclass[a4paper,twoside,openright]{book}
\usepackage{amssymb, amsthm, latexsym, amsmath, bm} 	
\usepackage{tikz}
\usetikzlibrary{cd,arrows,decorations.pathreplacing}
\usepackage[outline]{contour} %
\usepackage{setspace} %
\usepackage{graphicx}
\usepackage{emptypage} %
\usepackage[labelformat=simple]{subcaption}
\usepackage{float} %
\usepackage{units} %
\usepackage[top=3.5cm, bottom=3.5cm, left=4cm, right=3cm]{geometry}
\usepackage{enumitem} %
\usepackage{relsize} %
\usepackage[margin=2\parindent, font={small, it}]{caption} %
\usepackage{xifthen} %
\usepackage{xstring} %
\usepackage{cancel} %
\usepackage{mathtools} %

\usepackage[backend=biber,style=alphabetic,bibencoding=utf8]{biblatex}
\makeatletter
\@ifpackageloaded{biblatex}{\addbibresource{bibliography.bib}}{\bibliography{bibliography}}
\makeatother

\newenvironment{absolutelynopagebreak}
  {\par\nobreak\vfil\penalty0\vfilneg
   \vtop\bgroup}
  {\par\xdef\tpd{\the\prevdepth}\egroup
   \prevdepth=\tpd}

\usepackage{remreset}
\makeatletter
  \@removefromreset{section}{chapter}
\makeatother

\numberwithin{equation}{section} %
\numberwithin{figure}{section}

\makeatletter \let\c@equation\c@figure \makeatother %

\usepackage{hyperref}%
\usepackage[capitalise,nameinlink,noabbrev]{cleveref}
\crefname{equation}{}{}
\crefname{enumi}{}{}
\crefname{thm}{Theorem}{Theorems}
\crefformat{enumi}{#2(#1)#3}
\crefrangeformat{thm}{Theorems~#3#1#4 to~#5#2#6}
\crefrangeformat{lemma}{Lemmas~#3#1#4 to~#5#2#6}
\crefrangeformat{prop}{Propositions~#3#1#4 to~#5#2#6}
\crefrangeformat{corr}{Corollaries~#3#1#4 to~#5#2#6}
\crefrangeformat{conj}{Conjectures~#3#1#4 to~#5#2#6}
\crefrangeformat{defn}{Definitions~#3#1#4 to~#5#2#6}
\crefrangeformat{notation}{Notations~#3#1#4 to~#5#2#6}
\crefrangeformat{eg}{Examples~#3#1#4 to~#5#2#6}
\crefrangeformat{egs}{Examples~#3#1#4 to~#5#2#6}
\crefrangeformat{remark}{Remarks~#3#1#4 to~#5#2#6}
\crefrangeformat{remarks}{Remarks~#3#1#4 to~#5#2#6}
\crefrangeformat{section}{Sections~#3#1#4 to~#5#2#6}
\crefrangeformat{enumi}{#3(#1)#4 to~#5(#2)#6}
\crefmultiformat{thm}{Theorems~#2#1#3}{ and~#2#1#3}{, #2#1#3}{ and~#2#1#3}
\crefmultiformat{lemma}{Lemmas~#2#1#3}{ and~#2#1#3}{, #2#1#3}{ and~#2#1#3}
\crefmultiformat{prop}{Propositions~#2#1#3}{ and~#2#1#3}{, #2#1#3}{ and~#2#1#3}
\crefmultiformat{corr}{Corollaries~#2#1#3}{ and~#2#1#3}{, #2#1#3}{ and~#2#1#3}
\crefmultiformat{conj}{Conjectures~#2#1#3}{ and~#2#1#3}{, #2#1#3}{ and~#2#1#3}
\crefmultiformat{defn}{Definitions~#2#1#3}{ and~#2#1#3}{, #2#1#3}{ and~#2#1#3}
\crefmultiformat{notation}{Notations~#2#1#3}{ and~#2#1#3}{, #2#1#3}{ and~#2#1#3}
\crefmultiformat{eg}{Examples~#2#1#3}{ and~#2#1#3}{, #2#1#3}{ and~#2#1#3}
\crefmultiformat{egs}{Examples~#2#1#3}{ and~#2#1#3}{, #2#1#3}{ and~#2#1#3}
\crefmultiformat{remark}{Remarks~#2#1#3}{ and~#2#1#3}{, #2#1#3}{ and~#2#1#3}
\crefmultiformat{remarks}{Remarks~#2#1#3}{ and~#2#1#3}{, #2#1#3}{ and~#2#1#3}
\crefmultiformat{section}{Sections~#2#1#3}{ and~#2#1#3}{, #2#1#3}{ and~#2#1#3}
\crefmultiformat{enumi}{#2(#1)#3}{ and~#2(#1)#3}{, #2(#1)#3}{ and~#2(#1)#3}

\theoremstyle{plain}
\newtheorem{thm}[equation]{Theorem}
\newtheorem{lemma}[equation]{Lemma}
\newtheorem{prop}[equation]{Proposition}
\newtheorem{corr}[equation]{Corollary}
\newtheorem{conj}[equation]{Conjecture}

\theoremstyle{definition} 
\newtheorem{defn}[equation]{Definition}
\newtheorem{notation}[equation]{Notation}
\newtheorem{eg}[equation]{Example}
\newtheorem{egs}[equation]{Examples}

\theoremstyle{remark}

\newtheorem{convention}[equation]{Convention}

\crefname{thm}{Theorem}{Theorems}
\crefname{lemma}{Lemma}{Lemmas}
\crefname{prop}{Proposition}{Propositions}
\crefname{corr}{Corollary}{Corollaries}
\crefname{conj}{Conjecture}{Conjectures}
\crefname{defn}{Definition}{Definitions}
\crefname{notation}{Notation}{Notations}
\crefname{eg}{Example}{Examples}
\crefname{egs}{Examples}{Examples}
\crefname{remark}{Remark}{Remarks}
\crefname{remarks}{Remarks}{Remarks}

\usepackage[shortcuts]{extdash}

\newcommand{\R}{\mathbb{R}}
\newcommand{\C}{\mathbb{C}}
\newcommand{\Q}{\mathbb{Q}}
\newcommand{\Z}{\mathbb{Z}}
\newcommand{\N}{\mathbb{N}}
\newcommand{\F}{\mathbb{F}}
\newcommand{\iso}{\cong}
\newcommand{\homeo}{\cong}

\newcommand{\del}{\partial}
\newcommand{\dd}{\mathop{}\! d}
\newcommand{\union}{\cup}
\newcommand{\tensor}{\otimes}
\newcommand{\inv}{^{-1}}
\newcommand{\sminus}{\!\smallsetminus\!} %
\newcommand{\cross}{\!\times\!} %
\newcommand{\set}[1]{\{#1\}} %
\newcommand{\gappy}[1]{\enskip #1\enskip} %

\newcommand{\nogapcite}[2][]{%
\ifthenelse{ \isempty{#1} }{\hspace{1sp}\cite{#2}}{\hspace{1sp}\cite[#1]{#2}}%
}%

\DeclareMathOperator{\id}{id}
\DeclareMathOperator{\pt}{pt}
\DeclareMathOperator{\e}{e}
\DeclareMathOperator{\ind}{ind}
\DeclareMathOperator{\nmc}{nmc}
\newcommand{\MB}{\mathrm{MB}}

\newcommand{\restr}[2]{{%
  \left.\kern-\nulldelimiterspace %
  #1 %
  \vphantom{\big|} %
  \right|_{#2} %
}}

\newcommand{\orbits}[2]{%
\StrLen{#1}[\topLen]%
\StrLen{#2}[\botLen]%
^{%
\IfSubStr{#1}{0}{\emptyset\leftarrow}{}%
\IfSubStr{#1}{h}{h_+\rightarrow}{%
\ifthenelse{\botLen > 1}{\protect\phantom{h_+\rightarrow}}{}}%
\IfSubStr{#1}{e}{e_-}{}%
}%
_{%
\ifthenelse{\topLen > 2}{\protect\phantom{\emptyset\leftarrow}}{}%
\IfSubStr{#2}{e}{e_+\rightarrow}{%
\ifthenelse{\topLen > 1}{\protect\phantom{e_+\rightarrow}}{}}%
\IfSubStr{#2}{h}{h_-}{}%
}}

\newcommand{\figureaxes}[6]{
\StrLen{#3}[\eastLen]%
\StrLen{#4}[\northLen]%
\StrLen{#5}[\westLen]%
\StrLen{#6}[\southLen]%
\ifthenelse{\eastLen > 0}{\draw[-latex] (#1,#2)--+(0.7,0) node [right]{$#3$};}{}%
\ifthenelse{\northLen > 0}{\draw[-latex] (#1,#2)--+(0,0.7) node [above]{$#4$};}{}%
\ifthenelse{\westLen > 0}{\draw[-latex] (#1,#2)--+(-0.7,0) node [left]{$#5$};}{}%
\ifthenelse{\southLen > 0}{\draw[-latex] (#1,#2)--+(0,-0.7) node [below]{$#6$};}{}%
}

\newcommand{\refDiagram}[1]{\hyperref[#1]{Diagram~\ref*{#1}}}
\newcommand{\refNamedThm}[2]{\hyperref[#2]{#1~(\ref*{#2})}}
\newcommand{\refListInThm}[2]{\hyperref[#2]{\cref*{#1}(\ref*{#2})}}

\begin{document}

\onehalfspacing
\setlength{\unitlength}{1cm}

\begin{titlepage}
    \begin{center}
		
        \includegraphics[width=0.5\textwidth]{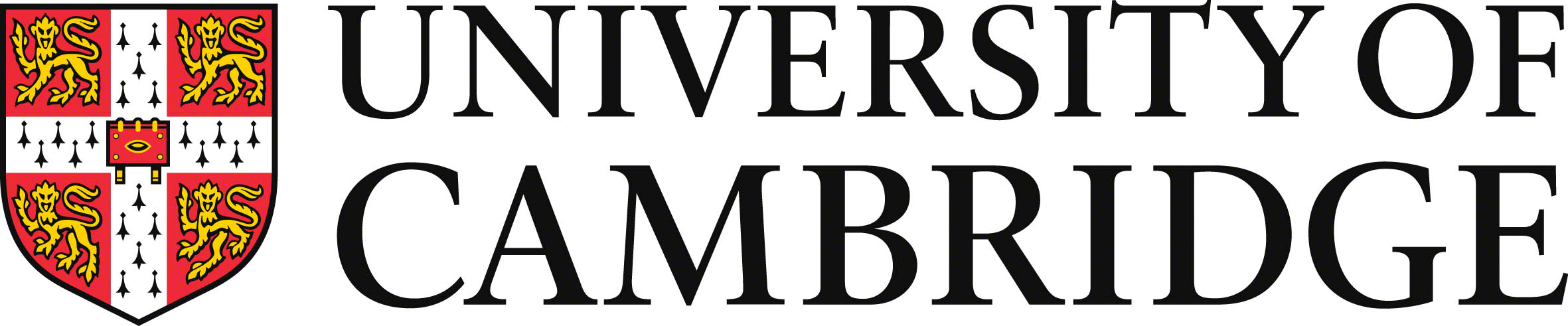}
        
				\vspace{2.5cm}
        
        \huge
        \textbf{Embedded contact knot homology \\ and a surgery formula}
        
        \vspace{2cm}
        
				\includegraphics[width=0.5\textwidth]{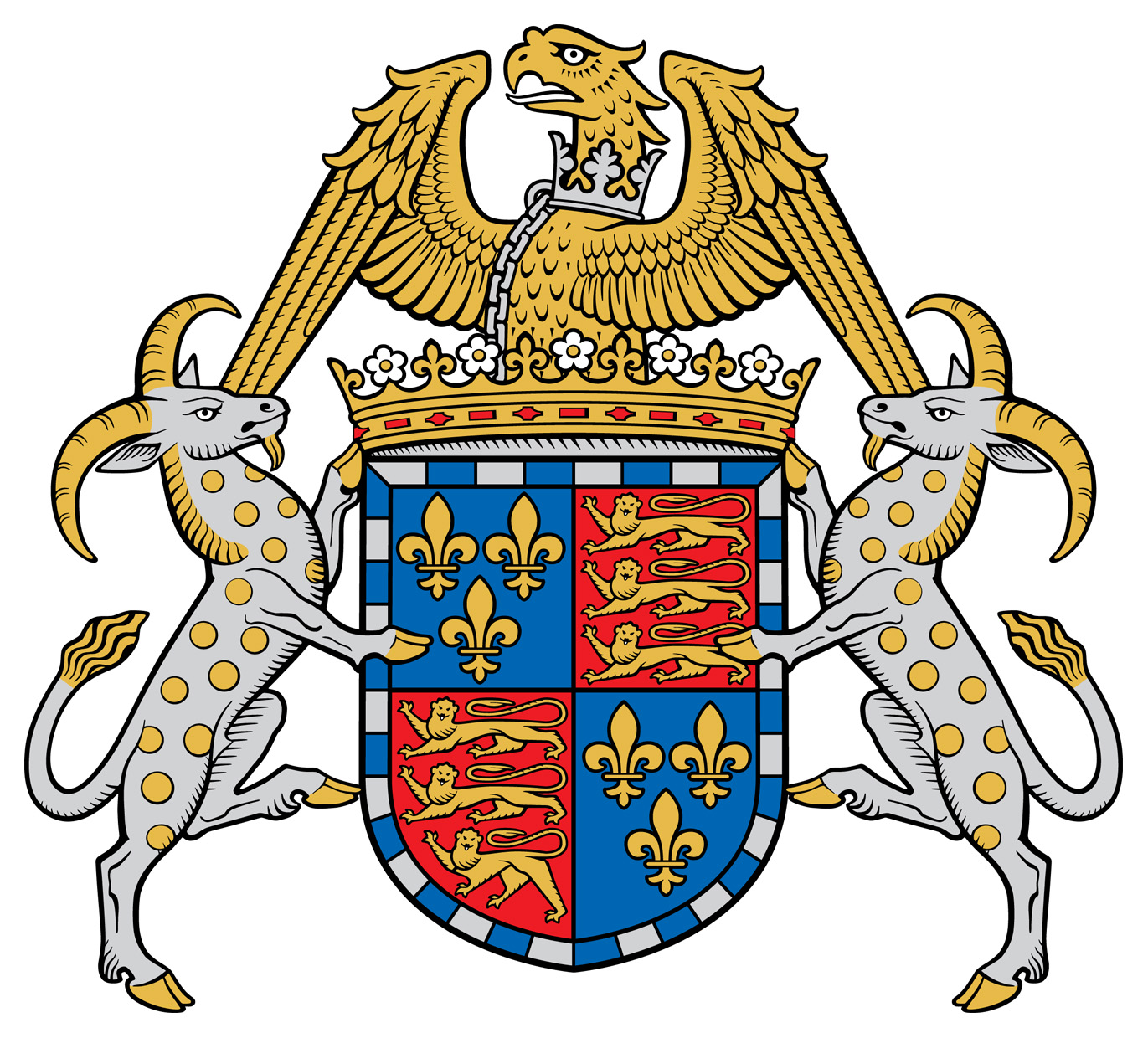}
        
				\vspace{1.6cm}
        
        \LARGE
        \textbf{Thomas Brown} \\
        
				\vspace{0.3cm}
				
        \large
        St John's College\\  
        University of Cambridge\\
        May 2018
        
        \vfill

        \large
        \textit{This dissertation is submitted for the degree of}\\
       Doctor of Philosophy

    \end{center}
\end{titlepage}
\cleardoublepage

\thispagestyle{empty}

\vspace*{\fill}

{
	\doublespacing

	\section*{\centering Declaration}

	\bigskip

	\noindent This dissertation is the result of my own work and includes nothing which is the outcome of work done in collaboration except as specified in the text.

	\bigskip

	\noindent It is not substantially the same as any that I have submitted, or, is being concurrently submitted for a degree or diploma or other qualification at the University of Cambridge or any other University or similar institution except as specified in the text. I further state that no substantial part of my dissertation has already been submitted, or, is being concurrently submitted for any such degree, diploma or other qualification at the University of Cambridge or any other University or similar institution except as specified in the text.

}

\vspace*{\fill}
\vspace*{\fill}
\vspace*{\fill}

\cleardoublepage

\thispagestyle{empty}

\vspace*{\fill}
\section*{\centering Abstract}
Embedded contact homology is an invariant of closed oriented contact 3-manifolds first defined by Hutchings, and is isomorphic to both Heegard Floer homology (by the work of Colin, Ghiggini and Honda) and Seiberg-Witten Floer cohomology (by the work of Taubes).  The embedded contact chain complex is defined by counting closed orbits of the Reeb vector field and certain pseudoholomorphic curves in the symplectization of the manifold. As part of their proof that ECH=HF, Colin, Ghiggini and Honda showed that if the contact form is suitably adapted to an open book decomposition of the manifold, then embedded contact homology can be computed by considering only orbits and differentials in the complement of the binding of the open book; this fact was then in turn used to define a knot version of embedded contact homology, denoted ECK, where the (null-homologous) knot in question is given by the binding.  

In this thesis we start by generalizing these results to the case of rational open book decompositions, allowing us to define ECK for rationally null-homologous knots.  In its most general form this is a bi-filtered chain complex whose homology yields ECH of the closed manifold.  There is also a hat version of ECK in this situation which is equipped with an Alexander grading equivalent to that in the Heegaard Floer setting, categorifies the Alexander polynomial, and is conjecturally isomorphic to the hat version of knot Floer homology. 

The main result of this thesis is a large negative $n$-surgery formula for ECK.  Namely, we start with an (integral) open book decomposition of a manifold with binding $K$ and compute, for all $n$ greater than or equal to twice the genus of $K$, ECK of the knot $K(-n)$ obtained by performing ($-n$)-surgery on $K$.  This formula agrees with Hedden's large $n$-surgery formula for HFK, providing supporting evidence towards the conjectured equivalence between the two theories. 

Along we the way, we also prove that ECK is, in many cases, independent of the choices made to define it, namely the almost complex structure on the symplectization and the homotopy type of the contact form.  We also prove that, in the case of integral open book decompositions, the hat version of ECK is supported in Alexander gradings less than or equal to twice the genus of the knot. 
\vspace*{\fill}
\vspace*{\fill}
\cleardoublepage

\tableofcontents
\cleardoublepage

\chapter*{Introduction}
\addcontentsline{toc}{chapter}{Introduction}
\markboth{INTRODUCTION}{INTRODUCTION}

\subsubsection{Embedded contact homology adapted to rational open book decompositions}
Let $M$ be a closed, oriented 3-manifold equipped with a contact form $\alpha$.  The embedded contact homology groups,
\[ ECH(M,\alpha)\quad\text{and}\quad \widehat{ECH}(M,\alpha),\]
arise from chain complexes which are generated by \emph{orbit sets}, collections of embedded closed orbits of the Reeb vector field of $\alpha$, and equipped with a differential which counts pseudoholomorphic curves in the symplectization $\R\cross M$ between orbit sets.

The corresponding chain complexes are defined in terms of both $\alpha$ and the chosen almost complex structure $J$ on the symplectization.  However it is shown~\cite{Taubes10} that the homology groups are independent of both these choices.  Furthermore, we have isomorphisms
\begin{equation} \label{ECH_HF_iso_intro}
ECH(M) \iso HF^+(-M)\quad\text{and}\quad \widehat{ECH}(M) \iso \widehat{HF}(-M)
\end{equation}
with Heegaard Floer homology~\cite{CGH_HF_ECH}, as well as isomorphisms
\begin{equation} 
ECH(M) \iso \widehat{HM}(M)\quad\text{and}\quad \widehat{ECH}(M) \iso \widetilde{HM}(M) 
\end{equation}
with Seiberg-Witten Floer cohomology~\cite{Taubes10}.   The isomorphisms between HF and HM have also been proven directly by Kutluhan, Lee and Taubes~\cite{KLT10}.

In this thesis we are concerned with a version of ECH adapted to knots $K\subset M$, embedded contact knot homology, which is analogous to knot Floer homology in the Heegaard Floer setting.  To define this we construct contact forms on $M$ which are in a sense ``adapted'' to the knot $K$.  In particular, we consider the case where $(M,K)$ admits a \emph{rational open book decomposition}, that is we can find a decomposition $M=\nu(K)\union N$ where $\nu(K)$ is a neighbourhood of the knot and 
\[ N = \frac{\Sigma\cross[0,1]}{(x,1)\sim(\phi(x),0)} \]
is a mapping torus over some oriented genus $g$ surface $\Sigma$ with one boundary component such that the restriction of the monodromy $\phi$ to $\del\Sigma$ is a fractional rotation.  It is required that, under the identification between $\del(\nu(K))$ and $\del N$, the meridian of the knot is identified with a (closed) integral curve of $\del_t$ (here $t$ is the $[0,1]$ coordinate).

Given this decomposition we construct contact forms $\alpha_N$ on $N$ which are \emph{extendable to the rational open book $M$} (c.f.~\cref{extendable_contact_form}), which in particular means that the boundary $\del N$ is a \emph{negative Morse-Bott torus} (c.f.\ \cref{morse-bott_def}) for $\alpha_N$ and the Reeb vector field on $N$ is transverse to all the pages $\Sigma\cross\set{t}$ of the open book.

We define (c.f.~\cref{relative_ECH_groups}) a collection of \emph{relative ECH groups},
\[ ECH(N,\del N, \alpha_N)\quad\text{and}\quad \widehat{ECH}(N,\del N, \alpha_N), \]
which involves some careful manipulation of the $S^1$-family of Reeb orbits foliating $\del N$---in particular we define the groups above by picking two distinguished orbits in $\del N$ which we denote by $e_-$ and $h_-$.  These two orbits are included when constructing the chain complexes above but we ignore all other orbits in $\del N$. 

The first main result of this thesis is \cref{relative_ECH_equals_ECH} which expresses the embedded contact homology groups of $M$ solely in terms of orbits and differentials in $N$, to obtain isomorphisms
\[ ECH(M) \iso ECH(N,\del N, \alpha_N)\quad\text{and}\quad \widehat{ECH}(M)\iso \widehat{ECH}(N,\del N, \alpha_N). \]
These isomorphisms are proved in the case of integral open book decompositions by Colin, Ghiggini and Honda~\cite{CGH_ECH_OBD}, and the generalization of their argument to the rational case is the subject of \cref{ECH_via_rational_open_book_decompositions_chapter} of this thesis.

\subsubsection{Embedded contact knot homology for rational knots}
The structure of the extension of $\alpha_N$ to $M$ gives rise to two additional orbits in $\nu(K)$, denoted by $e_+$ and $h_+$, and three pseudoholomorphic curves
\begin{equation}\label{three_differentials_intro}
h_+ \longrightarrow \emptyset,\quad h_+\longrightarrow e_-\quad\text{and}\quad e_+\longrightarrow h_-
\end{equation}
which contribute to the ECH differential. Here we are using the notation $\Gamma_+ \longrightarrow \Gamma_-$ to denote a pseudoholomorphic curve in $\R\cross M$ with positive end at $\Gamma_+$ and negative end at $\Gamma_-$ (c.f~\cref{moduli_spaces_of_holomorphic_curves}).  Let $\mathcal{P}$ denote the set of embedded orbits in $\mathrm{int}(N)$.  We introduce a complex denoted by
\begin{equation}\label{ECK_complex_intro_eqn}
ECC\orbits{0he}{eh}(\mathrm{int}(N),\alpha_N,J)
\end{equation}
(c.f.~\cref{ECK_complex_defn}), which is generated by orbit sets constructed from $\set{e_\pm, h_\pm}\union \mathcal{P}$, and has a differential obtained by counting $J$-pseudoholomorphic curves in $N$ as well as the three curves in \cref{three_differentials_intro} above.  The homology of the complex in \cref{ECK_complex_intro_eqn} is isomorphic to $ECH(M)$.

In analogy to knot Floer homology, this complex can be visualized graphically by plotting dots in the plane: an orbit set is represented by a dot at coordinates $(i,j)$ if it contains the orbit $e_+$ with multiplicity $i$ and its \emph{Alexander grading} is $j$.  Here, the Alexander grading is defined to be the number of times an orbit set winds around the mapping torus in the $[0,1]$ direction and is equivalent to the well-known Alexander grading from knot Floer homology (c.f.~\cref{ECK_hat_categorifies_Alex_poly_section}).  For an excellent introduction to knot Floer homology, refer to Manolescu~\cite{Mano14}.

No differentials arising from pseudoholomorphic curves in $\mathrm{int}(N)$ can affect the multiplicity of $e_+$ or the Alexander grading. From the description of the complex in \cref{ECK_complex_intro_eqn} we therefore see that the only differentials affecting these two gradings are
\begin{align*}
e_+&\longrightarrow h_-, \text{ which decreases the horizontal coordinate by 1, and} \\
h_+&\longrightarrow \emptyset, \text{ which decreases the vertical coordinate by 1.}
\end{align*}
Hence the complex is bi-filtered. We define \emph{embedded contact knot homology}, denoted $ECK(K,\alpha_N,J)$, to be the bi-filtered homotopy type of this complex.

Furthermore, in this setting $\widehat{ECH}(M)$ can be computed as the homology of the subcomplex lying in the column $i=0$, which is filtered by Alexander grading.  This motivates the definition of $\widehat{ECK}(M,\alpha_N,J)$, which is taken to be the homology of
\[ ECC\orbits{he}{h}(\mathrm{int}(N),\alpha_N,J), \]
the associated graded complex of this zeroth column.

$\widehat{ECK}$ was first defined in terms of sutured contact homology by Colin, Ghiggini, Honda and Hutchings~\cite{CGHH11}.  The full complex \cref{ECK_complex_intro_eqn} was first discussed by Spano~\cite{Spano17} but the above definition of ECK as the bi-filtered homotopy type of this full complex is new as far as the author is aware.

We also have a conjectured equivalence between embedded contact knot homology and knot Floer homology, namely that
\begin{equation} \label{ECK_HFK_hat_conj_intro}
\widehat{ECK}(K,\alpha_N,J) \iso \widehat{HFK}(-K) 
\end{equation}
and that the full complexes
\begin{equation} \label{ECK_HFK_conj_intro}
ECK(K,\alpha_N,J)\quad\text{and}\quad CFK^+(-K)
\end{equation}
have the same bi-filtered homotopy type.

\subsubsection{A large negative $n$-surgery formula}

The main result of this thesis is a large negative $n$-surgery formula for embedded contact knot homology.  More explicitly, start with an \emph{integral} open book decomposition, let $n>2g$ and identify a neighbourhood of $\del N$ with $\del N\cross[-\epsilon, \epsilon]$, where
\[ \del N\cross[-\epsilon, 0) \subset \nu(K)\quad\text{and}\quad \del N\cross(0,\epsilon] \subset \mathrm{int}(N). \]
We can choose a basis $(m,l)$ for $H_1(\del N)$ where $m$ is given by some Reeb orbit and $l$ is parallel to and oriented by the boundary of the page $\Sigma$.  Then since $\del N$ is negative Morse-Bott, according to this basis the Reeb slope is infinite at $\del N$, very large and positive on $\del N\cross(0,\epsilon]$, and very large and negative on $\del N\cross[-\epsilon,0)$.  The slope $-n$ is obtained by rotating the Reeb vector field in the negative direction until it is parallel to $nm-l$.

Let $M(-n)$ denote the manifold obtained by Dehn filling along the slope $-n$ and write $K(-n)$ for the knot arising as the core of the attached solid torus.  For $0\le j < n$, write 
\[  \widehat{ECK}(K(-n),\alpha';[j]) := \bigoplus_{j' \equiv j \ (n)} \widehat{ECK}(K(-n),\alpha';j'), \]
where here $\alpha'$ is a contact form on $M(-n)\sminus K(-n)$ adapted to $K(-n)$ and we are able to drop the choice of almost complex structure from the notation since $\widehat{ECK}$ is independent of $J$ (c.f.~\cref{ECC_independent_of_alpha}).  Before stating the surgery formula we must introduce the following notation.

\begin{notation}\label{complexes_A_j_and_B_i}
Consider the bi-filtered full ECK complex
\[ ECC\orbits{0eh}{he}(\mathrm{int}(N), \alpha_N,J). \]
We introduce the notation $A_j$ and $B_i$ as follows.
\begin{itemize}
\item The zeroth column ($i=0$) of the full complex, given by the subcomplex generated by orbit sets with ($e_+$)-multiplicity 0, is filtered by Alexander grading (or $j$ coordinate). Write $A_j$ for the subcomplex of this column with Alexander grading \emph{less than or equal to} $j$.  
\item The $2g$-th row ($j=2g$) of the full complex, generated by orbit sets with Alexander grading $2g$, is filtered by $e_+$-grading (or $i$ coordinate). Write $B_i$ for the subcomplex of this row with $e_+$-grading \emph{less than} $i$.
\end{itemize}
\end{notation}

\begin{thm}\label{surgery_formula}
Fix $n>2g$.  Then for $0\le j < n$, $\widehat{ECK}(K(-n),\alpha';[j])$ is given as follows:
\begin{itemize}
\item If $2g \le j < n$, $\widehat{ECK}(K(-n),\alpha';[j])$ lies entirely in Alexander grading $j$, and is isomorphic to the homology of $A_j$, which in turn is isomorphic to $\widehat{ECH}(M)$.
\item If $0 \le j < 2g$, $\widehat{ECK}(K(-n),\alpha';[j])$ lies entirely in Alexander gradings $j$ and $j+n$.  The part in Alexander grading $j$ is isomorphic to the homology of $A_j$, and the part in Alexander grading $j+n$ is isomorphic to the homology of $B_{2g-j}$.
\end{itemize}
\end{thm}

\begin{figure}\centering
  \begin{tikzpicture}  %
		\draw [-latex,thick] (-0.6,0) -- (-0.6,5.4) node [midway, left] {$A$};
		\draw [-latex,thick] (2,0.5) -- +(3.5,3.5);
		\node [rotate=45] at (3.75+0.2,2.25-0.2) {power of $e_+$};
		
		\begin{scope}
			\path[clip] (-0.5,-0.5) rectangle (6.5,6.5);
			\foreach \x in {0,...,6} {
				\foreach \y in {0,...,4}
					\fill (\x,\x+\y) circle (0.05);
				\draw [->] (\x+0.85,1+\x) -- +(-0.7,0);
				\draw [->] (\x+0.85,3+\x) -- +(-0.7,0);
				\draw [->] (\x,\x+1.85) -- +(0,-0.7);
				\draw [->] (\x,\x+3.85) -- +(0,-0.7);
				}
		\end{scope}

		\draw [rounded corners=8pt,gray](-0.3, 4-0.3) rectangle (2+0.3, 4+0.3);
		\draw [gray](2-0.3,2-0.3) rectangle (2+0.3,3+0.3);

		\foreach \x in {2,...,6}
			\draw [loosely dashed,thick] (\x+0.6,6.6) -- +(0.6,0.6);

	\end{tikzpicture}
  \caption{A diagram showing the bi-filtered complex $ECC\protect\orbits{0he}{eh}(\mathrm{int}(N), \alpha_N,J)$ in the case of the torus knot $T(2,5)$ (c.f.~\cref{ECK_for_torus_knots_section}) and the two subcomplexes $A_1$ (rectangle) and $B_{2g-1}$ (oval).  The complex $A_1$ is technically defined as a subcomplex of the zeroth column but we have shifted the rectangle here to highlight the well-known ``bent complex'' structure from the surgery formula on HFK.}
	\label{surgery_formula_figure}
\end{figure}
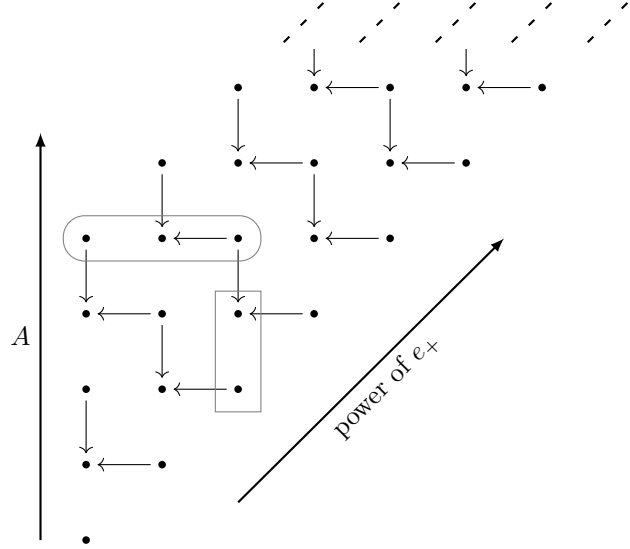

This formula agrees with the large (positive) $n$-surgery formula for knot Floer homology~\cite[Theorem 4.1]{Hedden}, supporting the conjectured equivalence between the two theories.  See \cref{surgery_formula_figure}.  Notice that chance of sign in the surgery coefficient is consistent with the changes of sign in \cref{ECH_HF_iso_intro} and the conjectured equivalences \cref{ECK_HFK_hat_conj_intro} and \cref{ECK_HFK_conj_intro} above.

\subsubsection{Other results}

Along the way, we also establish some results concerning ECK.  Three of the most important are that
\begin{enumerate}
\item The $e_+$-filtered homotopy type of the graded complex associated to the Alexander filtration on the full complex \cref{ECK_complex_intro_eqn} is invariant under the choice of almost complex structure $J$ and isotopies of $\alpha$ and, in the case of an integral open book decomposition, the bi-filtered homotopy type of \cref{ECK_complex_intro_eqn} is also invariant under such choices (\cref{ECC_independent_of_alpha,ECK_invariant_ZOBD});
\item $\widehat{ECK}$ categorifies the Alexander polynomial for rationally null-homologous knots in arbitrary manifolds (in the sense that its graded Euler characteristic is equal to
\[(1-[\mu])\tau(M\sminus K),\] where $\tau$ is the Turaev torsion and $\mu$ is the homology class of the meridian of the knot (\cref{ECK_hat_categorifies_Alex_poly})), and hence $\widehat{ECK}$ and $\widehat{HFK}$ are isomorphic at the level of Euler characteristic; and
\item in the case of an integral open book decomposition, $\widehat{ECK}$ is supported in Alexander gradings less than or equal to $2g$ (\cref{ECK_supported_genus}).
\end{enumerate}
Again, these three facts support the conjectures above.

\subsubsection{Acknowledgements}

First and foremost, I express enormous gratitude to my supervisor, Jake Rasmussen, for all the support he has given me throughout my PhD.  The time and attention he dedicates to all his students is greatly appreciated, and I certainly would not be where I am today without his generous and continual help and advice.  It must also be noted that it was Jake's captivating Master's course on 3-manifolds which first inspired me to pursue a PhD four and a half years ago.  Many thanks also go towards my two examiners, Vincent Colin and Ailsa Keating, for their insightful comments and helpful feedback.

This PhD could not have taken place without the generous funding from the St John's College Benefactors Scholarship, which also kindly supported me during my Master's studies.  My PhD was also supported by contributions from DPMMS and St John's College, for which I am very grateful.  I also wish to acknowledge Sue Colwell, the Tutor for Graduate Affairs at St John's, whose commitment to the graduate community of the college is second to none.

I am very grateful for several mathematical discussions which have taken place at conferences and workshops around the world; in particular my conversations with Paolo Ghiggini, Gilberto Spano and Vincent Colin, which sparked much of my interest towards the area of embedded contact knot homology.

To my PhD companions, Nate Davis, Christian Lund, Nina Friedrich and Claudius Zibrowius, thank you for your friendship and company over the years---I wish you all the best for the remainder of your academic and/or mathematical careers.  I also wish to thank Jeremy Judge: for our many scintillating discussions on the philosophy of learning, for allowing me to impart mathematical enthusiasm onto your students (and in turn acquire much from them), and for introducing me to the book \emph{Bounce}.  Good luck in Chicago.

Special thanks must go towards my partner---and resident \LaTeX\ guru---Chi-H\'e Elder, who has always been there for me over the last three years, providing unwavering support especially during the more difficult moments of my PhD.  Thank you for everything.  Many thanks also go to my two dance partners, Liane Dupont and Hanna Wikstrom, for doing a wonderful job of distracting me from the pains of PhD life when at practice, and of course for putting up with an overly analytic mathematician for a dance partner.

Finally, I wish to thank my family, who have of course been there for me for much longer than everyone else combined.  Most of all, to my mother, whose love and support is unparalleled and without whom I would not be who I am today.

\chapter{Embedded contact homology}\label{ECH_chapter}
\section{Definitions}
Embedded contact homology is a dynamical 3-manifold invariant originally defined by Hutchings~\cite{Hu14}, defined in terms of closed orbits of the Reeb vector field associated to the contact structure.  In this section we will give an exposition of embedded contact homology, following Hutchings~\cite{Hu14}.

\subsection{Contact geometry and Reeb dynamics}

\begin{defn}
Let $M$ be a closed oriented 3-manifold. A \emph{contact form} on $M$ is a 1 form $\alpha$ such that $\alpha \wedge d\alpha > 0$.
\end{defn}

Associated to a contact form we obtain a \emph{contact structure}, $\xi$, which is the 2-plane field defined by $\xi := ker(\alpha)$.  A contact form also gives rise to a \emph{Reeb vector field}, $R_\alpha$, the unique vector field such that $\iota_{R_\alpha}(d\alpha) = 0$ and $\alpha(R_\alpha)=1$.  Here $\iota_{R_\alpha}(d\alpha)$ denotes the contraction of $d\alpha$ with $R_\alpha$.  The Reeb vector field is always transverse to the contact structure $\xi$.  When the contact form is clear from the context, we will drop it from the notation and refer to the Reeb vector field by $R$. 

To define embedded contact homology, we consider the dynamics of the Reeb vector field.  A \emph{Reeb trajectory} is path $\gamma$ in $M$ such that
\[ \gamma'(t) = R(\gamma(t)) \text{ for all }t,\]
and a \emph{Reeb orbit} is a closed Reeb trajectory. The \emph{action} of a Reeb orbit is defined by
\[ \mathcal{A}(\gamma) := \int_\gamma \alpha. \]
A Reeb orbit is called \emph{simple} if it is injective.  Suppose that a Reeb orbit is paramaterized by
\[\gamma:[0,T]\to M,\]
so that $\gamma(0)=\gamma(T)$.  For $m>0$, the \emph{finite cover} $\gamma^m$ is given by the Reeb orbit
\[ [0,mT] \to M \]
defined at $t$ by evaluating $\gamma$ at $t \text{ mod } T$.  Every Reeb orbit is a finite cover of a simple Reeb orbit.

Suppose that $\gamma:[0,T]\to M$ is an injective Reeb orbit and that $\varphi_T:M\to M$ is the Reeb flow at time $T$.  The \emph{linearized return map} associated to $\gamma$ is the map 
\[f_\gamma:\xi_{\gamma(0)} \to \xi_{\gamma(T)}=\xi_{\gamma(0)}\] obtained by restricting the differential $d\varphi_T$ at $\gamma(0)$ to $\xi_{\gamma(0)}$.

\begin{defn}~
\begin{enumerate}
\item A Reeb orbit $\gamma$ is \emph{non-degenerate} if $f_\gamma$, and every power thereof, does not have 1 as an eigenvalue.
\item A contact form $\alpha$ is \emph{non-degenerate} if every Reeb orbit is non-degenerate.
\end{enumerate}
\end{defn}

Non-degeneracy is a generic condition, so any contact form can be made non-degenerate by performing a small $C^\infty$ perturbation~\cite[Lemma 2]{Bo03}.

Since the Lie derivative of $d\alpha$ with respect to $R$ is zero, $f_\gamma$ is a symplectic transformation of $(\xi_{\gamma(0)}, d\alpha)$.  If it does not have 1 as an eigenvalue then by the classification of elements in $Sp(2,\R)=SL_2(\R)$, it is exactly one of either
\begin{itemize}
\item \emph{hyperbolic}, if its eigenvalues are real, equal to $\set{\lambda, \lambda\inv}$, and the return map looks locally like a squeeze mapping, or
\item \emph{elliptic}, if its eigenvalues lie on the unit circle in $\C$, and the return map looks locally like a rotation.
\end{itemize}

\begin{defn}
A non-degenerate Reeb orbit $\gamma$ is called \emph{hyperbolic} (resp.\ \emph{elliptic}) if its linearized first return map $f_\gamma$ is hyperbolic (resp.\ elliptic).

A hyperbolic orbit is called \emph{positive} (resp.~\emph{negative}) if the eigenvalues of $f_\gamma$ are positive (resp.~negative).
\end{defn}

We are now ready to start defining the chain complex $ECC(M,\alpha)$ over the field $\F_2=\Z/2\Z$.  Assume that $\alpha$ is non-degenerate.  The generators of $ECC(M,\alpha)$ are orbit sets, defined below.

\begin{defn}
An \emph{orbit set} is a finite set $\Gamma = \set{(\gamma_i, m_i)}$ such that the following hold:
\begin{itemize}
\item each $\gamma_i$ is a simple Reeb orbit and $i\ne j \implies \gamma_i \ne \gamma_j$,
\item each $m_i$ is a positive integer, and
\item if $\gamma_i$ is hyperbolic then $m_i=1$.
\end{itemize}
The \emph{underlying set} of an orbit set is
\[ |\Gamma| := \bigcup_i \mathrm{Im}(\gamma_i). \]  
\end{defn}

The empty orbit set, $\emptyset$, is permitted as a generator of the complex.  Orbit sets are written multiplicatively as $\prod{\gamma_i^{m_i}}$, with $\emptyset$ written as 1.  This notation must not be confused with the notation for finite covers of orbits; if $\gamma$ is a simple orbit then the notation $\gamma^m$ can be interpreted either as
\begin{itemize}
\item an \emph{orbit}, the $m$-fold finite cover of $\gamma$, or
\item the \emph{orbit set} $\set{(\gamma,m)}$.
\end{itemize}
In the rare instances that these notations appear in close proximity it will be made clear which we are referring to---most of the time it will be clear from the context.  In later proofs, we will use multiplicative notation in the following sense: if $\Gamma=\prod{\gamma_i^{m_i}}$ and $\Gamma'=\prod{{\gamma_i}^{m'_i}}$ (note that here we are allowing $m_i$ and $m_i'$ to be zero for notational simplicity) then
\begin{align*}
\Gamma\cdot\Gamma' &= \prod{\gamma_i^{m_i+m_i'}}\text{ and} \\
\Gamma/\Gamma' &= \prod{\gamma_i^{m_i-m_i'}}.
\end{align*}

\begin{convention}
We will encounter orbit sets written in the multiplicative notation with $m_i<0$ or with $m_i>1$ for $\gamma_i$ hyperbolic---any orbit set where this is the case is equal to 0.
\end{convention}

The \emph{homology class} of an orbit set is defined to be 
\[ [\Gamma] := \sum{m_i[\gamma_i]}, \]
where $[\gamma_i]$ is the homology class of $\gamma_i$ (note that $[\Gamma]\ne[|\Gamma|]$).  This will usually be taken as an element of $H_1(M,\Z)$ but may occasionally be taken as an element of $H_1(|\Gamma|,\Z)$.  The \emph{action}, or \emph{length}, of an orbit set is
\[ \mathcal{A}(\Gamma) := \sum{m_i\mathcal{A}(\gamma_i)}.\]

\subsection{Moduli spaces of holomorphic curves}\label{moduli_spaces_of_holomorphic_curves}

The ECH differential arises by counting pseudoholomorphic curves in the symplectization of the contact manifold $M$. The \emph{symplectization} of a contact manifold $(M,\alpha)$ is the manifold $\R\cross M$ with symplectic form $d(e^s\alpha)$, where $s$ denotes the $\R$ coordinate.  We then consider almost complex structures $J$ which are \emph{adapted to $\alpha$}, in the sense that
\begin{itemize}
\item $J$ is $s$-invariant,
\item $J(\del_s) = R$, and
\item on each $\set{s}\cross M$, $J$ maps $\xi$ to itself and is compatible with $d\alpha$ in the sense that $d\alpha(\cdot, J\cdot)$ is a Euclidean metric on $\xi$. 
\end{itemize}
For the purposes of this thesis, pseudoholomorphic maps between almost complex manifolds will often be referred to simply as holomorphic maps.

We can now set up the idea of what it means for there to exist a differential between two orbit sets $\Gamma_+$ and $\Gamma_-$.  Let $(\hat{F},j)$ be a closed and not necessarily connected Riemann surface and consider (pseudo)holomorphic curves
\[ u:(F,j) \to (\R\cross M, J) \]
where $F$ is obtained by removing a finite number of puncture points from $\hat{F}$. A curve $u$ is \emph{pseudoholomorphic} if it satisfies the generalized Cauchy-Riemann equation
\[ du\circ j = J\circ du. \]
A puncture point $x\in \hat{F}$ is said to be \emph{asymptotic to a Reeb orbit $\gamma\in M$} if there exists a small neighbourhood $U\subset \hat{F}$ of $x$ such that $u(U)$ is asymptotic to $\R\cross \gamma \subset \R\cross M$.  Such a puncture point is also called an \emph{end} of the holomorphic curve $u$ and is said to be a \emph{positive end} if $u(U)$ is asymptotic to $[0,+\infty)\cross \gamma$ and a \emph{negative end} if it is asymptotic to $(-\infty,0]\cross \gamma$.  In addition, if a puncture point is asymptotic to a Reeb orbit $\gamma$ and $\gamma$ is an $m$-fold cover of a simple Reeb orbit $\gamma'$, then we say that the puncture point is asymptotic to $\gamma'$ with multiplicity $m$.

We say that the positive (resp.~negative) end of a holomorphic curve $u$ has \emph{multiplicity $m$ (with partition $(m_1, \dots, m_k)$)} at a simple Reeb orbit $\gamma'$ if $m = \sum{m_j}$ and $u$ has $k$ positive (resp.~negative) ends at $\gamma'$ with multiplicities $m_1, \dots, m_k$.  Here $(m_1, \dots, m_k)$ is an unordered tuple.  Finally, we say that the \emph{positive (resp.~negative) end of $u$} is equal to an orbit set $\Gamma = \set{(\gamma_i, m_i)}$ if its positive (resp.~negative) end has multiplicity $m_i$ at the simple orbit $\gamma_i$ and it has no positive (resp.~negative) ends at any other simple orbits.

Notice the subtle difference between the three definitions concerning ``ends'' above:
\begin{itemize}
\item an end of $u$ at $\gamma$ with multiplicity $m$, which just considers a single puncture point;
\item the positive (resp.~negative) end of $u$ at $\gamma$, with multiplicity $m$, which considers all positive (resp.~negative) puncture points ending at $\gamma$ and gives rise to a partition of $m$; and
\item the positive (resp.~negative) end of $u$, which is equal to some orbit set and considers all positive (resp.~negative) ends of $u$.
\end{itemize}
We say that $u$ is a holomorphic curve \emph{from $\Gamma_+$ to $\Gamma_-$}, sometimes written $\Gamma_+\longrightarrow\Gamma_-$, if the positive end of $u$ is equal to $\Gamma_+$ and the negative end is equal to $\Gamma_-$.  We will also refer to the orbits making up $\Gamma_+$ as \emph{positive orbits} and to the orbits making up $\Gamma_-$ as \emph{negative orbits}.

\begin{defn}
Let $\Gamma_+$ and $\Gamma_-$ be orbit sets.  $\mathcal{M}_J(\Gamma_+, \Gamma_-)$ is defined to be the moduli space of holomorphic curves
\[ u:(F,j) \to (\R\cross M, J) \]
from $\Gamma_+$ to $\Gamma_-$.  Furthermore, if $*$ is some predicate, then $\mathcal{M}^*_J(\Gamma_+, \Gamma_-)$ denotes the moduli space of those curves satisfying $*$.
\end{defn}

Notice that a holomorphic curve from $\Gamma_+$ to $\Gamma_-$ acts as a cobordism between the two orbit sets so $\mathcal{M}_J(\Gamma_+, \Gamma_-)$ can only be non-empty if $[\Gamma_+]=[\Gamma_-]$.

We say that the almost complex structure $J$ is \emph{regular} if near any connected holomorphic curve $u$, from $\Gamma_+$ to $\Gamma_-$, the moduli space $\mathcal{M}_J(\Gamma_+, \Gamma_-)$ is transversely cut out and hence a finite dimensional manifold.  Such $J$ exist by Dragnev~\cite{Dr04}.  We will always assume that $J$ is chosen to be regular.  The dimension of the moduli space near $u$ is given by the Fredhom index, $\ind(u)$, however the definition of ECH involves a different index, called the ECH index, defined in the following section.

\subsection{The ECH Index}\label{The_ECH_index}

The ECH differential is defined by counting certain holomorphic curves with ECH index equal to 1.  In this section we will give a detailed exposition of this index, the definition of which involves three parts:
\begin{enumerate}
\item the \emph{Conley-Zehnder indices} of the orbits at either end of the holomorphic curve $u$,
\item the \emph{relative first Chern class} of the curve, and
\item the \emph{self-intersection number} of the curve.
\end{enumerate}
All three of these are defined with respect to trivializations of $M$ near the Reeb orbits involved, however the final index is independent of this trivialization.  Furthermore, the ECH index is defined only in terms of the (relative) homology class of the closure of the curve, defined below.

\begin{defn}~
\begin{itemize}
\item $u_M := \pi_M\circ u$, where $\pi_M$ denotes projection from $\R\cross M$ to $M$.
\item $u_\R := \pi_\R\circ u$, where $\pi_\R$ denotes projection from $\R\cross M$ to $\R$.
\end{itemize}
\end{defn}

\begin{defn}
Let $H_2(M, \Gamma_+, \Gamma_-)$ be the set of relative homology classes in $H_2(M,|\Gamma_+|\union|\Gamma_-|)$ represented by chains $Z$ with
\[ [\del Z] = [\Gamma_+] - [\Gamma_-] \in H_1(|\Gamma_+|\union|\Gamma_-|). \]
Define the \emph{homology class} of a holomorphic curve $u$ by
\[ [u] := [\overline{u_M}] \in H_2(M, \Gamma_+, \Gamma_-). \]
\end{defn}

When discussing $H_2(M, \Gamma_+, \Gamma_-)$ we will again, in analogy to the holomorphic set-up, refer to the curves $\gamma_i^+$ making up the orbit set $\Gamma_+$ as \emph{positive orbits} and the curves $\gamma_i^-$ making up $\Gamma_-$ as \emph{negative orbits}.

\begin{defn}
Let $\gamma$ be a simple orbit in $M$.  A \emph{trivialization} of ($M$ near) $\gamma$ is a homeomorphism
\[ \tau : (V, \gamma) \xrightarrow{\iso} (S^1 \cross D^2, S^1 \cross \set{0}) \]
where $V\subset M$ is some neighbourhood of the orbit $\gamma$.

Trivializations are considered up to isotopy and hence are determined uniquely by a \emph{framing} of the orbit.  Since the orbit is transverse to the contact structure $\xi$, this framing can be taken as a non-zero section of the bundle $\restr{\xi}{\gamma}$ over $\gamma$.  We will denote the framing associated to a trivialization $\tau$ by $\sigma_\tau$.

If $\Gamma_+ = \set{(\gamma^+_i, m^+_i)}$ and $\Gamma_- = \set{(\gamma^-_i, m^-_i)}$ are orbit sets in $M$, then a trivialization of the pair $(\Gamma_+, \Gamma_-)$ consists of a trivialization for each $\gamma^+_i$ and each $\gamma^-_i$.  Such a trivialization will be also be denoted by $\tau$, and in this case $\sigma_\tau$ will denote the associated framing as a section of the bundle $\restr{\xi}{|\Gamma_+|\union|\Gamma_-|}$ over $|\Gamma_+|\union|\Gamma_-|$.
\end{defn}

Note that there is a free transitive $\Z$-action on the space of trivializations of a simple orbit.  We will use the sign convention adopted by Hutchings in 2009~\cite{Hu09} and let the element $1\in\Z$ act on a trivialization $\tau$ by adding a single positive twist to the framing associated with $\tau$.  This is the opposite sign convention to Hutchings' earlier paper~\cite{Hu02}.  We will write the action as $\tau + k$ for $k\in\Z$.

\subsubsection{The Conley-Zehnder Index}

Recall that when defining the return map $f_\gamma$ associated to a Reeb orbit $\gamma:[0,T]\to M$ we linearized the time $T$ flow of $R$ near $\gamma$ to obtain an element of $Sp(2,\R)$.  The Conley-Zehnder index $\mu_\tau(\gamma)$ is defined by linearizing the time $t$ flow for all $t\in[0,T]$, which gives us a path in $Sp(2, \R)$, starting at $\id_{Sp(2,\R)}$ and ending at $f_\gamma$.  We will not give the full definition here, instead defining $\mu_\tau$ only for the cases required for our purposes.

\begin{defn}
Let $\gamma$ be a simple Reeb orbit in $M$ with trivialization $\tau$, and let $\gamma^m$ be the $m$-fold cover of $\gamma$.  
\begin{itemize}
\item If $\gamma$ is hyperbolic, then with respect to the trivialization $\tau$, the linearized flow from time $t=0$ to $t=T$ will not only apply a squeeze factor, but also rotate by an angle of $n\pi$, where $n\in\Z$.  In this case we define $\mu_\tau(\gamma):=n$.  Furthermore, $\mu_\tau(\gamma^m) = mn$.  If $\gamma$ is positive (resp.~negative) hyperbolic then $n$ is even (resp.~odd).  
\item If $\gamma$ is elliptic, the linearized return map is rotation by the angle $2\pi\theta$ where $\theta\in\R/\Z$.  Armed with a trivialization $\tau$ we can use the linearized flow of $R$ from $t=0$ to $t=T$ to lift $\theta$ to a value in $\R$; in this case we define $\mu_\tau(\gamma):=2\lfloor\theta\rfloor + 1$.  Furthermore, $\mu_\tau(\gamma^m) = 2\lfloor m\theta \rfloor + 1$.  Note that since all powers of $\gamma$ are non-degenerate, $\theta$ is irrational.  
\end{itemize}
\end{defn}

For both hyperbolic and elliptic orbits, the sign convention is such that if the flow of $R$ near $\gamma$ winds around the Reeb orbit more positively (resp.~negatively) than $\sigma_\tau$ then $\mu_\tau(\gamma)$ is positive (resp.~negative).  Recall that $\gamma$ is oriented by $R$ and $\xi$ is oriented by $d\alpha$ so the notion of positive/negative winding is well-defined.

If we change the trivialization $\tau$ of $\gamma$, say by letting $\tau' = \tau+k$, then
\[ \mu_{\tau'}(\gamma) = \mu_{\tau'}(\gamma) - 2k. \]
This is because if $\sigma_{\tau'}$ winds around $\gamma$ $k$ times more than $\sigma_\tau$, then the flow of $R$ will wind around $k$ times less with respect to $\sigma_{\tau'}$, which decreases the rotation angle by $2\pi k$.  By the same reasoning,
\[ \mu_{\tau'}(\gamma^m) = \mu_{\tau'}(\gamma^m) - 2km. \]

\begin{defn}[{\nogapcite[Section 2.3]{Hu02}}]
If $\Gamma=\set{(\gamma_i, m_i)}$ is an orbit set with trivialization $\tau$, then we define the Conley-Zehnder index of $\Gamma$ by
\[ \mu_\tau := \sum_i{ \sum_{k=1}^{m_i}{ \gamma_i^k } }. \]
In other words for each pair $(\gamma_i, m_i)$ we sum the Conley-Zehnder indices of all finite covers of the simple orbit $\gamma_i$ up to and including the $m_i$-fold cover.
\end{defn}

\subsubsection{The relative first Chern class}

Let $Z$ be an embedded surface representing some homology class $[Z] \in H_2(M, \Gamma_+,\Gamma_-)$.

\begin{defn}[{\nogapcite[Section 2.2]{Hu02}}]
The \emph{relative first Chern class}, $c_1(\restr{\xi}{Z},\tau)$, is defined to be the algebraic count of zeros of a generic section $\sigma$ of $\restr{\xi}{Z}$, with the boundary condition that
\[ \restr{\sigma}{|\Gamma_+|\union|\Gamma_-|} = \sigma_\tau. \]
The sign at each zero is defined with respect to the orientation of $Z$ and the orientation on $\restr{\xi}{Z}$ coming from $d\alpha$.
\end{defn}

Suppose, as usual, that $\Gamma_+ = \set{(\gamma^+_i, m^+_i)}$ and  $\Gamma_- = \set{(\gamma^-_i, m^-_i)}$.  If we change the trivialization at some $\gamma^+_i$, say from $\tau_i$ to $\tau'_i = \tau_i + k$, then we introduce $km_i$ new zeros to our generic section near $\gamma^+_i$.  Taking into account the orientation of $Z$ near $\gamma^+_i$, it is easy to see that these zeros contribute positively (resp.~negatively) if $k$ is positive (resp.~negative).  If we perform the same procedure at an orbit $\gamma^-_i$ instead, then the orientation of $Z$ near the boundary is reversed and the introduced zeros contribute negatively (resp.~positively) if $k$ is positive (resp.~negative).

\subsubsection{The self-intersection number}

In order to define a notion of self-intersection on $H_2(M, \Gamma_+,\Gamma_-)$, we must first represent an element $[Z]\in H_2(M, \Gamma_+,\Gamma_-)$ by a surface immersed in a 4-dimensional space, so that it makes sense to count intersection points.

\begin{defn}~
\begin{enumerate}
\item An \emph{admissible representative} of $[Z]\in H_2(M, \Gamma_+,\Gamma_-)$ is an immersed surface
\[ S: (\Sigma,\del \Sigma) \to \big([0,1]\cross M, \set{1}\cross |\Gamma_+|\union \set{0}\cross |\Gamma_-|\big) \]
such that
\begin{enumerate}
\item $\Sigma$ is an oriented compact surface with boundary,
\item $S$ is an embedding on $\mathrm{int}(\Sigma)$ and is transverse to $\set{0,1}\cross M$ at the boundary, and
\item the projection of $\mathrm{Im}(S)$ onto $M$ is a representative of $[Z]$.
\end{enumerate}
Each component of $\del\Sigma$ is called an \emph{end} of $S$, and the end is called \emph{positive} (resp.~\emph{negative}) if it maps into $\set{1}\cross |\Gamma_+|$ (resp.~$\set{0}\cross |\Gamma_-|$).
\item Let
\[ S_s := \mathrm{Im}(S)\cap\big(\set{s}\cross M \big), \]
which is a 1-manifold for generic $s\in[0,1]$.  For $\epsilon>0$ small enough, $S_{1-\epsilon}$ (resp.~$S_\epsilon$) is a collection of curves in a neighbourhood of $\set{1}\cross|\Gamma_+|$ (resp.~$\set{0}\cross|\Gamma_-|$), which forms a closed braid $B^+_{1-\epsilon,i}$ with $m^+_i$ strands around each $\gamma^+_i$ (resp.~$B^-_{\epsilon,i}$ with $m^-_i$ strands around each $\gamma_i^-$).
Define a \emph{$\tau$-admissible representative} of $[Z]$ to have the further properties that for $\epsilon>0$ small enough, the braids $B^+_{1-\epsilon,i}$ around $\gamma^+_i$ consist of $m^+_i$ (resp.~$B^-_{\epsilon,i}$ around $\gamma_i^-$ consist of $m_i^-$) $\tau$-trivial longitudes.
\end{enumerate}
\end{defn}

Admissible representatives can be thought of as compact, embedded analogues of holomorphic curves in $\R\cross M$; punctured surfaces have been replaced by compact surfaces with boundary  and non-compact asymptotic ends by constraints on the images of the boundary components.  Every embedded holomorphic curve $u$ gives rise to an admissible representative of its homology class by compactification, however in the general case we first perturb our surface to make any intersection points transverse, before removing them by replacing a small neighbourhood of each point (which takes the form of the cone on the Hopf link) with an annulus (the Seifert surface of the Hopf link).  In addition, obtaining a $\tau$-admissible representative from $u$ requires further subtle manipulations; any ends of multiplicity $m>1$ must be replaced with a homologous surface which instead has $m$ ($\tau$-trivial) ends of multiplicity 1.

\begin{defn}[{\nogapcite[Section 2.4]{Hu02}}]
Suppose $[Z]\in H_2(M,\Gamma_+,\Gamma_-)$.  Pick two $\tau$-admissible representatives of $[Z]$, $S$ and $S'$, which intersect transversely on their interiors, and define the self-intersection number
\[ Q_\tau([Z],[Z]) := S\cdot S'. \]
\end{defn}

Suppose we change the trivialization by replacing $\tau_i$ with $\tau_i'=\tau_i+k$ at some positive orbit $\gamma^+_i$ of multiplicity $m^+_i$.  As $1-\epsilon$ approaches 1, the $m^+_i$ longitudes of $\gamma^+_i$ in $\mathrm{Im}(S)\cap\set{1-\epsilon}\cross M$ are required to form $k$ additional full positive twists around $\gamma^+_i$.  Forming these additional twists as $1-\epsilon$ increases introduces $k(m^+_i)^2$ new intersections of positive sign.  If instead we change the trivialization at a \emph{negative} orbit $\gamma_i^-$ of multiplicity $m_i^-$, then as $\epsilon$ decreases towards 0 we introduce $k(m_i^-)^2$ new intersections of negative sign.

We are now ready to define the ECH index.

\subsubsection{Definition of the ECH Index}

\begin{defn}
Suppose $[Z]\in H_2(M,\Gamma_+,\Gamma_-)$, and $\tau$ is a trivialization of $(\Gamma_+,\Gamma_-)$. The ECH index of $[Z]$ is defined by
\begin{equation}\label{ECH_index_formula}
I([Z]) := c_1(\restr{\xi}{Z},\tau) + Q_\tau([Z],[Z]) + \mu_\tau(\Gamma_+) - \mu_\tau(\Gamma_-).
\end{equation}
Furthermore, if $u$ is a holomorphic curve, we define
\[ I(u) := I([u]). \]
\end{defn}

It easy to see from the above discussion that the definition of the ECH index is independent of the choice of trivialization $\tau$; replacing $\tau_i$ with $\tau_i' = \tau_i +k$ at a positive orbit $\gamma^+_i$ of multiplicity $m^+_i$ increases $c_1([Z],\tau)$ by $km^+_i$, increases $Q_\tau([Z],[Z])$ by $k(m^+_i)^2$, and decreases $\mu_\tau(\Gamma)$ by $\sum_{m=1}^{m^+_i}{2km}$.  At a negative orbit, the contributions cancel similarly.

The ECH index satisfies an \emph{index parity formula}, which gives rise to an absolute $\Z/2$-homological grading which will prove useful later when computing Euler characteristics in \cref{ECK_hat_categorifies_Alex_poly_section}.  The formula involves the well-known \emph{Lefschetz sign}, which is defined for each Reeb orbit in terms of the linearized return map, $f_\gamma$, by
\begin{equation}
\epsilon(\gamma) := \mathrm{sign}(\mathrm{det}(\id-f_\gamma))\in\set{1,-1}.
\end{equation}
It is straightforward to verify that
\begin{equation}\label{lemma_lefschetz_sign}
\epsilon(\gamma) = \begin{cases*}
1&\text{if $\gamma$ is elliptic or negative hyperbolic, and}\\
-1&\text{if $\gamma$ is positive hyperbolic.}
\end{cases*}
\end{equation}
Furthermore, by examining the definition of the ECH index, one obtains the following proposition.
\begin{prop}[{Index parity formula~\cite[Proposition 1.6]{Hu02}}]\label{index_parity_formula}
If $[Z]\in H_2(M,\Gamma_+,\Gamma_-)$, then
\[ (-1)^{I([Z])} = \epsilon(\Gamma_+)\epsilon(\Gamma_-), \]
where $\epsilon$ is defined on orbit sets multiplicatively, by the formula
\[ \epsilon\left( \prod{\gamma_i^{m_i}} \right) := \prod{\epsilon(\gamma_i)^{m_i}}. \]
\end{prop}
As a result, we can define the following $\Z/2$-grading on $ECC(M,\alpha)$:
\begin{defn}
If $\Gamma$ is an orbit set, then define
\[ I(\Gamma) := \begin{cases*}
0 & \text{if $\epsilon(\Gamma)=1$, and}\\
1 & \text{if $\epsilon(\Gamma)=-1$}
\end{cases*}\quad\in \Z/2.\]
\end{defn}
As mentioned at the start of this section, the ECH differential is defined by counting certain holomorphic curves with ECH index 1. For this reason we will call $I$ the \emph{absolute $\Z/2$-homological grading} on the ECC complex.  The result of this discussion is that, when computing Euler characteristics later in \cref{ECK_hat_categorifies_Alex_poly_section}, the contribution of an orbit set to the Euler characteristic is given precisely by the Lefschetz sign. 

We will now set up some terminology which will enable us to state the final result of this section, namely a relation between the Fredholm index, ind, and the ECH index.

\pagebreak %
\begin{defn}
Let $u:F\to \R\cross M$ and $u':F'\to \R\cross M$ be two holomorphic curves.  We write $u\union u'$ to denote the map $F\union F' \to \R\cross M$ and say that $u$ and $u'$ are \emph{disjoint} if their images are disjoint, writing $u \sqcup u'$.

Now consider a single curve $u:F\to \R\cross M$ and write $F$ as a disjoint union of its connected components, $F=\bigsqcup F_i$.  We call each $u_i = \restr{u}{F_i}$ a \emph{component} of $u$ and say that $u_i$ is a \emph{multiply-covered component} if it is a (possibly branched) non-trivial covering map onto its image and \emph{simply-covered} otherwise. Finally, we say that a component $u_i$ is \emph{repeated} if there exists at least one other $u_j$ such that $u_i=u_j$, when considered up to reparametrization.  If every component of $u$ is simply-covered and not repeated, then we say that $u$ is \emph{simply-covered}.
\end{defn}

\begin{thm}[Index inequality, {\cite[Theorem 4.15]{Hu09}}]
If $u$ is a simply-covered holomorphic curve, then
\[ \ind(u) \le I(u) \]
and equality holds if and only if $u$ is an embedding and $u$ satisfies certain \emph{partition conditions} defined by Hutchings~\cite[Definition 4.7]{Hu02}.
\end{thm}

\subsection{Defining ECH and a hat version}\label{ECH_def_section}

We are now in a position to define the ECH differential.  Recall that as a vector space, $ECC(M,\alpha)$ is generated by orbit sets over $\F_2$.  Let $J$ be a regular almost complex structure adapted to $\alpha$ and let $\mathcal{M}_J^{I=k,\nmc}(\Gamma_+,\Gamma_-)$ denote the moduli space of holomorphic curves between $\Gamma_+$ and $\Gamma_-$ with ECH index $k$, and where the modifier nmc means that no component is multiply-covered (although repeated components are permitted).  Also note that  for any predicate $*$ there is an $\R$-action on $\mathcal{M}^*_J(\Gamma_+,\Gamma_-)$ defined by translation in the $\R$ direction.  We define the quotient
\[\widehat{\mathcal{M}}_J^*(\Gamma_+,\Gamma_-) := \mathcal{M}_J^*(\Gamma_+,\Gamma_-)/\R; \]
$\widehat{\mathcal{M}}_J^{I=1,\nmc}(\Gamma_+,\Gamma_-)$ is a discrete and finite set~\cite[Lemma 7.19]{HT07}.

\begin{defn}
We define the \emph{ECC chain complex}, 
\[ ECC(M,\alpha, J) := (ECC(M,\alpha), \del), \]
where the \emph{ECC differential} $\del$ is defined with respect to $J$ as follows:
\begin{equation}
\del \Gamma = \sum_{\Gamma'}{\#\widehat{\mathcal{M}}_J^{I=1,\nmc}(\Gamma,\Gamma')\cdot \Gamma'}
\end{equation}
where $\#$ denotes the mod 2 count.
\end{defn}

The summation in the above definition is finite~\cite[Lemma 7.19]{HT07} and the differential satisfies $\del^2=0$~\cite[Theorem 7.20]{HT07}.

\begin{defn}
\emph{Embedded Contact Homology} is defined to be the homology of the ECC chain complex:
\[ ECH(M,\alpha) := H_*(ECC(M,\alpha,J)). \]
\end{defn}

The homology is independent of the contact form $\alpha$, the contact structure $\xi$ and the almost complex structure $J$~\cite{KLT10,CGH_HF_ECH}.  Interestingly, this independence has not been proven directly, but through isomorphisms with Seiberg-Witten Floer cohomology (in the case of Kuluhan, Lee and Taubes) and Heegaard Floer homology (in the case of Colin, Ghiggini and Honda).  Note that since holomorphic curves only exist between orbit sets with the same homology class, ECH splits canonically as a direct sum
\[ ECH(M,\alpha) \iso \bigoplus_{A\in H_1(M,\Z)}ECH(M,\alpha,A), \]
although this splitting is not independent of $\alpha$.

Colin, Ghiggini and Honda define a hat version of ECH~\cite[Section 2.5]{CGH_ECH_OBD}, using the $U$-map defined by Hutchings~\cite[Section 2.5]{HuWeinstein}:

\begin{defn}
Let $z\in\R\cross M$ be a generic point away from any Reeb orbit.
\begin{enumerate}
\item $\mathcal{M}_J^{I=2,*}(\Gamma_+,\Gamma_-;z)$ is defined to be the space of ECH index 2 curves satisfying the predicate $*$ and that pass through the point $z$.
\item The $U$-map is a degree 2 map defined by
\[ U\Gamma = \sum_{\Gamma'}{\#\mathcal{M}_J^{I=2,\nmc}(\Gamma,\Gamma';z)\cdot\Gamma'}. \]
\end{enumerate}
\end{defn}
This is a well-defined chain map~\cite[Lemma 2.6]{HuWeinstein}.

\begin{defn}
$\widehat{ECH}(M,\alpha)$ is defined to be the homology of the mapping cone of $U$.
\end{defn}

$\widehat{ECH}$ is also independent of $\alpha$, $\xi$ and $J$, as shown indirectly by Colin, Ghiggini and Honda via an isomorphism with $\widehat{HF}$~\cite[Theorem 1]{CGH_HF_ECH}.  Also, since the $U$-map respects the splitting over $H_1(M,\Z)$, $\widehat{ECH}$ also splits:
\[ \widehat{ECH}(M,\alpha) \iso \bigoplus_{A\in H_1(M,\Z)}\widehat{ECH}(M,\alpha,A). \]

\section{Topological constraints on holomorphic curves}\label{top_constraints_on_holo_curves}

In this section we will briefly discuss the topological nature of the holomorphic curves which make up the ECH differential and the $U$-map.  Later, in \cref{topological_constrants_in_the_MB_settting}, we will see how further topological constraints arise in the Morse-Bott setting.

\begin{defn}
A \emph{trivial cylinder} over a simple orbit $\gamma$ is the holomorphic curve with image $\R\cross\gamma$.  A \emph{multiply-covered} trivial cylinder is a curve which is a (possibly branched) cover of a trivial cylinder.
\end{defn}

\begin{thm}[{\nogapcite[Proposition 3.7]{Hu14}}]\label{index_0_1_and_2_holo_curves}
Let $J$ be a regular almost complex structure adapted to $\alpha$ and $u$ be a holomorphic curve . Then the following holds:
\begin{itemize}
\item $I(u) \ge 0$,
\item $I(u)=0$ if and only if $u$ is a (possibly empty) union of (possibly repeated, possibly multiply-covered) trivial cylinders,
\item If $I(u)=1$ then $u=u_0 \sqcup u_1$, where $I(u_0)=0$ and $u_1$ is a connected, embedded curve with $\mathrm{ind}(u)=I(u_1)=1$.
\item If $I(u)=2$ then $u=u_0 \sqcup u_2$, where $I(u_0)=0$ as above and $u_2$ is an embedded curve with $\mathrm{ind}(u)=I(u_2)=2$.  $u_2$ is either a connected curve or a disjoint union of two embedded curves with Fredholm and ECH indices 1.
\end{itemize}
\end{thm}

If $u$ has ECH index 1 (resp.~2) then we will call $u_1$ (resp.~$u_2$) the \emph{non-connector part} of $u$ and we will call $u_0$ the \emph{connector part} of $u$.

Note that any holomorphic curve $u$ can be augmented by the addition of trivial cylinders with index 0, provided they are disjoint from $u$. This fact plays a very important role in the behaviour of the ECH differential.

\section{ECH = HF}\label{ECH_equals_HF_section}

In a series of papers~\cite{CGH_HF_ECH_I,CGH_HF_ECH_II,CGH_HF_ECH_III}, Colin, Ghiggini and Honda recently proved that embedded contact homology and Heegaard Floer homology are equivalent.  More precisely:

\begin{thm}[{\nogapcite[Theorem 1]{CGH_HF_ECH}}]
Let $M$ be a closed oriented contact 3-manifold.  Then
\begin{align}
& ECH(M) \iso HF^+(-M)\text{, and}\\
& \widehat{ECH}(M) \iso \widehat{HF}(-M).
\end{align}
\end{thm}

The proof starts by finding an open book decomposition of $M$ which is adapted to the contact structure.  It is then possible to define ECH solely in terms of the monodromy of the open book decomposition away from the binding, utilising a very similar construction called \emph{periodic Floer homology}, or PFH, which we will discuss in detail later in \cref{ECK_hat_supported_in_low_Alexander_degrees}.  

At the same time, we can use the open book decomposition to exhibit a Heegaard diagram for $M$, whose construction is similarly only dependent on the monodromy data.

It is then possible to express the generators of Heegaard Floer homology as short flow lines between the $\alpha$ and $\beta$ curves.   The authors construct symplectic cobordisms between these arcs and the orbits in PFH and vice versa.  It is then shown that these cobordisms induce chain maps whose induced maps on homology are mutual inverses.

Kuluhan, Lee and Taubes also proved the isomorphism above by passing through Seiberg-Witten Floer cohomology.  In their joint paper~\cite{KLT10} they prove that ECH = HM, and Taubes showed that HM = HF~\cite{Taubes10}.

There is also a conjectured isomorphism (c.f.~\cref{ECK_HFK_conjecture}) between knot versions of the two theories, which will be discussed in more detail later in \cref{embedded_contact_knot_homology}.

\section{ECH via rational open book decompositions (statement)}

As mentioned above, the first part of Colin, Ghiggini and Honda's proof that ECH = HF uses an open book decomposition of $M$ adapted to the contact structure on $M$ to express ECH solely in terms of the monodromy of the open book decomposition.

Our first step towards understanding the behaviour of ECH under knot surgery operations is to show that a similar result holds for \emph{rational} open book decompositions.

\begin{defn}\label{OBD1}
An \emph{open book decomposition} (with connected binding) for a closed 3-manifold $M$ is a pair $(K,\pi)$ such that
\begin{itemize}
\item $K$ is a null-homologous knot in $M$.
\item $\pi:M\sminus K\to S^1$ is a fibration such that each fibre $\pi^{-1}(\theta)$ is the interior of a compact surface in $M$ with boundary $K$.  Each fibre surface is called a \emph{page} of the open book decomposition.
\end{itemize}
\end{defn}

\begin{egs}\label{OBD_egs}~
\begin{enumerate}
\item The unknot $U \subset S^3$ admits an open book decomposition: $S^3 \sminus U$ is homeomorphic to $S^1\cross \mathrm{int}(D^2)$ and the closure of each page, $\set{\theta}\cross D^2$ has boundary equal to $U$.
\item The trefoil knot $T\subset S^3$, or indeed any fibred knot, admits an open book decomposition by Seifert surfaces in a completely analogous way.  In the case of the trefoil, each page of the decomposition is a punctured torus.
\end{enumerate}
\end{egs}

For the purposes of this thesis we will use a slightly different definition of an open book decomposition, which is given below and defined in terms of a mapping torus.

\begin{defn}
Let $\Sigma$ be a surface and $\phi:\Sigma\to\Sigma$ a diffeomorphism.  Define an equivalence relation $\sim$ on $\Sigma\times[0,1]$ by $(x,1)\sim(\phi(x),0)$ for all $x\in\Sigma$. The \emph{mapping torus} is the 3-manifold 
\[ M(\phi) := \Sigma\cross[0,1]/\!\sim. \]
We refer to $\phi$ as the \emph{monodromy} of the mapping torus.
\end{defn}

By abuse of notation, we will refer to points in $M(\phi)$ by $(x,t)$ rather than $[(x,t)]$.  Similarly, we will describe curves in $M(\phi)$ by defining arcs, in $\Sigma\cross[0,1]$, which form loops under the equivalence relation.

\begin{defn}\label{OBD2}
An \emph{open book decomposition} of a manifold $M$ with respect to a knot $K\subset M$ is a monodromy $\phi:\Sigma\to\Sigma$ such that
\begin{enumerate}
\item $\phi$ is a diffeomorphism fixing $\del \Sigma$ pointwise,
\item The mapping torus $N:= M(\phi)$ is equal to $M\sminus \nu(K)$, where $\nu(K)$ is a small open neighbourhood around $K$.
\item\label{fourth_point} the boundary of each page $\Sigma\cross\set{t}$ is homologous to $K$ in $\overline{\nu(K)}$.  In other words, $M$ is formed from $N$ by Dehn filling along the boundary curve
\[ m=\set{x}\cross[0,1]\subset \del\Sigma\cross[0,1],\]
 where $x$ is some point in $\del\Sigma$.
\end{enumerate}
\end{defn}

\begin{egs}~
\begin{enumerate}
\item The monodromy for the open book decomposition of the unknot in \cref{OBD_egs} is the identity map $D^2\to D^2$.
\item The monodromy for the right handed trefoil knot is given by the composition of two Dehn twists on the punctured torus, written $ab$ in Bell's notation~\cite{knottable}.
\end{enumerate}
\end{egs}

In the definition above, we refer to each $\Sigma\cross\set{t}$ as a page of the open book decomposition. In this thesis we will only deal with open book decompositions where $\Sigma$ has one boundary component; we  refer to the null-homologous curve $l=\del\Sigma\cross\set{0}\in\del N$ as the \emph{canonical longitude} of $K$, oriented as $\del\Sigma$.  We also fix an arbitrary point $x\in\del\Sigma$ and refer to the closed curve $m$ in part \cref{fourth_point} of the definition as the \emph{meridian} of the knot $K$, oriented so that $m\cdot l=1$ with respect to $\del(\overline{\nu(K)})$.  

Since any fibration $\pi:N\to S^1$ with fibre a surface $\Sigma$ is equivalent to the mapping torus of some monodromy $\phi:\Sigma\to\Sigma$, the two definitions of an open book decomposition above are equivalent.  Note however that the $N$ in \cref{OBD2} is compact and slightly smaller than the $N$ which is open in \cref{OBD1}.

\begin{defn}
A \emph{$p/q$-fractional twist} on $\del N$ is a map
\[ r_{p/q}: \del \Sigma \to \del \Sigma \]
such that there exists an orientation-preserving identification of $\del\Sigma$ with $S^1\subset\C$ such that $\restr{\phi}{\del\Sigma}:S^1\to S^1$ is given by multiplication by $e^{2\pi ip/q}$.
\end{defn}

\begin{defn}
A \emph{rational open book decomposition} of a manifold $M$ with respect to a knot $K\subset M$ is a monodromy $\phi:\Sigma\to\Sigma$ such that
\begin{enumerate}
\item $\phi$ is a diffeomorphism such that $\restr{\phi}{\del\Sigma}$ is a $p/q$-fractional twist $r_{p/q}$.
\item The mapping torus $N:= M(\phi)$ is equal to $M\sminus \nu(K)$, where $\nu(K)$ is a small open neighbourhood around $K$.
\item $M$ is formed from $N$ by surgery along a closed curve $d\subset\del N$ which we call the \emph{degeneracy slope}.  The curve $d$ is obtained by starting at some $(x,0)\in\del\Sigma\cross\set{0}$ and flowing under the vector field $\del/\del t$ until we form a closed loop.
\end{enumerate}
\end{defn}
Again, we will only consider rational open book decompositions where $\Sigma$ has a single boundary component.

Since $\restr{\phi}{\del\Sigma}$ is a $p/q$-fractional twist, the degeneracy slope will intersect each page $q$ times.  In fact, if our identification of $\del\Sigma$ with $S^1$ identifies $x$ with $1$ then $d$ will intersect $\del\Sigma\cross\set{0}=S^1$ at the following points in the following order:
\[ 1, e^{2\pi ip/q}, e^{2\pi i 2p/q}, \dots, e^{2\pi i (q-1)p/q}. \]

Since $M$ is obtained by surgery along $d$, $d$ is a meridian of $K$.  Again, denote by $l$ the curve $\del\Sigma\cross\set{0}$ but note that $d\cdot l = q$ so $l$ is not a longitude for $K$; instead, pick a longitude in $\del(\overline{\nu(K)})$ for $K$ and denote it by $\lambda$.

In an open book decomposition, $K$ is null-homologous since it is homologous to $\del\Sigma\cross\set{0}$.  However in a rational open book decomposition this is not necessarily the case.  For example, the decomposition of the lens space $L(p,q)$ into two tori is a rational open book decomposition, where $K$ is the core of one torus and the other torus represents $N$.  Then $[K]$ generates $H_1(L(p,q);\Z)\iso \Z/p$.

In general, let $m$ denote a curve in $\del N$ which intersects each page once. (For example, let $m$ be homotopic to  the union of the arc $\set{x}\cross[0,1]$ and an arc in $\del\Sigma\cross\set{0}$ joining $(\phi(x),0)$ and $(x,0)$.)  Since $m\cdot l=1$, $[m]$ and $[l]$ form a basis for $H_1(\del(\overline{\nu(K)});\Z)$.  Then $m\cdot d= k$ for some $k\in\Z$ and hence
\[ [d]=q[m] + k[l]\in H_1(\del(\overline{\nu(K)});\Z). \]
In addition, if we write $[\lambda] = a[m]+b[l]$ for some $a,b\in\Z$, then
\[\begin{split}
q[K] = q[\lambda] &= q(a[m]+b[l])\\
 &= a[d] + (qb-ak)[l]\\
 & = 0 \in H_1(M;\Z).
\end{split}\]
Here we are using $[l]=[\del\Sigma]=0$ and $[d]=0$ since it is a meridian for $K$.  Hence $[K]=0\in H_1(M;\Q)$, which motivates the term ``rational'' open book decomposition.

In the following chapter we will define a pair of ``relative'' embedded contact homology groups in the case of a rational open book decomposition, $ECH(N,\del N, \alpha, J)$ and $\widehat{ECH}(N,\del N, \alpha, J)$ (c.f.~\cref{relative_ECH_groups}).  The first of these groups comes equipped with a $U$-map and they both split as a direct sum over the homology of $N$ relative to the degeneracy slope $d$,
\[ H_1(N,[d]) := coker( d_*:H_1(S^1;\Z) \to H_1(N;\Z) ). \]
In \cref{ECH_via_rational_open_book_decompositions_chapter} we will then prove the following result, which is adapted from Colin, Ghiggini and Honda's result for integral open book decompositions~\cite[Theorem 1.1.1]{CGH_ECH_OBD}.

\begin{thm}\label{relative_ECH_equals_ECH}\label{RELATIVE_ECH_EQUALS_ECH}
Suppose that $\phi:N\to N$ is the monodromy for a rational open book decomposition of $M$, a closed, oriented, connected 3-manifold.  Then we can equip $N$ and $M$ with contact forms $\alpha$ and $\hat{\alpha}$ respectively such that $\restr{\hat{\alpha}}{N}$ differs from $\alpha$ only by a small $\mathcal{C}^\infty$ perturbation near $\del N$.  Furthermore, $\alpha$ is non-degenerate on $\mathrm{int}(N)$ and its associated Reeb vector field is negative Morse-Bott on $\del N$ and foliates $\del N$ by meridians of $K$.  We then have isomorphisms
\begin{align*}
&\text{1. }ECH(N,\del N, \alpha, J) \iso ECH(M,\hat{\alpha})\quad\text{and} \\
&\text{2. }\widehat{ECH}(N,\del N, \alpha, J) \iso \widehat{ECH}(M,\hat{\alpha})
\end{align*}
where $J$ is a choice of regular adapted almost complex structure for $\alpha$. In addition, the isomorphism in (1) is compatible with the $U$-maps on both sides and both isomorphisms are compatible with the splitting over $H_1(N,[d])\iso H_1(M)$.
\end{thm}
The constructions of $\alpha$ and $\hat{\alpha}$ in the above theorem will be discussed in \cref{contact_forms_and_rational_open_book_decompositions} and the Morse-Bott terminology will be explained in \cref{Morse-Bott_contact_homology}.  Part (1) is proven on \cpageref{proof_of_rel_ECH_equals_ECH_pt_1} and part (2) on \cpageref{proof_of_rel_ECH_equals_ECH_pt_2}.

\chapter{Morse-Bott theory}

\section{Morse-Bott contact homology}\label{Morse-Bott_contact_homology}

In order to prove \cref{relative_ECH_equals_ECH}, we must start by defining a Morse-Bott version of embedded contact homology, $ECH_\MB(M,\alpha,J)$.  This gives rise to some tools which help us understand the behaviour of holomorphic curves near \emph{Morse-Bott tori} in $M$ and also allows the construction of \emph{finite energy foliations} which explicitly compute some holomorphic curves contributing to the ECH differential in a neighbourhood of the knot $K$.  It also means that it makes sense to define a version of embedded contact homology in the case where there is a Morse-Bott torus at $\del N$.  This section follows the construction of Morse-Bott contact homology originally due to Bourgeois~\cite{Bo02}. 

\begin{defn} \label{morse-bott_def}
Let $\alpha$ be a contact form on a connected, oriented, compact 3-manifold $M$ and $R$ its Reeb vector field.  Suppose that $\mathcal{N}$ is an $S^1$ family of simple Reeb orbits foliating an embedded torus $T_\mathcal{N}\subset M$.  For every $p\in T_\mathcal{N}$ choose a symplectic basis $\set{v_1,v_2}$ of $(\xi_p, d\alpha)$ such that $v_2$ is tangent to $T_\mathcal{N}$.  The linearized return map of $R$ at any point $p\in T_\mathcal{N}$ is degenerate in the $v_2$ direction, hence represents a parabolic element of $SL_2(\R)$ of the form
\[ \begin{pmatrix} 1 & 0 \\ a & 1 \end{pmatrix}. \]
\begin{enumerate}
\item The family $\mathcal{N}$ is a \emph{Morse-Bott family} of Reeb orbits if the linearized return map is non-degenerate in the $v_1$ direction at every $p\in T_\mathcal{N}$, i.e.~$a\ne0$.
\item We say that the Morse-Bott torus is \emph{positive (resp.~negative)} if $a$ is positive (resp.~negative) at all $p$.
\item The contact form $\alpha$ is said to be a \emph{Morse-Bott contact form} if all orbits are either non-degenerate or part of a Morse-Bott family.
\end{enumerate}
\end{defn}

In this section we will see that it is possible to perturb $\alpha$ in an arbitrarily small neighbourhood of $T_\mathcal{N}$ to form two non-degenerate Reeb orbits, one elliptic and one positive hyperbolic.  If the Morse-Bott torus is positive (resp.~negative) then the return map of the elliptic orbit will be a small positive (resp.~negative) rotation.  See \cref{Morse-Bott_perturbation} for an illustration of how these orbits arise under the perturbation.

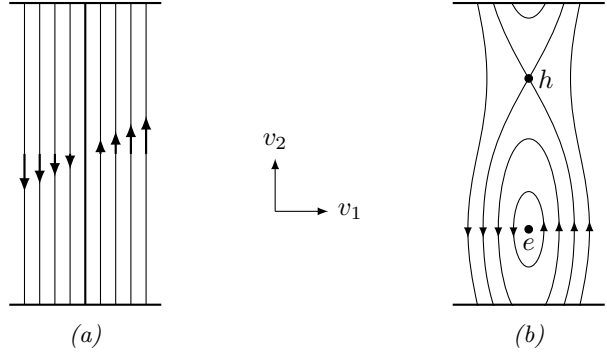
\begin{figure}\centering
	\begin{subfigure}{0.3\textwidth}\centering
		\begin{tikzpicture}%
			\draw [thick](-1,0) --(1,0);
			\draw [thick](-1,4) --(1,4);
			\foreach \x in {-0.8, -0.6,...,-0.2,0.2,0.4,...,0.8} {
				 \draw [thin](\x,0)--(\x,4);
				 \draw [-latex,thick,shorten >=-3pt] (\x,2)--(\x,2+0.5*\x);
			}
			\draw [thick](0,0)--(0,4);
		\end{tikzpicture}
		\caption{}
		\label{Morse-Bott_perturbation_a}
	\end{subfigure}
	\begin{subfigure}[c]{0.1\textwidth}\centering
	  \begin{tikzpicture}
	    \figureaxes{0}{0}{v_1}{v_2}{}{}
		\end{tikzpicture}
	\end{subfigure}
  \begin{subfigure}{0.3\textwidth}\centering
		\begin{tikzpicture}%
		 \path [clip] (-1,-0.01) rectangle (1,4.01);
			\draw [thick](-1,0) --(1,0);
			\draw [thick](-1,4) --(1,4);
			\draw [thin] (0.55,-1) to [out=90,in=-90] (0.8,1) to [out=90,in=-90] (0.55,3) to [out=90,in=-90] (0.8,5);
			\draw [thin] (-0.55,-1) to [out=90,in=-90] (-0.8,1) to [out=90,in=-90] (-0.55,3) to [out=90,in=-90] (-0.8,5);
			\draw [thin] (0,3) to [out=-120, in=90] (-0.6,1) to [out=-90,in=120] (0,-1);
			\draw [thin] (0,3) to [out=120, in=-90] (-0.6,5) to [out=90,in=-120] (0,7);
			\draw [thin] (0,3) to [out=-60, in=90] (0.6,1) to [out=-90,in=60] (0,-1);
			\draw [thin] (0,3) to [out=60, in=-90] (0.6,5) to [out=90,in=-60] (0,7);
			\draw [thin] (0,1) ellipse (0.4 and 1.2);
			\draw [thin] (0,5) ellipse (0.4 and 1.2);
			\draw [thin] (0,1) ellipse (0.2 and 0.5);
			\draw [fill] (0,1) circle (0.05) node [below] {$e$};
			\draw [fill] (0,3) circle (0.05) node [right] {$h$};
			\foreach \x in {0.2,0.4,0.6,0.8} {
				\draw [-latex,shorten >=-3pt] (\x,1)-- +(90:0.01);
				\draw [-latex,shorten >=-3pt] (-\x,1)-- +(-90:0.01);
			}
			
		\end{tikzpicture}
		\caption{}
		\label{Morse-Bott_perturbation_b}
	\end{subfigure}
	\caption{In both diagrams above, the top and bottom are identified to form an annulus and we then cross with $S^1$ to form the region $T_\mathcal{N}\cross(-\epsilon,\epsilon)$.  The flows as indicated in both diagrams represent the return map of the Reeb vector field.  In (a), we have highlighted the central vertical line which is a positive Morse-Bott torus foliated by orbits which wind around in the $S^1$ direction once.  $v_1$ and $v_2$ represent the vectors normal to and parallel to $T_\mathcal{N}$ respectively. In (b), the Morse-Bott torus has been perturbed into two non-degenerate orbits, one hyperbolic and one elliptic.}
	\label{Morse-Bott_perturbation}
\end{figure}

To formalize the notions of ``small'', we have to fix $L>0$ and consider only orbits of action less than $L$.  Then, since $M$ is compact we have a finite number of such orbits and Morse-Bott tori, $\mathcal{N}_1,\dots,\mathcal{N}_k$ (this would not be the case if we allowed orbits of any length; for example, in \cref{Morse-Bott_perturbation_a} the Morse-Bott tori form a dense subset).  We say that $L>0$ is an \emph{action constraint} for $\alpha$ if no Reeb orbit of $\alpha$ has action equal to $L$ and we say that $\alpha$ is $L$-non-degenerate if all orbits of action less than $L$ are non-degenerate.

To make the perturbation, start by picking a Morse-Bott function $\bar{g} :S^1 \to \R$ on each $\mathcal{N}_i\iso S^1$ with two critical points.  The pull back of this function under the quotient map $T_{\mathcal{N}_i} \to \mathcal{N}_i$ gives an map $g_i:T_{\mathcal{N}_i}\to \R$ which is invariant under the flow of $R$.  We then pick a smooth function $g:M\to \R$ such that the following hold:
\begin{itemize}
\item $g$ is supported in disjoint neighbourhoods of the tori $T_{\mathcal{N}_i}$, and away from any non-degenerate orbits of action less than $L$.
\item There exist small neighbourhoods of each Morse-Bott torus $T_{\mathcal{N}_i}$ which we can identify with $T_{\mathcal{N}_i}\cross (-\epsilon,\epsilon)$; on each such neighbourhood $g(x,\rho) = g_i(x)$.
\end{itemize}

We now use this function $g$ to perturb our contact form $\alpha$.
\begin{prop}[{\nogapcite[Lemmas 2.3 and 2.4]{Bo02}}]\label{perturbation_fn_g}
For every action constraint $L>0$ we can find a $g$ as above and an $\epsilon$ arbitrary small such that if we let $\alpha_\epsilon=(1+\epsilon g)\alpha$ then
\begin{enumerate}
\item $\alpha_\epsilon$ is $L$-non-degenerate,
\item each family of orbits $\mathcal{N}_i$ is perturbed into a pair of non-degenerate orbits, $e_i$ and $h_i$, which are elliptic and hyperbolic respectively and have $\alpha_\epsilon$-action less than $L$,
\item if $\mathcal{N}$ is positive (resp.~negative) then all multiple covers $e_i^m$ with $\alpha_\epsilon$-action less than $L$ have Conley-Zehnder index 1 (resp.~$-1$) with respect to the trivialization coming from the Morse-Bott torus. (I.e.~the rotation angle $\theta$ of $e_i$ can be made small enough so that $\lfloor m\theta\rfloor = 0$ (resp.~$\lfloor m\theta\rfloor = -1$).)
\end{enumerate}
\end{prop}

For the rest of this chapter we will work with Morse-Bott contact forms such as $\alpha$ and small perturbations $\alpha_\epsilon$.  We will construct a Morse-Bott analogue of embedded contact homology which, for certain ``nice'' Morse-Bott contact forms $\alpha$ can be approximated by $ECH(M,\alpha_\epsilon, J_\epsilon)$. At the same time, we will exploit topological restrictions which arise in the Morse-Bott setting, giving us some control over the behaviour of these complexes.

As a chain complex, $ECC_\MB(M,\alpha,J)$ is generated by orbit sets of $R_\alpha$ where from each Morse-Bott torus $\mathcal{N}_i$ we choose two distinguished orbits $e_i$ and $h_i$ corresponding to the maximum and minimum of the Morse-Bott perturbation $\bar{g}$. (If $\mathcal{N}$ is positive then $e_i$ corresponds to the maximum and $h_i$ to the minimum. If $\mathcal{N}$ is negative then the correspondence is switched.)  With respect to $\alpha$ these two orbits are geometrically indistinguishable from the remaining orbits in $\mathcal{N}_i$, but they can be thought of intuitively as the limit as $\epsilon$ tends to $0$ of the orbits $e_i$ and $h_i$ of $R_{\alpha_\epsilon}$.  It therefore is natural that we treat each $e_i$ as an elliptic orbit and each $h_i$ as a hyperbolic orbit. 

More precisely, let $\mathcal{P}'$ denote the set of simple non-degenerate orbits of $\alpha$; $\mathcal{P}_\MB=\mathcal{P}'\union(\bigcup_i\mathcal{N}_i)$ denote the set of all Reeb orbits of $\alpha$; and $\mathcal{P}=\mathcal{P}'\union(\bigcup_i\set{e_i,h_i})$.  The complex $ECC_\MB(M,\alpha,J)$ is generated over $\F_2$ by orbit sets of the form $\Gamma = \set{(\gamma_i,m_i)}$ where $\gamma_i$ are orbits in $\mathcal{P}$ and $m_i=1$ for hyperbolic non-degenerate orbits and the distinguished ``hyperbolic'' Morse-Bott orbits $\set{h_i}$.

\subsection{Morse-Bott buildings}

The definition of the differential in the Morse-Bott setting involves counting \emph{Morse-Bott buildings}, which are a generalization of holomorphic curves.  The definition and notation is rather unwieldy but the ideas are quite intuitive.  First we must discuss moduli spaces of holomorphic curves in the Morse-Bott setting.

\begin{defn}
Suppose that $\alpha$ is a Morse-Bott contact form on $M$ with adapted almost complex structure $J$.  Let $\Gamma_+ = \set{(\gamma_i^+, m_i^+)}$ be an orbit set where each $\gamma_i^+$ is a simple non-degenerate orbit and $\mathcal{N}_+ = \set{(\mathcal{N}_i^+, \widetilde{m}_i^+)}$ be a collection of Morse-Bott tori with multiplicities.  Similarly let  $\Gamma_- = \set{(\gamma_i^-, m_i^-)}$ and $\mathcal{N}_- = \set{(\mathcal{N}_i^-, \widetilde{m}_i^-)}$.

Define the moduli space
\[ \mathcal{M}_J(\Gamma_+, \mathcal{N}_+, \Gamma_-, \mathcal{N}_-) \]
of holomorphic curves $u$ such that
\begin{itemize}
\item the positive end of $u$ is equal to the orbit set $\Gamma_+\union \set{(\widetilde{\gamma}_i^+,\widetilde{m}_i^+)}$, where each $\widetilde{\gamma}_i^+$ is a simple orbit in $\mathcal{N}_i^+$ and
\item similarly the negative end of $u$ is equal to the orbit set $\Gamma_-\union \set{(\widetilde{\gamma}_i^-,\widetilde{m}_i^-)}$, where each $\widetilde{\gamma}_i^-$ is a simple orbit in $\mathcal{N}_i^-$.
\end{itemize}
\end{defn}

The almost complex structure $J$ is \emph{regular} if for all orbit and Morse-Bott torus sets $\Gamma_+, \mathcal{N}_+, \Gamma_-, \mathcal{N}_-$, and for all $u$ in the resulting moduli space containing no multiply-covered covered components, the moduli space near $u$ is transversely cut out and hence a manifold.  This is a generic condition by Bourgeois~\cite[Proposition 6.4]{Bo02}.

We will now state the definition of a Morse-Bott building.  For the purposes of this definition, we think of $\mathcal{P}_\MB$ as a union of Morse-Bott families, where the non-degenerate orbits in $\mathcal{P}'\subset\mathcal{P}_\MB$ are each treated as a Morse-Bott family consisting of a single orbit.

Define an \emph{end pair} within a Morse-Bott family $\mathcal{N}$ to be an ordered pair of orbits $(\widetilde{\gamma}^-, \widetilde{\gamma}^+)$ such that
\begin{itemize}
\item there exists $m$ so that $\widetilde{\gamma}^\pm$ is an $m$-fold cover of some simple orbit $\gamma^\pm\in\mathcal{N}$, and
\item there is a gradient flow line $\delta$ from $\gamma^-$ to $\gamma^+$ in $\mathcal{N}$ with respect to the Morse-Bott function $\bar{g}$. 
\end{itemize}
$\delta$ can be thought of as an $m$-fold unbranched cover of the annulus in $T_{\mathcal{N}}$ connecting $\gamma^-$ and $\gamma^+$.  If $\mathcal{N}$ is a Morse-Bott family consisting of a simple non-degenerate orbit, or if $\widetilde{\gamma}^-=\widetilde{\gamma}^+$, then $\delta$ is taken as a flow line of length 0.

We say that two collections of orbits $\set{\widetilde{\gamma}^-_j}$ and $\set{\widetilde{\gamma}^+_j}$ can be \emph{paired with respect to the flow lines $\set{\delta_j}$} if, up to some reordering of the orbits, we have for every $j$ that $( \widetilde{\gamma}^-_j, \widetilde{\gamma}^+_j)$ forms an end pair with flow line $\delta_j$.

\begin{defn}
Let $\Gamma_+$ and $\Gamma_-$ be Morse-Bott orbit sets constructed from $\mathcal{P}$.  A \emph{{$J$-holomorphic} Morse-Bott building} $\widetilde{u}$ is a collection of $J$-holomorphic maps, $\set{u_i:F_i\to \R\cross M\gappy{|}i=1,\dots,n}$, which we call the \emph{holomorphic part}, and a set of gradient flow lines in $\mathcal{P}_\MB$, $\set{\delta_{i,j}, i=0,\dots,n,j=1,\dots,j_i}$, such that the following hold:
\begin{itemize}
\item The number of positive ends of $u_1$ is $j_0$, the number of negative ends of $u_n$ is $j_n$, and for $i=1,\dots,n-1$, the number of negative ends of $u_i$ and the number of positive ends of $u_{i+1}$ are both equal to $j_i$.
\item For $i=1,\dots,n-1$, the negative orbits of $u_i$ and the positive orbits of $u_{i+1}$ can be paired with respect to the flow lines $\set{\delta_{i,j}, j=1,\dots,j_i}$ in the manner described above. 
\item There is a collection of orbits $\set{\gamma_{0,j}}$, realising the orbit set $\Gamma_+$, which can be paired with the positive orbits of $u_1$ with respect to $\set{\delta_{0,j},j=1,\dots,j_0}$.  Similarly there is a pairing (with respect to $\set{\delta_{n,j},j=1,\dots,j_n}$) between the negative orbits of $u_n$ and a collection of orbits realising the orbit set $\Gamma_-$. 
\end{itemize}
The set of all Morse-Bott buildings between $\Gamma_+$ and $\Gamma_-$ will be denoted $\mathcal{M}^\MB_J(\Gamma_+,\Gamma_-)$.
\end{defn}

It turns out that in order to relate Morse-Bott contact homology with embedded contact homology, we have to consider only Morse-Bott contact forms which satisfy a property restricting the types of Morse-Bott buildings which can arise.

\begin{defn}\label{niceness_defn} Let $u$ be a Morse-Bott building.
\begin{enumerate}
\item We say that $u$ is \emph{nice} if the non-connector part of $u$ has only one component.
\item We say that the contact form $\alpha$ is \emph{nice} if for every generic almost complex structure $J$, every $J$-holomorphic Morse-Bott building with ECH index 1 is nice\footnote{Note that the definition of ECH index is easily extended to Morse-Bott buildings since they too naturally give rise to a relative homology class $[\widetilde{u}]\in H_2(M,\Gamma_+,\Gamma_-)$.}.
\item Suppose that $U=\set{(\Gamma_+,\Gamma_-)}$ is a collection of pairs of orbit sets and $d$ is a differential constructed by counting only Morse-Bott buildings between pairs of orbits in $U$.  We say that $\alpha$ is \emph{nice with respect to $d$} if every such contributing Morse-Bott building is nice.
\end{enumerate}
\end{defn}

Once we restrict to only considering nice contact forms, we then form the Morse-Bott differential by only counting \emph{very nice} Morse-Bott buildings:

\begin{defn}
A Morse-Bott building $u$ is \emph{very nice} if it is nice and every connector component of $u$ is a trivial cylinder over a simple orbit.
\end{defn}

\begin{defn}\label{MB_homology_def}
Recall that the Morse-Bott chain complex is generated by orbit sets constructed from $\mathcal{P}$, where we have chosen two distinguished orbits $e_i$ and $h_i$ from each Morse-Bott torus $\mathcal{N}_i$.  Define the Morse-Bott differential by counting very nice ECH index 1 Morse-Bott buildings between orbit sets.  The Morse-Bott contact chain complex is then denoted by
\[ ECC_\MB(M, \alpha, J). \]
We also define an \emph{action-bounded} version,
\[ ECC^L_\MB(M,\alpha,J), \]
obtained by considering only those orbit sets with action less than $L$.  Morse-Bott contact homology is obtained by taking the homology of these complexes. 
\end{defn}

The proof that the Morse-Bott differential satisfies $d_\MB^2=0$ follows from the proposition below.  First we must introduce an action-bounded version of the standard ECH complex; define $ECC^L(M,\alpha_\epsilon, J_\epsilon)$ to be the chain complex generated by orbit sets with action less than $L$.  Note that since $d\alpha_\epsilon>0$ on any holomorphic curve, a simple application of Stokes' theorem implies that if there exists a holomorphic curve from $\Gamma_+$ to $\Gamma_-$ then $\mathcal{A}_{\alpha_\epsilon}(\Gamma_+)>\mathcal{A}_{\alpha_\epsilon}(\Gamma_-)$.  This implies that $ECC^L(M,\alpha_\epsilon, J_\epsilon)$ is a subcomplex of $ECC(M,\alpha_\epsilon, J_\epsilon)$.  Similarly, $ECC^L_\MB(M,\alpha, J)$ is a subcomplex of $ECC_\MB(M,\alpha, J)$.

\begin{prop}[{\nogapcite[Lemma 4.4.5 and Proposition 4.4.7]{CGH_ECH_OBD}}]\label{non-degen_close_to_MB}
Let $L>0$, suppose that $\alpha$ is either a nice Morse-Bott contact form or a nice $L$-non-degenerate Morse-Bott contact form and let $J$ be a regular adapted almost complex structure for $\alpha$.  Then, for all $\epsilon>0$ sufficiently small, the $L$-non-degenerate perturbation $\alpha_\epsilon$ from \cref{perturbation_fn_g} has the properties below.  Let $J_\epsilon$ be a regular adapted almost complex structure for $\alpha_\epsilon$.
\begin{enumerate}
\item If $\Gamma_+$, $\Gamma_-$ are two orbit sets constructed from $\mathcal{P}$ then
\[ \mathcal{M}_{J}^{\MB,I=1,\mathrm{vn}}(\Gamma_+,\Gamma_-)/\R \iso \mathcal{M}_{J_\epsilon}^{I=1,\nmc}(\Gamma_+,\Gamma_-)/\R, \]
where ``vn'' means that only very nice Morse-Bott buildings are counted and recall that ``nmc'' means that only holomorphic curves with no multiply-covered components are counted.
\item \label{non-degen_close_to_MB_b} There is an isomorphism of chain complexes
\begin{equation}\label{ECC_L_equals_ECCMB_L} ECC^L_\MB(M,\alpha,J) \iso ECC^L(M,\alpha_\epsilon, J_\epsilon). \end{equation}
\end{enumerate}
\end{prop}

In particular, this implies that $ECC_\MB(M,\alpha,J)$ is a chain complex, since the restriction of $d_\MB^2$ to $ECC^L_\MB(M,\alpha,J)$ is zero for every $L$, and hence $d_\MB^2=0$ on the unbounded complex.

In this thesis we will be careful to only appeal to Morse-Bott contact homology when the contact forms are nice with respect to the differentials we are computing, and hence we are permitted to apply the above proposition.  However we will sometimes wish to prove the existence of certain Morse-Bott buildings even in the case when $\alpha$ is not necessarily nice, and for this we will use the following Lemma.

\begin{lemma}[{\nogapcite[Theorem 4.4.3]{CGH_ECH_OBD}}]\label{MB_building_exists_epsilon_small_enough}
Let $L>0$, suppose that $\alpha$ is a (not necessarily nice) Morse-Bott contact form on $M$ and let $J$ be a regular adapted almost complex structure for $\alpha$. Choose a sequence $\epsilon_i$ tending to $0$ and for each $i$ choose an $L$-non-degenerate perturbation $\alpha_{\epsilon_i}$ as in \cref{perturbation_fn_g} and a regular adapted almost complex structure $J_{\epsilon_i}$.

Then for fixed orbits sets $\Gamma_+$, $\Gamma_-$, all sequences $u_i \in \mathcal{M}_{J_{\epsilon_i}}(\Gamma_+,\Gamma_-)$ have a subsequence converging to some (not necessarily nice) Morse-Bott building $u\in \mathcal{M}^\MB_J(\Gamma_+,\Gamma_-)$.
\end{lemma}

The above lemma means that for a fixed $L>0$ and Morse-Bott contact form $\alpha$, it is always possible to choose a perturbation $\alpha_\epsilon$ such that, for any two orbit sets  $\Gamma_+$, $\Gamma_-$ of length less than $L$, the existence of a $J_\epsilon$-holomorphic curve in $\mathcal{M}_{J_\epsilon}(\Gamma_+,\Gamma_-)$ implies the existence of a $J$-holomorphic Morse-Bott building in $\mathcal{M}_J^\MB(\Gamma_+,\Gamma_-)$.

To see this, let $Z$ denote the set of pairs $(\Gamma_+,\Gamma_-)$ of orbit sets with action less than $L$ such that there exists a $J$-holomorphic Morse-Bott building between $\Gamma_+$ and $\Gamma_-$.  Similarly define $Z_\epsilon$ for $J_\epsilon$.  We wish to argue that for $\epsilon$ small enough, $Z_\epsilon\subset Z$.  But if this is not the case then we can find a single pair $(\Gamma_+, \Gamma_-)$ and a sequence $\epsilon_i$ tending to $0$ with $J_{\epsilon_i}$-holomorphic curves between $\Gamma_\pm$ for every $i$ but no $J$-holomorphic Morse-Bott building between the pair.  This of course contradicts \cref{MB_building_exists_epsilon_small_enough}.

\section{Direct limits}

In order to obtain an isomorphism relating ECH and Morse-Bott ECH of the form of \cref{ECC_L_equals_ECCMB_L} in the unbounded case, Colin, Ghiggini and Honda use a direct limit argument to let $L$ tend to $\infty$.  We will give an exposition of this argument in this section. Letting $L$ tend to $\infty$ on the right hand side of \cref{ECC_L_equals_ECCMB_L} is straightforward since for $L<L'$ there is a canonical inclusion map 
\begin{equation} \label{MB_inclusion_map}
i_{L,L'}:ECC^L_\MB(M,\alpha, J) \to ECC^{L'}_\MB(M,\alpha, J)
\end{equation}
and in the direct limit we obtain $ECC_\MB(M,\alpha,J)$ by definition.  However on the left hand side, $L\to\infty$ implies that $\epsilon\to0$ and hence the contact forms $\alpha_\epsilon$ (and almost complex structures $J_\epsilon$) are not fixed.  Hence some work must be done to apply a direct limit argument on this side.

\subsection{Maps induced by exact symplectic cobordisms}\label{direct_limits_through_cobordism_maps}

The key to taking a direct limit is to construct \emph{cobordism maps} between pairs of contact forms $\alpha_\epsilon$ and $\alpha_{\epsilon'}$.  These cobordism maps will be covered in great detail later in \cref{cobordism_maps_via_SW_theory} and for now we will only give a brief overview, stating only the results which are needed for this section.  

An \emph{exact symplectic cobordism} between two contact manifolds $(M_+,\alpha_+)$ to $(M_-,\alpha_-)$ is a symplectic 4-manifold $(X,d\lambda)$ such that $\del X = M_+ \union (-M_-)$ and $\restr{\lambda}{M_\pm}=\alpha_\pm$. Given such a cobordism, we obtain a map 
\begin{equation}\label{cobordism_map_bounded} 
\Phi^L(X,d\lambda) : ECH^L(M_+,\alpha_+) \to ECH^L(M_-,\alpha_-)
\end{equation}
for all $L>0$ which satisfies a collection of naturality conditions~\cite[Theorem 1.9]{HT_Arnold13}.  It is induced by a non-canonical chain map on ECC defined in terms of the isomorphism with Seiberg-Witten theory---it is not currently known how to construct the map solely in terms of pseudoholomorphic curves in $X$.  Note that in \cref{cobordism_map_bounded} we have omitted the choice of the almost complex structure---this is legitimate since the groups are independent of $J$~\cite[Theorem 1.3]{HT_Arnold13}.

All exact symplectic cobordisms we consider will be \emph{interpolating cobordisms}, which are defined via isotopies of $M$ in the following way.  Suppose that $\alpha_0$ and $\alpha_1$ are homotopic contact forms through $\alpha_t$. Then by Gray's Stability Theorem (see for example Geiges~\cite[Theorem 2.2.2]{Ge08}) there exists an isotopy of $M$,
\[ \phi :[0,1]\cross M\to M, \]
and a one-parameter family of smooth functions on $M$,
\[ f_t : M \to \R_{>0},  \]
such that
\begin{itemize}
\item $\phi_t^*(f_t\alpha_0)=\alpha_t$ for all $t\in[0,1]$, and
\item $\phi_0=\id$ and $f_0=1$.
\end{itemize}
Now use this data to construct $\lambda_\phi := f \phi^*(\alpha_0)$, a 1-form on $[0,1]\cross M$.  If $\frac{\del f}{\del t}>0$ then $([0,1]\cross M, d\lambda_\phi)$ is an exact symplectic cobordism.

In general, we can always find an interpolating cobordism between any two homotopic contact forms, up to rescaling by a constant, by the proposition below.  Furthermore, the induced map $\Phi_\phi := \Phi([0,1]\cross M, d\lambda_\phi)$ depends only on the homotopy class of $\phi$ relative to its endpoints~\cite[Lemma 3.1.6]{CGH_ECH_OBD}.

\begin{prop}\label{interpolating_cobordism_between_homotopic_forms}
Suppose that $\alpha_0$ and $\alpha_1$ are homotopic contact forms.  Then there exists $A>0$ such that we can find an interpolating cobordism between $\alpha_1$ and $A\alpha_0$.
\end{prop}
\begin{proof}
Following the above discussion, we have 
\[ \phi_t^*(f_t\alpha_0)=\alpha_t \]
for some positive family $f_t$ of smooth functions on $M$ with $f_0=1$.  If $\frac{\del}{\del t}f_t >0$ then we are done with $A=1$.  Otherwise, our aim is to find some positive function $\lambda(t)$ so that
\begin{itemize}
\item $\frac{\del}{\del t}(\lambda(t)f_t) > 0$ everywhere,
\item $\lambda(1)=1$
\end{itemize}
We can then write 
\[ \phi_t^*(g_t\alpha_0)=\lambda(t)\alpha_t \]
where $g_t=\lambda(t)f_t$ and by the above discussion obtain an exact symplectic cobordism between $\alpha_1$ and $\lambda(0)\alpha_0$.  

To define $\lambda$, suppose that $f_t$ is bounded below by $c>0$ and $\frac{\del f_t}{\del t}$ is bounded below by $-C\le 0$.  Let $A'>C/c$ and define
\[  \lambda(t) = e^{A'(t-1)}. \]
Then
\[\begin{split}
\frac{\del}{\del t}(\lambda(t)f_t) &= \lambda(t)\frac{\del f_t}{\del t} + \lambda'(t)f_t \\
&= e^{A'(t-1)}( \frac{\del f_t}{\del t} + A' f_t ) \\
&> e^{A'(t-1)}( -C + A' c ) \\
&> 0
\end{split}\]
and we are done.
\end{proof}

The next step towards a direct limit is to construct maps where the action constraint increases.  

\begin{lemma}[{\nogapcite[Lemma 3.1.7]{CGH_ECH_OBD}}]\label{interpolating_cobordism_maps}
Suppose that $\alpha_+ = f\phi^*(\alpha_-)$ where $f\ge 1$ is a smooth function on $M$ and $\phi$ is isotopic to the identity map through $\phi_t$.  If $\alpha_+$ and $\alpha_-$ are $L$ and $L'$ non-degenerate respectively, and if $L'>L/f$, then there is a map
\[ ECH^{L}(M, \alpha_+) \to ECH^{L'}(M, \alpha_-) \]
which depends only on the relative homotopy class of $\phi_t$.  Furthermore, if $f\equiv1$ and $\phi_t\equiv\id$ then the map is induced by inclusion.

If we also have $0<L'<L''$, and a third contact form $\alpha'_+$ related to $\alpha_+$ in the same way, then the following triangle commutes:
\[\begin{tikzcd}
ECH^{L''}(M, \alpha'_+)\arrow[rr]\arrow[dr] && ECH^L(M, \alpha_-) \\
& ECH^{L'}(M, \alpha_+)\arrow[ur] &
\end{tikzcd}\]
\end{lemma}
The composition property of the above lemma allows us to form a directed system and take a direct limit in the next few paragraphs.

Now suppose we are working with a sequence of homotopic contact forms, $\set{\alpha_i}_{i=1}^\infty$.
\begin{defn}
Let $\alpha$ be a contact form on $M$.  We say that the sequence $\set{\alpha_i}$ is \emph{commensurate} to $\alpha$ if there is a constant $0<c<1$ such that for each $i\in\N$ there is diffeomorphism $\phi_i$ of $M$ which is isotopic to the identity and a function $f_i: M\to \R$ such that $\phi_i^*\alpha_i = f_i\alpha$ and $c<|f_i|_{C^0}<\frac{1}{c}$.
\end{defn}

\begin{thm}[{\nogapcite[Corollary 3.2.3]{CGH_ECH_OBD}}]\label{commensurate_contact_forms_direct_limit}
Let $\set{\alpha_i}$ be commensurate to $\alpha$ with constant $c$ and $L_i$ tending to $\infty$ such that each $\alpha_i$ is $L_i$-non-degenerate and $L_{i+1}>\frac{1}{c^3}L_i$ for all $i$.  Then the interpolating cobordism maps from \cref{interpolating_cobordism_maps}, 
\[ \Phi_{i,i'} : ECH^{L_i}(M, \alpha_i) \to ECH^{L_{i'}}(M, \alpha_{i'}), \]
form a directed system and
\[ ECH(M,\alpha) = \lim_{i\to\infty}ECH^{L_i}(M, \alpha_i). \]
\end{thm}

The requirement that $L_{i+1}>\frac{1}{c^3}L_i$ has the purpose of ensuring that we can find interpolating cobordisms between consecutive ECH groups---recall the assumption on $L$ and $L'$ in \cref{interpolating_cobordism_maps}.

Note that it is not necessarily the case that $\alpha_i$ tends to $\alpha$, so this result still holds if $\lim_{i\to\infty}\alpha_i$ is non-degenerate.  In particular, note that if $\set{\alpha_i}$ is commensurate to $\alpha$ then it is also commensurate to each $\alpha_i$ in turn.  In light of this we may often simply write that a sequence of contact forms $\set{\alpha_i}$ \emph{is commensurate}.

\subsection{Direct limits in the Morse-Bott setting}

Suppose that $\alpha$ is a nice Morse-Bott contact form on $M$ and $L_i$ is an increasing sequence of action constraints which tend to $\infty$.  Colin, Ghiggini and Honda construct a sequence of perturbed 1-forms $\alpha_i$, each of which is $L_i$-non-degenerate and which tend to $\alpha$.  In this section we will outline the method behind this construction.

Recall that in \cref{perturbation_fn_g} we picked, for each action constraint $L>0$, a function $g$ and $\epsilon>0$ to obtain the $L$-non-degenerate form $\alpha_\epsilon= (1+\epsilon g)\alpha$.  Furthermore each Morse-Bott torus of action less than $L$ was perturbed into a pair of non-degenerate orbits, such that all appropriate covers of the elliptic orbit had Conley-Zehnder index $\pm 1$.

We now choose $g_i$ and $\epsilon_i$ for each $L_i$ according to \cref{perturbation_fn_g} to obtain the contact forms $\alpha_i := (1+\epsilon_i g_i)\alpha$.  These are chosen in such a way that in addition the following hold:
\begin{enumerate}
\item Each $\epsilon_i$ is also small enough so that \cref{non-degen_close_to_MB} holds for each $L_i$, 
\item for $i<i'$ there are interpolating cobordism maps
\[  \Phi_{i,i'} : ECH^{L_i}(M, \alpha_i) \to ECH^{L_{i'}}(M, \alpha_{i'}), \]
\item there exists a sequence of positive constants $c_k$ such that $||g_i||_{C^k} \le c_k$ for all $i, k$, and
\item $\alpha_i$ has no Reeb orbits with $\alpha_i$-action in a small neighbourhood of $\set{L_i, L_{i-1}}$.  More precisely, if we define $a_i = (1+\epsilon_i c_0)^2$, then the neighbourhood is given by
\[ I_i := [a_i^{-2}L_i, a_i^2 L_i] \union [a_{i-1}^{-2}L_{i-1}, a_{i-1}^2 L_{i-1}]. \]
\end{enumerate}

The final point has the effect of controlling the introduction of new orbits when we move from $\alpha_i$ to $\alpha_{i+1}$ and is used to ensure that the interpolating cobordisms behave like inclusions---indeed, consider the following diagram:
\begin{equation}\begin{tikzcd}\label{interpolating_cobordisms_commute_with_MB_inclusion}
ECH^{L_i}(M,\alpha_i)\arrow[rr,"\iso"]\arrow[d,"\Phi_{i,i+1}"'] &&  ECH^{L_i}_\MB(M,\alpha) \arrow[d,"i"] \\
ECH^{L_{i+1}}(M,\alpha_{i+1})\arrow[rr,"\iso"] &&  ECH^{L_{i+1}}_\MB(M,\alpha)
\end{tikzcd}\end{equation}

It is easy to assume at first sight that this diagram commutes and that each interpolating cobordism map on the left is simply the map induced by inclusion at the chain level.  However recall that the maps $\Phi_{i,i+1}$ are defined via exact symplectic cobordisms and Seiberg-Witten theory so their behaviour is not as clear as we would like.  

Fortunately it is possible to show that \refDiagram{interpolating_cobordisms_commute_with_MB_inclusion} does indeed commute.  By \cref{interpolating_cobordism_maps} the map $\Phi_{i,i+1}$ is equal to the composition of interpolating cobordism maps
\[  ECH^{L_i}(M,\alpha_i) \to ECH^{a_i L_i}(M, \alpha_{i+1}) \to ECH^{L_{i+1}}(M, \alpha_{i+1}), \]
where we  are assuming that $a_i L_i < L_{i+1}$ by taking $\epsilon_i$ small enough.
The fact that $\alpha_i$ has no orbits with action in the region $I_i$ is used~\cite[Lemma 4.5.3]{CGH_ECH_OBD} to show that the first map is an isomorphism and this in turn is then used~\cite[Lemma 4.5.5]{CGH_ECH_OBD} to show that it is in fact equal to the ``obvious'' map induced by inclusion.  It then follows that, since the second map is induced by inclusion by \cref{interpolating_cobordism_maps}, \refDiagram{interpolating_cobordisms_commute_with_MB_inclusion} does indeed commute as expected.

Since the contact forms $\set{\alpha_i}$ are commensurate, we can apply \cref{commensurate_contact_forms_direct_limit} and take direct limits to obtain the following result.

\begin{thm}[{\nogapcite[Theorem 4.5.9]{CGH_ECH_OBD}}]
Let $\alpha$ be a nice  Morse-Bott contact form and choose $\set{\alpha_i}$ as above. Then we can take the direct limit of the isomorphism 
\[ ECH^{L_i}(M,\alpha_i) \xrightarrow{\iso} ECH^{L_i}_\MB(M,\alpha) \]
from \refListInThm{non-degen_close_to_MB}{non-degen_close_to_MB_b} to obtain an isomorphism
\[ ECH_\MB(M,\alpha) \iso ECH(M). \]
\end{thm}

\section{Topological constraints in the Morse-Bott setting}\label{topological_constrants_in_the_MB_settting}

As mentioned earlier, the purpose of switching to Morse-Bott contact homology is two-fold.  Morse-Bott contact forms are simpler to handle than the perturbed non-degenerate contact forms---in particular it is possible to construct finite energy foliations to understand explicitly where pseudoholomorphic curves arise and how they contribute to the differential.  We will discuss this in more detail later.  Another advantage, which we will discuss in this section, is that Morse-Bott tori give rise to further topological constraints on the existence of holomorphic curves.

The three main topological constraints are positivity of intersection, the Blocking Lemma and the Trapping Lemma.  These are all due to Colin, Ghiggini and Honda and will be detailed in \cref{pos_of_intersection_section,blocking_trapping_lemmas_section}.  Before stating these results we must build up a notion of ``slope'' and will develop an original tool which we call ``slope calculus''.

\subsection{Slope calculus}\label{section_slope_calculus}

Let $\alpha$ be a contact form on $M$ and suppose that $T$ is an oriented torus which is foliated by Reeb trajectories.  If the foliation consists of closed leaves then it is required to be a Morse-Bott family of orbits; if the foliation consists of dense leaves then we say that the torus is \emph{foliated by Reeb trajectories of irrational slope}.

The \emph{slope} of $T$ is an element of
\[ \mathbb{P}^+H_1(T; \R) = (H_1(T;\R)\sminus\set{0})/\!\sim,\]
where $\sim$ is the equivalence relation generated by multiplication by positive real numbers.  If $T$ has closed leaves then the slope $\delta$ is rational and defined simply as $[\gamma]$ where $\gamma$ is one of the orbits in $T$.  If $T$ has dense leaves then $\delta$ is irrational and can be defined by considering arbitrarily long Reeb trajectories in $T$.

\begin{lemma}[{\nogapcite[Lemma 5.1.1]{CGH_ECH_OBD}}]\label{finite_singularities}
The map $u_M:=\pi_M\circ u$ is transverse to the Reeb vector field $R$ away from a finite number of points in the domain of $u$.  In particular, it is an immersion away from a finite number of points.
\end{lemma}
In this section we will often, by abuse of notation, refer to the image of the map $u_M$ also by $u_M$.

Now let $T$ be any oriented torus in $M$.  Then by \cref{finite_singularities} any holomorphic curve $u$ which has no ends intersecting $T$ has a finite number of singularities on $T$, and hence gives rise to a homology class
\[ [u_T]:=[u_M\cap T]\in H_1(T;\Z),\]
which we call the \emph{homology class of $u$ at $T$}, orientated as follows:  identify a neighbourhood of $T$ with $T\cross[-\epsilon, \epsilon]$ such that the coordinate $y\in[-\epsilon,\epsilon]$ is oriented as the normal direction to $T\cross \set{0}$ with respect to the orientation on $T$.  This convention means that $T\cross\set{0}$ is oriented as the boundary of $T\cross(-\epsilon, 0]$. We then orient $u_M\cap T=u_M\cap (T\cross\set{0})$ as the boundary of $u_M\cap (T\cross(-\epsilon,0])$.  The homology class $[u_T]$ also defines a slope in $\mathbb{P}^+H_1(T; \R)$.

\begin{prop}\label{fundamental_principle_of_slope_calculus}
Let $T$ be an oriented (not necessarily foliated) torus in $M$ and suppose that the positive end of $u$ has multiplicity $m$ at some simple Reeb orbit $\gamma\subset T$, and that it has no other ends in the region $T\cross[-\epsilon, \epsilon]$.  Then 
\[ [u_{\epsilon}] = [u_{-\epsilon}] - m[\gamma] \in H_1(T\cross[-\epsilon, \epsilon]; \Z), \]
where $\gamma$ is oriented by the Reeb vector field $R$, $u_y := u_{T\cross\set{y}}$ and each $T\cross\set{y}$ is oriented as $T$.
If instead it is the negative end of $u$ with multiplicity $m$ at $\gamma$ then
\[ [u_{\epsilon}] = [u_{-\epsilon}] + m[\gamma]. \]
If $u$ has no ends in the region $T\cross[-\epsilon, \epsilon]$ then
\[ [u_{\epsilon}] = [u_{-\epsilon}] .\]
\end{prop}

\begin{proof}
First suppose that the positive end of $u$ has multiplicity $m$ at $\gamma\subset T$ and consider the surface
\[ S := \overline{u_M}\cap (T\cross[-\epsilon, \epsilon]),\]
which, considered at first set-wise without orientation, has boundary
\begin{align*} 
\del S &= \big((\del\overline{u_M})\cap (T\cross[-\epsilon, \epsilon])\big) \union \big(u_M\cap(T\cross\set{-\epsilon, \epsilon})\big) \\
       &= \gamma \union u_{-\epsilon} \union u_{\epsilon}. 
\end{align*}
We claim that when $S$ is considered as a 2-chain, its boundary is
\begin{equation}\label{boundary_of_S_formula}
dS = m\gamma - u_{-\epsilon} + u_{\epsilon}.
\end{equation}
To see this, first note that $S$ inherits a canonical orientation from $u$, as does every $\overline{u_M}\cap(T\cross[y_1,y_2])$.  Then the orientation on $u_\epsilon$ coming from $S$ agrees with that coming from $\overline{u_M}\cap(T\cross(\epsilon/2,\epsilon])$ which is equal to that in the definition of $[u_\epsilon]=[u_{T\cross\set{\epsilon}}]$ above.

The orientation on $u_\epsilon$ coming from $S$ agrees with that coming from $\overline{u_M}\cap(T\cross[-\epsilon,-\epsilon/2])$.  This is the opposite orientation to that in the definition of $[u_{-\epsilon}]=[u_{T\cross\set{-\epsilon}}]$ above, explaining the minus sign in \cref{boundary_of_S_formula}.

Since every end of $u$ at $\gamma$ is a positive end, and $S$ is oriented by $J(\del_s) = R$, the induced orientation on $\gamma$ agrees with the orientation of $R$.  Noting that the total multiplicity of $u$ at $\gamma$ is $m$ completes our justification of \cref{boundary_of_S_formula}.

On the level of homology, \cref{boundary_of_S_formula} becomes
\[ m[\gamma] - [u_{-\epsilon}] + [u_{\epsilon}] = 0 \in H_1(T\cross[-\epsilon, \epsilon]; \Z), \]
which completes the proof.

The case where $u$ has negative ends at $\gamma$ follows from the fact that the induced orientation from $S$ on negative ends of $u$ is opposite to that of $R$.  When $u$ has no ends in the region $T \cross[-\epsilon, \epsilon]$ the result is clear.
\end{proof}

The above proposition gives rise to a theory of ``slope calculus'':

\begin{prop}[Slope calculus]\label{slope_calculus}
Suppose that some region of $M$ can be identified with $T^2\cross[0,1]$ and $u\in\mathcal{M}_J(\Gamma,\Gamma')$ is some holomorphic curve such that every end $\gamma$ of $u$ in the region $T^2\cross[0,1]$ lies in some slice $T^2\cross\set{y_\gamma}$.  Write $\Gamma=\set{(\gamma_i, m_i)}$ and $\Gamma'=\set{(\gamma_i', m_i')}$.

Then for all $y$ such that $u$ has no ends in $T\cross\set{y}$, the homology class $[u_y] = [u_{T\cross\set{y}}]$ can be computed as follows:
\begin{equation}
[u_y] = [u_0] - \sum_{\set{i|y_{\gamma_i}<y}}m_i[\gamma_i] + \sum_{\set{i|y_{\gamma_i'}<y}}m_i'[\gamma_i'].
\end{equation}
\end{prop}

The proof follows immediately from \cref{fundamental_principle_of_slope_calculus}.  Despite the name, slope calculus actually computes the homology class of each slice, $[u_y]$.  However we will use it later to compute the \emph{slope} of $u$ at each torus $T\cross\set{y}$, which proves to be a very useful topological constraint when applied to positivity of intersection and the Blocking Lemma.  The computations above also easily generalize to the theory of Morse-Bott buildings, meaning that we can also perform slope calculus in the Morse-Bott setting.

\subsection{Positivity of intersection}\label{pos_of_intersection_section}

\begin{prop}[Positivity of intersections in dimension 4~{\cite[Theorem 7.1]{MW95}}]\label{positivity_of_intersections_4}
Let $u, u'$ be two connected pseudoholomorphic curves with non-equal images and let $S(u,u')=\set{(p,p')\gappy{|}u(p)=u'(p')}$. Then every intersection point $(p,p')\in S(u,u')$ is isolated and the intersection number at $(p,p')$ is greater than or equal to 1, where equality holds if and only if $u$ and $u'$ are transverse immersions near $(p,p')$.
\end{prop}

Now suppose that $T$ is an oriented torus and let $\langle \cdot,\cdot \rangle$ denote the intersection pairing on $H_1(T;\R)$.

\begin{defn}
Suppose that $T$ is an oriented Morse-Bott torus of slope $s\in \mathbb{P}^+H_1(T; \R)$ and choose some $\delta\in H_1(T;\Z)$.  Then we say that $\delta\cdot s > 0$ (resp.~$\delta\cdot s = 0$) if $\langle \delta,\gamma \rangle>0$ (resp.~$\langle \delta,\gamma \rangle=0$) for every representative $\gamma$ of $s$.
\end{defn}

\begin{prop}[Positivity of intersections in dimension 3~{\cite[Lemma 5.2.2]{CGH_ECH_OBD}}]\label{positivity_of_intersections_3}
Suppose that $T\subset M$ is an oriented Morse-Bott torus of slope $s \in \mathbb{P}^+H_1(T; \R)$ and that $u$ is a holomorphic curve such that no ends of $u$ intersect $T$.  Then $[u_T]\cdot s \ge 0$, and furthermore $[u_T]\cdot s=0$ if and only if $u_T=\emptyset$.
\end{prop}
A brief proof of this proposition was given by Colin, Ghiggini and Honda; here we will flesh out the details.
\begin{proof}
First assume that $\gamma\subset T$ is a Reeb orbit with rational slope $s\in \mathbb{P}^+H_1(T; \Q)$.  Then $\R\cross\gamma$ is a holomorphic curve, so positivity of intersections in dimension 4 (\cref{positivity_of_intersections_4}) implies that the image of $u$ intersects $\R\cross\gamma$ positively at finitely many points.  Furthermore, by \cref{finite_singularities} we can assume that $\gamma \subset T$ is chosen such that it is transverse to $u_M$.

The proof that the corresponding intersection points in $u_T\cap \gamma$ contribute positively to $\langle [u_T],[\gamma]\rangle$ is essentially an exercise in orientations and basis manipulation.  To start, let  
\[ p \in u_T\cap \gamma \subset T \]
 be an intersection point and pick a basis
\[ (R,v_1, v_2) \]
for $T_pM$ such that
\begin{itemize}
\item $R$ is the value of the Reeb vector field at $p$,
\item $v_i\subset \ker(\alpha)=\xi$ for $i=1,2$,
\item $v_2\subset T_p(T)$, and
\item $v_2=\restr{J}{\xi}(v_1)$.
\end{itemize}
Recall that a neighbourhood of $T$ was identified with $T\cross(-\epsilon,\epsilon)$ in order to define the orientations on $T$ and $u_T:=u_M\cap T$ in \cref{section_slope_calculus}.  By negating both $v_i$ if necessary we can assume that this identification is made such that $v_1$ points in the \emph{positive} $y$ direction.

We claim that the basis above is positively oriented; this follows from the fact that $\alpha\wedge d\alpha$ is a positive volume form on $M$ and
\[ (\alpha\wedge d\alpha)(R,v_1,v_2) = \alpha(R)\cdot d\alpha(v_1,Jv_1) > 0 \]
since $R\subset \ker(d\alpha)$, $v_i\subset\ker(\alpha)$ and $\restr{J}{\xi}$ is compatible with $d\alpha$. 

Since we have transversality between $u_M$ and $R$ at $p$, there exist real numbers $a_1,a_2$ such that 
\[ a_1R+v_1,\text{ }a_2R+v_2 \subset T_p(u_M). \]
This in turn means that there exist real numbers $b_1, b_2$ such that
\[ b_1\del_s + a_1R+v_1,\text{ }b_2\del_s+a_2R+v_2 \subset T_{\hat{p}}(\mathrm{Im}(u))\subset T_{\hat{p}}(\R\cross M), \]
where here $\hat{p} \in \mathrm{Im}(u)\cap \R\cross \gamma$ is the 
corresponding intersection point in the symplectization.

Furthermore, as $J(v_1)=v_2$, we must have that
\[ J(b_1\del_s + a_1R+v_1) = b_2\del_s+a_2R+v_2\]
and hence $b_1=a_2$ and $b_2=-a_1$.  Moreover, since the orientation of $u_M$ is inherited from the action of $J$ on $u$, under the image of the derivative $d_{\hat{p}}\pi_M$ we obtain a positively oriented basis 
\[ (a_1R+v_1, a_2R+v_2) \]
for $T_p(u_M)$.  In particular, note that $a_2R+v_2$ is tangent to $u_T$. 

Now we have positive bases for both $T\cross(-\epsilon,0]$ and $u_M\cap ( T\cross(-\epsilon,0])$.  As a result, since $T$ and $u_T$ are oriented as their respective boundaries (c.f.~\cref{section_slope_calculus}), 
\begin{itemize}
\item $T_p(u_T)$ is oriented by $a_2R+v_2$ and
\item $T_p(T)$ is oriented by $(v_2,R)$, or equivalently by $(a_2R+v_2,R)$.
\end{itemize}
Of course, we also have that $T_p(\gamma)$ is oriented by $R$ and it therefore follows that the intersection point $p$ makes a positive contribution to $\langle [u_T],[\gamma]\rangle$ as required.

To complete the proof of the rational case, first note that if $u_T=\emptyset$ then $[u_T]\cdot s=0$. Conversely, if $u_T\neq\emptyset$ then $\gamma$ can be chosen such that there is at least one (positively contributing) intersection point $p$, and hence $[u_T]\cdot s>0$.

The case where the slope $s$ is irrational is discussed in the proof by Colin, Ghiggini and Honda.
\end{proof}

\subsection{The Blocking and Trapping Lemmas}\label{blocking_trapping_lemmas_section}

The Blocking Lemma below follows immediately from \cref{positivity_of_intersections_3}.

\begin{prop}[The Blocking Lemma~{\cite[Lemma 5.2.3]{CGH_ECH_OBD}}]\label{blocking_lemma}
Let $T\subset M$ be an oriented Morse-Bott torus of slope $s$ and $u$ a holomorphic curve. Then:
\begin{enumerate}
\item If $u$ is homotopic, relative to its ends, to a map whose image is disjoint from $T\cross\R$, then $u$ itself is disjoint from $T\cross\R$.
\item If we identify a neighbourhood of $T$ with $T\cross[-\epsilon, \epsilon]$ such that $u$ has no ends in $T\cross([-\epsilon, \epsilon]\sminus\set{0})$, and if one of $[u_{\pm\epsilon}]$ has slope $\pm s$, then $u$ has an end in $T$.
\end{enumerate}
\end{prop}

The Trapping Lemma does not follow as simply, requiring further analysis to understand the ends of holomorphic curves near Morse-Bott tori.  Here we will just state the result.  

\begin{prop}[The Trapping Lemma~{\cite[Lemma 5.3.2]{CGH_ECH_OBD}}]\label{trapping_lemma}
Suppose $T\subset M$ is a Morse-Bott torus and that $u$ is a holomorphic map with a one-sided end at some $\gamma\subset T$. Then if $T$ is positive (resp.~negative), the end is positive (resp.~negative).
\end{prop}

Here we are calling an end \emph{one-sided} if a small neighbourhood of the end does not intersect $T$, i.e.~its projection to $M$ lands entirely on one side of $T$.

\section{ECH groups for manifolds with boundary}

In this section we will use Morse-Bott homology to define some ECH groups for manifolds with boundary.  These definitions are originally due to Colin, Ghiggini and Honda~\cite[Section 7.1]{CGH_ECH_OBD}, although here we employ slightly different notation.

\begin{defn}
Suppose that $M$ is a manifold with torus boundary, and that $\alpha$ is a Morse-Bott contact form which is non-degenerate on $\mathrm{int}(M)$ and foliates $\del M$ by Reeb trajectories of either irrational or rational slope.  Let $\mathcal{P}$ denote the set of simple orbits of $R_\alpha$ in $\mathrm{int}(M)$ and $J$ be a regular almost complex structure adapted to $\alpha$.
\begin{enumerate}[series=ECH_boundary_defs]
\item $ECH(\mathrm{int}(M),\alpha, J)$ is defined to be the homology of $ECC(\mathrm{int}(M),\alpha,J)$, the chain complex generated by orbit sets constructed from simple orbits in $\mathcal{P}$.
\item $ECH(M,\alpha,J)$ for $\del M$ irrational is defined to be $ECH(\mathrm{int}(M),\alpha,J)$.
\end{enumerate}
\end{defn}
If $\del M$ is foliated by closed leaves of rational slope then we have a Morse-Bott family of orbits.  Recall that, in the closed case, Morse-Bott homology is defined by choosing two orbits $e$, $h$ in each Morse-Bott torus and counting Morse-Bott buildings between orbit sets constructed from non-degenerate orbits and such $e$ and $h$.  We perform the same procedure in the situation with boundary:
\pagebreak %
\addtocounter{equation}{-1}
\begin{defn}[continued]
Suppose that $\del M$ is a Morse-Bott family of closed orbits and choose two orbits $e$ and $h$ in $\del M$.  
\begin{enumerate}[resume=ECH_boundary_defs]
\item $ECH(M,\alpha,J)$ is defined to be the homology of $ECC_\MB(M,\alpha,J)$, the Morse-Bott complex which is generated by orbit sets constructed from $\mathcal{P}\union\set{e,h}$ as in the definition of Morse-Bott contact homology, \cref{MB_homology_def}.
\item Let $\mathcal{P}_\del$ be a subset of the collection $\set{e,h}$.  We say that $e$ (or $h$) is \emph{included} if it is in $\mathcal{P}_\del$. Consider the chain complex generated by orbit sets constructed from $\mathcal{P}\union\mathcal{P}_\del$, and denote it by $ECC^e_h(\mathrm{int}(M),\alpha,J)$, where we omit $e$ or $h$ from the notation if they are not included. For example, $ECC_\MB(M,\alpha,J)=ECC^e_h(\mathrm{int}(M),\alpha,J)$ and $ECC^e(\mathrm{int}(M),\alpha,J)$ is the complex which Colin, Ghiggini and Honda denote by $ECC^\flat(M,\alpha,J)$~\cite[Section 7.1]{CGH_ECH_OBD}.  The corresponding homology groups are denoted by replacing ECC with ECH.
\end{enumerate}
\end{defn}
These complexes can be shown to satisfy $d^2=0$ using the Blocking and Trapping lemmas~\cite[Section 7.1]{CGH_ECH_OBD}.  Note that we do not claim that any of these homology groups are independent of the chosen almost complex structures. 

We will finish this chapter by defining the relative embedded contact homology groups as seen in the statement of \cref{relative_ECH_equals_ECH}.  This is a generalization of Colin, Ghiggini and Honda's construction in the case of an integral open book decomposition~\cite[Remark 7.3.3]{CGH_ECH_OBD}.

\begin{defn}\label{relative_ECH_groups}
Suppose that $M$ is a manifold with torus boundary, and that $\alpha$ is a Morse-Bott contact form which is non-degenerate on $\mathrm{int}(M)$, and $\del M$ is a Morse-Bott torus of rational slope.  Let $J$ be a regular almost complex structure adapted to $\alpha$. We define the \emph{relative ECH groups} to be
\begin{align*}
ECH(M, \del M, \alpha,J) &:= ECH^{e}(\mathrm{int}(M),\alpha,J)/\!\sim \\
\widehat{ECH}(M, \del M, \alpha,J) &:= ECH^e_h(\mathrm{int}(M),\alpha,J)/\!\sim, 
\end{align*}
where $\sim$ is the equivalence relation generated by $e\sim \emptyset$.
\end{defn}

\chapter{ECH via rational open book decompositions}\label{ECH_via_rational_open_book_decompositions_chapter}

\section{Contact forms adapted to rational open book decompositions}\label{contact_forms_and_rational_open_book_decompositions}
\markright{9.\ \ CONTACT FORMS ADAPTED TO RATIONAL OBDs}

The first step to proving \cref{relative_ECH_equals_ECH} is to construct a contact form which is adapted in a suitable way to the rational open book decomposition
\[ M=\nu(K)\union N.\]
The form takes on a particular behaviour in $\nu(K)$ and this gives us full control over the Reeb orbits and holomorphic curves in this region---we then use a spectral sequence argument to compute $ECH(M)$ solely in terms of Reeb orbits and holomorphic curves in $N$.  This argument is completely analogous to that employed by Colin, Ghiggini and Honda~\cite[Section 6]{CGH_ECH_OBD} in the case of integral open books; we make some subtle adjustments in the rational case and also some generalizations which prove useful in \cref{a_knot_version_of_ECH_chapter}.  

Recall that $N=M(\phi)$ is the mapping torus corresponding to the rational open book decomposition and $\nu(K)\iso S^1\cross D^2$ is the neighbourhood of the binding $K=S^1\cross\set{0}\subset M$.  The monodromy $\phi$ rotates $\del\Sigma$ by a $p/q$-twist.

We further decompose $\overline{\nu(K)}$ into two pieces
\[ \overline{\nu(K)} = V\union T^2\cross[1,2]. \]

Here $V$ itself is homeomorphic to $S^1\cross D^2$ and is parametrized by cylindrical coordinates $(t_0,\theta_0,y_0)$, where $\theta_0\in \R/\Z$ parametrizes the $S^1$ direction, and $(t_0,y_0)\in \R/\Z \cross [0,1]$ are polar coordinates for $D^2$.  Write $T_{\rho}:= \set{y=\rho}$.

The piece $T^2\cross[1,2]$ is parametrized by $(t_1,\theta_1,y_1) \in (\R/\Z)^2\cross[1,2]$ and the two pieces are glued by the identity homeomorphism
\begin{align*}
T_1 &\to T^2\cross\set{1}\\
(t,\theta,1) &\mapsto (t,\theta,1).
\end{align*}
$T^2\cross \set{2}=\del \overline{\nu(K)}$ is identified with $\del N$ by a more complex homeomorphism.  Recall that $M$ is obtained by surgery along the degeneracy slope $d \subset \del N$.  We parametrize $\del N\homeo T^2$ in a way which comes from the mapping cylinder construction and is therefore a non-standard parametrization of $T^2$.

Recall that the mapping cylinder is constructed as $\Sigma\cross[0,1]/\!\sim$ where the equivalence relation is given by 
\[ (x,1)\sim (\phi(x),0) \]
and $\restr{\phi}{\del \Sigma}=r_{p/q}$, a $p/q$-fractional twist.  An equivalent definition is
\[ N = 	\Sigma\cross\R/\!\sim, \quad (x,t+1)\sim (\phi(x),t), \]
and this gives rise to the following parametrization of $\del N$:
\[ \del N = (\R\cross \R/\Z) / \!\sim, \quad (t_2+1,\theta_2)\sim( t_2, \theta_2+\frac{p}{q}). \]
This can be thought of as a tiling of $\R^2$ by parallelograms: let $t_2$ denote the horizontal coordinate and $\theta_2$ the vertical coordinate.  Then cut $\R^2$ into parallelograms along vertical lines $\set{t_2=m}_{m\in\Z}$ and lines of slope $-\frac{p}{q}$ given by $\set{\theta_2=-\frac{p}{q}t_2+m}_{m\in\Z}$.  See \cref{R2_parallelogram_tiling}.

\begin{figure}\centering
  \begin{tikzpicture}%
		\draw [->,thick] (0,0) -- (0,2.8) node [above left]  {$\theta_2$};
		\draw [->,thick] (0,0) -- (2.8,0) node [below right] {$t_2$};
		{
		\path [clip] (-0.4,-0.5)--(2.4,0.2)--(2.4,3)--(-0.4,2.3)--cycle;
		\draw [thin] (0,-2) -- (0,4);
		\draw [thin] (1,-2) -- (1,4);
		\draw [thin] (2,-2) -- (2,4);
		\draw [thin] (-2,-0.5) -- (4,1);
		\draw [thin] (-2,0.5) -- (4,2);
		\draw [thin] (-2,1.5) -- (4,3);
		}
	\end{tikzpicture}
  \caption{The tiling of $\R^2$ in the case $p/q=-1/4$.}
	\label{R2_parallelogram_tiling}
\end{figure}
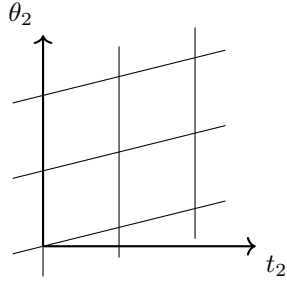

Now let $m\in\Z$ be such that $mp+1\equiv0 \gappy{\mathrm{ mod }}q$ and consider the matrix
\[ A=
\begin{pmatrix}
q & m \\
0 & \frac{1}{q} 
\end{pmatrix}
\] 
as a homeomorphism $\R^2\to\R^2$.  This descends to a map $T^2\cross\set{2} \to \del N$ (also denoted by $A$), since for $(n,n')\in\Z^2$,
\[\begin{split}
A(n,n') &= (qn+mn', \frac{n'}{q}) \\
&\sim (0, \frac{n'}{q} + (qn+mn')\frac{p}{q}) \\
&= (0, np + n'\frac{1+mp}{q}) \\
&\sim (0,0)
\end{split}\]
as $\frac{1+mp}{q}\in\Z$.  It is easy to see that $A^{-1}$ also descends to the quotient spaces and hence $A$ is a homeomorphism between the two tori.

We use $A$ to identify $T^2\cross\set{2}$ and $\del N$.  To see that this is the correct choice, we must check that it identifies the meridian slope of $K$ with the degeneracy slope $d\in\del N$.  By considering the identification of $V$ with $T^2\cross[1,2]$, it is easy to see that the curve $t\mapsto (t,0,2)$ is a meridianal curve for $K$.  Under the homeomorphism $A$ this curve is identified with the curve $t\mapsto (qt,0)\in \del N$, which is precisely the degeneracy slope.

\subsection{Contact forms on \texorpdfstring{$V$}{V}, \texorpdfstring{$T^2\cross[1,2]$}{T²x[1,2]} and  a neighbourhood of \texorpdfstring{$\del N$}{∂N}}

In this section, we will discuss families of smooth one-forms on $V$, $T^2\cross[1,2]$ and $\nu(\del N)$ in terms of the coordinates at the start of \cref{contact_forms_and_rational_open_book_decompositions}. We will then understand under what conditions these give rise to contact forms and compute their respective Reeb vector fields.

\paragraph{1. ($V$)}On $V$, we consider one-forms which can be written as
\[ f(y_0)\dd t_0 + g(y_0)\dd\theta_0 \]
for some $f,g:[0,1]\to\R$.  For smoothness to hold near $\set{y_0=0}$ we must have $f(0)=0$ and all derivatives of odd degree of both $f$ and $g$ at $\rho=0$ must vanish.  (This requirement can be easily verified via a simple change of coordinates, from polar to cartesian.)

\paragraph{2. ($T^2\cross[1,2]$)} On $T^2\cross[1,2]$ we consider 1-forms which can be similarly written as
\[ f(y_1)\dd t_1 + g(y_1)\dd \theta_1 \]
for some $f,g:[1,2]\to\R$.

\paragraph{3. ($\nu(\del N)$)}On $N$, we consider 1-forms which, on a neighbourhood $\del N \cross[2,2+\epsilon)$ of $\del N$, take the form
\[ f(y_2)\dd t_2 + g(y_2)\dd \theta_2, \]
where here $y_2$ denotes the $[2,2+\epsilon)$ coordinate.

\begin{lemma}\label{contact_condition_lemma}
The above forms on $T^2\cross[1,2]$ and $\del N\cross[2,2+\epsilon)$ are contact forms if and only if
\begin{equation}\label{contact_condition_eq}   
gf' - fg' > 0
\end{equation}
for all $y_1\in[1,2]$ and $y_2\in[2,2+\epsilon)$. The above form on $V$ is a contact form if and only if \cref{contact_condition_eq} holds for all $y_0\in(0,1]$ and
\[ \lim_{y_0\to0} \frac{gf' - fg'}{y_0}>0 .\]
\end{lemma}
\begin{proof}
Since all three pieces are oriented the same way, the calculations are identical aside from notation.  For this reason we will omit the subscripts for the purposes of the proof.
\[\begin{split}
\alpha\wedge d\alpha &= \left(f\dd t + g\dd\theta\right)\wedge \left( f'\dd y\wedge dt + g'\dd y\wedge d\theta \right) \\
                     &= f g'\dd t\wedge dy\wedge d\theta + gf' \dd\theta \wedge dy\wedge ds \\
										 &= \left(gf' - fg'\right) \dd t\wedge d\theta\wedge d y
\end{split}\]
Hence $\alpha$ satisfies the contact condition if and only if $gf' - fg' > 0$.  The extra condition as $y_0$ tends to $0$ on $V$ follows by considering a simple change of coordinates, from polar to cartesian.
\end{proof}

\begin{lemma}\nogapcite[Lemmas 6.1.1 and 6.2.2]{CGH_ECH_OBD}\label{reeb_vector_field_lemma}
The Reeb vector fields on $T^2\cross[1,2]$ and $\del N\cross[2,2+\epsilon)$ are given by
\[ R = \frac{1}{gf'-fg'}\left( -g'\del_t + f'\del_\theta \right) \]
where we have omitted the subscripts for ease of notation.

The Reeb vector field on $V$ is given by the same formula above away from $\set{y_0=0}$ and is parallel to $\del_\theta$ when $y_0=0$.
\end{lemma}

There is a nice geometrical interpretation of the above two lemmas:  if we plot the parametric curve $(f,g)$ in the plane then \cref{contact_condition_lemma} implies that the curves must move \emph{strictly clockwise} around the origin.  \cref{reeb_vector_field_lemma} implies that the \emph{slope} of the Reeb vector field (defined with respect to the basis $(\del_\theta, \del_t)$) is always equal to the \emph{negative of the gradient of the curve}. See \cref{contact_forms_on_no_mans_land_and_V_fig}. These two facts allow us to very quickly construct contact forms on these three pieces of $M$ and immediately understand their Reeb vector fields.

\subsection{Constructing a contact form adapted to the open book decomposition}\label{contact_forms_on_3_pieces}

In this section we will define families of contact forms on $V$ and $T^2\cross[1,2]$ which will be used as building blocks in the next section when constructing forms on $M$.  We will also discuss the boundary conditions which must be satisfied by contact forms on $N$ in order to glue nicely with these families of forms. 

The forms we use on $V$ and $T^2\cross[1,2]$ are based on those defined by Colin, Ghiggini and Honda~\cite[Section 6]{CGH_ECH_OBD}.  We will begin with a form on $V$.

\begin{defn}\label{contact_form_on_V_defn}
Let $\mu\in\R$ and define a contact form $\alpha_{V,\mu}$ on $V$ by
\[ \alpha_{V,\mu} = f(y_0)\dd t + g(y_0) \dd \theta,\]
where:
\begin{enumerate}
\item $f$ and $g$ satisfy the contact condition in \cref{contact_condition_lemma}.
\item $(f(y_0),g(y_0))= (y_0^2, C-y_0^2)$ near $y_0=0$, where $C>0$ is a large constant.
\item $(f(y_0),g(y_0)) = (1 - (y_0-1)^2, \mu - (y_0-1))$ near $y_0=1$. \label{gluing_condition_V_to_no_mans_land}
\item The slope of the parametrized curve $(f,g)$, $\frac{g'}{f'}$ is monotonic, decreasing from $-1$ to $-\infty$ as $y_0$ increases from 0 to 1.
\end{enumerate}
See \cref{contact_forms_on_no_mans_land_and_V_fig_a}.
\end{defn}

\begin{figure}\centering
	\begin{subfigure}[b]{0.45\textwidth}\centering
		\begin{tikzpicture}%
			\draw [->] (0,0) -- (0,4.8) node [above left] {$g$};
			\draw [->] (0,0) -- (4.8,0) node [below right] {$f$};
			\draw [thick,-*,shorten >=-3pt] (0,3.75) to [out=-45,in=135] (1.5,2.25) to [out=-45,in=90] (2,1) node [below right] {$(f(1),g(1))=(1,\mu)$};
			\draw (0,0)--(4,2) node [above] {slope $=\mu$};	
			\draw (0,3.75) -- (-0.1,3.75) node [left] {$C$};
			
		\end{tikzpicture}
		\caption{}
		\label{contact_forms_on_no_mans_land_and_V_fig_a}
	\end{subfigure}
	\begin{subfigure}[b]{0.45\textwidth}\centering
		\begin{tikzpicture}%
			\draw [->] (0,0) -- (0,4.8) node [above left] {$g$};
			\draw [->] (0,0) -- (4.8,0) node [below right] {$f$};
			\draw [thick,*-*,shorten >=-3pt,shorten <=-3pt](1.5,1) to [out=90,in=-90] (2,3);
			\draw (0,0)--(3,4.5);
			\node [right,inner sep=4pt] at (1.5,1) {$(f_\delta(2),g_\delta(2))=(a_1,b_1)$};
			\node [right,inner sep=4pt] at (2,3) {$(f_\delta(1),g_\delta(1))$};
			\node [right] at (2.75,4) {slope $=\mu$};	
		\end{tikzpicture}
		\caption{}
		\label{contact_forms_on_no_mans_land_and_V_fig_b}
	\end{subfigure}
  \caption{The contact forms on $V$ and $T^2\cross[1,2]$ from \cref{contact_form_on_V_defn,contact_form_on_no_mans_land_def} respectively.  The curve in (a) is linear near $y_0=0$ and parabolic near $y_0=1$.  The curve in (b) is parabolic near $y_1=1$ and $y_1=2$ and its slope is never shallower than $1/\delta$.}
	\label{contact_forms_on_no_mans_land_and_V_fig}
\end{figure}
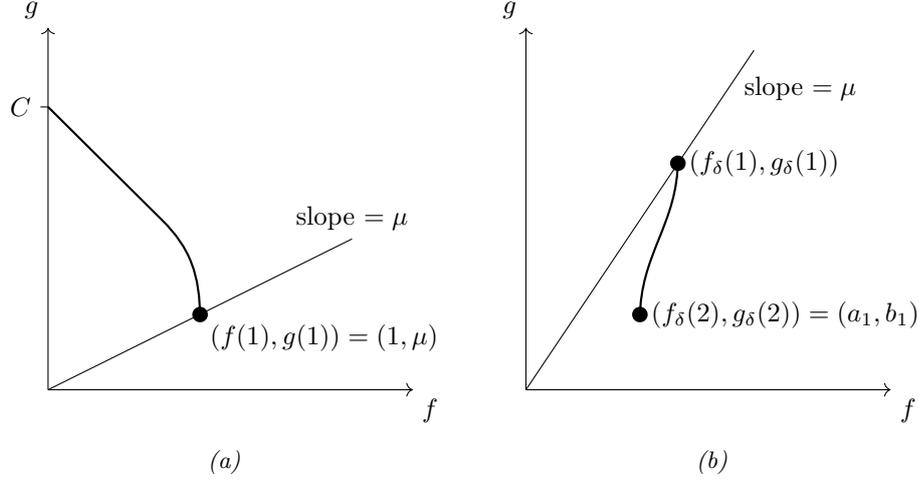

We now turn to the next piece, $T^2\cross[1,2]$.

\begin{defn}\label{contact_form_on_no_mans_land_def}
Fix $a_1, c>0$ and $b_1\in\R$ and let $\mu\in \R$ be such that the ray from the origin of slope $\mu$ in $\R_{>0}\cross\R$ lies above the point $(a_1,b_1)$.  Then for all small irrational $\delta>0$, define a contact form on $T^2\cross[1,2]$ by 
\[ \alpha_\delta = f_\delta(y_1)\dd t_1 + g_\delta(y_1)\dd\theta_1, \]
where:
\begin{enumerate}
\item $f_\delta$ and $g_\delta$ satisfy the contact condition in \cref{contact_condition_lemma}.
\item $(f_\delta(y_1), g_\delta(y_1)) = (a_1 + c(y_1-2)^2,b_1 - c(y_1-2))$ near $y_1=2$.
\item \label{contact_form_on_no_mans_land_def_c}The point $(f_\delta(1), g_\delta(1))$ lies on the ray of slope $\mu$ from the origin in $\R_{>0}\cross\R$ and in particular $f_\delta(1)>0$.
\item \label{contact_form_on_no_mans_land_def_d}$(f_\delta(y), g_\delta(y)) = (f_\delta(1) - c_\delta(y-1)^2,g_\delta(1) - c_\delta(y-1))$ near $y_1=1$, where $c_\delta=f_\delta(1)$.
\item The slope of the parametrized curve $(f_\delta,g_\delta)$, $\frac{g_\delta'}{f_\delta'}$, is never shallower than $\frac{1}{\delta}$.  It is infinite at $y_1=1$ and $y_1=2$, equal to $\frac{1}{\delta}$ at $y_1=3/2$, and strictly monotonic on $\set{1\le y_1\le \frac{3}{2}}$ and $\set{\frac{3}{2}\le y_1\le 2}$.
\end{enumerate}
See \cref{contact_forms_on_no_mans_land_and_V_fig_b}.
\end{defn}

Points \cref{contact_form_on_no_mans_land_def_c,contact_form_on_no_mans_land_def_d} imply that near $y_1=1$, 
\[ (f_\delta(y), g_\delta(y)) = c_\delta \cdot \left(1 - (y_1-1)^2, \mu - (y_1-1)\right) \]
and hence every $\alpha_\delta$ glues smoothly to the form $c_\delta\alpha_{V,\mu}$ on $V$.  Note that when we perform this gluing we obtain a positive Morse-Bott torus of infinite slope at $\del V=T^2\cross\set{1}$ (this follows from \cref{reeb_vector_field_lemma}).

Also note that our definition is slightly more general than that of Colin, Ghiggini and Honda, since we have introduced the positive constant $c$.  In this chapter we do not require this generalization, so will take $c=1$, but later in \cref{a_knot_version_of_ECH_chapter} we will need this slightly more general construction so that we can rescale forms by a positive constant when constructing product cobordisms (c.f.~\cref{product_region_defn}).

We will now introduce a second form on $V$ which is a small perturbation of $\alpha_{V\mu}$ to make it non-degenerate on $\mathrm{int}(V)$---the perturbed form will be denoted $\alpha_{V,\mu}'$.  The perturbation is performed in such a way that $c_\delta\alpha_{V\mu}'$ still glues smoothly to $\alpha_\delta$ on $T^2\cross[1,2]$ and $\del V$ remains a positive Morse-Bott torus.  In addition, the form $\alpha_{V,\mu}'$ satisfies a technical condition~\cite[Lemma 8.1.1]{CGH_ECH_OBD} ensuring the existence of tori $T_{\rho_i}$, with $\rho_i$ tending to $1$, foliated by Reeb trajectories of irrational slope---these tori are used by Colin, Ghiggini and Honda to compute the relative ECH of $V$ (c.f.~\cref{int_V_orbits_cancel}) and will also be used later in the proof of \cref{cobordism_respects_eta_filtration}.

We will end this section by discussing the conditions which must be satisfied by any form on $N$ near $\del N$ in order to glue nicely to $\alpha_\delta$ under the identification $A$.  Recall that on a neighbourhood of $T^2\cross\set{2}$, say $T^2\cross(2-\epsilon,2]$, $\alpha_\delta$ takes the form
\[\alpha_\delta = (a_1+c(y_1-2)^2)\dd t_1 + (b_1-c(y_1-2))\dd \theta_1 \]
for some $a_1,c>0$, $b_1\in\R$.  Extend this form to $T^2\cross(2-\epsilon,2+\epsilon)$, and identify this region with $\del N\cross(2-\epsilon,2+\epsilon)$ by the map 
\[ A: T^2\cross(2-\epsilon,2+\epsilon) \to \del N \cross(2-\epsilon,2+\epsilon). \]
The pull-back of $\alpha_\delta$ by $A^{-1}$ is given by
\begin{align}
(A^{-1})^*\alpha_\delta &= \bigg(a_1+c(y_2-2)^2\bigg)\left(\frac{1}{q}\dd t_2-m\dd\theta_2\right) + \bigg(b_1-c(y_2-2)\bigg)\bigg(q\dd \theta_2\bigg) \nonumber \\ 
                 &= \left(\frac{a_1}{q} + \frac{c}{q}(y_2-2)^2\right)\dd t_2 + \bigg(qb_1-ma_1 - qc(y_2-2) - mc(y_2-2)^2\bigg)\dd\theta_2 \nonumber \\
								&= \left(a_2 + \frac{c}{q}(y_2-2)^2\right)\dd t_2 + \bigg(b_2 - qc(y_2-2) - mc(y_2-2)^2\bigg)\dd\theta_2, \label{contact_form_on_nu_del_N}
\end{align}
where here we have set 
\begin{equation}\label{a_1_b_1_a_2_b_2_relations}
\begin{split}
a_2 &= \frac{a_1}{q}>0,\quad\text{and}\\
b_2&=qb_1-ma_1.
\end{split}
\end{equation}
(Recall that $q$ and $m$ are the integers used to define the transformation $A$ in \cref{contact_forms_and_rational_open_book_decompositions}.)

\begin{lemma}\label{contact_form_near_del_N_lemma}
Suppose that $\alpha$ is a contact form on $N$ which takes the form of \cref{contact_form_on_nu_del_N} on a neighbourhood $\del N\cross[2,2+\epsilon)$ of $\del N$.

Then provided the identification is performed with a sufficiently small $\epsilon>0$, the Reeb vector field in this neighbourhood has infinite slope at $y_2=2$ and has strictly decreasing positive slope on $\set{2< y_2 < 2+\epsilon}$. See \cref{contact_form_on_nu_del_N_fig}.
\end{lemma}

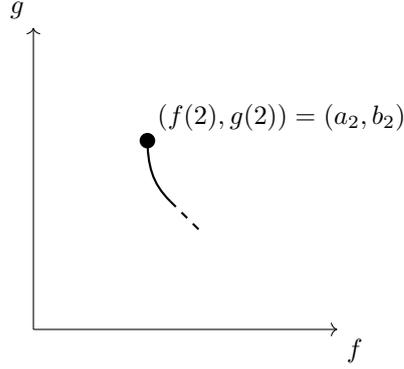
\begin{figure}\centering
  \begin{tikzpicture}%
		\draw [->] (0,0) -- (0,4) node [above left] {$g$};
		\draw [->] (0,0) -- (4,0) node [below right] {$f$};
		\draw [thick,*-,shorten <=-3pt](1.5,2.5) node [above right] {$(f(2),g(2))=(a_2,b_2)$} to [out=-90,in=135] (1.8,1.7);
		\draw [thick,dashed] (1.8,1.7) to (2.2,1.3);
	\end{tikzpicture}
  \caption{The form on $\del N\cross[2,2+\epsilon)$ as described in \cref{contact_form_near_del_N_lemma}.  The slope is infinite at $y_2=2$ and rotates as we move away from the boundary.}
	\label{contact_form_on_nu_del_N_fig}
\end{figure}

\begin{proof}
On $\del N\cross[2,2+\epsilon)$, $\alpha$ takes the form
\[ f(y_2)\dd t_2 + g(y_2)\dd\theta_2, \]
where
\begin{align*}
f(y_2) &= a_2 + \frac{c}{q}(y_2-2)^2,\text{ and}\\
g(y_2) &=b_2 - qc(y_2-2) - mc(y_2-2)^2.
\end{align*}
We know that the contact condition of \cref{contact_condition_lemma} is satisfied since this is the pullback of the contact form $\alpha_\delta$.  This fact can also be verified easily since
\[ gf'-fg' = a_2qc+2c(\frac{b_2}{q}+a_2m)(y_2-2)-c^2(y_2-2)^2, \]
which is positive for small values of $y_2-2>0$ as $a_2$, $c$ and $q$ are positive.

The slope of the parametrized curve is
\[ \frac{g'}{f'} = \frac{-q^2}{2(y_2-2)}-mq, \]
which is negative for small $y_2-2>0$, and asymptotic to $-\infty$ as $y-2$ tends to $0$ from above.

Finally, by \cref{reeb_vector_field_lemma}, the Reeb vector field on $\del N\cross[2,2+\epsilon)$ is parallel to
\[ -g'\del t_2 + f'\del\theta_2, \]
and hence satisfies the claims in the statement of the lemma.
\end{proof}

Note that as a result of this lemma, we see that the torus $\del N$ is a \emph{negative} Morse-Bott torus.

\begin{defn}\label{extendable_contact_form}
We say that a contact-form $\alpha$ on $N$ is \emph{extendable to the ($p/q$-rational) open book} $M$ if it takes the form of \cref{contact_form_on_nu_del_N} on a neighbourhood of $\del N$, or if it agrees with \cref{contact_form_on_nu_del_N} to infinite order along $\del N$ (and hence can still be glued to the forms $\alpha_\delta$ smoothly).
\end{defn}

\subsection{A contact form on \texorpdfstring{$N$}{N}, and its extension to \texorpdfstring{$M$}{M}}\label{a_contact_form_on_N}

In this section we will construct a contact form $\alpha$ on $N$ which is extendable to the $p/q$-rational open book $M$, taking the form of \cref{contact_form_on_nu_del_N} for $c=1$ and some $a_2>0$ and $b_2\in\R$.  We can then extend $\alpha$ to a contact form on $M$ as follows.

First, referring to \cref{a_1_b_1_a_2_b_2_relations}, set
\begin{align*}
a_1 &= qa_2 >0,\quad\text{and}\\
b_1 &= \frac{b_2}{q}+ma_2.
\end{align*}
Next choose an appropriate slope $\mu$ as in \cref{contact_form_on_no_mans_land_def}.  Then for all $\delta>0$, the forms $\alpha$, $\alpha_\delta$ and $c_\delta\alpha'_{V,\mu}$ glue together to form a valid contact form on $M$.  Refer to \cref{contact_form_on_nu_K_a}.

To construct $\alpha$, first recall that $N$ is the mapping cone of some diffeomorphism $\phi:\Sigma\to\Sigma$ such that $\restr{\phi}{\del\Sigma}=r_{p/q}$.  As in the previous section, identify a neighbourhood of $\del N$ with $\del N\cross[2,2+\epsilon)$.  We will parametrize this neighbourhood by coordinates
\[ (t,\theta,y), \]
dropping the subscripts ``2'' used in the previous sections. We will now construct a contact form on $N$, by an adaptation of the argument used by Colin, Ghiggini and Honda~\cite[Section 9.3.1]{CGH_ECH_OBD}.

Let $d_\Sigma$ denote the exterior derivative on $\Sigma$. Take a 1-form $\beta$ on $\Sigma$ such that $\omega:=d_\Sigma\beta$ is a positive area form for $\Sigma$ and with the boundary constraint that $\beta = c(2-y)\dd\theta$ on a neighbourhood $\del \Sigma\cross[2,2+\epsilon)$ of $\del \Sigma$.  Here $y$ denotes the $[2,2+\epsilon)$ coordinate and $c>0$ is a small constant.

We know that at the boundary the monodromy takes the form
\[\phi(\theta,2) = (r_{p/q}(\theta),2)=(\theta+p/q,2).\]
We say that a map $\Sigma\to\Sigma$ \emph{restricts to $r_{p/q}$ on a neighbourhood of $\del\Sigma$} if there exists some $\epsilon'>0$ such that
\[ \phi(\theta,y) = (r_{p/q}(\theta),y) = (\theta+p/q,y) \]
on $\del \Sigma\cross[2,2+\epsilon')$; by perturbing the diffeomorphism $\phi$, we may assume that it satisfies this requirement.  Let $\mathrm{Symp}(\Sigma, r_{p/q}, \omega)$ denote the group of symplectomorphisms of $(\Sigma,\omega)$ which restrict to $r_{p/q}$ on a neighbourhood of $\del \Sigma$, noting that $\phi$ is not necessarily in $\mathrm{Symp}(\Sigma,r_{p/q},\omega)$.

\begin{lemma}
$\phi$ is isotopic, through maps restricting to $r_{p/q}$ on a neighbourhood of $\del\Sigma$, to an element of $\mathrm{Symp}(\Sigma, r_{p/q}, \omega)$.
\end{lemma}
\begin{proof}
Since $\phi$ is an orientation-preserving diffeomorphism, $\phi^*\omega=f\omega$ for some $f>0$.  Then 
\[ \omega_t=t\phi^*\omega+(1-t)\omega= (tf + 1-t)\omega \]
is an area form for all $t\in[0,1]$, and we can apply Moser's trick to obtain an isotopy $\Phi_t:\Sigma\to \Sigma$, equal to the identity map on a neighbourhood of $\del\Sigma$, such that
\[ \Phi_t^*(\omega_t) = \omega_0 = \omega. \]
Then, letting $\phi_t = \phi\circ \Phi_t$, we obtain
\[ \phi_1^*\omega = (\phi\circ\Phi_1)^*\omega = \Phi_1^* \phi^* \omega = \Phi_1^* \omega_1 = \omega. \tag*{\qedhere}\]
\end{proof}

\begin{lemma}
Any $\phi\in\mathrm{Symp}(\Sigma, r_{p/q}, \omega)$ is isotopic, through $\mathrm{Symp}(\Sigma, r_{p/q}, \omega)$, to an element $\phi_1$ such that
$\phi_1^*\beta - \beta = d_\Sigma f$ for some positive function $f$ on $\Sigma$ which is constant near $\del\Sigma$.
\end{lemma}
The proof of this is identical to that given by Colin, Ghiggini and Honda~\cite[Section 9.3.1]{CGH_ECH_OBD}, with the only difference arising from the fact that $\restr{\phi}{\del\Sigma}\ne \id$.  However it is still the case that $\phi^*\beta=\beta$ near $\del\Sigma$, so the result holds.

As a result of the above two lemmas we can assume that $N=M(\phi)$, where $\phi\in\mathrm{Symp}(\Sigma, r_{p/q}, \omega)$ has the property that $\phi^*\beta - \beta = d_\Sigma f$ for some $f$ which is identically equal to some constant $C>0$ near $\del\Sigma$. 

We now construct a contact form on $N$.  Start by taking a function $\chi:[0,1]\to[0,1]$ such that
\begin{itemize}
\item $\chi(0)=0$ and $\chi(1)=1$,
\item $\dot{\chi}\ge0$, and
\item $\chi$ is constant near $0$ and $1$.
\end{itemize}
Then choose some $C'>0$ and define $\beta_t := \chi(t)\phi^*\beta+(1-\chi(t))\beta$ and $f_t := \dot{\chi}(t)f+C'$, which is greater or equal to $C'$ everywhere.

We then define $\alpha_1 = f_t\dd t+\beta_t$ on $\Sigma\cross[0,1]$.  Near 0 and 1, $\chi(t)$ and $\beta_t$ are constant and $f_t\equiv C'$, and hence $\alpha_1$ descends to the mapping torus $M(\phi)$.  To verify that $\alpha_1$ is a contact form we make the following calculations.  First note that
\[ \dot{\beta}_t = \dot{\chi}(t)(\phi^*\beta - \beta) = \dot{\chi}(t)(d_\Sigma f) = d_\Sigma f_t \]
and
\[ d_\Sigma\beta_t = \chi\phi^*d_\Sigma\beta + (1-\chi)d_\Sigma\beta = \chi\phi^*\omega + (1-\chi)\omega = \omega. \]
Then
\[ d\alpha_1 = d_\Sigma f_t \wedge dt + d_\Sigma \beta_t+dt\wedge \dot{\beta}_t = d_\Sigma f_t \wedge dt + \omega +dt\wedge d_\Sigma f_t = \omega,\]
so $\alpha_1\wedge d\alpha_1 = f_t>0$ and hence $\alpha_1$ is a contact form with Reeb vector field $R_{\alpha_1}=f_t\del_t$.  We now go a step further than Colin, Ghiggini and Honda, and reparametrize the $t$ coordinate so that $\alpha$ is invariant in the $t$ direction near $\del N$.

More precisely, let
\[ t' = \frac{1}{C'+C}(C't + \chi(t)C), \]
then
\[ (C'+C)\dd t' = (C'+\dot\chi(t)C) \dd t \]
and hence near $\del N$, where $f_t = C'+\dot\chi(t)C$ and $\beta_t=\beta=c(2-y)\dd\theta$, we have
\[ \alpha_1 = (C'+C)\dd t' + c(2-y)\dd\theta. \]

The final stage in this section is to perturb $\alpha_1$ near $\del N$ so that it will glue nicely to the contact forms $\alpha_\delta$ on $T^2\cross[1,2]$.

On $T^2\cross[2,2+\epsilon)$, we currently have $\alpha_1=f(y)\dd t + g(y) \dd\theta$ where $f(y)=C'+C$ and $g(y)=c(2-y)$.  Extend $\alpha_1$ to $\tilde{N}:=T^2\cross[2-\epsilon,2]\union N$ by extending $f$ and $g$ to $[2-\epsilon,2]$ such that the following hold:
\begin{itemize}
\item $f$ and $g$ satisfy the contact condition in \cref{contact_condition_lemma},
\item $f$ and $g$ remain close to $f(2)$ and $g(2)$, and
\item near $y=2-\epsilon$,
\begin{align*}
f(y) &= f(2-\epsilon) + \frac{1}{q}(y-(2-\epsilon))^2,\quad\text{and} \\
g(y) &= g(2-\epsilon) - q(y-(2-\epsilon)) - m(y-(2-\epsilon))^2,
\end{align*}
where $q$ and $m$ are the integers used to define the transformation $A$ in \cref{contact_forms_and_rational_open_book_decompositions}.
\end{itemize}
\begin{figure}\centering
  \begin{tikzpicture}%
		\draw [->] (0,0.5) -- (0,4) node [above left] {$g$};
		\draw [->] (0,1.5) -- (4,1.5) node [below] {$f$};
		\draw [thick,*-*,shorten <=-3pt, shorten >=-3pt](1.5,2.5) node [above right] {$(f(2-\epsilon),g(2-\epsilon))$} to [out=-90,in=90] (2,1.5) node [above right] {$(f(2),g(2))=(C+C',0)$};
		\draw [thick] (2,1.5) to (2,1.2);
		\draw [thick,dashed] (2,1.2) to (2,0.7);

	\end{tikzpicture}
  \caption{The extended contact form $\alpha_2$, defined on $\tilde{N}$, which is pushed back onto $N$ by the diffeomorphism $\Phi$ to form $\alpha_3$.  The shape of the curve near $y=2-\epsilon$ is defined to match that of \cref{contact_form_on_nu_del_N_fig} so that it is extendable to the rational open book.}
	\label{contact_form_on_extension_of_N_fig}
\end{figure}
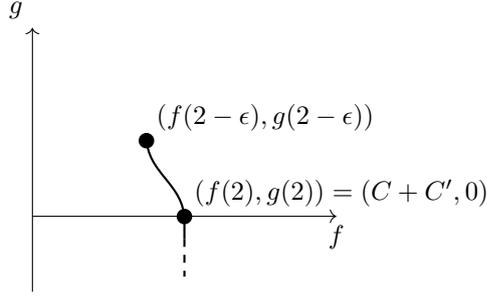
See \cref{contact_form_on_extension_of_N_fig}. Denote this contact form on $\tilde{N}$ by $\alpha_2$.  We use a deformation retract to push this form back onto $N$:  let $a(y)$ be a smooth function on $\tilde{N}$ which is 1 when $y\le2$ and is supported on $T^2\cross[2-\epsilon,2+\frac{1}{2}\epsilon]$.  Then define a diffeomorphism
\[ \Phi: \tilde{N} \to N \]
by flowing along the vector field $X=a(y)\del_y$.  Define $\alpha_3 := \Phi_*\alpha_2$ and also let $a_2=f(2-\epsilon)>0$ and $b_2=g(2-\epsilon)$.  Then $\alpha_3$ is extendable to the $p/q$-rational open book as required.

We are now ready to define $\alpha$, which is a small non-degenerate perturbation of $\alpha_3$:

\begin{defn}\label{contact_form_alpha_on_N}
Define the contact form $\alpha$ on $N$ by making a small $C^\infty$ perturbation of $\alpha_3$ so that $\alpha$ is non-degenerate on $\mathrm{int}(N)$ and the only Morse-Bott torus that remains is $\del N$.  This perturbation is done in such a way that $\alpha$ and $\alpha_3$ agree to infinite order along $\del N$ so that $\alpha$ is still extendable to the rational open book.
\end{defn}

\section{Computing ECH via a filtration argument}\label{computing_ech_via_a_filtration_argument}

Choose a sequence $\delta_i$ tending to $0$ and the corresponding $\alpha_{\delta_i}$ and $c_{\delta_i}\alpha_{V,\mu}'$ on $T^2\cross[1,2]$ and $V$ respectively.  For simplicity of notation we will drop the $\mu$ for the remainder of this chapter, writing only $\alpha_V'$.  Note that the values $c_{\delta_i}>0$ are decreasing and convergent---their limit is positive as can be seen in \cref{contact_forms_on_3_pieces}.  These two pieces glue to the contact form $\alpha$ on $N$ to produce a sequence of commensurate Morse-Bott contact forms $\alpha_i$ on $M$.  See \cref{contact_form_on_nu_K_a}. As $\delta_i$ tends to $0$, the action of all Reeb orbits in $T^2\cross(1,2)$ tend to infinity, so choose a sequence of real numbers $L_i$ tending to $\infty$ such that all orbits in $T^2\cross(1,2)$ have $\alpha_i$-action greater than $L_i$.  The idea of the main theorem is to use a direct limit argument applied to the homology groups $ECH^{L_i}_\MB(M,\alpha_i)$.  This way we can essentially ignore all orbits in the region $T^2\cross(1,2)$---for this reason it is referred to as \emph{no man's land}.

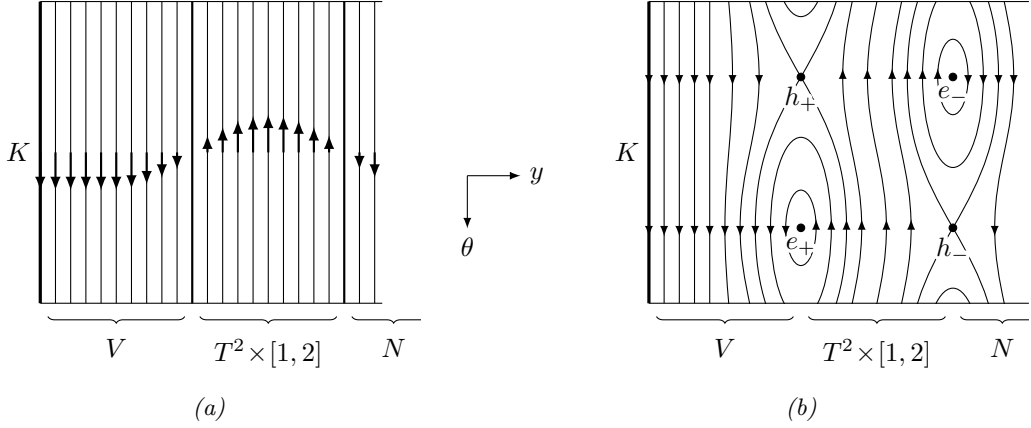
\begin{figure}\centering
  \begin{subfigure}{0.44\textwidth}\centering
		\begin{tikzpicture} %
			\draw  (0,0) --(4.5,0);
			\draw (0,4) --(4.5,4);
			\draw [very thick] (0,0)--(0,4); %
			\node [left] at (0,2) {$K$};

			\foreach \x in {0.2, 0.4,...,4.4} {
				 \draw [thin](\x,0)--(\x,4);
			}
			\draw [thick](2,0)--(2,4);
			\draw [thick](4,0)--(4,4);
			\foreach \x in {0,0.2,0.4,...,1} {
				\draw [-latex,thick,shorten >=-3pt] (\x,2)-- ++(0,-0.4);
			}
			\foreach \x in {-0.8,-0.6,...,-0.2,0.2,0.4,...,1.8,2.2,2.4} {
				\draw [-latex,thick,shorten >=-3pt] (2+\x,2)-- ++(0,{0.4*sin(90*\x)});
			}

			\draw [decorate,decoration={brace,mirror,raise=5pt}] (0.1,0) --node[below=10pt] {$V$} (1.9,0) ;
			\draw [decorate,decoration={brace,mirror,raise=5pt}] (2.1,0) --node[below=10pt] {$T^2\cross[1,2]$} (3.9,0) ;
			\begin{scope}
			\path [clip] (4,-1) rectangle (5,0);
			\draw [decorate,decoration={brace,mirror,raise=5pt}] (4.1,0) --node[below=10pt] {$N$} (5.2,0) ;
			\end{scope}
		\end{tikzpicture}
		\caption{}
		\label{contact_form_on_nu_K_a}
	\end{subfigure}
	\begin{subfigure}[c]{0.1\textwidth}
	  \begin{tikzpicture}
		  \figureaxes{0}{0}{y}{}{}{\theta}
		\end{tikzpicture}
	\end{subfigure}
	\begin{subfigure}{0.44\textwidth}\centering
		\begin{tikzpicture} %
			\draw  (0,0) --(5,0);
			\draw (0,4) --(5,4);
			\draw [very thick] (0,0)--(0,4); %
			\node [left] at (0,2) {$K$};

			\begin{scope}
				\path [clip] (-1,0) rectangle (5,4);

				\draw [thin] (1.1,-1) to [out=90,in=-90] (1,1) to [out=90,in=-90] (1.1,3) to [out=90,in=-90] (1,5);
				\draw [-latex,shorten >=-3pt] (1,1)-- +(-90:0.01);
				\draw [-latex,shorten >=-3pt] (1.1,3)-- +(-90:0.01);
				\foreach \x in {0,3.35} {
					\draw [thin] (\x+1.45,-1) to [out=90,in=-90] (\x+1.2,1) to [out=90,in=-90] (\x+1.45,3) to [out=90,in=-90] (\x+1.2,5);
					\draw [-latex,shorten >=-3pt] (\x+1.2,1)-- +(-90:0.01);
					\draw [-latex,shorten >=-3pt] (\x+1.45,3)-- +(-90:0.01);
				}
				\foreach \x in {0,0.325,0.65} {
					\draw [thin] (\x+2.55,-1) to [out=90,in=-90] (\x+2.8,1) to [out=90,in=-90] (\x+2.55,3) to [out=90,in=-90] (\x+2.8,5);
					\draw [-latex,shorten >=-3pt] (\x+2.8,1)-- +(90:0.01);
					\draw [-latex,shorten >=-3pt] (\x+2.55,3)-- +(90:0.01);
				}
				
				\foreach \y in {0,4} {
					\draw [thin] (2,\y+3) to [out=-120, in=90] (1.4,\y+1) to [out=-90,in=120] (2,\y-1);
					\draw [thin] (2,\y+3) to [out=-60, in=90] (2.6,\y+1) to [out=-90,in=60] (2,\y-1);
					\draw [thin] (4,5-\y) to [out=-120, in=90] (3.4,3-\y) to [out=-90,in=120] (4,1-\y);
					\draw [thin] (4,5-\y) to [out=-60, in=90] (4.6,3-\y) to [out=-90,in=60] (4,1-\y);

					\draw [thin] (2,\y+1) ellipse (0.4 and 1.2);
					\draw [thin] (4,3-\y) ellipse (0.4 and 1.2);
				}

				\draw [thin] (2,1) ellipse (0.2 and 0.5);
				\draw [thin] (4,3) ellipse (0.2 and 0.5);

				\foreach \x in {0,0.2,0.4,...,0.8} {
					\draw [thin] (\x,0)--(\x,4);
					\draw [-latex,shorten >=-3pt] (\x,1)-- +(-90:0.01);
					\draw [-latex,shorten >=-3pt] (\x,3)-- +(-90:0.01);
				}

				\draw [-latex,shorten >=-3pt] (1.4,1)-- +(-90:0.01);
				\draw [-latex,shorten >=-3pt] (1.6,1)-- +(-90:0.01);
				\draw [-latex,shorten >=-3pt] (2.2,1)-- +(90:0.01);
				\draw [-latex,shorten >=-3pt] (2.4,1)-- +(90:0.01);
				\draw [-latex,shorten >=-3pt] (2.6,1)-- +(90:0.01);
				\draw [-latex,shorten >=-3pt] (1.8,1)-- +(-90:0.01);
				\draw [-latex,shorten >=-3pt] (3.4,3)-- +(90:0.01);
				\draw [-latex,shorten >=-3pt] (3.6,3)-- +(90:0.01);
				\draw [-latex,shorten >=-3pt] (3.8,3)-- +(90:0.01);	
				\draw [-latex,shorten >=-3pt] (4.2,3)-- +(-90:0.01);
				\draw [-latex,shorten >=-3pt] (4.4,3)-- +(-90:0.01);
				\draw [-latex,shorten >=-3pt] (4.6,3)-- +(-90:0.01);
				
			\end{scope}

			\draw [fill] (2,1) circle (0.05) node [below] {\contourlength{1.5pt}\contour{white}{$e_+$}};
			\draw [fill] (2,3) circle (0.05) node [below] {\contourlength{1.5pt}\contour{white}{$h_+$}};
			\draw [fill] (4,1) circle (0.05) node [below] {\contourlength{1.5pt}\contour{white}{$h_-$}};
			\draw [fill] (4,3) circle (0.05) node [below] {\contourlength{1.5pt}\contour{white}{$e_-$}};

			\draw [decorate,decoration={brace,mirror,raise=5pt}] (0.1,0) --node[below=10pt] {$V$} (1.9,0) ;
			\draw [decorate,decoration={brace,mirror,raise=5pt}] (2.1,0) --node[below=10pt] {$T^2\cross[1,2]$} (3.9,0) ;
			\begin{scope}
				\path [clip] (4,-1) rectangle (5,0);
				\draw [decorate,decoration={brace,mirror,raise=5pt}] (4.1,0) --node[below=10pt] {$N$} (5.2,0) ;
			\end{scope}
		\end{tikzpicture}
		\caption{}
		\label{contact_form_on_nu_K_b}
	\end{subfigure}
  \caption{In these figures we see a single page of the open book decomposition of $\nu(K)=V\union T^2\cross [1,2]$; spin each figure around $K$ and identify the top and bottom to obtain a solid torus.  The flows represent the return map of the Reeb vector field.  In (a) we see a positive Morse-Bott torus at $\del V$ which is perturbed into two non-degenerate orbits $h_+$ and $e_+$ in (b).  Likewise the negative Morse-Bott torus at $\del N$ is perturbed into $e_-$ and $h_-$.}
	\label{contact_form_on_nu_K}
\end{figure}

Two Morse-Bott tori we cannot ignore however are the tori $T^2\cross\set{1}=\del V$ and $T^2\cross\set{2}=\del N$.  For this reason we will perturb each $\alpha_i$ to obtain an $L_i$-non-degenerate contact form $\alpha_i'$, in the manner laid out in \cref{perturbation_fn_g}.  $\del N$ is a negative Morse-Bott torus and is perturbed into two orbits which we label $e_-$ and $h_-$; $\del V$ is a positive Morse-Bott torus and is perturbed into two orbits $e_+$ and $h_+$.  See \cref{contact_form_on_nu_K_b}. Hence we have two chain complexes for each $i$: $ECC^{L_i}(M, \alpha_i')$ with ECH differential $d_i$ and $ECC^{L_i}_\MB(M, \alpha_i)$ with Morse-Bott differential $d^\MB_i$.  Each pair of chain complexes are trivially isomorphic as vector spaces but the differentials are not necessarily equal since the contact forms $\alpha_i$ are not necessarily nice (recall \cref{non-degen_close_to_MB}).  

\begin{convention}Note that here we have dropped the choice of (regular, adapted) almost complex structure for these complexes from the notation, and we will continue to make this omission for the remainder of the chapter.
\end{convention}

Recall that our aim is to express $ECH(M)$ entirely in terms of the ``relative'' embedded contact homology of $N$.  For this reason it is necessary to understand the behaviour of the orbits outside $N$.  As well as the four orbits $e_-$, $h_-$, $e_+$ and $h_+$, the other orbits which we must understand are those in $\mathrm{int}(V)$.  The idea of the proof by Colin, Ghiggini and Honda~\cite{CGH_ECH_OBD} is to use a filtration argument to exclude these orbits:  the authors define a filtration on the collection of orbit sets which is given by the number of times an orbit set winds around the $\theta$ direction \emph{inside $V$}.  They then prove that on the first page of this spectral sequence all generators with filtration level greater than 0 are killed - hence the spectral sequence converges at the first page and the embedded contact homology can be computed without ever considering the orbits in $\mathrm{int}(V)$ at all.

In practice, the proof is a lot more technical, since the Morse-Bott and non-degenerate contact forms, as well as the direct limits, must be handled carefully.  From this point forwards, all results carry directly over from the work of Colin, Ghiggini and Honda~\cite{CGH_ECH_OBD}---we will simply outline the method of the proof to complete the section.

\subsection{A sequence of filtrations on the ECH complexes}\label{a_sequence_of_filtrations_on_the_ECH_complexes}

The first step is to identify each of $ECC^{L_i}(M, \alpha_i')$ and $ECC^{L_i}_\MB(M, \alpha_i)$, as vector spaces, with a subspace of $ECC(V,\alpha_V')\tensor ECC(N,\alpha)$.  This can be done canonically since we do not see any orbits in the no man's land and $\restr{\alpha_i}{V}=c_{\delta_i}\alpha_V'$ so all Reeb orbits in $V$ coincide up to rescaling (recall that the perturbed form $\alpha_i'$ leaves all $\alpha_i$-orbits of length less than $L_i$ untouched).  This means that every orbit set can be written as $\gamma\tensor\Gamma$ where $\gamma$ and $\Gamma$ are orbit sets in $V$ and $N$ respectively. 

We next define, for every $i$, an ascending filtration $\mathcal{F}_i$ on $ECC^{L_i}(M, \alpha_i')$.  This gives rise to a spectral sequence $E^r(\mathcal{F}_i)$ which converges to $ECH^{L_i}(M, \alpha_i')$.  Every $\mathcal{F}_i$ is defined in the same way:

\begin{defn}\label{filtrations_def}
Identify $H_1(V;\Z)$ with $\Z$ such that the core of $V$, oriented in the positive $\theta$ direction, is identified with 1, and denote this identification by $\eta$.

Define $\mathcal{F}_i$ by
\[ \mathcal{F}_i(\gamma\tensor\Gamma) = \eta(\gamma) \ge 0, \]
extended linearly to $ECC^{L_i}(M,\alpha_i')$, and define
\[ \mathcal{F}_i^p = \set{ x\in ECC^{L_i}(M,\alpha_i') \gappy{|} \mathcal{F}_i(x)\le p}. \]
\end{defn}

We will apply slope calculus (\cref{slope_calculus}) and positivity of intersection (\cref{positivity_of_intersections_3}) to the region
\[ (\del V \sminus K)\union (T^2\cross[1,2])\iso T^2\cross (0,2]\]
to see that both differentials $d_i$, $d^\MB_i$ respect the filtration $\mathcal{F}_i$.

To see this for $d_i^\MB$, suppose that $u$ is the non-connector part of a Morse-Bott building between $\gamma\tensor\Gamma$ and $\gamma'\tensor\Gamma'$ and choose some torus $T_{\rho_i}$ of large irrational slope $r_i$ such that all orbits in $\gamma$ and $\gamma'$ lie in the region $\set{y_0 < \rho_i}$ (c.f.~the definition of $\alpha_V'$).  Then by slope calculus the homology class $[u_{T^2\cross\set{\rho_i}}]$ of the slice of $u$ through $T^2\cross\set{\rho_i}$ is equal to
\[ [\gamma']-[\gamma]\in H_1(T^2\cross(0,2]).\]
Positivity of intersection then states that
\begin{align*}
 0 &< ([\gamma']-[\gamma])\cdot r_i \\
\iff 0 &< \big((\eta(\gamma')-\eta(\gamma))[l] + k[m]\big)\cdot \big(r_i[m] + [l]\big) = k - r_i(\eta(\gamma')-\eta(\gamma)), 
\end{align*}
where $k\ge 0$ is the number of times that $\gamma\union\gamma'$ winds around the core $K$.  As we can choose $i$ to make $r_i$ arbitrarily large, this inequality holds if and only if $\eta(\gamma')\le\eta(\gamma)$.

To argue the same result for the differential $d_i$, we must invoke \cref{MB_building_exists_epsilon_small_enough}, assuming that $\alpha_i'$ is close enough to $\alpha_i$ that there exists a Morse-Bott building for every holomorphic curve.  We can then proceed by the argument above.

In the Morse-Bott case we can go a step further: if $\eta(\gamma')=\eta(\gamma)$ then the slope of $u$ at $T^2\cross\set{\rho_i}$ is infinite (i.e.~parallel to the meridianal slope) or zero.  If non-zero, then since the slope of the Morse-Bott torus $\del V$ is also infinite, the \refNamedThm{Blocking Lemma}{blocking_lemma} implies that $u$ is one-sided at $\del V$.  Furthermore the \refNamedThm{Trapping Lemma}{trapping_lemma} implies that such one-sided ends of $u$ at $\del V$ can only be positive ends.  Similarly, $u$ can only have negative one-sided ends at $\del N$.  We will state this as a proposition:

\begin{prop}\label{filtration_non-decreasing_must_be_one-sided}
If $u$ is the non-connector part of a Morse-Bott building which does not decrease $\mathcal{F}_i$-filtration level then $u$ can only have one-sided positive ends at $\del V$ and one-sided negative ends at $\del N$.
\end{prop}

This is an important result since we can apply the following Proposition:

\begin{prop}\label{one-sided_implies_nice}
Suppose that $u\in \mathcal{M}_J^{MB, I=1}(\Gamma_+, \Gamma_-)$ is a Morse-Bott building such that every end of each non-connector holomorphic component of $u$ at a Morse-Bott torus is one-sided.  Then $u$ is nice.
\end{prop}
\begin{proof}
This is an adaptation of work by Colin, Ghiggini and Honda~\cite[Lemma 7.1.2]{CGH_ECH_OBD}, which states that if a Morse-Bott contact form only has Morse-Bott tori at the boundary then it is nice.

Denote the connected components of the non-connector part of $u$ by $\set{u_i \gappy{|} 1\le i\le m}$ and let $\set{\mathcal{N}_j}$ be the set of Morse-Bott tori containing ends of the non-connector part of $u$.  Since all such ends are one-sided the \refNamedThm{Trapping Lemma}{trapping_lemma} implies that if $\mathcal{N}_j$ is positive (resp.~negative) then all ends of $u$ at $\mathcal{N}_j$ are positive (resp.~negative).

Consider a single end of some $u_i$ at some $\gamma$ in $\mathcal{N}_j$, and without loss of generality assume that it is a positive end.  Then according to the definition of a Morse-Bott building there is a gradient flow line in $\mathcal{N}_j$ ending at $\gamma$.  Ignoring possible interruptions at connector components, the other end of this gradient flow line must be an orbit in $\Gamma_+$, since no non-connector components have negative ends in $\mathcal{N}_j$.

The idea is to replace $u$ by a simpler Morse-Bott building $\tilde{u}$ with the same ECH index.  We can then compute a bound for $I(\tilde{u})$ which forces $m=1$.  Construct the Morse-Bott building $\tilde{u}$ by augmenting each $u_i$ at the positive end with gradient flow line from the corresponding orbit in $\Gamma_+$, and at the negative end with a gradient flow line to the corresponding orbit in $\Gamma_-$.  Then the ECH indices of $u$ and $\tilde{u}$ are equal, since they define the same relative homology class and have the same ends.  

We will now compute a lower bound for $I(\tilde{u})$ to complete the proof. Write each $u_i$ as a $k_i$-th cover of some simply covered $J$-holomorphic curve $v_i$ and augment each $v_i$ to form a simply-covered very nice building $\tilde{v}_i$.  Such a Morse-Bott building can be perturbed to form a $J_\epsilon$-holomorphic map $v_{i,\epsilon}$~\cite[Theorem 4.4.3(2)]{CGH_ECH_OBD}, which, by \cref{index_0_1_and_2_holo_curves} has ECH index greater than 0.

Finally, by a result of Hutchings~\cite[Theorem 5.1]{Hu09},
\[ 1 = I(u) = I(\tilde{u}) \ge \sum_{1\le i\le m}{k_i I(v_{i,\epsilon})} \ge \sum_{1\le i\le m}{k_i}, \]
hence $m=1$ and $k_1=1$.
\end{proof}

Let $d_{i,0}$ and $d^\MB_{i,0}$ denote the parts of $d_i$ and $d^\MB_i$ respectively which do not decrease filtration level $\mathcal{F}_i$.  Then \cref{filtration_non-decreasing_must_be_one-sided,one-sided_implies_nice} imply that even though the Morse-Bott forms $\alpha_i$ are not necessarily nice, they are nice \emph{with respect to the differentials $d^\MB_{i,0}$} (c.f.~\cref{niceness_defn}).

Therefore, by \cref{non-degen_close_to_MB},
\[ d^\MB_{i,0} = d_{i,0} \]
for all $i$ and we can use Morse-Bott theory to compute $E^1(\mathcal{F}_i)$.

\subsection{Taking limits and passing to the first page}

Colin, Ghiggini and Honda~\cite[Section 8.4]{CGH_ECH_OBD} analyze a \emph{finite energy foliation} of psuedoholomorphic curves in $T^2\cross[1,2]$ to show that there are precisely two Morse-Bott differentials
\begin{align*}
e_+ &\longrightarrow h_- \\
h_+ &\longrightarrow e_-
\end{align*}
in $T^2\cross[1,2]$.  There are also two gradient trajectories from $e_+$ to $h_+$, and two from $h_-$ to $e_-$, which cancel since we are working over the field $\F_2$.  Hence we obtain the following proposition:

\begin{prop}[{\nogapcite[Lemma 9.5.3]{CGH_ECH_OBD}}]\label{differential_on_E_0_i}
The differential $d_{i,0}$ on $E^0(\mathcal{F}_i)$ is given by:
\[ d_{i,0}(\gamma\tensor\Gamma) = d_V(\gamma)\tensor\Gamma + (\gamma/e_+)\tensor h_-\Gamma + (\gamma/h_+)\tensor e_-\Gamma + \gamma\tensor d_N(\Gamma). \]
\end{prop}
In this proposition $d_V$ and $d_N$ denote the Morse-Bott differentials on $V$ and $N$ respectively.

The intuition behind the existence of the above differentials in $T^2\cross[1,2]$ can be understood in the perturbed case as seen in \cref{perturbed_alpha_i_showing_differentials_fig}---to construct this figure we are employing \cref{foliation_perp_to_R_xi} below, which is a generalization of Wendl's constructions of finite energy foliations\cite[Section 4.2]{Wendl}.

Before stating the proposition we need a few ingredients---first suppose that $U\subset M$ is homeomorphic to $S\cross S^1$ for some surface $S$, and that both the Reeb vector field $R$ and almost complex structure $J$ are $S^1$-invariant on $U$.  Write $R=R_\xi + a\del_t$, where $R_\xi\in\xi$, $t$ is the $S^1$ coordinate and $a:S\to\R$.  Now write $J(R_\xi) = X + b\del_t$ where $b:S\to\R$ and $X$ is tangent to $S$ and also $S^1$-invariant.
\begin{prop}\label{foliation_perp_to_R_xi}
Let $U\homeo S\cross S^1$, $R$, $J$ and $X$ be as above.  Then the symplectization $\R\cross U\homeo \R\cross S\cross S^1$ is foliated by pseudoholomorphic curves of the form
\[ Z\cross S^1, \]
where $Z\subset \R\cross S$ is an integral curve of $\del_s + X$.  This foliation is $\R\cross S^1$-invariant and its projection to $U$ is a foliation by cylinders of the form $\gamma\cross S^1$ where $\gamma$ is an integral curve of $X$.  At those points where $R$ is parallel to $\del_t$, and hence $X=0$, the foliation is by Reeb orbits (i.e.~$\gamma$ is a singleton).
\end{prop}
\begin{proof}
Consider the the $J$-invariant plane field
\[ \langle \del_t, J(\del_t) \rangle, \]
where $t$ denotes the $S^1$ coordinate, defined on $\R\cross U$.  Note that the function $a$ defined above is everywhere non-zero since $R$ is transverse to $\xi$, and hence we can write
\[ \begin{split}
J(\del_t) &= J\left(\frac{1}{a}(R-R_\xi)\right) \\
&= \frac{1}{a}(-\del_s - J(R_\xi)) \\
&= \frac{1}{a}(-\del_s - X - b\del_t).
\end{split}\]
As a result, we have
\[ \begin{split}
\langle \del_t, J(\del_t) \rangle &= \langle \del_t, \frac{1}{a}(-\del_s - X - b\del_t) \rangle \\
&= \langle \del_t, -\del_s-X-b\del_t \rangle \\
&= \langle \del_t, \del_s + X \rangle.
\end{split}\]
It is easy to see that this plane field is integrable, since $X$ depends only on $S$ and hence the Lie bracket vanishes.  Therefore it is supported on a regular foliation which is $\R\cross S^1$-invariant and furthermore has leaves of the form $Z\cross S^1$ as described in the statement of the proposition.
\end{proof}

\begin{figure}\centering
	\begin{tikzpicture} %
		\draw  (0,0) --(4.5,0);
		\draw (0,4) --(4.5,4);
		\draw [very thick] (0,0)--(0,4); %
		\node [left] at (0,2) {$K$};

		\draw [thick] (2,0)--(2,4);
		\draw [-latex,shorten >=-3pt] (2,2.3)--(2,2.31);
		\draw [-latex,shorten >=-3pt] (2,3.6)--(2,3.59);
		\draw [thick] (4,0)--(4,4);
		\draw [-latex,shorten >=-3pt] (4,1.7)--(4,1.71);
		\draw [-latex,shorten >=-3pt] (4,0.4)--(4,0.39);
		\draw [thick] (0,3) --(4.45,3);
		\draw [thick,dashed] (4.45,3)--(5,3);
		\draw [-latex,shorten >=-3pt] (1,3)--(0.99,3);
		\draw [-latex,shorten >=-3pt] (2.7,3)--(2.71,3);
		\draw [-latex,shorten >=-3pt] (4.4,3)--(4.39,3);
		\draw [thick] (2,1) --(4.45,1);
		\draw [thick,dashed] (4.45,1)--(5,1);
		\draw [-latex,shorten >=-3pt] (3.3,1)--(3.31,1);
		\draw [-latex,shorten >=-3pt] (4.4,1)--(4.39,1);

		\begin{scope}
		\path [clip] (-1,0) rectangle (4,4);
		\draw [thin] (2,1) to [out=75,in=-135] (2.5,2.5) to [out=45,in=-165] (4,3);
		\draw [thin] (2,1) to [out=60,in=-135] (2.8,2.2) to [out=45,in=-150] (4,3);
		\draw [thin] (2,1) to [out=30,in=-135] (3.2,1.8) to [out=45,in=-120] (4,3);
		\draw [thin] (2,1) to [out=15,in=-135] (3.5,1.5) to [out=45,in=-105] (4,3);
		\foreach \y in {-2,2} {
			\draw [thin] (2,3+\y) to [out=-15,in=135] (3.5,2.5+\y) to [out=-45,in=105] (4,1+\y);
			\draw [thin] (2,3+\y) to [out=-30,in=135] (3.2,2.2+\y) to [out=-45,in=120] (4,1+\y);
			\draw [thin] (2,3+\y) to [out=-60,in=135] (2.8,1.8+\y) to [out=-45,in=150] (4,1+\y);
			\draw [thin] (2,3+\y) to [out=-75,in=135] (2.5,1.5+\y) to [out=-45,in=165] (4,1+\y);
		}
		\draw [thin] (2,1) to [out=100,in=0] (0,2.7);
		\draw [thin] (2,1) to [out=120,in=0] (0,2);
		\draw [thin] (2,1) to [out=150,in=0] (0,1.4);
		\draw [thin] (2,1) to [out=-150,in=0] (0,0.6);
		\draw [thin] (2,1) to [out=-120,in=0] (0,0);
		\foreach \y in {0,4} 
			\draw [thin] (2,1+\y) to [out=-100,in=0] (0,-0.7+\y);
		\end{scope}

		\draw [fill] (2,1) circle (0.05) node [left] {$e_+$};
		\draw [fill] (2,3) circle (0.05) node [above left] {$h_+$};
		\draw [fill] (4,1) circle (0.05) node [below right] {$h_-$};
		\draw [fill] (4,3) circle (0.05) node [below right] {$e_-$};

		\draw [decorate,decoration={brace,mirror,raise=5pt}] (0.1,0) --node[below=10pt] {$V$} (1.9,0) ;
		\draw [decorate,decoration={brace,mirror,raise=5pt}] (2.1,0) --node[below=10pt] {$T^2\cross[1,2]$} (3.9,0) ;
		\begin{scope}
		\path [clip] (4,-1) rectangle (5,0);
		\draw [decorate,decoration={brace,mirror,raise=5pt}] (4.1,0) --node[below=10pt] {$N$} (5.2,0) ;
		\end{scope}
	\end{tikzpicture}
  \caption{The structure of differentials in $\nu(K)$ between the orbits $h_\pm$, $e_\pm$ and $\emptyset$.  This picture can be obtained by applying \cref{foliation_perp_to_R_xi} to \cref{contact_form_on_nu_K_b}.  In particular, the integral curves of $X$ seen above arise by flowing in a perpendicular manner to the flow lines in \cref{contact_form_on_nu_K_b} (oriented by turning to the left).  The lines in bold correspond to cylinders with ECH index one, and the arrows point in the negative $s$ direction, i.e.~against the flow of $X$.  The curved lines correspond to cylinders with ECH index 2.  Finally, the two dashed lines to the right are shown only as a reminder that differentials in $\mathrm{int}(N)$ can only have negative ends at $\del N$ by the \refNamedThm{Trapping Lemma}{trapping_lemma}.}
	\label{perturbed_alpha_i_showing_differentials_fig}
\end{figure}
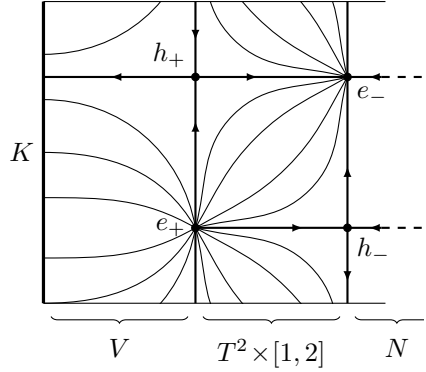

The next step in the proof of \cref{relative_ECH_equals_ECH} is to let $i$ tend to $\infty$.  For each $i$, the spectral sequence $E^r(\mathcal{F}_i)$ converges to $ECH^{L_i}(M,\alpha_i')$, and since the contact forms $\alpha_i$ are commensurate,
\[ \lim_{i\to\infty}ECH^{L_i}(M,\alpha_i') \iso ECH(M),\]
where the direct limit is taken with respect to the interpolating cobordism maps
\[ \Phi_i:ECH^{L_i}(M,\alpha_i') \to ECH^{L_{i+1}}(M,\alpha_{i+1}'). \]
Recall from \cref{direct_limits_through_cobordism_maps} that the above maps are induced by non-canonical chain maps
\[ \hat\Phi_i:ECC^{L_i}(M,\alpha_i') \to ECC^{L_{i+1}}(M,\alpha_{i+1}') \]
and in fact these induce chain maps~\cite[Section 9.6]{CGH_ECH_OBD}
\[ E^r(\mathcal{F}_i) \to E^r(\mathcal{F}_{i+1}) \]
for all $r$. The resulting direct limit is itself a spectral sequence, denoted\footnote{Note that we write $E^r(\mathcal{F})$ simply for notational purposes; we are not claiming that the spectral sequence arises from some filtration $\mathcal{F}$.} $E^r(\mathcal{F})$, converging to
\[ \lim_{i\to\infty}ECH^{L_i}(M,\alpha_i') \iso ECH(M).\]

Furthermore, the equation in \cref{differential_on_E_0_i} passes to the direct limit:

\begin{prop}[{\nogapcite[Lemma 9.6.3]{CGH_ECH_OBD}}]\label{differential_on_E_0}
As a vector space,
\[ E^0(\mathcal{F}) \iso ECC(V,\alpha_V')\tensor ECC(N,\alpha), \]
and the differential $d_0$ on this page is given by:
\begin{equation}\label{differential_d_0}
 d_0(\gamma\tensor\Gamma) = d_V(\gamma)\tensor\Gamma + (\gamma/e_+)\tensor h_-\Gamma + (\gamma/h_+)\tensor e_-\Gamma + \gamma\tensor d_N(\Gamma).
\end{equation}
\end{prop}

Notice that the direct limit argument has successfully allowed us to ignore all orbits in the no man's land.  The next step is to show that we can also ignore all those in the interior of $V$.

Note that the filtration levels $\mathcal{F}_i$ as defined in \cref{filtrations_def} give rise to a grading on every $E^r(\mathcal{F}_i)$, which passes to the direct limit: write $E^r_p(\mathcal{F})$ for the summand of $E^r(\mathcal{F})$ generated by orbit sets $\gamma\tensor\Gamma$ with filtration grading $\eta(\gamma)=p$.  We will call this the $\eta$-grading.

\begin{thm}\label{F_spectral_seq_converges_to_ECH_at_first_page_thm}
$E^1_p(\mathcal{F})=0$ for $p>0$. As a result, the spectral sequence converges at the first page:
\[ E^1(\mathcal{F}) = E^1_0(\mathcal{F}) \iso ECH(M). \]
\end{thm}

The proof of this theorem involves defining a second filtration $\mathcal{G}$ on the zero-th page $E^0(\mathcal{F})$, which gives rise to a second spectral sequence $E^r(\mathcal{G})$ converging to $E^1(\mathcal{F})$.  Computing the first page of $E^r(\mathcal{G})$ will give us the result that $E^1_p(\mathcal{F})=0$ for $p>0$, and also provides the set-up which is used to prove that $E^1_0(\mathcal{F}) \iso ECH(N,\del N,\alpha)$, thus completing the proof of \cref{relative_ECH_equals_ECH}.  We will outline the proof of \cref{F_spectral_seq_converges_to_ECH_at_first_page_thm} below, following the method of Colin, Ghiggini and Honda~\cite[Section 9.6]{CGH_ECH_OBD}.

The filtration $\mathcal{G}$ is defined on $E^0(\mathcal{F})$ by first decomposing $E^0(\mathcal{F})$ into four pieces corresponding to the presence of the orbits $h_+$ and $h_-$:
\[ E^0(\mathcal{F}) \iso C_{0,0}\oplus C_{0,1}\oplus C_{1,0}\oplus C_{1,1}, \]
where $C_{i,j}$ denotes the subspace generated by orbit sets where the multiplicity of $h_-$ is $i$ and the multiplicity of $h_+$ is $j$.  Define the (ascending) filtration $\mathcal{G}$ by $\mathcal{G}(C_{i,j})=j-i$; in order to compute the spectral sequence corresponding to $\mathcal{G}$ we must first understand the explicit behaviour of the differential $d_0$ with respect to $h_+$ and $h_-$ respectively.  Write $d^\flat_V$ and $d^\flat_N$ for the differentials on $ECC^{e_+}(\mathrm{int}(V),\alpha_V')$ and $ECC^{e_-}(\mathrm{int}(N),\alpha)$ respectively.
\begin{itemize}
\item If $\Gamma\in ECC^{e_-}(\mathrm{int}(N),\alpha)$, we can write
\[ d_N(\Gamma) = d^\flat_N(\Gamma) + h_-d_N'(\Gamma),\] where $d_N^\flat(\Gamma)$ and $d_N'(\Gamma)$ do not contain $h_-$.  Furthermore, 
\[d_N(h_-\Gamma) = h_- d_N(\Gamma) = h_- d^\flat_N(\Gamma)\]
since no holomorphic curves can have a positive end at the negative Morse-Bott torus $\del N$.
\item If $\gamma\in ECC^{e_+}(\mathrm{int}(V), \alpha_V')$, we can write
\[d_V(h_+\gamma) = h_+ d^\flat_V(\gamma) + d_V'(h_+\gamma),\] where $d_V^\flat(h_+\gamma)$ and $d_V'(h_+\gamma)$ do not contain $h_+$.  Furthermore, 
\[d_V(\gamma) = d^\flat_V(\gamma)\]
since no holomorphic curves can have a negative end at the positive Morse-Bott torus $\del V$.
\end{itemize}
Referring to the description of $d_0$ in \cref{differential_d_0} we see that it can be written as
\[ d_0 = d_{\mathcal{G},0} + d_{\mathcal{G},1}, \]
where $d_{\mathcal{G},0}$ is the part of $d_0$ which does not decrease $\mathcal{G}$-filtration level, and $d_{\mathcal{G},1}$ is the part which does.  In particular, the components of $d_{\mathcal{G},1}$ can be seen in the following diagram, verifying that $\mathcal{G}$ is indeed a valid ascending filtration on $E^0(\mathcal{F})$:

\begin{equation}\label{d_G_1_diagram}\begin{tikzcd}
C_{0,1}\arrow[rrrr,"1\tensor h_-d'_N + \cdot/e_+\tensor h_-"]\arrow[d,"d'_V\tensor 1 + \cdot/h_+ \tensor e_-"'] &&&& C_{1,1}\arrow[d,"d'_V\tensor 1 + \cdot/h_+ \tensor e_-"] \\
C_{0,0}\arrow[rrrr,"1\tensor h_-d'_N + \cdot/e_+\tensor h_-"] &&&& C_{1,0}
\end{tikzcd}\end{equation}

We now pass to the first page of the spectral sequence, $E^1(\mathcal{G})$, by taking homology within each $C_{i,j}$, and obtain the same diagram as above but with each $C_{i,j}$ replaced with
\[  H_*(C_{i,j}) \iso (h_+)^j ECH^{e_+}(\mathrm{int}(V),\alpha_V') \tensor (h_-)^i ECH^{e_-}(\mathrm{int}(N),\alpha). \]

We can now apply the following lemma.

\begin{lemma}[{\nogapcite[Theorem 8.1.2]{CGH_ECH_OBD}}]\label{int_V_orbits_cancel}
\[ ECH^{e_+}(\mathrm{int}(V),\alpha_V') \iso \F_2[e_+], \]
the polynomial ring generated by the orbit $e_+$.
\end{lemma}

It is the result of this lemma which essentially allows us to ignore all the orbits in the interior of $V$, i.e.~with $\eta$-grading greater than 0. More precisely, if we write $E^r_p(\mathcal{G})$ for the summand of $E^r(\mathcal{G})$ with $\eta$-grading $p$, then, for $p>0$,
\[ E^1_p(\mathcal{G}) = 0,\]
we have convergence at the first page, and
\[ E^1_p(\mathcal{F}) \iso E^\infty_p(\mathcal{G})=0. \]

\Cref{F_spectral_seq_converges_to_ECH_at_first_page_thm} now follows trivially from the fact that $E^r(\mathcal{F})$ converges to $ECH(M)$.

\subsection{The isomorphism with relative ECH}\label{the_isomorphism_with_relative_ECH}

In this section we will show that $E^1_0(\mathcal{F}) \iso ECH(N,\del N,\alpha)$, using \cref{int_V_orbits_cancel} again together with the fact the only remaining holomorphic curve arising from $d_V'$ and contributing to the first page is
\[ h_+ \longrightarrow \emptyset, \]
manifested as a disk in $V$ intersecting $K$ once and with positive boundary at $h_+$.  This follows~\cite[Proposition 8.4.5 and Lemma 8.4.8]{CGH_ECH_OBD} by an explicit construction of a finite energy foliation on $\mathrm{int}(V)$ with respect to the contact form $\alpha_V$ and the existence of some isotopic foliation for the perturbed $\alpha_V'$.  Again, for intuition refer to the perturbed set-up as seen in \cref{perturbed_alpha_i_showing_differentials_fig}.

We can now refer back to \refDiagram{d_G_1_diagram} and write the first page of the spectral sequence, $E^1_0(\mathcal{G})$ as:
\begin{equation}\label{G_filtration_first_page}\begin{tikzcd}
h_+\F_2[e_+]\tensor ECH^{e_-}(\mathrm{int}(N), \alpha)\arrow[d,"\cdot/h_+\tensor(1+e_-)"']\arrow[rr,"\cdot/e_+\tensor h_-"] && h_+\F_2[e_+]\tensor h_-ECH^{e_-}(\mathrm{int}(N), \alpha)\arrow[d,"\cdot/h_+\tensor(1+e_-)"] \\
\F_2[e_+]\tensor ECH^{e_-}(\mathrm{int}(N), \alpha)\arrow[rr,"\cdot/e_+\tensor h_-"] && \F_2[e_+]\tensor h_-ECH^{e_-}(\mathrm{int}(N), \alpha).
\end{tikzcd}\end{equation}

With this in mind we introduce the following notation and a selection of complexes:

\begin{notation}
Let $e_1,\dots,e_n$ be a collection of elliptic orbits and $h_1,\dots,h_m$  hyperbolic.  Define the vector space
\[ \mathcal{R}[e_1,\dots,e_n,h_1,\dots,h_m] := \F_2[e_1,\dots,e_n,h_1,\dots,h_m]/(h_1^2,\dots,h_m^2). \]
\end{notation}

\begin{defn}\label{ECK_complex_defn}
Define the chain complex
\[ ECC\orbits{0he}{eh}(\mathrm{int}(N),\alpha) := \mathcal{R}[h_+,e_+]\tensor ECC\orbits{e}{h}(\mathrm{int}(N),\alpha),\]
equipped with the differential
\[ d(\gamma\tensor\Gamma) = \gamma/h_+\tensor(1+e_-)\Gamma + \gamma/e_+\tensor h_-\Gamma + \gamma\tensor d_N(\Gamma). \]
We also have the option of not including some or all of the orbits $e_\pm$, $h_\pm$, giving rise to other complexes such as:
\begin{itemize}
\item $ECC\orbits{0he}{}(\mathrm{int}(N),\alpha)$, which is obtained by ``cancelling'' the differential from $e_+$ to $h_-$ and is denoted by $ECH^\natural(N,\alpha)$ by Colin, Ghiggini and Honda~\cite[Section 9.7]{CGH_ECH_OBD}, and
\item $ECC\orbits{0he}{h}(\mathrm{int}(N),\alpha)$, denoted by $\widehat{ECH}^\natural(N,\alpha)$ by Colin, Ghiggini and Honda.
\end{itemize} 
\end{defn}

The intuition behind this definition reflects the behaviour of the differential in \refDiagram{G_filtration_first_page} in that $ECC\orbits{0he}{eh}(\mathrm{int}(N),\alpha)$ is the full complex, with homology equal to $E^1_0(\mathcal{F})\iso ECH(M)$.  The notation, although seemingly cumbersome, is intuitive since it is indicative of the differential on the complex.

\begin{thm}\label{ECH0heeh_E10F_ECH_iso}
\[ ECH\orbits{0he}{eh}(\mathrm{int}(N),\alpha) \iso E^1_0(\mathcal{F}), \]
and hence is isomorphic to $ECH(M)$.
\end{thm}
\begin{proof}
Define an ascending filtration $\bar{\mathcal{G}}$ on $ECC\orbits{0he}{eh}(\mathrm{int}(N),\alpha)$ analogously to the filtration $\mathcal{G}$, by setting $\bar{\mathcal{G}}(\gamma\tensor\Gamma)$ to be the multiplicity of $h_+$ minus the multiplicity of $h_-$.  Then the $\bar{\mathcal{G}}$-preserving part of the differential is given by $d^\flat_N$, and hence the first page of the spectral sequence $E^r(\bar{\mathcal{G}})$ is given by \refDiagram{G_filtration_first_page} and hence is isomorphic to $E^1(\mathcal{G})$.  It therefore follows that $E^2(\bar{\mathcal{G}})\iso E^2(\mathcal{G})$ and both spectral sequences converge at this point.  The chain of isomorphisms
\[  ECH\orbits{0he}{eh}(\mathrm{int}(N),\alpha) \iso E^\infty(\bar{\mathcal{G}}) \iso E^\infty(\mathcal{G}) \iso E^1_0(\mathcal{F})\]
completes the proof.
\end{proof}

The Cancellation Lemma below makes rigorous the notion of ``cancelling'' the differential from $e_+$ to $h_-$ mentioned above; the process is non-trivial since the complex is infinite and the single holomorphic curve from $e_+$ to $h_-$ actually manifests itself as infinitely many differentials in the complex (e.g.~$e_+^i \longrightarrow e_+^{i-1}h_-$ for all $i$).

\begin{lemma}[Cancellation Lemma]\label{cancellation_lemma}
Suppose that $C$ is any variant of an ECC chain complex and that $u$ is a differential between two simple orbits $\gamma_1$ and $\gamma_2$, one of which is hyperbolic and one of which is elliptic. Furthermore suppose that
\begin{itemize}
\item no other differentials in the complex have ends at $\gamma_1$, and
\item if $\gamma_2$ is at the positive (resp.~negative) end of $u$ then all other differentials can only have positive (resp.~negative) ends at $\gamma_2$.
\end{itemize}
Then $C$ is chain homotopic to the complex, denoted $C'$, where we do not include the two orbits $\gamma_1$ and $\gamma_2$. 
\end{lemma}
\begin{proof}
This proof will involve two filtrations, $\mathcal{F}$ and $\mathcal{G}$, which are different from the two filtrations used previously in this section.

Note that any generator of $C$ can be written as
\[ \gamma_1^i\gamma_2^j\Gamma' \]
where $\Gamma'\in C'$ does not contain either of the $\gamma_i$.  Define the first filtration $\mathcal{F}$ on $C$ by
\[ \mathcal{F}(\gamma_1^i\gamma_2^j\Gamma') = i+j. \]
By the hypotheses of the Lemma, ends of differentials at $\gamma_2$ are either all negative or all positive. In the negative case, the $\mathcal{F}$-filtration level can only increase, so $\mathcal{F}$ defines a descending filtration.  In the positive case, we obtain an ascending filtration.

To compute $E^1(\mathcal{F})$, we define a second filtration, $\mathcal{G}$ on $E^0(\mathcal{F})$, by
\[ \mathcal{F}(\gamma_1^i\gamma_2^j\Gamma') = i. \]
$\mathcal{G}$ is also a valid filtration since there is only one differential involving $\gamma_1$.  If $\gamma_1$ is at the positive (resp.\ negative) end of $u$ then $\mathcal{G}$ is an ascending (resp.\ descending) filtration.

The zeroth page $E^0(\mathcal{G})$ is isomorphic to 
\[ \bigoplus_{i,j}\gamma_1^j\gamma_2^{i-j}\cdot C', \]
and so the first page is isomorphic as a vector space to the same formula but with a copy of $H_*(C')$ in each summand.

Write $E^1_i(\mathcal{G})$ to denote the subcomplex of $E^1(\mathcal{G})$ in $\mathcal{F}$-grading $i$.  Then, since one of $\gamma_1$ and $\gamma_2$ is hyperbolic, $E^1_i(\mathcal{G})$ is supported in just two $\mathcal{G}$-filtration levels: 
\begin{itemize}
\item if $\gamma_1$ is hyperbolic, then the two filtration levels are 0 and 1, and
\item if $\gamma_2$ is hyperbolic, then the two filtration levels are $i$ and $i-1$.
\end{itemize}

We therefore have the following structure for $E^1_i(\mathcal{G})$:
\begin{itemize}
\item $E^1_0(\mathcal{G})\iso H_*(C')$, and
\item for $i>0$, $ E^1_i(\mathcal{G})$ is given by the diagram
\[ \gamma_1^j\gamma_2^{i-j} H_*(C') \xrightarrow{\gamma_1^j\gamma_2^{i-j}\Gamma'\mapsto\gamma_1^{j-1}\gamma_2^{i-j+1}\Gamma'} \gamma_1^{j-1}\gamma_2^{i-j+1} H_*(C') \]
if $u$ is a differential from $\gamma_1$ to $\gamma_2$, and
\[ \gamma_1^{j-1}\gamma_2^{i-j+1} H_*(C') \xrightarrow{\gamma_1^{j-1}\gamma_2^{i-j+1}\Gamma'\mapsto\gamma_1^{j}\gamma_2^{i-j}\Gamma'} \gamma_1^{j}\gamma_2^{i-j} H_*(C') \] if $u$ is a differential from $\gamma_2$ to $\gamma_1$.  (In the two diagrams above, $j=1$ if $\gamma_1$ is hyperbolic and $j=i$ if $\gamma_2$ is hyperbolic.) Note that both these maps are isomorphisms.
\end{itemize}
It therefore follows that the second page is given by
\[ E^2_i(\mathcal{G}) \iso \begin{cases*}
H_*(C') &\text{if $i=0$, and}\\
0 &\text{if $i\neq 0$.}
\end{cases*} \]
Hence we have convergence at this page, so the first page of the spectral sequence corresponding to $\mathcal{F}$ is also given by this formula.  This in turn implies that the spectral sequence corresponding to $\mathcal{F}$ also converges at this point, so we obtain
\[  H_*(C) \iso H_*(C') \]
as required.
\end{proof}

\begin{prop}[{\nogapcite[Lemma 9.7.1]{CGH_ECH_OBD}}]\label{relative_ECH_iso_augmented_complex_prop}
There is an isomorphism
\[ ECH(N,\del N,\alpha) \iso ECH\orbits{0he}{}(\mathrm{int}(N),\alpha). \]
\end{prop}
\begin{proof}
We will reproduce here the proof by Colin, Ghiggini and Honda for completeness. In addition, the exact triangle argument involved will prove useful later in the proof of \cref{A_j_iso}.

The complex $ECC\orbits{0he}{}(\mathrm{int}(N),\alpha)$ can be exhibited as the mapping cone of the map
\begin{align*}
f: ECC\orbits{e}{}(\mathrm{int}(N),\alpha) &\to ECC\orbits{e}{}(\mathrm{int}(N),\alpha) \\
\Gamma &\to (1+e_-)\Gamma.
\end{align*}
We therefore obtain a triangle
\[\begin{tikzcd}
ECC\orbits{e}{}(\mathrm{int}(N),\alpha) \arrow[r,"i"] & ECC\orbits{0he}{}(\mathrm{int}(N),\alpha) \arrow[r,"\cdot/h_+"] & ECC\orbits{e}{}(\mathrm{int}(N),\alpha) \arrow[r,"f"]& ~
\end{tikzcd}\]
which gives rise to an exact triangle on homology:
\[\begin{tikzcd}
ECH\orbits{e}{}(\mathrm{int}(N),\alpha)\arrow[rr,"f_*"] && ECH\orbits{e}{}(\mathrm{int}(N),\alpha)\arrow[dl,"i_*"] \\
 & ECH\orbits{0he}{}(\mathrm{int}(N),\alpha)\arrow[ul,"(\cdot/h_+)_*"] & 
\end{tikzcd}\]
Now note that the map
\[ f_* : [\Gamma] \mapsto [\Gamma] + [e_-\Gamma] \]
is injective and hence the map induced by $\cdot/h_+$ is zero.  Then
\[ ECH\orbits{0he}{}(\mathrm{int}(N),\alpha) \iso \frac{ECH\orbits{e}{}(\mathrm{int}(N),\alpha)}{\mathrm{Im}(f_*)} = ECH(N,\del N,\alpha) \]
as required.
\end{proof}

\begin{proof}[Proof of \cref{relative_ECH_equals_ECH}, part (1)]\label{proof_of_rel_ECH_equals_ECH_pt_1}
We have isomorphisms
\[ ECH(N,\del N,\alpha) \iso ECH\orbits{0he}{}(\mathrm{int}(N),\alpha) \iso ECH\orbits{0he}{eh}(\mathrm{int}(N),\alpha), \]
where the first follows from \cref{relative_ECH_iso_augmented_complex_prop} above and the second by the \refNamedThm{Cancellation Lemma}{cancellation_lemma}.  \Cref{ECH0heeh_E10F_ECH_iso} then completes the proof.
\end{proof}

\subsection{The proof of the hat part of Theorem \protect\ref*{relative_ECH_equals_ECH}}

The structure of the proof of the second part of \cref{relative_ECH_equals_ECH} is as follows.  Recall that the hat version of ECH is defined as the homology of the mapping cone of a $U$-map on the embedded chain complex, which is in turn defined by counting holomorphic curves of ECH index 2 passing through a generic point $z\in M$ (c.f.~\cref{ECH_def_section}).

We will start by defining $U$-maps on the complexes 
\[ECC\orbits{0he}{eh}(\mathrm{int}(N),\alpha)\quad\text{and}\quad ECC\orbits{0he}{}(\mathrm{int}(N),\alpha),\]
 denoted by $U^\sharp$ and $U^\natural$ respectively.  We will then show that the homology of the mapping cones of all three $U$-maps are isomorphic, and moreover isomorphic to the homology of the complex
\[ ECC\orbits{0he}{h}(\mathrm{int}(N),\alpha). \]
These results will then be used to prove that the hat versions of ECH and relative ECH are isomorphic.

\begin{defn}~
\begin{enumerate}
\item Define an action on $ECC\orbits{0he}{}(\mathrm{int}(N),\alpha)$, denoted $U^\natural$, by
\[  U^\natural\cdot \gamma\tensor\Gamma := \gamma\tensor d_N'(\Gamma). \]
\item Define an action on $ECC\orbits{0he}{eh}(\mathrm{int}(N),\alpha)$, denoted $U^\sharp$, by
\[  U^\sharp\cdot \gamma\tensor\Gamma := \gamma/e_+ \tensor \Gamma. \]
\end{enumerate}
\end{defn}

\begin{prop}\label{Cone_U_natural_iso_ECC_0he_h}
The mapping cone of $U^\natural$, $C(U^\natural)$, is isomorphic as a chain complex to $ECC\orbits{0he}{h}(\mathrm{int}(N),\alpha)$.
\end{prop}
\begin{proof}
The proof follows easily once we write out the diagram representing $C(U^\natural)$:
\[ ECC\orbits{0he}{}(\mathrm{int}(N),\alpha) \xrightarrow{d'_N} ECC\orbits{0he}{}(\mathrm{int}(N),\alpha),\]
where in addition there are differentials contained within each summand.

This is isomorphic to the following diagram representing $ECC\orbits{0he}{h}(\mathrm{int}(N),\alpha)$:
\[ ECC\orbits{0he}{}(\mathrm{int}(N),\alpha) \xrightarrow{h_-\cdot d'_N} h_-\cdot ECC\orbits{0he}{}(\mathrm{int}(N),\alpha).\tag*{\qedhere}\]
\end{proof}

\begin{prop}\label{U_natural_and_U_sharp_iso_prop}
The mapping cones $C(U^\natural)$ and $C(U^\sharp)$ are chain homotopic.
\end{prop}
\begin{proof}
We begin by defining the following chain map
\begin{align}
\sigma : ECC\orbits{0he}{}(\mathrm{int}(N),\alpha) &\to ECC\orbits{0he}{eh}(\mathrm{int}(N),\alpha)\label{sigma_quasi_iso}\\
\Gamma &\mapsto \sum_{i=0}^\infty e_+^i \tensor (d'_N)^i(\Gamma) \nonumber
\end{align}
where of course the sum is finite since $(d'_N)^i(\Gamma)$ is zero for $i$ large enough. It is easy to check that this map commutes with the $U$-maps:
\[ \sigma(\Gamma)/e_+ = \sum_{i=1}^\infty e_+^{i-1} \tensor (d'_N)^i(\Gamma) = \sum_{i=0}^\infty e_+^{i} \tensor (d'_N)^{i+1}(\Gamma) = \sigma(d'_N(\Gamma)) \]
and hence induces a map, denoted by $\hat\sigma$, between the corresponding mapping cones.  

We therefore have a map between long exact sequences:
\begin{equation}\label{map_between_U_long_exact_sequences_diag}\begin{tikzcd}
\cdots\arrow[r,"U^\natural_*"] & ECH\orbits{0he}{}\arrow[r]\arrow[d,"\sigma_*"] & H_*(C(U^\natural))\arrow[r]\arrow[d,"\hat\sigma_*"] & ECH\orbits{0he}{}\arrow[r,"U^\natural_*"]\arrow[d,"\sigma_*"] & \cdots \\
\cdots\arrow[r,"U^\sharp_*"] & ECH\orbits{0he}{eh}\arrow[r] & H_*(C(U^\sharp))\arrow[r] & ECH\orbits{0he}{eh}\arrow[r,"U^\sharp_*"] & \cdots,
\end{tikzcd}\end{equation}
where here we are omitting ``$(\mathrm{int}(N),\alpha)$'' from the notation for brevity.

By examining the proof of the \refNamedThm{Cancellation Lemma}{cancellation_lemma}, where in this case we have
\[ (\gamma_1, \gamma_2) = (e_+, h_-), \]
we see that the map induced by $\sigma$ on the first page of the spectral sequence corresponding to $\mathcal{G}$ yields an isomorphism
\[\begin{tikzcd}
ECH\orbits{0he}{}(\mathrm{int}(N),\alpha) \arrow[r,"\sigma_*", "\iso"'] & E^1_0(\mathcal{G}) \iso ECH\orbits{0he}{eh}(\mathrm{int}(N),\alpha)
\end{tikzcd}\]
and hence $\sigma_*$ in \refDiagram{map_between_U_long_exact_sequences_diag} above is an isomorphism.  The five lemma completes the proof.
\end{proof}

\begin{prop}[{\nogapcite[Lemma 9.9.1]{CGH_ECH_OBD}}]\label{ECH_0he_h_iso_relative_ECH}~
\[ ECH\orbits{0he}{h}(\mathrm{int}(N),\alpha) \iso \widehat{ECH}(N,\del N,\alpha). \]
\end{prop}
The proof of this Proposition is exactly the same as that of \cref{relative_ECH_iso_augmented_complex_prop}, except the orbit $h_-$ is included throughout.

The final step is to show that $H_*(C(U^\natural)) \iso \widehat{ECH}(M)$---some care is required to deal with the sequence of contact forms $\alpha_i'$ on $M$.  Recall that the $U$-maps on $ECC(M,\alpha_i')$ are defined in terms of a generic basepoint $z$ in the interior of $V$; let $U_i$ denote the corresponding map on $ECC^{L_i}(M,\alpha_i')$.

It is shown~\cite[Lemma 9.9.3]{CGH_ECH_OBD} that there is a single index 2 $J_i'$-holomorphic curve passing through $z$; it has a single positive end at $e_+$ and no negative end---see \cref{perturbed_alpha_i_showing_differentials_fig}.  Therefore each $U_i$ is given by the map
\[ \gamma\tensor \Gamma \to \gamma/e_+\tensor \Gamma,\]
which explains the motivation behind the definition of $U^\sharp$.

The proof of \cref{relative_ECH_equals_ECH} is the completed by  constructing, for every $i\in\N$, an ``action bounded'' version of \refDiagram{map_between_U_long_exact_sequences_diag}.   More precisely, the top row is an $L_i$-bounded copy of the top row of \refDiagram{map_between_U_long_exact_sequences_diag}; the bottom row is given by
\[\begin{tikzcd}
\cdots \arrow[r,"(U_i)_*"] & ECH^{L_i}(M,\alpha_i')\arrow[r] & \widehat{ECH}^{L_i}(M,\alpha_i') \arrow[r]& ECH^{L_i}(M,\alpha_i')\arrow[r,"(U_i)_*"] & \cdots
\end{tikzcd}\]
where the two unlabeled horizontal maps are induced by the mapping cone inclusion and quotient maps; and the chain maps $\sigma_i$ and $\hat\sigma_i$ are constructed in a completely analogous way to the construction of $\sigma$ and $\hat\sigma$ in the proof of \cref{U_natural_and_U_sharp_iso_prop} above.  (For a precise description of this diagram refer to Colin, Ghiggini and Honda~\cite[Section 9.9]{CGH_ECH_OBD}---their construction carries over directly to our setting.)

Finally, we take the direct limit: the maps $(\sigma_i)_*$ limit to the isomorphism $\sigma_*$ and hence we can then honestly apply the five lemma to obtain an isomorphism
\begin{equation}\label{U_natural_ECH_hat_iso} \lim_{i\to\infty}(\hat\sigma_i)_* : H_*(C(U^\natural)) \xrightarrow{\iso} \widehat{ECH}(M). \end{equation}

\begin{proof}[Proof of \cref{relative_ECH_equals_ECH}, part (2)]\label{proof_of_rel_ECH_equals_ECH_pt_2}
This follows from \cref{Cone_U_natural_iso_ECC_0he_h,U_natural_and_U_sharp_iso_prop,ECH_0he_h_iso_relative_ECH} and \cref{U_natural_ECH_hat_iso} above.
\end{proof}

\chapter{A knot version of ECH}\label{a_knot_version_of_ECH_chapter}

\section{Embedded contact knot homology}\label{embedded_contact_knot_homology}

Embedded contact knot homology was first defined in terms of sutured contact homology~\cite[Section 7.2]{CGHH11} and later in terms of the ECC complexes from the previous chapter~\cite[Theorem 10.3.2]{CGH_ECH_OBD} for integral open book decompositions.  In this section we will generalize the definition to rationally null-homologous knots in rational open book decompositions.

Let $M$ be a closed 3-manifold equipped with a rational open book decomposition with binding $K$ and suppose that $\alpha$ is a contact form on $N$ which is adapted to the decomposition $M=N\union \nu(K)$ following the setup from the previous chapter.  As before, $K$ is the core of $\nu(K)$ and $N$ is the mapping torus of some monodromy $\phi:\Sigma \to \Sigma$ which restricts to a $p/q$-twist at the boundary.  Recall that the full ECC complex of $M$ can be written as
\[ ECC\orbits{0he}{eh}(\mathrm{int}(N),\alpha, J)\]
where $J$ is regular and adapted to the symplectization $\R\cross N$.

\begin{defn}
We define the \emph{Alexander grading} by $A(\gamma) := \langle \gamma, \Sigma\cross\set{0} \rangle$ for orbits in $N$, $A(h_+)=A(e_+):=q$, and on the full ECC complex by
\[ A\left( \prod{\gamma_i^{m_i}} \right) = \sum{m_i A(\gamma_i)}, \]
i.e.~the number of times that an orbit set winds around the $S^1$-direction within $N$.
\end{defn}

This defines an ascending filtration, which we call the \emph{knot filtration}, since the only differentials which can change $A$ are those arising from the holomorphic curve from $h_+$ to $\emptyset$.  In analogy to knot Floer homology, we consider this filtration on both the hat versions and full versions of the ECC complexes.
\pagebreak %
\begin{defn}~
\begin{enumerate}
\item Define the \emph{hat version of embedded contact knot homology} to be
\begin{equation}
\widehat{ECK}(K, \alpha) :=  H_*\left(ECC\orbits{he}{h}(\mathrm{int}(N),\alpha,J)\right),
\end{equation}
that is, the homology of the associated graded complex associated to the knot filtration on $\widehat{ECC}$. (Note that we have dropped the differential $h_+\longrightarrow \emptyset$ from the notation.)
\item Define the \emph{full version of embedded contact knot homology}, $ECK(K, \alpha)$, to be the bi-filtered homotopy type of the complex
\begin{equation}\label{full_ECK_complex}
ECC\orbits{0he}{eh}(\mathrm{int}(N), \alpha,J), 
\end{equation}
where the bi-filtering is taken with respect to the knot filtration and the filtration arising from the power of $e_+$. (Notice that similarly differentials can only decrease the power of $e_+$.)
\end{enumerate}
Note that in both these definitions we chosen to omit the ambient manifold, $M$, in which $K$ lies from the notation.
\end{defn}

The definition of $\widehat{ECK}$ was first given by Colin, Ghiggini, Honda and Hutchings in terms of sutured contact homology\cite[Section 7.2]{CGHH11} and they showed that it is well-defined with respect to the choice of (regular, adapted) almost complex structure $J$ and invariant under homotopies of $\alpha$~\cite[Theorem 10.2.2]{CGH_ECH_OBD}.  The full ECK complex was first discussed by Spano~\cite[Definition 2.1]{Spano17}.  However Spano's definition is slightly weaker since it applies only to integral open book decompositions and is defined as the homology of the graded complex associated to the Alexander filtration rather than in terms of bi-filtered homotopy type as seen here.  

The fact that full ECK in part (2) is well-defined with respect to the choice of (regular, adapted) almost complex structure is only conjectured, however we will prove later that
\begin{enumerate}
\item we have well-definedness in the case of an \emph{integral} open book decomposition (\cref{ECK_invariant_ZOBD}), and
\item in the rational case, the $e_+$-filtered homotopy type of
\begin{equation}\label{ECK_complex_no_A_decreasing_diff}
ECC\orbits{he}{eh}(\mathrm{int}(N), \alpha,J)
\end{equation}
is independent of $J$ within each Alexander grading (\cref{ECC_independent_of_alpha}).
\end{enumerate}

Spano uses a dynamical reformulation of the Alexander polynomial to show that, in the case of integral open book decompositions (where $K$ is null-homologous), $\widehat{ECK}(K,\alpha)$ categorifies the Alexander polynomial~\cite[Theorem 0.4]{Spano17}.  We will generalize this result to our setting of rationally null-homologous knots in rational open book decompositions in \cref{ECK_hat_categorifies_Alex_poly_section}.  Furthermore, we have the following conjecture:
\pagebreak %
\begin{conj}\label{ECK_HFK_conjecture}~
\begin{equation}\label{hat_conjecture_iso}
\widehat{ECK}(K,\alpha) \iso \widehat{HFK}(-K)
\end{equation}
and
\begin{equation}\label{full_conjecture} ECK(K,\alpha) \simeq CFK^+(-K), \end{equation}
where $\simeq$ denotes bi-filtered chain homotopy type.
\end{conj}
The isomorphism in \cref{hat_conjecture_iso} was first conjectured (in terms of sutured contact homology and sutured Floer homology) by Colin, Ghiggini, Honda and Hutchings~\cite[Conjecture 1.5]{CGHH11}.   Spano restated \cref{hat_conjecture_iso} in terms of ECK (for integral open book decompositions)~\cite[Conjecture 1.47]{Spano17} and also stated a variant of \cref{full_conjecture} for his version of full ECK~\cite[Conjecture 2.5]{Spano17}.

\subsection{Symmetries of the complex}\label{symmetries_of_the_complex}

In analogy to knot Floer homology, we can represent the bi-filtered complex \cref{full_ECK_complex} graphically by plotting dots in the plane.  A generator is represented by a dot at coordinates $(i,j)$ if its power of $e_+$ is $i$ and its Alexander grading is $j$.  Differentials arising from the holomorphic curve from $e_+$ to $h_-$ are represented by horizontal arrows decreasing the $i$ coordinate by 1, and those arising from the holomorphic curve from $h_+$ to $\emptyset$ are represented by vertical arrows decreasing the $j$ coordinate by 1.  All other differentials do not change the coordinates.

Recall that there is a $U$-action on this complex given by the map $\Gamma \mapsto \Gamma/e_+$; this is a diagonal map, decreasing both the $e_+$ grading and the Alexander grading by 1.  

The subcomplex obtained by just taking those generators with $i=0$ yields $\widehat{ECH}(M)$, and the associated graded of this column gives $\widehat{ECK}(K,\alpha)$.  Furthermore, the full complex has translational symmetry as described in the proposition below.

\begin{prop}\label{ECK_translational_symmetry}
For each $i'\ge 0$, consider the quotient complex
\[ C_{i'} := ECC\orbits{0he}{eh}(\mathrm{int}(N),\alpha,J; i\ge i'), \]
generated by orbit sets with coordinates $(i,j)$ such that $i\ge i_0$.  In particular $C_0$ is the full ECK complex.  Then the diagonal map
\begin{align*}
C_{i'} &\to C_0  \\
\Gamma &\mapsto \Gamma/e_+^{i'}
\end{align*}
is an isomorphism of chain complexes.
\end{prop}
\begin{proof}
At the level of generators, the proposition holds trivially, since the structure is given as a tensor product
\[ \mathcal{R}[e_+] \tensor ECC\orbits{0he}{h}(\mathrm{int}(N),\alpha,J). \]
It remains to show that the differentials also satisfy the translational symmetry.  First note that the only non-filtration preserving differentials arise from the two holomorphic curves
\[  h_+\longrightarrow\emptyset\quad\text{and}\quad e_+\longrightarrow h_-;\]
such differentials clearly appear in accordance with the translational symmetry. 

Suppose we have a non-filtration preserving differential
\begin{equation} \label{differential_ith_column}
  e_+^i\Gamma_+ \longrightarrow e_+^i\Gamma_-,
	\end{equation}
where $\Gamma_\pm$ are orbit sets (with equal Alexander grading) constructed from the orbits $h_\pm$, $e_-$ and those in $\mathrm{int}(N)$.   

Then by slope calculus (\cref{slope_calculus}) the differential necessarily contains $i$ trivial cylinders over the orbit $e_+$, and hence by removing these we obtain a differential
\begin{equation} \label{differential_zeroth_column}
\Gamma_+ \longrightarrow \Gamma_-,
\end{equation}
which appears in the zeroth column ($i=0$) of the full complex.

Conversely, any such differential in the zeroth column of the form in \cref{differential_zeroth_column} naturally gives rise to a differential in the $i$-th column of the form in \cref{differential_ith_column} by appending the necessary trivial cylinders.
\end{proof}

Notice that by definition the full complex is supported in the ``upper triangular'' region
\[ \set{ (i,j)\in \Z^2 \gappy{|} i\ge0, j\ge i}. \]
Furthermore, we will show later (\cref{ECK_supported_genus}) that, for an integral open book decomposition, $\widehat{ECK}(K,\alpha)$ is supported within Alexander gradings $0\le j\le 2g$, where $g$ is the genus of the page $\Sigma$.  Therefore, when considered up to bi-filtered chain homotopy type, the full complex is in fact supported in the region
\[ \set{ (i,j)\in \Z^2 \gappy{|} i\ge0, i\le j \le i+2g }. \]
We will call this region the \emph{diagonal region of height $2g$}.  This fact, along with \cref{ECK_translational_symmetry}, implies that the entirety of the full ECK complex can be computed, up to bi-filtered homotopy type, by considering only the (finite) complex
\begin{equation}\label{ECK_first_column_up_to_2g}
ECC\orbits{0he}{h}(\mathrm{int}(N),\alpha,J;A\le 2g),
\end{equation}
which we define to be the subcomplex of the zeroth column ($i=0$) generated by orbit sets with Alexander grading less than or equal to $2g$. More precisely, we have the following theorem.

\begin{thm}\label{full_ECK_computable_from_ECK_hat}
In the case of an integral open book decomposition of genus $g$, the full ECK complex is bi-filtered homotopy equivalent to the complex with underlying vector space
\[ \mathcal{R}[e_+]\tensor ECC\orbits{0he}{h}(\mathrm{int}(N),\alpha,J;A\le 2g), \]
and with differential
\[  e_+^i\tensor \Gamma \mapsto e_+^{i-1}\tensor h_-\Gamma + e_+^{i}\tensor d\Gamma, \]
where $d$ is the differential on the complex in \cref{ECK_first_column_up_to_2g}.
\end{thm}

In \cref{ECK_for_torus_knots_section} we will compute ECK of a family of torus knots and the structure and symmetries discussed above will be apparent. In particular see \cref{ECK_full_complex_diagram}.

\subsection{\texorpdfstring{$\widehat{ECK}$}{ECK hat} categorifies the Alexander polynomial}\label{ECK_hat_categorifies_Alex_poly_section}

In the case of an integral open book decomposition of $S^3$ with binding $K$, $\widehat{ECK}(K,\alpha)$ categorifies the Alexander polynomial, $\Delta_K(t)$~\cite[Theorem 0.4]{Spano17}.  This is, of course, also the case for knot Floer homology and provides strong evidence for their conjectured equivalence in this case.  In this section, we will generalize this result to rational open book decompositions of arbitrary 3-manifolds $M$.

First it is necessary to discuss what is meant by ``categorifying the Alexander polynomial'' in this setting, and in order to answer that we must first define the \emph{graded Euler characteristic} of $\widehat{ECK}$ and introduce the concept of the \emph{Turaev torsion} of the knot complement.

Recall that $\widehat{ECK}(K,\alpha)$ splits as a direct sum over $H_1(M\sminus K;\Z)$---write
\[ \widehat{ECK}(K,\alpha;A) \]
to denote the summand of $\widehat{ECK}(K,\alpha)$ lying in homology class $A$.
\begin{defn}
The \emph{graded Euler characteristic} of $\widehat{ECK}$ is defined by the equation
\[ \chi\left( \widehat{ECK}(K,\alpha) \right) := \sum_{A\in H_1(M\sminus K;\Z)} \chi\left( \widehat{HFK}(K,\alpha;A) \right)\cdot [A] \in \Z[H_1(M\sminus K;\Z)], \]
where on the right $\chi$ denotes the normal Euler characteristic and $[A]$ is the group ring element corresponding to the homology class $A$.
\end{defn}

The \emph{Turaev torsion} is a 3-manifold invariant which is a generalization of the Reidemeister torsion, defined by considering the torsion of an acyclic chain complex arising from the universal abelian cover of the manifold in question.  For an excellent introduction and precise definition refer to Turaev~\cite{Turaev1}.  

Let $H$ denote the homology group $H_1(M\sminus K;\Z)$.  The Turaev torsion, denoted $\tau(M\sminus K)$, is defined as an element of
\[ \frac{Q\big( \Z[H] \big)}{\pm H}, \]
where $Q\big( \Z[H] \big)$ is the total ring of fractions, i.e.~the localization of the group ring $\Z[H]$ by the multiplicative set of non-zerodivisors.  In our setting we will consider the torsion as an infinite power series in the group ring $\Z[H]$. The quotient by $\pm H$ means that the torsion is only defined up to multiplication by monomials in the group ring.

For manifolds with torus boundary, $\tau$ is easily computed via a presentation of the fundamental group.  We will briefly describe this process, following Turaev~\cite[Section II.1]{Turaev3}.  Without loss of generality assume that the presentation is geometric in the sense that it comes from a handle decomposition of $M\sminus K$; such a presentation will have generators $a_1,\dots,a_{n+1}$ and relators $w_1,\dots,w_n$ for some $n\in\N$.  Furthermore assume that $a_1$ is chosen so that its image $[a_1]\in H_1(M\sminus K)$ has infinite order.

Let $\phi$ denote the natural map $\Z[\pi_1(M\sminus K)]\to \Z[H] $ and consider the matrix $A$ with entries
\[ \phi\left(\frac{\del w_i}{\del a_j}\right),\quad 1\le i\le n,\ 1\le j\le n+1, \]
where here we are employing the Fox derivative~\cite{Fox}.  Let $A'$ be the $n\cross n$ minor obtained by deleting the first column.  Then we have~\cite[Section II.1.7]{Turaev3}
\[ \tau(M\sminus K) \gappy{\dot{=}} \mathrm{det}(A')\cdot(1-[[a_1]])^{-1}, \]
where $\dot{=}$ denotes equality up to multiplication by a monomial in $\Z[H]$. 
\begin{eg}
If $K$ is a knot in $S^3$, then $H_1(S^3\sminus K;\Z)\iso\Z$ and the group ring is given simply by the polynomial ring $\Z[t,t^{-1}]$ and as such we write the generators as $t^j$ rather than $[j]$.  Then the Turaev torsion is related to the standard Alexander polynomial~\cite[Theorem 11.8]{Turaev1} by the equation
\[ \Delta_K(t)\gappy{\dot{=}} (1-t)\cdot\tau(M\sminus K). \]
For example, the torus knot $T(2,5)$ has
\[\begin{split}
\tau\left(S^3\sminus T(2,5)\right) \gappy{&\dot{=}} \frac{1-t+t^2-t^3+t^4}{1-t} \\
\gappy{&\dot{=}} 1 + t^2 + t^4 + t^5+t^6+\cdots.
\end{split}\]
\end{eg}

We will now digress briefly to discuss how knot Floer homology is related to the Turaev torsion.  First recall that the set of spin$^c$ structures on $M\sminus K$ forms an affine copy of $H_1(M\sminus K;\Z)$, and that, after a non-canonical choice of zero element, we therefore have a splitting
\[  \widehat{HFK}(K) = \bigoplus_{A\in H_1(M\sminus K;\Z)}\widehat{HFK}(K;A). \]
The \emph{graded Euler characteristic} of $\widehat{HFK}$ is then defined analogously to that of $\widehat{ECK}$, by
\[ \chi\left( \widehat{HFK}(K) \right) := \sum_{A\in H_1(M\sminus K;\Z)} \chi\left( \widehat{HFK}(K;A) \right)\cdot [A] \in \Z[H_1(M\sminus K;\Z)]. \]
Notice that, since we made a non-canonical identification between spin$^c$ structures and elements of $H_1(M\sminus K;\Z)$, this Euler characteristic is well-defined only up to multiplication by monomials in the group ring.
\begin{lemma}[{\nogapcite[Proposition 2.1]{RasRas}}]\label{RasRas_lemma}
$\widehat{HFK}$ ``categorifies the Alexander polynomial'' in the sense that
\[ \chi\left( \widehat{HFK}(K) \right) \gappy{\dot{=}} (1-[\mu])\cdot\tau(M\sminus K), \]
where $\mu\in H_1(M\sminus K;\Z)$ is the homology class of the meridian of $K$.
\end{lemma}
\addtocounter{equation}{-2}
\begin{eg}[continued]
Return to the torus knot $K=T(2,5)\subset S^3$ and suppose that we perform ($p\mu+q\lambda$)-Dehn surgery on $K$, where $\mu$ and $\lambda$ are a choice of meridian and null-homologous longitude for $K$.  Then $[p\mu+q\lambda]=t^p$ and hence
\[ \begin{split}
\chi\left( \widehat{HFK}(K) \right) \gappy{&\dot{=}} (1-t^p)\cdot\tau(M\sminus K) \\
\gappy{&\dot{=}} 1+t^2+t^4+t^5+t^6 +\cdots+ t^{p-2} + t^{p-1} + t^{p+1} + t^{p+3}.
\end{split}\]
\end{eg}
\addtocounter{equation}{1}

\begin{thm}\label{ECK_hat_categorifies_Alex_poly}
$\widehat{ECK}$ ``categorifies the Alexander polynomial'', in the sense that
\[ \chi\left( \widehat{ECK}(K,\alpha) \right) \gappy{\dot{=}} (1-[\mu])\cdot\tau(M\sminus K), \]
and hence $\widehat{ECK}$ and $\widehat{HFK}$ are isomorphic at the level of Euler characteristic.
\end{thm}

The proof of this theorem follows the ideas of Spano~\cite{Spano17}, so first we will discuss some of the ingredients of his set-up.  The most important is the well-known \emph{twisted Leftschetz zeta function} which is a formal polynomial obtained by counting periodic orbits of flows.

Let $N$ be a manifold with torus boundary and suppose that $R$ is a vector field on $N$ which is non-degenerate, in the sense that all closed orbits of $R$ are isolated.  Let $\phi_{R}$ denote the flow of $R$ on $N$, and suppose that $\gamma$ is a periodic orbit of $\phi_{R}$.  Recall from \cref{The_ECH_index} that associated to $\gamma$ is a Lefschetz sign, $\epsilon(\gamma)\in\set{1,-1}$.\footnote{Actually, in \cref{The_ECH_index} we only defined the Lefschetz sign for Reeb orbits, but the definition of the linearized return map $f_\gamma$ is easily extended to closed non-degenerate orbits in a non-Reeb setting, and hence the Lefschetz sign is well-defined.}
\begin{defn}~
\begin{enumerate}
\item The \emph{local Lefschetz zeta function} of $\phi_{R}$ near $\gamma$ is the formal power series
\[ \zeta_\gamma(t) := \sum_{i\ge 1} \epsilon(\gamma^i)\frac{t^i}{i} \in \Z[t]. \]
\item The \emph{twisted Lefschetz zeta function} of $\phi_{R}$ is defined\footnote{Note that Spano's definition is denoted by $\zeta_\rho(\phi_{R})$ where $\rho$ denotes an abelian cover of $N$; here we are taking $\rho$ to be the universal abelian cover.} by taking the following product across all simple orbits of $R$:
\[ \zeta(\phi_{R}) := \prod_{\gamma}\zeta_\gamma([[\gamma]]) \in \Z[H_1(N;\Z)]. \]
\end{enumerate}
\end{defn}
The lemma below follows immediately from \cref{lemma_lefschetz_sign} by a simple calculation using the Taylor series for the natural logarithm.
\begin{lemma}\label{local_lefschetz_for_elliptic_and_hyperbolic}
If $\gamma$ is a Reeb orbit, then
\[ \zeta_\gamma(t) = f_\gamma(t) := \begin{cases*}
1-t &\text{if $\gamma$ is positive hyperbolic,}\\
1+t &\text{if $\gamma$ is negative hyperbolic, and}\\
(1-t)^{-1}=1+t+t^2+\cdots &\text{if $\gamma$ is elliptic.}\\
\end{cases*}\]
\end{lemma}
The proof of \cref{ECK_hat_categorifies_Alex_poly} proceeds by exploiting the following relation between the twisted Lefschetz zeta function and the Turaev torsion, due to Fried\footnote{Note that Fried calls the Turaev torsion the ``Alexander quotient'', and denotes it by $\mathrm{ALEX}(N)$.}:
\begin{lemma}[{\nogapcite[Theorem 7]{Fried}}]\label{zeta_fn_is_ALEX_torsion}
Suppose that $R$ is a non-degenerate flow on $N$ which is
\begin{enumerate}
\item circular, meaning that there exists a map $\theta:N\to S^1$ such that $d\theta(R)>0$; and
\item transverse to $\del N$ and pointing \emph{out} of $N$.
\end{enumerate}
Then
\[ \zeta(\phi_{R}) \gappy{\dot{=}} \tau(N). \]
\end{lemma}
\begin{proof}[Proof of \cref{ECK_hat_categorifies_Alex_poly}]
To start, note that the Reeb vector field $R$ associated to $\alpha$ is circular and non-degenerate on the interior of $N$, but at the boundary it is neither non-degenerate nor pointing out of $N$.  We will alter $R$ slightly so that it satisfies all the hypotheses of \cref{zeta_fn_is_ALEX_torsion}.

To do that, let $L>0$ and suppose that $\del N\cross[2,2+\epsilon_L]$ is a neighbourhood of $\del N$ such that all orbits with action less than $L$ lie outside this neighbourhood.  Then on this region replace the vector field $R$ with
\[ R_L := R - a_L(y)\del y, \]
where $a_L:[2,2+\epsilon_L]\to\R$ is a smooth non-negative function which is zero on a neighbourhood of $2+\epsilon_L$ and positive on a neighbourhood of 2.  Then $R_L$ is transverse to $\del N$ and identical to $R$ near all orbits of action less than $L$.  Furthermore the Morse-Bott torus of Reeb orbits at $\del N$ has been eliminated, so $R_L$ is non-degenerate.

We can therefore apply \cref{zeta_fn_is_ALEX_torsion} to obtain:
\begin{equation}\label{HFK_zeta_function_L_eqn}
\tau(N) \gappy{\dot{=}} \zeta(\phi_{R_L}).
\end{equation}
But this equality holds for all $L>0$ and furthermore notice that, for every $A\in H_1(N;\Z)$, the coefficient of $[A]$ on the right hand side of \cref{HFK_zeta_function_L_eqn} stabilizes once $L$ is larger than the action of all orbit sets with homology class $A$.  We can therefore let $L$ tend to $\infty$ on the right hand side of \cref{HFK_zeta_function_L_eqn}, replacing $\zeta(\phi_{R_L})$ with
\[ \zeta\left(\restr{\phi_R}{\mathrm{int}(N)}\right), \]
the twisted Lefschetz zeta function obtained by considering all simple orbits of $R$ in $\mathrm{int}(N)$, a set which (recall \cref{Morse-Bott_contact_homology}) is denoted by $\mathcal{P}'$.

Using \cref{local_lefschetz_for_elliptic_and_hyperbolic}, we now replace the contribution to $\zeta\left(\restr{\phi_R}{\mathrm{int}(N)}\right)$ from each orbit with the contribution $f_\gamma([\gamma])$.  This yields
\[ \tau(N) \gappy{\dot{=}} \prod_{\gamma\in\mathcal{P}'}f_\gamma([[\gamma]]).\]

Next, we turn to the Euler characteristic of $\widehat{ECK}$, which we will compute as the graded Euler characteristic of the chain complex
\[  ECC\orbits{he}{h}(\mathrm{int}(N),\alpha,J). \]
We claim that this is given by the following product
\begin{equation}\label{chi_ECK_orbit_product}
 \chi\left( ECC\orbits{he}{h}(\mathrm{int}(N),\alpha,J)\right ) = \prod_{\gamma\in\set{h_+,e_-,h_-}\union\mathcal{P}'} f_\gamma([[\gamma]]),
\end{equation}
and remark that, since
\[\begin{split}
\prod_{\gamma\in\set{h_+,e_-,h_-}\union\mathcal{P}'} f_\gamma([[\gamma]]) &= (1-[[h_+]])(1-[[h_-]])(1-[[e_-]])^{-1}\prod_{\gamma\in\mathcal{P}'} f_\gamma([\gamma]) \\
&= (1-[\mu])\prod_{\gamma\in\mathcal{P}'} f_\gamma([\gamma])
\end{split}\]
(as $[h_+]=[h_-]=[e_-]=\mu$), this claim completes the proof.

We will verify \cref{chi_ECK_orbit_product} by recursively considering the contribution to the Euler characteristic from simple orbits, recalling the absolute $\Z/2$-homological grading from \cref{The_ECH_index}.  Suppose that we have simple orbits $\set{\gamma_i}_{i\in\N}$ and let $\mathcal{P}_k$ denote the collection of orbit sets constructed from orbits $\gamma_1$ through $\gamma_k$.  Then if
\[ \mathcal{P}_{k-1} = \set{\Gamma_1,\Gamma_2,\dots} \]
and $\gamma_k$ is positive hyperbolic, i.e.~with $\Z/2$-homological grading 1 and Lefshetz sign $-1$, then
\[ \mathcal{P}_k = \set{\Gamma_1,\Gamma_2,\dots}\union\set{\gamma_k\Gamma_1,\gamma_k\Gamma_2,\dots } \]
and $\epsilon(\gamma_k\Gamma_j) =-\epsilon(\Gamma_j)$, so
\[ \sum_{\Gamma\in\mathcal{P}_k}{\epsilon(\Gamma)[[\Gamma]]} = (1-[[\gamma_k]])\cdot\sum_{\Gamma\in\mathcal{P}_{k-1}}{\epsilon(\Gamma)[[\Gamma]]}. \]
Similarly a negative hyperbolic orbit, with grading 0 and Lefschetz sign $+1$, contributes by a factor of $(1+[[\gamma]])$ and an elliptic orbit, all of whose covers have grading 0 (and Lefschetz sign $+1$), contributes
\[ 1+[[\gamma]]+[[\gamma]]^2+[[\gamma]]^3+\dots = (1-[[\gamma]])^{-1}.\]
Therefore, when we take the contributions to the Euler characteristic from all the simple orbits, we obtain the product seen on the right hand side of \cref{chi_ECK_orbit_product}.
\end{proof}

\subsection{ECK for a family of torus knots}\label{ECK_for_torus_knots_section}

In this section we will compute ECK of the torus knots $T(2,n)\subset S^3$ for odd $n$ greater than or equal to 3.  It is well-known that for these knots, the knot complement $N$ is fibred over $S^1$, where each fibre is the Seifert surface of the knot (with genus $(n-1)/2$) and the monodromy has finite order $2n$. This construction is outlined in Kauffman's book~\cite[Chapter 19 (pp.~393-6)]{Kauffman}.  Taking the quotient of the monodromy action yields an orbifold which is the base orbifold of the Seifert-fibred fibration of the knot. As a result we see two elliptic orbits, $e_2$ and $e_n$, in the knot complement, with Alexander gradings $2$ and $n$ respectively.  All other orbits have Alexander gradings greater than or equal to $2n$.

By \cref{full_ECK_computable_from_ECK_hat}, we can compute the entirety of ECK by considering only the orbits $e_\pm$, $h_\pm$, $e_2$ and $e_n$ and differentials between them.  The differentials in the interior of $N$, with ends at $e_2$ and $e_n$, are a priori harder to understand, but can be deduced by exploiting the isomorphism 
\begin{equation}\label{ECH_torus_knot_is_F_eqn}
\widehat{ECH}(N,\del N,\alpha,J) \iso \widehat{ECH}(S^3) \iso \F_2
\end{equation}
arising from \cref{relative_ECH_equals_ECH}.  Recall that the left hand side of this equation is defined via an equivalence relation $e_-\sim\emptyset$; this means that it is possible to compute $\widehat{ECH}$ by considering the complex
\begin{equation}\label{complex_only_int_and_h_minus}
 ECC\orbits{}{h}(\mathrm{int}(N),\alpha,J),
\end{equation}
augmented with additional Alexander filtration-decreasing differentials
\[ \Gamma_1 \longrightarrow \Gamma_2 \]
whenever there is a an honest differential in the full ECK complex between the orbit sets
\[ \Gamma_1 \quad\text{and}\quad e_-^{A(\Gamma_1)-A(\Gamma_2)}\Gamma_2 \]
(with non-connector part contained entirely in the interior of $N$).

Furthermore, by the \refNamedThm{Cancellation Lemma}{cancellation_lemma}, the complex \cref{complex_only_int_and_h_minus} computes $\widehat{ECK}$ of the torus knot $T(2,5)$.  By \cref{ECK_supported_genus} we know that $\widehat{ECK}$ vanishes in Alexander gradings greater than $2g$, and this results in the following explicit description of $\widehat{ECK}(T(2,5))$:
\begin{equation}\label{ECK_hat_torus_knot_complex_table}
\begin{tabular}{c|c|c|c|c|c|c|c}
$A$ grading & 0 & 1 & 2 & 3 & $\cdots$ & $n-2$ & $n-1$ \\
\hline
 & $\emptyset$ & $h_-$ & $e_2$ & $h_-e_2$ & $\cdots$ & $h_-e_2^{g-1}$ & $e_2^g$ 
\end{tabular}
\end{equation}
We will now exploit \cref{ECH_torus_knot_is_F_eqn} to prove the existence of a differential
\[ e_2 \longrightarrow e_-h_-. \]
First note that, when we augment the complex \cref{ECK_hat_torus_knot_complex_table} with Alexander filtration-decreasing differentials as discussed above, the resulting complex must, by \cref{ECH_torus_knot_is_F_eqn}, have homology $\F_2$. By recalling the absolute $\Z/2$-homological grading on ECC, all such differentials must have a copy of $h_-$ at precisely one end; by the \refNamedThm{Trapping Lemma}{trapping_lemma} it must therefore be at the negative end.  This means that the only possibility is for such differentials to arise from holomorphic curves of the form
\[ e_2^i \longrightarrow h_- e_-^{i-2i'-1}e_2^{i'}. \]
We now consider $e_-^i$ for each $i$ in turn.  First, the empty set---we cannot have any differentials of the above form either into or out of $\emptyset$, so this must survive as the (sole) generator of $\widehat{ECH}(S^3)$.  Hence all other generators in \cref{ECK_hat_torus_knot_complex_table} must be killed.

Next consider $e_2$---it is clear from the classification above that the only possibility is a differential
\[ e_2 \longrightarrow h_-e_- \]
which manifests itself as a Alexander filtration-decreasing differential
\[ e_2 \longrightarrow h_- \]
in the complex \cref{ECH_torus_knot_is_F_eqn}.  By adding trivial cylinders we also obtain differentials
\[ e_2^i \longrightarrow h_-e_2^{i-1} \]
for each $i> 1$, which yields the complex
\begin{equation}\label{T_2_5_complex_first_column}
\begin{tikzpicture}
\fill (0,0) circle (0.05) node[above] {$\emptyset$};
\fill (1.5,0) circle (0.05) node[above] {$h_-$};
\fill (3,0) circle (0.05) node[above] {$e_2$};
\fill (4.5,0) circle (0.05) node[above] {$h_-e_2$};
\fill (6,0) circle (0.05) node[above] {$e_2^2$};
\node at (7.5,0) {$\cdots$};
\fill (9,0) circle (0.05) node[above] {$h_-e_2^{g-1}$};
\fill (10.5,0) circle (0.05) node[above] {$e_2^{g}$};
\draw [->] (3 -0.15,0) -- +(-1.2,0);
\draw [->] (6 -0.15,0) -- +(-1.2,0);
\draw [->] (10.5 -0.15,0) -- +(-1.2,0);
\end{tikzpicture}
\end{equation}

We claim that there are no other holomorphic curves of the form
\[ e_2^i \longrightarrow h_-e_-^{i-2i'-1}e_2^{i'} \]
with $i'<i-1$; this follows by arguing via the ECH index.  Indeed, consider the orbit sets $e_2^i$ and $h_-e_2^i$.  The existence of index 1 curves
\[\begin{tikzcd}
& h_+e_2^i\arrow[dl]\arrow[dr] & & h_-e_2^i\arrow[dl] \\
e_2^i & & e_-e_2^i &
\end{tikzcd}\]
(recall that there are two index 1 curves between $h_-$ and $e_-$ which cancel in the chain complex) implies that the homological grading in \cref{T_2_5_complex_first_column} is identical to the Alexander grading.  Therefore no other differentials are possible.

By \cref{full_ECK_computable_from_ECK_hat}, this complex is filtered chain homotopic to the zeroth column of the full ECK complex and furthermore the full ECK complex can be computed by taking the tensor product with $\mathcal{R}[e_+]$ and considering all occurences of the holomorphic curve
\[ e_+\longrightarrow h_-. \]
In the case of $T(2,5)$ this results in differentials
\[ e_+^ie_2^j \longrightarrow h_-e_+^{i-1}e_2^j, \quad\text{$i\ge 1$, $0\le j < g$.}\]

\begin{absolutelynopagebreak} %
The resulting complex, shown in \cref{ECK_torus_knot_fig} for the case $n=5$, is exactly what we expect, since it is isomorphic to $CFK^+(-T(2,n))$, providing additional evidence towards \cref{ECK_HFK_conjecture}.
\end{absolutelynopagebreak} %

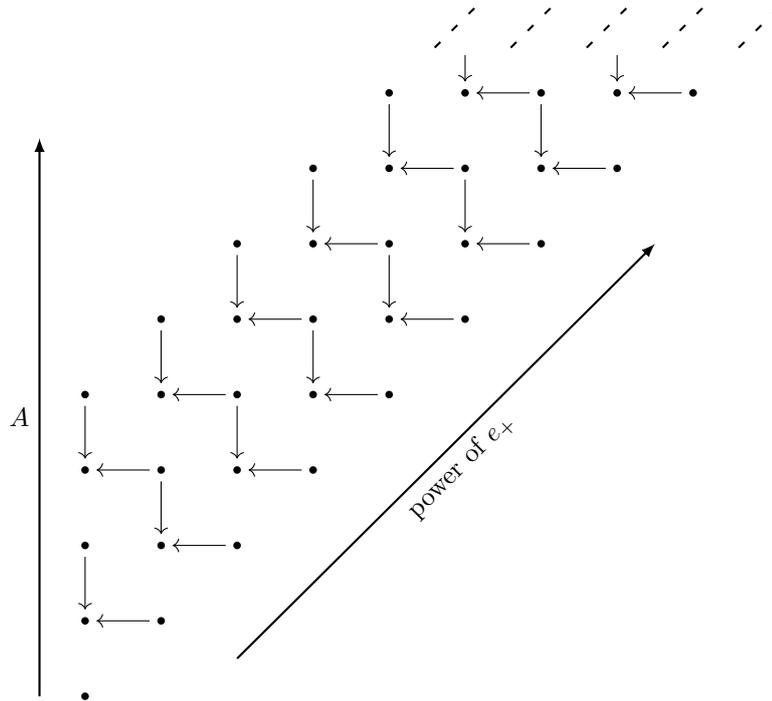
\begin{figure}\label{ECK_torus_knot_fig}\centering
	\begin{tikzpicture}  %
		\draw [-latex,thick] (-0.6,0) -- (-0.6,7.4) node [midway, left] {$A$};
		\draw [-latex,thick] (2,0.5) -- +(5.5,5.5);
		\node [rotate=45] at (4.75+0.2,3.25-0.2) {power of $e_+$};
		
		\begin{scope}
			\path[clip] (-0.5,-0.5) rectangle (8.5,8.5);
			\foreach \x in {0,...,8} {
				\foreach \y in {0,...,4}
					\fill (\x,\x+\y) circle (0.05);
				\draw [->] (\x+0.85,1+\x) -- +(-0.7,0);
				\draw [->] (\x+0.85,3+\x) -- +(-0.7,0);
				\draw [->] (\x,\x+1.85) -- +(0,-0.7);
				\draw [->] (\x,\x+3.85) -- +(0,-0.7);
				}
		\end{scope}

		\foreach \x in {4,...,8}
			\draw [loosely dashed,thick] (\x+0.6,8.6) -- +(0.6,0.6);

	\end{tikzpicture}
  \caption{A diagram representing $ECK(T(2,5))$.  The vertical coordinate represents the Alexander grading and the horizontal coordinate represents the power of $e_+$.  Moreover, the $i$-th column is a copy of \cref{T_2_5_complex_first_column} multiplied by $e_+^i$. The lower (resp.\ upper) diagonal of horizontal arrows consists of differentials of the form $e_+^i \longrightarrow h_-e_+^{i-1}$ (resp.\ $e_+^ie_2 \longrightarrow h_-e_+^{i-1}e_2$).  The homology of the first column gives $\widehat{ECH}(S^3)$ and the associated graded of the first column gives $\widehat{ECK}(T(2,5))$.}
	\label{ECK_full_complex_diagram}
\end{figure}

\section{Cobordism maps via Seiberg-Witten theory}\label{cobordism_maps_via_SW_theory}

Before progressing further with ECK, it is necessary to understand in greater detail the cobordism maps from \cref{direct_limits_through_cobordism_maps} and how they are induced by non-canonical chain maps via Seiberg-Witten theory.  The purpose for this is to obtain invariance results for ECK (c.f.~\cref{J_invariance_bounded_case_thm,ECC_independent_of_alpha,ECK_invariant_ZOBD}), which are in turn used to prove \cref{ECK_supported_genus} and also the surgery formula (\cref{surgery_formula}).

Recall that an exact symplectic cobordism between two contact manifolds $(M_+,\alpha_+)$ and $(M_-,\alpha_-)$ is a symplectic 4-manifold $(X,d\lambda)$ such that $\del X = M_+ \union (-M_-)$ and $\restr{\lambda}{M_\pm}=\alpha_\pm$.  We will often denote such a cobordism simply by $(X,\lambda)$.  We can identify a neighbourhood of $M_+$ with $(-\epsilon,0]\cross M_+$ and a neighbourhood of $M_-$ with $[0,\epsilon)\cross M_-$ and furthermore it is possible to form this identification in such a way that $\alpha$ is given by $e^s \alpha_\pm$ on these neighbourhoods.  We then form the \emph{completion} of $X$ by attaching the half symplectizations to either end of $X$:
\[ \bar{X} := \big((-\infty,0]\cross M_-\big) \union X \union \big([0,\infty)\cross M_+\big). \]
More explicitly, the one-form on $\bar{X}$ is defined by the formula
\begin{equation*}
\bar{\lambda} :=
\begin{cases*}
e^s \alpha_- &\text{on $(-\infty,0]\cross M_-$} \\
\lambda       &\text{on $X$} \\
e^s \alpha_+ &\text{on $[0,\infty)\cross M_+$}.
\end{cases*}
\end{equation*}
An almost complex structure $J$ on $(\bar{X},\bar{\lambda})$ is said to be \emph{cobordism-admissible} if it is compatible with $d\lambda$ on $X$ and adapted in the sense of \cref{moduli_spaces_of_holomorphic_curves} to $\alpha_\pm$ on the attached ends.  Such $J$ will be used extensively when relating Seiberg-Witten theory to ECH later.

An important definition which we need moving forwards is that of a \emph{product region}.

\begin{defn}[{\nogapcite[Definition 1.7]{HT_Arnold13}}]\label{product_region_defn}
Suppose that $(Z,\alpha_0)$ is an open contact manifold and we have an embedding $[c_-,c_+]\cross Z$ into $X$ such that the following is true:
\begin{itemize}
	\item $\set{c_\pm} \cross Z$ is mapped to $M_\pm$ and $(c_-,c_+)\cross Z$ is mapped to $\mathrm{int}(X)$.
	\item The pullback of $\lambda$ is $e^s \alpha_0$, where $s$ is the $(c_-, c_+)$ coordinate.
	\item The pullback $J_0$ of $J$ on $\bar{X}$ to $[c_-,c_+]\cross Z$ has the following properties:
		\begin{itemize}
			\item $J_0$ restricted to $\mathrm{ker}(\alpha_0)$ is invariant in the $s$-direction, and
			\item $J_0(\del_s)=f(s)R_{\alpha_0}$, where $f$ is a positive function.
		\end{itemize}
\end{itemize}
We call the image of $[c_-,c_+]\cross Z$ in $X$ a \emph{product region}.
\end{defn}

Product regions are interesting because when we form the completion $(\bar{X},\bar{\lambda})$, the subspace 
\[ \bigg( \big((-\infty,0]\cross Z\big) \union \big([c_-,c_+]\cross Z\big) \union \big([0,\infty)\cross Z\big) , d\bar{\lambda}\bigg) \]
is symplectomorphic to the symplectization
\[ \big(\R\cross Z, d(e^s\alpha_0)\big) \]
and we can therefore apply the results of \cref{top_constraints_on_holo_curves,topological_constrants_in_the_MB_settting}, including positivity of intersection, the Blocking Lemma and the Trapping Lemma.

In fact, since the \refNamedThm{Trapping Lemma}{trapping_lemma} is a result concerning only the ends of holomorphic curves, it also applies away from product regions:

\begin{lemma}[The Trapping Lemma for cobordisms]\label{trapping_lemma_cobordisms}
Suppose that $(X,\lambda)$ is an exact symplectic cobordism between $M_+$ and $M_-$ and that $T_+$ (resp.\ $T_-$) is a Morse-Bott torus in $M_+$ (resp.\ $M_-$).   Suppose also that $u$ is a holomorphic curve in the completion $(\bar{X},J)$ and $u$ has a one-sided positive (resp.\ negative) end at some $\gamma \subset T_+$ (resp.\ $T_-$).  Then $T_+$ (resp.\ $T_-$) must be a positive (resp.\ negative) Morse-Bott torus.
\end{lemma}
\begin{proof}
The proof follows since the almost complex structure on the ends of the completion is equal to a symplectization.  Therefore the ends of $u$ are constrained by the same results as the ends of $u$ in a symplectization, so the normal \refNamedThm{Trapping Lemma}{trapping_lemma} applies. 
\end{proof}

In addition, positivity of intersection in four dimensions (\cref{positivity_of_intersections_4}) clearly still applies in any completion $(\bar{X},J)$.  However positivity of intersection in three dimensions (\cref{positivity_of_intersections_3}) and the \refNamedThm{Blocking Lemma}{blocking_lemma} do not apply in general, since we have lost the notion of a product cylinder.

In very special cases, the entirety of $X$ is a product region.  Then $X=[c_-,c_+]\cross M$ and is a cobordism between $(M,e^{c_+}\alpha)$ and $(M,e^{c_-}\alpha)$.  The cobordism-admissible almost complex structure, when restricted to the ends of the completion $\bar{X}$, is adapted on the two pieces to the forms $e^{c_\pm}\alpha$ respectively.  The lemma below examines how the almost complex structures adapted to $\alpha$ and $e^{c_\pm}\alpha$ are related.

\begin{lemma}\label{rescale_alpha_J_iso}
Suppose that $(M,\alpha)$ is a closed contact manifold and let $J$ be a regular almost complex structure on $\R \cross M$ which is adapted to $\alpha$. Then for any $c\in\R$ there is a regular almost complex structure $J^c$ which is adapted to $e^c\alpha$ such that:
\begin{enumerate}
\item On each slice $\set{s}\cross M$, $\restr{J}{\xi}=\restr{J^c}{\xi}$, where $\xi$ is the contact plane field.
\item\label{moduli_spaces_diffeo} For any two orbit sets $\Gamma_+$ and $\Gamma_-$, the moduli spaces $\mathcal{M}_J(\Gamma_+,\Gamma_-)$ and $\mathcal{M}_{J^c}(\Gamma_+,\Gamma_-)$ are diffeomorphic.
\item\label{canon_rescaling_item} If $L'=e^{c}L$ then we have a canonical rescaling isomorphism 
\[ s^c : ECC^L(M,\alpha,J) \xrightarrow{\iso} ECC^{L'}(M,e^c \alpha,J^c). \]
\end{enumerate}
\end{lemma}

\begin{proof}
Define $J^c$ on $\R\cross M$ by 
\begin{align*}
\restr{J^c}{\xi}&:=\restr{J}{\xi},\quad\text{and}\\
J^c(\del_s)&:= R_{e^c \alpha} = e^{-c}R_\alpha = e^{-c}J(\del_s).
\end{align*}
Hence $J^c$ is adapted to $e^c\alpha$.  There is a canonical bijection between the complexes
\[ ECC^L(M,\alpha,J)\quad\text{and}\quad ECC^{L'}(M,e^c\alpha,J^c),\]
since the orbits are identical up to rescaling.  Hence point \cref{canon_rescaling_item} and the regularity of $J^c$ follow once we establish point \cref{moduli_spaces_diffeo}. 

Consider the self-diffeomorphism of $\R\cross M$ given by
\[ (s,x) \mapsto (e^c s,x). \]
This sends $J$-holomorphic curves to $J^c$-holomorphic curves, and hence the moduli spaces $\mathcal{M}_J(\Gamma_+,\Gamma_-)$ and $\mathcal{M}_{J^c}(\Gamma_+,\Gamma_-)$ are diffeomorphic. 
\end{proof}

Later we will see that the chain map induced by a product cobordism $[c_-,c_+]\cross M$ is just the identity map, but to see this we need to understand more Seiberg-Witten theory.

\subsection{Seiberg-Witten theory}

The introduction to Seiberg-Witten theory which follows will just be enough to understand the parallels with embedded contact homology and not the precise details of the objects involved.  For a full introduction refer to the work of Kronheimer and Mrowka~\cite{KM_monopoles07}.  The setup below is taken from Hutchings and Taubes~\cite[Sections 2 and 4]{HT_Arnold13}.

Seiberg-Witten Floer cohomology is an invariant assigned to a closed oriented connected 3-manifold $M$ equipped with a metric $g$.  In our case the metric $g$ arises as follows: given a contact structure $\alpha$ on $M$ choose an adapted almost complex structure $J$ as in the definition of ECC.  This determines a metric on $M$ by requiring that the Reeb vector field $R_\alpha$ has length 1, $R_\alpha$ is orthogonal to the contact field $\xi$, and $\restr{g}{\xi}$ is compatible with $\restr{J}{\xi}$ and $\restr{d\alpha}{\xi}$ in the sense that
\begin{equation}\label{g_alpha_J_equation_SW_theory}
 g(v,w) = \frac{1}{2}d\alpha(v,Jw)
\end{equation}
for all $v, w\in \xi$.  The reason for the factor of 1/2 here is for consistency with earlier papers of Taubes (c.f.~\cite[Remark 2.2]{HT_Arnold13}).

The Seiberg-Witten Floer cohomology chain complex, which we will denote in analogy to ECC by
\[ \widehat{CM}^*(M, \alpha, J) = \bigoplus_{\mathfrak{s}\in \mathrm{Spin}^c(M)} \widehat{CM}^*(M, \alpha, J ; \mathfrak{s}),\]
is generated over $\F_2$ by solutions to a perturbed version of the \emph{Seiberg-Witten equations} on $(M,g)$.  The splitting over Spin$^c$ structures in the equation above is anologous to that in Heegaard Floer homology and the splitting with respect to $H_1(M;\Z)$ of embedded contact homology.

The setup of the perturbed Seiberg-Witten equations are not important to us, except we note that they involve choosing some data, namely a particular exact 2-form $\mu$, a large perturbation parameter $r$ and a perturbation $\eta$ which is suitably generic to ensure transversality.  The resulting chain complex depends on these choices (although we omit them from the notation) but the homology is canonically independent of these choices and also of the choice of $g$ (and by extension, $\alpha$ and $J$).  

Solutions to the equations are pairs 
\[\mathfrak{c} = (\mathbb{A}, \Psi), \]
where $\mathbb{A}$ is a connection and $\Psi$ is a section on the \emph{Spinor bundle}, a rank 2 Hermitian vector bundle over $M$ making up part of the data for the Spin$^c$ structure $\mathfrak{s}$.  A priori, solutions fall into one of two categories: irreducible or reducible. However, when it comes to the correspondance between embedded contact homology and Seiberg-Witten Floer cohomology, it is only the irreducible solutions which interest us. This is because, given certain data $(\alpha, J, \mu, r, \eta)$, there is a correspondance between orbit sets generating ECC and irreducible solutions to the Seiberg-Witten equations.  (This is outlined by Taubes~\cite[Theorem 4.2]{Taubes10}.)

The notion of the action on embedded contact homology has an analogous concept in Seiberg-Witten theory, which is a functional defined on solutions to the Seiberg-Witten equations called the \emph{energy}.  For the right choice of data the energy $E(\mathbb{A})$ is approximately equal to $2\pi\mathcal{A}(\Gamma)$ and this motivates the definition of the filtered complex
\[ \widehat{CM}^*_L(M, \alpha, J), \]
which is generated by those solutions with energy less than $2\pi L$ and is a valid subcomplex for $r$ sufficiently large and carefully chosen generic $\mu$ and $\eta$.

It is more intuitive to think of the Seiberg-Witten complex in a Morse-theoretic manner; indeed there is a functional $\mathfrak{a}_\eta$ defined on the space of such pairs $(\mathbb{A}, \Psi)$ such that solutions to the Seiberg-Witten equations are precisely the critical points of $\mathfrak{a}_\eta$.  In this setting a differential between two irreducible solutions $\mathfrak{c}_\pm$ is simply a downward flowing gradient line, called an index 1 \emph{instanton} and denoted by $\mathfrak{d}$. The resulting filtered cohomology depends only on $\alpha$ and not $J$ or any of the other choices of data made when defining the complex.

Any cobordism $(X,\lambda)$ between two 3-manifolds $M_\pm$ induces a chain map on the Seiberg-Witten Floer complex.  Suppose that $J$ is a cobordism-admissible almost complex structure on the completion
\[ \bar{X} := (-\infty,0]\cross M_- \union X \union [0,\infty)\cross M_+.  \]
We then choose a large perturbation parameter $r$ and data $(g, \mu, \eta)$ on $\bar X$ which is compatible with the data $(g_\pm, \mu_\pm, \eta_\pm)$ on either end.  In particular the metric $g$ is defined in terms of $\lambda$ and $J$.  The chain map
\[ \widehat{CM}(X,\lambda) : \widehat{CM}^*(M_+, \alpha_+, J_+) \to \widehat{CM}^*(M_-, \alpha_-, J_-) \]
is defined, at least between the irreducible solutions which concern us, by counting index 0 instantons $\mathfrak{d}$ which limit to $\mathfrak{c_+}$ at the positive end and $\mathfrak{c_-}$ at the negative end.  Furthermore for $r$ large enough and carefully chosen generic $\mu$ and $\eta$, the chain map restricts to the filtered complexes:
\[ \widehat{CM}_L(X,\lambda) : \widehat{CM}^*_L(M_+, \alpha_+, J_+) \to \widehat{CM}^*_L(M_-, \alpha_-, J_-). \]
Both chain maps depend on the choices made (although again, they are omitted from the notation) but on homology the unfiltered map depends only on $X$ and the filtered map depends only on $X$ and $\lambda$.

The final aspect of Seiberg-Witten theory which we will need later concerns homotopies $(X,\lambda_t,J_t, g_t, \mu_t, \eta_t)$ through cobordisms.  Such a homotopy induces a chain homotopy between the two chain maps obtained at either end, which we denote
\[ H_L(X,\lambda_t) : \widehat{CM}^*_L(M_+, \alpha_+, J_+) \to \widehat{CM}^*_L(M_-, \alpha_-, J_-) \]
which is defined by counting index $-1$ instantons which appear during the homotopy.  (There is an analogous construction in the unfiltered case which we will not need.)  The fact that both chain maps and chain homotopies are supported on instantons will be used extensively in the proof of \cref{supported_on_holo_curves} and in \cref{cobordism_maps_on_ECK}.

\subsection{Constructing the chain map on ECC}

We are now in a position where we can work towards defining the (non-canonical) chain map  
\[ \hat{\Phi}^L(X,\lambda) : ECC^L(M, \alpha_+, J_+) \to ECC^L(M,\alpha_-, J_-) \]
which induces the cobordism maps on ECH first seen in \cref{direct_limits_through_cobordism_maps}.  Notice that here we have set $M=M_+=M_-$; this will be the situation for the remainder of the thesis.

\begin{prop}{\nogapcite[Lemma 3.6]{HT_Arnold13}}\label{L-flat_approx}
Let $L>0$ and $\epsilon>0$.  To every pair $(\alpha, J)$ on a closed manifold $M$ with $J$ regular, there exists an \emph{$L$-flat approximation} $(\alpha', J')$ with $J'$ regular and a homotopy $(\alpha_t, J_t)$ between $(\alpha, J)$ and $(\alpha', J')$ with the following properties:
\begin{enumerate}
\item $(\alpha', J')$ is \emph{$L$-flat} near every Reeb orbit of length less than $L$, where the precise definition of L-flatness is given by Taubes~\cite[(4-1)]{Taubes10}.
\item Every $\alpha_t$ is L-nongenerate and every $J_t$ is a regular adapted almost complex structure for $\alpha_t$.
\item The homotopy leaves the Reeb orbits of $\alpha$ with length less than $L$, and their lengths, unchanged.
\item Each $(\alpha_t, J_t)$ only differs from $(\alpha, J)$ in an $\epsilon$-neighbourhood of the Reeb orbits of $\alpha$ of length less than $L$. 
\item\label{isomorphic_chain_complexes} The chain complexes $ECC^L(M,\alpha_t,J_t)$ are isomorphic for all $t\in[0,1]$, under the canonical identification of orbits.
\end{enumerate}
\end{prop}

Fix $\epsilon>0$ and find $L$-flat approximations $(\alpha_\pm', J_\pm')$.  $L$-flat approximations are useful since they allow us to step into the world of Seiberg-Witten theory.  

\begin{prop}{\nogapcite[Proposition 3.1]{HT_Arnold13}}
Suppose that the auxiliary data $(\mu,r,\eta)$ is chosen so that $\widehat{CM}_L(M,\alpha,J)$ is defined.  Then there is a canonical chain map
\begin{equation}
\Theta_{M,\alpha,J} : \widehat{CM}_L(M,\alpha,J) \to ECC^L(M,\alpha,J).
\end{equation}
Furthermore, if $(\alpha,J)$ is $L$-flat, then the map an isomorphism of chain complexes.
\end{prop}

If $(X,\lambda)$ is an exact symplectic cobordism between $(\alpha_+,J_+)$ and $(\alpha_-,J_-)$ then it is possible to find a cobordism $(X, \lambda')$ between $L$-flat approximations $(\alpha_+',J_+')$ and $(\alpha_-',J_-')$ which is in a sense ``close'' to $(X,\lambda)$~\cite[Section 6.3]{HT_Arnold13}.  This cobordism then gives rise to a chain map
\begin{equation}\label{CM_cobordism_map}
\widehat{CM}_L(X, \lambda') : \widehat{CM}_L(M, \alpha_+', J_+') \to \widehat{CM}_L(M, \alpha_-', J_-').
\end{equation}

We can now define the chain map $\hat{\Phi}^L(X, \lambda)$:
\begin{defn}\label{cobordism_chain_map_def}
Define the chain map $\hat{\Phi}^L(X, \lambda)$ by composing $\widehat{CM}_L(X, \lambda')$ with the isomorphisms $\Theta_{M,\alpha'_\pm, J'_\pm}$ and the isomorphisms from \refListInThm{L-flat_approx}{isomorphic_chain_complexes} on either side.

More explicitly, the chain map is given by the following composition:
\[\begin{tikzcd}
ECC^L(M,\alpha_+,J_+) \arrow[r,"\iso"] &
ECC^L(M,\alpha_+',J_+') \arrow[rr,"\Theta_{M,\alpha_+',J_+'}"', "\iso"] && \widehat{CM}_L(M,\alpha_+', J_+') \arrow[d,"{\widehat{CM}_L(X,\lambda')}"]\\
ECC^L(M,\alpha_-,J_-) & ECC^L(M,\alpha_-',J_-') \arrow[l,"\iso"']&& \widehat{CM}_L(M,\alpha_-', J_-') \arrow[ll,"\Theta_{M,\alpha_-',J_-'}", "\iso"'].
\end{tikzcd}\]
\end{defn}
Note that the above definition is a priori dependent on the choice of approximation $(X,\lambda')$, the choice of cobordism-admissible almost complex structure $J$ on the completion $(\bar{X},\bar{\lambda'})$, and the choices of $(\mu,r,\eta)$ made when defining the chain map on $\widehat{CM}$. However the map on the level of homology is independent of these choices~\cite[Theorem 1.9]{HT_Arnold13}.   For our purposes the dependence of the chain map on these choices is not a problem---we will take these maps forwards, remembering that they are not canonical but omitting the choices from the notation for brevity.

\begin{lemma}\label{supported_on_holo_curves}
Let $L>0$ and $(X=[0,1]\cross M, \lambda)$ be an exact symplectic cobordism from $(M, \alpha_+, J_+)$ to $(M,\alpha_-, J_-)$ with $J_\pm$ regular.  Fix a cobordism-admissible almost complex structure $J$ on the completion $(\bar{X},\bar{\lambda})$. Then the induced chain map $\hat{\Phi}^L(X,\lambda)$ has the following properties:
\begin{enumerate}
\item\label{chain_map_holo_curves} If $\langle \hat{\Phi}^L(X,\lambda)(\Gamma_+), \Gamma_- \rangle \neq 0$ then there is an index 0 (possibly broken) $J$-holomorphic curve in $(\bar{X},J)$ between $\Gamma_+$ and $\Gamma_-$.  Furthermore, if the only index 0 holomorphic curves between $\Gamma_+$ and $\Gamma_-$ are contained in a product region, then
\[ \langle \hat{\Phi}^L(X,\lambda)(\Gamma_+), \Gamma_- \rangle = 1. \]
\item\label{homotopies_holo_curves} If $(X, \lambda_t)$ is a homotopy through exact symplectic cobordisms and $J_t$ is a collection of corresponding almost complex structures on the completions $(\bar{X}, \bar{\lambda_t})$, then the maps corresponding to $t=0$ and $t=1$ are chain homotopic via some homotopy $H^L(X,\lambda_t)$ such that if $\langle H^L(X,\lambda_t)(\Gamma_+), \Gamma_- \rangle \neq 0 $ then for some $t\in[0,1]$ there exists a (possibly broken) $J_t$-holomorphic curve between $\Gamma_+$ and $\Gamma_-$.
\end{enumerate}
\end{lemma}

Broken holomorphic curves, as mentioned in the above lemma, are a slight generalization of holomorphic curves where the curve is allowed to ``break'' at intermediate orbit sets. More precisely, a \emph{broken holomorphic curve} between orbit sets $\Gamma_+$ and $\Gamma_-$ is a finite collection $u_1,\dots,u_k$ where
\begin{enumerate}
\item there exists some unique $k'$ such that $u_{k'}$ is a $J$-holomorphic curve in $\bar{X}$;
\item $u_1,\dots,u_{k'-1}$ are $J_+$-holomorphic curves in $\R\cross M$ and $u_{k'+1},\dots,u_k$ are $J_-$-holomorphic curves in $\R\cross M$;
\item the positive end of $u_1$ is equal to $\Gamma_+$, the negative end of $u_k$ is equal to $\Gamma_-$, and for each $i=1,\dots,k-1$ the negative end of $u_i$ coincides with the positive end of $u_{i+1}$ at some orbit set $\Gamma_i$; and
\item the curves $u_1,\dots,u_{k'-1},u_{k'+1},\dots,u_k$ are \emph{not} $\R$-invariant.
\end{enumerate}
Slope calculus, positivity of intersection and the Blocking and Trapping Lemmas still apply in the case of broken holomorphic curves, so this added technicality has no significance as far as we are concerned.

We will prove \cref{supported_on_holo_curves} shortly but first will state the following proposition which is an immediate consequence of part \cref{chain_map_holo_curves} of the lemma.

\begin{prop}\label{prod_cob_is_id}
If $(X,\lambda)$ is a product cobordism from $e^{c_+}\alpha$ to $e^{c_-}\alpha$ then the induced map
\[ \hat{\Phi}{(X,\lambda)} : ECC^L(M,e^{c_+}\alpha,J^{c_+}) \to ECC^{L}(M,e^{c_-}\alpha,J^{c_-}) \]
maps every orbit set to itself. 
\end{prop}

Before we proceed with the proof of \cref{supported_on_holo_curves}, we make the following observation concerning Proposition 5.2 from Hutchings and Taubes' paper~\cite{HT_Arnold13}.  In the statement of this proof part (a) has two parts, (i) and (ii).  Part (a) is concerned with index 0 solutions to the Seiberg-Witten equations and states that (i) if the energy at the positive end of a solution is less than $2\pi L$ then the energy at the negative end is also less than $2\pi L$ and (ii) there exists a (possibly broken) holomorphic curve between the orbit sets in ECC corresponding to the positive and negative solutions under the map $\Theta_{M,\alpha,J}$.  Part (b) is concerned with index 1 solutions and only states a result concerning the energy bound.  Our observation is that the lemma below concerning index $-1$ (possibly broken) holomorphic curves follows trivially from the proof of Proposition 5.2(b):

\begin{lemma}\label{index_minus_1_holo_curves}
If there is an index $-1$ solution to the perturbed Seiberg-Witten equations as defined by Hutchings and Taubes~\cite[Section 4.2]{HT_Arnold13} at some $t$, with ends corresponding to solutions $\mathfrak{c}_\pm$, then there is a (possibly broken) $J_t$-holomorphic curve from $\Gamma_+$ to $\Gamma_-$, where $\Gamma_\pm = \Theta_{M,\alpha_\pm, J_\pm}(\mathfrak{c}_\pm)$.
\end{lemma}

This observation plays an important role in the following proof.

\begin{proof}[Proof of \cref{supported_on_holo_curves}]
Part \cref{chain_map_holo_curves} follows immediately from Hutchings and Taubes~\cite[Theorem 1.9]{HT_Arnold13} and Cristofaro-Gardiner~\cite[Theorem 5.1]{Cr13}.

To prove part \cref{homotopies_holo_curves}, let $(X=[0,1]\cross M, \lambda_t)$ be a strong homotopy of exact symplectic cobordisms with $J_t$ the corresponding cobordism-admissible almost complex structures on the completions $(\bar{X},\bar{\lambda_t})$. Let $s$ denote the $[0,1]$ coordinate of $X$.  Then there exists $\epsilon>0$ such that $\lambda_t$ is a symplectization on $\epsilon$-neighbourhoods of $s=0,1$ and hence the homotopy is supported on the region $s\in[\epsilon, 1-\epsilon]$.

The process performed by Hutchings and Taubes~\cite[Section 6.3]{HT_Arnold13} to obtain $(X, \lambda')$ from $(X, \lambda)$ involves choosing an $n>0$ and finding $L$-flat approximations to $(\alpha_+, J_+)$ and $(\alpha_-, J_-)$ which only differ in $\frac{1}{n}$-neighbourhoods of those Reeb orbits of action less than $L$.  The cobordism $(X, \lambda')$ (and corresponding $J'$) is then constructed by altering $(X, \lambda)$ (and $J$) on $\epsilon$-neighbourhoods of $s=0,1$, leaving the cobordism unchanged in the region $s\in[\epsilon, 1-\epsilon]$.

It is therefore possible to apply this same alteration to every $\lambda_t$ and $J_t$ to obtain a homotopy $(X, \lambda_t')$ of cobordisms (with corresponding $J'_t$) between the $L$-flat pairs $(\alpha_+', J_+')$ and $(\alpha_-', J_-')$.  Since this homotopy between $L$-flat pairs depends on $n$ we will denote it by $(X, \lambda_t^{(n)})$ and the almost complex structures by $J_t^{(n)}$.

Let $H_n$ denote the chain homotopy on $\widehat{CM}_L$ induced by the homotopy $(X, \lambda_t^{(n)})$.  Then \cref{index_minus_1_holo_curves} above implies that if $\langle H_n(\mathfrak{c}_+), \mathfrak{c}_-\rangle \neq 0$ then for some $t\in[0,1]$ there exists a (possibly broken) $J_t^{(n)}$-holomorphic curve between the orbit sets $\Gamma_+$ and $\Gamma_-$ where $\Gamma_\pm=\Theta_{M,\alpha_\pm, J_\pm}(\mathfrak{c}_\pm)$.

We will use a Gromov compactness argument similar to that employed by Hutchings and Taubes~\cite[Section 6.3]{HT_Arnold13} to complete the proof.  The argument proceeds by taking $n$ large enough and hence $J_t^{(n)}$ ``close'' enough to $J_t$ so that the existence of a (possibly broken) $J_t^{(n)}$-holomorphic curve between two orbit sets is sufficient for the existence of a (possibly broken) $J_t$-holomorphic curve as well.  More precisely, consider pairs $(\Gamma_+, \Gamma_-)$ where $\Gamma_\pm$ is an orbit set for $\alpha_\pm$ of action less than $L$.  Let $Z$ denote the set of pairs between which there exists a (possibly broken) $J_t$-holomorphic curve for some $t$, and likewise let $Z_n$ denote those pairs between which there exists a (possibly broken) $J_t^{(n)}$-holomorphic curve.

We will argue that for $n$ large enough $Z_n\subset Z$.  To see this, suppose that $(\Gamma_+, \Gamma_-)$ is a pair such that there exists a sequence $u_n$ of (possibly broken) $J_{t_n}^{(n)}$-holomorphic curves for some sequence $t_n$.  Pass to a subsequence so that $t_n$ converge to some fixed $t\in[0,1]$. Then $\set{J^{(n)}_{t_n}}$ converges to $J_t$ and we can apply Hutchings and Taubes~\cite[Lemma 6.8(a)]{HT_Arnold13} to see that the curves $u_n$ converge to a (possibly broken) $J_t$-holomorphic curve $u$ from $\Gamma_+$ to $\Gamma_-$.

To complete the proof, choose $n$ large enough so that $Z_n \subset Z$ and define the chain homotopy $H^L(X,\lambda_t)$ by composing $H_n$ with the isomorphisms $\Theta_{M,\alpha_\pm, J_\pm}$ on either side.
\end{proof}

\section{Cobordism maps on ECK}\label{cobordism_maps_on_ECK}

Before we are able to prove $J$-invariance of ECK as mentioned in \cref{embedded_contact_knot_homology}, we must see how exact symplectic cobordisms induce chain maps on ECK.  For the purposes of this thesis we will only consider cobordisms of the form $[c_-,c_+]\cross N$, but the results could easily be generalized to more general cobordisms.

\begin{defn}
If $N$ is a manifold with torus boundary, we say that $([c_-,c_+]\cross N, \lambda)$ is an \emph{exact symplectic cobordism} between contact forms $\alpha_+$ and $\alpha_-$ on $N$ if $d\lambda$ is a symplectic form on $[c_-,c_+]\cross N$ and $\restr{\lambda}{\set{c_\pm}\cross N} = \alpha_\pm$.
\end{defn}

Going forward, assume that $N$ is a mapping torus with a $p/q$-twist at the boundary and choose a corresponding transformation matrix $A$ as in \cref{contact_forms_and_rational_open_book_decompositions}, with which we identify $\del N$ and $T^2\cross\set{2}\subset T^2\cross[1,2]$.  

Let $([0,1]\cross N, \lambda)$ be an exact symplectic cobordism between forms $\alpha_+$ and $\alpha_-$ which are both extendable to a $p/q$-rational open book $M$.  Then under the pullback by $A$ they agree to infinite order at $y_1=2$ with
\[  f_\pm(y_1) \dd t_1 + g_\pm(y_1) \dd \theta_1, \]
where
\begin{align*}
f_\pm(y_1) &= a_\pm + c_\pm(y_1-2)^2,\quad\text{and} \\
g_\pm(y_1) &= b_\pm + c_\pm(y_1-2)
\end{align*}
for some $a_\pm,c_\pm>0$.  Finally, assume that $a_+>a_-$ and that, on $[0,1]\cross \del N\cross[2,2+\epsilon)$, $\lambda$ is given by a linear interpolation
\begin{equation} \label{exact_sympl_cob_near_del_N}
s\alpha_+ + (1-s)\alpha_-, 
\end{equation}
where here $s$ denotes the $[0,1]$ coordinate. 

In order to compose exact symplectic cobordisms later in the section, we will need the following technical definition.
\begin{defn}
Suppose that $\set{\lambda^{(i)}}$ are a collection of contact forms on $[0,1]\cross N$ of the above form.  We say that $\lambda^{(i)}$ are \emph{uniformly bounded at $\del N$ by the ray of slope $\mu$} if the ray in $\R_{>0}\cross\R$ from the origin of slope $\mu$ lies above the points $(a_\pm^{(i)}, b_\pm^{(i)})$ for every $i$. 
\end{defn}
Notice that since the values $a^{(i)}_\pm$ are always greater than 0, if the collection $\set{\lambda^{(i)}}$ is finite then there exists some ray which uniformly bounds all the $\lambda^{(i)}$.  If the collection is infinite then we require that the fractions $b^{(i)}_\pm/a^{(i)}_\pm$ are uniformly bounded above.

\begin{lemma} \label{extend_N_to_M}
Let $\lambda$ be a one-form on $[0,1]\cross N$ between $\alpha_+$ and $\alpha_-$ as above which is bounded at $\del N$ by the slope $\mu$.  Given any $\delta>0$, we can extend $\lambda$ to a one-form $\hat{\lambda}$ on $X=[0,1]\cross M$, where $M=V\union T^2\cross[1,2]\union N$, so that
\begin{itemize}
\item $([0,1]\cross M, \hat{\lambda})$ is an exact symplectic cobordism,
\item on each slice $\set{s}\cross T^2\cross[1,2]$, the form looks like those constructed in \cref{contact_form_on_no_mans_land_def} with the slopes of the Reeb vector fields no shallower than $\frac{1}{\delta}$, and
\item on a neighbourhood of $[0,1]\cross V$, the form looks like a product region, with the form on each slice $\set{s}\cross V$ equal to a constant multiple of $\alpha_{V,\mu}'$ as constructed in \cref{contact_form_on_V_defn}.
\end{itemize}
\end{lemma}

\begin{proof}
The idea of the proof is to extend the definition of $\lambda$ first to the region $[0,1]\cross T^2\cross[1,2]$ and then in turn to $[0,1]\cross V$.  The extension to $[0,1]\cross T^2\cross[1,2]$ will be of the form
\[ \hat\lambda = f_s(y_1) \dd t_1 + g_s(y_1) \dd \theta_1 \]
such that
\begin{itemize}
\item near $y_1=2$ the form is compatible with $\lambda$ under the identification with $[0,1]\cross N$;
\item near $y_1=1$ the form defines a product region with $(f_s(1),g_s(1))$ lying on the ray of slope $\mu$ from the origin; and
\item the derivative $d\hat\lambda$ is non-degenerate everywhere.
\end{itemize}
Furthermore, we will take the functions $f_s$ and $g_s$ to be linear interpolations:
\begin{align*}
f_s &=  sf_+ + (1-s)f_-\quad\text{and} \\
g_s &=  sg_+ + (1-s)g_-,
\end{align*}
where $f_\pm$ and $g_\pm$ are to be determined.
We can therefore write
\[\begin{split}
d\hat\lambda = & \big(f_+(y)-f_-(y)\big)\dd s \wedge dt + \big(g_+(y)-g_-(y)\big)\dd s\wedge d\theta \\
         & + \big(sf_+'(y) + (1-s)f_-'(y)\big) \dd y\wedge dt + \big(sg_+'(y) + (1-s)g'_-(y)\big) \dd y\wedge d\theta,
\end{split}\]
where we are dropping the subscripts from the notation for simplicity, and using coordinates
\[ (s,t,\theta,y) \]
for $[0,1]\cross T^2\cross[1,2]$.

The form $d\hat\lambda$ is a symplectic form if it is non-degenerate:
\begin{equation}\label{d_alpha_wedged_with_itself_eqn}\begin{split}
 0 &< d\hat\lambda\wedge d\hat\lambda \\
   &= 2\big(f_-(y)-f_+(y)\big)\big(sg_+'(y) + (1-s)g_-'(y)\big) \dd s\wedge dt\wedge d\theta \wedge dy \\
	 & - 2\big(g_+(y)-g_-(y)\big)\big(sf'_+(y) + (1-s)f_-'(y)\big) \dd s\wedge dt\wedge d\theta\wedge dy.
\end{split}\end{equation}
We will choose $f_\pm$ such that $f_-<f_+$ and $g_\pm$ so that $g_\pm'<0$.  Then the first summand of $d\hat\lambda\wedge d\hat\lambda$ above is positive.  We must ensure that in addition $f_\pm$ and $g_\pm$ are chosen such that the second summand is small enough to guarantee $d\hat\lambda\wedge d\hat\lambda>0$.

The behaviour of the form $\hat\lambda$ near $y=2$ is determined by the fact that it must be compatible under the gluing map $A$ with the form on $[0,1]\cross N$ near $[0,1]\cross\del N$.  The functions $f_\pm$ and $g_\pm$ must therefore be chosen of the form
\begin{equation}\label{forms_near_y_equals_2}\begin{split}
f_\pm(y) &= a_\pm+c_\pm(2-y)^2 \\
g_\pm(y) &= b_\pm+c_\pm(2-y)
\end{split}\end{equation}
near $y=2$, where the constants $a_\pm$, $b_\pm$ and $c_\pm$ are taken from the discussion above.

We also want the positive end of $\hat\lambda$ to be a constant multiple of the negative end near $y=1$ and so that the points $(f_\pm(1), g_\pm(1))$ lie on the ray of slope $\mu$ from the origin.  To ensure compatibility when gluing with the product region on $[0,1]\cross V$, the functions $f_\pm$ and $g_\pm$ must take the form
\[ \big(f_\pm(y),g_\pm(y)\big) = c'_\pm\cdot\big(  1-(1-y)^2 , \mu+(1-y) \big)  \]
near $y=1$ for some $0<c'_-<c'_+$.

Notice that even with the above constraints $f_\pm:[1,2]\to\R$ can be chosen so that $f_\pm(y)$ remains arbitrarily close to $f_\pm(2)$, and hence we can assume that when we choose the functions $f_\pm$ they will have the uniform lower bound
\[ |f_+(y) - f_-(y)| > C_1, \]
for some $C_1>0$.  In addition the functions $g_\pm:[1,2]\to\R$ can be chosen with some uniform upper bound
\[ |g_+(y) - g_-(y)| < C_2, \]
for some $C_2>0$.  Finally, we can assume that when $g_\pm$ are chosen we have $|g_\pm'(y)|>\delta_1$, for some $\delta_1>0$.

Now choose the functions $f_\pm$, $g_\pm:[1,2]\to\R$ so that
\begin{itemize}
\item $|f_\pm'(y)| < \frac{C_1\delta_1}{2C_2}$,
\item the slopes $\frac{g'_\pm(y)}{f'_\pm(y)}$ are never shallower than $\frac{1}{\delta}$,
\item $f_\pm$ and $g_\pm$ satisfy the above gluing requirements near $y=1$ and $y=2$, and
\item $f_\pm$, $g_\pm$ and $g_\pm'$ satisfy the bounds above concerning $C_1$, $C_2$ and $\delta_1$.
\end{itemize}
Then it is simple to check using \cref{d_alpha_wedged_with_itself_eqn} that $d\hat\lambda\wedge d\hat\lambda >0$ and hence we have an exact symplectic cobordism.  See \cref{one-form_lambda_on_no_mans_land}.  Finally define $\hat\lambda$ on $[0,1]\cross V$ by the equation
\[  s (c'_+ \alpha_{V,\mu}') + (1-s) (c'_- \alpha_{V,\mu}') \]
which defines a product region (recall \cref{product_region_defn}).  Furthermore, all gluing requirements are met so we obtain an exact symplectic cobordism 
\[ ([0,1]\cross M, \hat\lambda) \]
satisfying all the requirements of the lemma.
\end{proof}

\begin{figure}\centering
  \begin{tikzpicture}%
		\draw [->] (0,0) -- (0,4.8);
		\node [above left] at (0,4.8)  {$g$};
		\node [below right] at (4.8,0)  {$f$};
		\draw [->] (0,0) -- (4.8,0);
		
		\fill[lightgray] (1.8,0.5) to [out=90,in=-90] (2.2,2.2) to (3.8,3.8) to[out=-90,in=90] (3.2,0.8) to cycle;
		\draw (0,0)--(4,4);
		\node [above] at (4,4) {slope $=\mu$};
		\draw [very thick] (1.8,0.5) to [out=90,in=-90] (2.2,2.2);
		\draw [very thick] (3.2,0.8) to[out=90,in=-90] (3.8,3.8);
		\draw (1.8,0.5) to (3.2,0.8);
		
		\node [right] at (3.4,2) {$(f_+,g_+)$};
		\fill [white] (1.1,1.1) rectangle (1.5,1.5);
		\node [left] at (2,1.3) {$(f_-,g_-)$};
		\draw [fill] (3.2,0.8) circle (0.1) node [below right] {$(a_+,b_+)$};
		\draw [fill] (1.8,0.5) circle (0.1) node [below] {$(a_-,b_-)$};
		
	\end{tikzpicture}
  \caption{A diagram showing the construction of the one form $\lambda$ on $[0,1]\cross T^2\cross[1,2]$.  The left (resp.\ right) hand side of the shaded region is the path $(f_-,g_-)$ (resp.\ $(f_+,g_+)$), and the form interpolates linearly between these within the shaded region.  The image of curve $(f_s,g_s)$ moves strictly to the left as $s$ increases from 0 to 1.  Furthermore the form defines a product region near the ray of slope $\mu$ when $y$ is near 1, and obeys \cref{forms_near_y_equals_2} when $y$ is near 2.}
	\label{one-form_lambda_on_no_mans_land}
\end{figure}
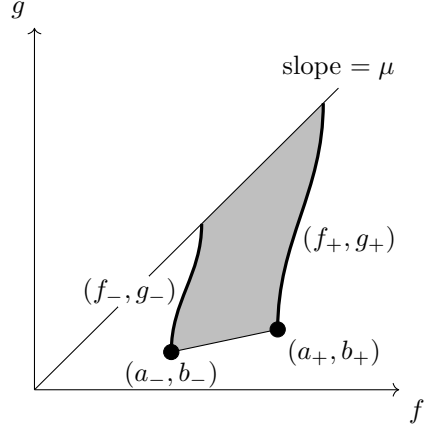

We must take a moment to discuss non-degenerate perturbations of the forms constructed above.  First note that given $L>0$ we can take $\delta>0$ small enough so that all Reeb orbits of $\hat\alpha_\pm$ in the region $T^2\cross(0,1)\subset M$ have $\hat\alpha_\pm$-action greater than $L$.  

We therefore obtain an exact symplectic cobordism $(X,\hat\lambda)$ between $\hat\alpha_+$ and $\hat\alpha_-$ where $\hat\alpha_\pm$ are both Morse-Bott contact forms with (at most) two Morse-Bott tori of action less than $L$, namely at $\del V$ and $\del N$.  Note that $X$ cannot induce a chain map yet since the construction in \cref{cobordism_chain_map_def} requires an exact symplectic cobordism between two honestly $L$-non-degenerate contact forms.  With this in mind, let $\hat\alpha_\pm'$ be $L$-non-degenerate forms on $M$, constructed by perturbing $\hat\alpha_\pm$ on small neighbourhoods of $\del V$ and $\del N$, following \cref{Morse-Bott_contact_homology}. Then update the exact symplectic cobordism $(X,\hat\lambda)$ to obtain an exact symplectic cobordism $(X,\hat\lambda')$ such that $\hat\lambda'$  defines a product region on some neighbourhood of $[0,1]\cross V$.

We will state the results of the above discussion as a theorem:

\begin{thm}\label{alpha_extension_thm}
Let $L>0$ and suppose that $([0,1]\cross N, \lambda)$ is an exact symplectic cobordism between $(N,\alpha_+, J_+)$ and $(N,\alpha_-, J_-)$, where $\alpha_\pm$ are Morse-Bott on $\del N$ and $L$-non-degenerate on $\mathrm{int}(N)$.  Assume that near $[0,1]\cross\del N$, $\lambda$ is of the form described at the start of \cref{cobordism_maps_on_ECK}.  Let $M=V\union T^2\cross[1,2] \union N$.  Then we can extend $\lambda$ to form an exact symplectic cobordism $([0,1]\cross M, \hat\lambda)$ and also obtain a perturbed version $([0,1]\cross M, \hat\lambda')$, with corresponding cobordism-admissible almost complex structures $\hat J$ and $\hat J'$ for the completions respectively, such that the following hold:
\begin{enumerate}
\item $\hat\alpha'_\pm$ are $L$-non-degenerate everywhere, and $\hat\alpha_\pm$ are $L$-non-degenerate everywhere except for two Morse-Bott tori at $\del V$ and $\del N$.
\item Neither $\hat\alpha_\pm$ nor $\hat\alpha'_\pm$ have any orbits of length less than $L$ in the no-man's land $T^2\cross(1,2)$.
\item Both $\hat\lambda$ and $\hat\lambda'$ are product regions on a neighbourhood of $[0,1]\cross V$.
\item\label{non_degen_approximation_to_MB} If there is a (possibly broken) $\hat J'$-holomorphic curve between two orbit sets $\Gamma_+$ and $\Gamma_-$ then there is a corresponding (possibly broken) $\hat J$-holomorphic Morse-Bott building between $\Gamma_+$ and $\Gamma_-$. 
\end{enumerate}
\end{thm}

We will say that such a cobordism $([0,1]\cross N,\lambda)$ is \emph{extendable to a cobordism between ($p/q$)-rational open books} $X=[0,1]\cross M$ (or simply \emph{extendable} for short) and will call $([0,1]\cross M, \hat\lambda)$ the \emph{$L$-Morse-Bott extension to $X$} and $([0,1]\cross M, \hat\lambda')$ the \emph{$L$-non-degenerate extension to $X$}.  All points in the theorem follow from \cref{extend_N_to_M} and the discussion above except the last---point \cref{non_degen_approximation_to_MB} follows from applying a version of \cref{MB_building_exists_epsilon_small_enough} in the cobordism setting (c.f.~\cite[Section 9.2]{CGH_ECH_OBD}).

For the next proposition, recall that $\eta$ is the function used to define the filtration on ECC in \cref{computing_ech_via_a_filtration_argument}, and counts how many times an orbit set winds around inside $V$ in the $\theta_0$-direction.

\begin{prop}\label{cobordism_respects_eta_filtration}
Let $L>0$ and suppose that $(X,\hat\lambda')$ is an $L$-non-degenerate extension of the form constructed in \cref{alpha_extension_thm}.  Let $\Gamma_+\in ECC^L(M,\hat\alpha'_+,\hat J'_+)$ be an orbit set with $\eta(\Gamma_+)=0$, i.e.~$\Gamma_+=e_+^i h_+^j \Gamma_+'$ where $\Gamma_+'$ is constructed from orbits in $N$.  Then $\eta(\hat{\Phi}^L(X,\hat \lambda')(\Gamma_+))=0$.  Furthermore, if $(X,\hat \lambda'_t)$ is a homotopy of such cobordisms then the same result holds for the induced chain homotopy $H^L(X,\hat \lambda'_t)$.
\end{prop}

\begin{proof}
The proof is the same for the chain map and the homotopy so we will only prove the proposition for the chain map.

Suppose that $\langle \hat{\Phi}^L(X,\hat \lambda')(\Gamma_+), \Gamma_- \rangle \neq 0$ and that $\eta(\Gamma_-)>0$.  Then by \cref{supported_on_holo_curves} there exists a (possibly broken) $\hat J'$-holomorphic curve in the completion $\bar{X}$ with positive end at $\Gamma_+$ and negative end at $\Gamma_-$.  Hence by \refListInThm{alpha_extension_thm}{non_degen_approximation_to_MB} there exists a (possibly broken) $\hat J$-Morse-Bott building $u$ between these two orbit sets.  Write $\Gamma_-=\gamma\tensor\Gamma_-'$ where $\gamma$ is constructed from orbits in $V$ and $\Gamma_-'$ is constructed from orbits in $N$.  Denote by $\lambda$ the longitude of $K$ in the $\theta_0$-direction and $d$ the meridian of $K$ oriented in the $t_0$-direction so that $d\cdot \lambda=1$.  We will apply slope calculus (\cref{slope_calculus}) to $u$ at the torus $T_{\rho_i}$, which recall from \cref{contact_forms_on_3_pieces} is the torus in $V$ at radius ${\rho_i}$ and is foliated by Reeb trajectories of $\alpha_V'$ with irrational slope $s=d+\epsilon_i\lambda$.  Take ${\rho_i}$ sufficiently close to 1 so that no orbits of $\gamma$ lie between $T_{\rho_i}$ and $\del V$. Then by slope calculus
\[ [u_{T_{\rho_i}}] = [\gamma] = \eta(\Gamma_-)\lambda + kd \]
for some $k\in\Z$, and
\[ [\gamma]\cdot s = (\eta(\Gamma_-)\lambda+kd)\cdot (d+\epsilon_i\lambda) = -\eta(\Gamma_-) + k\epsilon_i, \]
which, for $\rho_i$ close to 1 and hence $\epsilon_i$ close to 0, is strictly less than 0 and hence contradicts positivity of intersection (\cref{positivity_of_intersections_3}), which applies since $\hat{\lambda}'$ defines a product region on $[0,1]\cross V$.
\end{proof}

As a result of the above proposition both chain maps and chain homotopies of the above form restrict to maps on the lowest $\eta$-filtration level of the ECC complexes, so we obtain maps
\[ ECC\orbits{0he}{eh}(\mathrm{int}(N),\alpha_+,J_+;L) \to ECC\orbits{0he}{eh}(\mathrm{int}(N),\alpha_-,J_-;L). \]

We are now able to show that such maps respect the bi-filtrations on these complexes.

\begin{lemma}\label{cobordism_maps_behave_nicely_on_V}
Let $L>0$ and suppose that $(X,\hat\lambda')$ is an $L$-non-degenerate extension of $([0,1]\cross N,\lambda)$ of the form constructed in \cref{alpha_extension_thm}.
\begin{enumerate}
\item \label{chain_map_respects_filtrations}The map
\[ \hat{\Phi}^L(X,\hat\lambda') : ECC\orbits{0he}{eh}(\mathrm{int}(N),\alpha_+,J_+;L) \to ECC\orbits{0he}{eh}(\mathrm{int}(N),\alpha_-,J_-;L) \]
has the property that if $\langle\hat{\Phi}^L(X,\hat\lambda')(\Gamma_+), \Gamma_-\rangle \neq 0$ then
\begin{enumerate}
\item\label{alex_grading} $A(\Gamma_-)=A(\Gamma_+)$,
\item\label{power_of_plus_orbits} the powers of $e_+$ and $h_+$ in $\Gamma_-$ are equal to their respective powers in $\Gamma_+$, and
\item\label{power_of_minus_orbits} the powers of $e_-$ and $h_-$ in $\Gamma_-$ are greater than or equal to their respective powers in $\Gamma_+$.
\end{enumerate}
\item\label{chain_homotopy} Assume that $H^L(X,\hat\lambda'_t)$ is a chain homotopy induced by a homotopy of such cobordisms, and $\langle H^L(X,\hat\lambda'_t)(\Gamma_+), \Gamma_- \rangle \neq 0$.  Then
\begin{enumerate}
\item\label{homotopy_A_grading} $A(\Gamma_-)\le A(\Gamma_+)$,
\item\label{homotopy_plus_orbits} the powers of $e_+$ and $h_+$ in $\Gamma_-$ are less than or equal to their respective powers in $\Gamma_+$, and
\item\label{homotopy_minus_orbits} the powers of $e_-$ and $h_-$ in $\Gamma_-$ are greater than or equal to their respective powers in $\Gamma_+$.
\end{enumerate}
\end{enumerate}
\end{lemma}

\begin{proof}
Assume that $\langle \hat{\Phi}^L(X,\hat\lambda')(\Gamma_+), \Gamma_-\rangle \neq 0$ so there exists an index 0 (possibly broken) $\hat J'$-holomorphic curve between $\Gamma_+$ and $\Gamma_-$ by \cref{supported_on_holo_curves} and hence an index 0 (possibly broken) $\hat J$-holomorphic Morse-Bott building $u$ between $\Gamma_+$ and $\Gamma_-$ by \refListInThm{alpha_extension_thm}{non_degen_approximation_to_MB}.  Write $\Gamma_\pm = \gamma_\pm \tensor \Gamma'_\pm$, where $\gamma_\pm$ are contained in $V$ and $\Gamma'_\pm$ are contained in $N$.

Since $\hat\lambda'$ defines a product region on a neighbourhood of $\del V$, we apply the \refNamedThm{Blocking Lemma}{blocking_lemma} to see that the holomorphic part of $u$ is supported away from $\del V$.  Furthermore, since no Reeb orbits in $T^2\cross(1,2)$ have action less than $L$ it must be that $u$ is supported in $V\union N$.  But then $\restr{\hat\lambda}{[0,1]\cross V}$ is a product region so the part of $u$ contained in $V$ consists solely of trivial cylinders and hence $\gamma_+ = \gamma_-$.  This implies part \cref{power_of_plus_orbits} of the lemma.

On the other hand, the existence of the part of $u$ in $N$ means that $A(\Gamma'_+)=A(\Gamma'_-)$, and hence $A(\Gamma_+)=A(\Gamma_-)$.

The \refNamedThm{Trapping Lemma for cobordisms}{trapping_lemma_cobordisms} applied to $u$ at $\del N$ implies that, aside from trivial cylinders, only negative ends of $u$ can approach $\del N$, so part \cref{power_of_minus_orbits} holds.

The proof of part \cref{chain_homotopy} is similar, since the chain homotopy is also supported on (possibly broken) holomorphic curves by \cref{supported_on_holo_curves}.  The statement is slightly weaker since we have lost the index 0 constraint, however we can still apply the Trapping Lemma for cobordisms, this time at $\del V$ as well as $\del N$, to see that powers of orbits in $\del V$ can only decrease (hence part \cref{homotopy_plus_orbits} holds), while powers of orbits in $\del N$ can only increase (hence part \cref{homotopy_minus_orbits} holds).  Part \cref{homotopy_A_grading} holds by applying positivity of intersection (\cref{positivity_of_intersections_3}) to the knot $K=\set{y_0=0}\subset V$.
\end{proof}

The following theorem is an immediate corollary of the above lemma.

\begin{thm}\label{maps_on_ECK}
Suppose that $([0,1]\cross N,\lambda)$ is an extendable exact symplectic cobordism between $(N,\alpha_+, J_+)$ and $(N,\alpha_-, J_-)$, and that $([0,1]\cross N,\lambda_t)$ is a homotopy of such cobordisms.  Let $L>0$ and extend these cobordisms by \cref{alpha_extension_thm} to obtain cobordisms $(X=[0,1]\cross M,\hat\lambda')$ and $(X,\hat\lambda'_t)$ respectively.  Then the induced chain map $\hat{\Phi}^L(X,\hat\lambda')$ and chain homotopy $H^L(X,\hat\lambda'_t)$ yield maps:
\[ ECC\orbits{0he}{eh}(\mathrm{int}(N),\alpha_+, J_+;L) \to ECC\orbits{0he}{eh}(\mathrm{int}(N),\alpha_-, J_-;L) \]
which respect the Alexander and $e_+$-filtrations. Moreover the chain map fixes the Alexander grading and the multiplicities of both $e_+$ and $h_+$, and is supported entirely on index 0 Morse-Bott buildings in $N$.
\end{thm}

From this point forward we will denote such chain maps by $\Phi^L([0,1]\cross N,\lambda)$ and such chain homotopies by $H^L([0,1]\cross N,\lambda_t)$, dropping from the notation the dependence on the extension to $X$. We will complete this section by stating a version of \cref{prod_cob_is_id} for cobordisms of the form $[0,1]\cross N$.

\begin{prop}\label{prod_cob_for_torus_bdry_prop}
If $([0,1]\cross N, \lambda)$ is a product cobordism from $e^{c_+}\alpha$ to $e^{c_-}\alpha$ then the induced map
\[ \Phi^L([0,1]\cross N,\lambda) : ECC\orbits{0he}{eh}(\mathrm{int}(N),e^{c_+}\alpha,J^{c_+};L) \to ECC\orbits{0he}{eh}(\mathrm{int}(N),e^{c_-}\alpha,J^{c_-};L) \]
maps every orbit set to itself.
\end{prop}
\begin{proof}
This follows immediately from \refListInThm{supported_on_holo_curves}{chain_map_holo_curves} and the fact that the map is supported on index 0 holomorphic curves in $N$ by \cref{maps_on_ECK}.
\end{proof}

\subsection{Composing cobordisms}

In this section will we discuss the composition of cobordisms on ECK.  First take two extendable exact symplectic cobordisms
\[  ([0,1]\cross N,\lambda^{(1)})\quad\text{and}\quad ([0,1]\cross N,\lambda^{(2)})  \]
such that 
\[\restr{\lambda^{(1)}}{\set{0}\cross N} = \restr{\lambda^{(2)}}{\set{1}\cross N}.\]
We say that $\lambda^{(1)}$ and $\lambda^{(2)}$ are composable and denote their composition by
\[ ([0,2]\cross N, \lambda^{(1,2)}), \]
where
\[ \restr{\lambda^{(1,2)}}{\set{s}\cross N} = \begin{cases}
\restr{\lambda^{(1)}}{\set{s-1}\cross N} & \text{if $s\in[1,2]$,}\\
\restr{\lambda^{(2)}}{\set{s}\cross N} &\text{if $s\in[0,1]$.}
\end{cases}\]
One would hope that the induced chain maps satisfied the relation
\begin{equation}\label{cobordisms_compositions_equal_eqn}
\Phi^L([0,2]\cross N, \lambda^{(1,2)}) = \Phi^L([0,1]\cross N, \lambda^{(2)})\circ\Phi^L([0,1]\cross N, \lambda^{(1)}), 
\end{equation}
however this is not necessarily the case.  To see why, and to understand how these maps are in fact related, we must construct extensions to $\lambda^{(1)}$ and $\lambda^{(2)}$ such that the extensions are also composable,~i.e.
\[ \restr{\hat\lambda^{(1)}}{\set{0}\cross N} = \restr{\hat\lambda^{(2)}}{\set{1}\cross N} \]
on the whole of $M$.  The following proposition follows easily from the construction in \cref{extend_N_to_M}.

\begin{prop}\label{extensions_composable}
Suppose that we have a collection of extendable exact symplectic cobordisms, $\set{([0,1]\cross N, \lambda^{(i)})}_{i\in I}$, and furthermore that the $\lambda^{(i)}$ are uniformly bounded at $\del N$ by some ray of slope $\mu$.  Let $L>0$. Then it is possible to extend every $\lambda^{(i)}$ to an $L$-Morse-Bott extension and an $L$-non-degenerate extension with respect to the ray of slope $\mu$, such that whenever $([0,1]\cross N,\lambda^{(i)})$ and $([0,1]\cross N,\lambda^{(j)})$ are composable then their extensions are.
\end{prop}

In the remainder of this chapter, whenever we compose cobordisms on $[0,1]\cross N$ we will assume that the extensions used to define their induced chain maps are constructed according to the above proposition.  We can then define the extension of $\lambda^{(i,j)}$ to be the composition of the extensions of $\lambda^{(i)}$ and $\lambda^{(j)}$.

Note that compositions of such cobordisms do not necessarily have the linear form of \cref{exact_sympl_cob_near_del_N} near the boundary---instead the composition has a piece-wise linear appearance.  However, since all the extensions making up the composition were made with respect to the same ray of slope $\mu$, we still have a product region on a neighbourhood of $V$.  In addition, we can still apply the \refNamedThm{Trapping Lemma for cobordisms}{trapping_lemma_cobordisms}, since the ends of the completion have not changed.  Therefore the results of the previous section (namely \cref{cobordism_respects_eta_filtration}, \cref{cobordism_maps_behave_nicely_on_V} and \cref{maps_on_ECK}) can be applied to the induced chain maps and chain homotopies nonetheless.

\begin{prop}\label{composition_homotopic}
Assume that $([0,1]\cross N, \lambda^{(1)})$ and $([0,1]\cross N, \lambda^{(2)})$ are composable extendable exact symplectic cobordisms and denote their composition by
\[ ([0,2]\cross N, \lambda^{(1,2)}). \]
Then \cref{cobordisms_compositions_equal_eqn} holds up to bi-filtered chain homotopy,~i.e.
\begin{equation}\label{compostions_homotopic_eqn}
\Phi^L([0,2]\cross N, \lambda^{(1,2)}) - \Phi^L([0,1]\cross N, \lambda^{(2)})\circ\Phi^L([0,1]\cross N, \lambda^{(1)}) = \del H+H\del
\end{equation}
for some bi-filtered chain homotopy $H$.
\end{prop}
\begin{proof}
The proof follows by examining Hutchings and Taubes' proof of the analogous result at the level of cohomology---namely that the composition of maps induced by cobordisms equals the map induced by the composition~\cite[Proposition 5.4]{HT_Arnold13}.

Their proof uses a neck stretching argument to construct a chain homotopy between the two maps concerned as in \cref{compostions_homotopic_eqn}; our proof will proceed by simply verifying that the chain homotopy in question respects the bi-filtration. First write $(X^{(i)},\hat\lambda^{(i)})$ for the $L$-non-degenerate extension of $([0,1]\cross N,\lambda^{(i)})$, for $i=1,2$, and likewise $(X^{(1,2)},\hat\lambda^{(1,2)})$ for the extension of the composition.

In their proof, Hutchings and Taubes construct the following ``neck-stretched'' variation of the composition:
\[ X_R := X^{(2)}\gappy{\union} \big([-R,R]\cross M\big)\gappy{\union} X^{(1)}, \]
for $R>0$, equipped with a one-form $\lambda_R$ defined by
\[ \lambda_R := \begin{cases}
e^{-2R}\hat\lambda^{(2)}&\text{on $X^{(2)}$,}\\
e^{2s}\alpha_0&\text{on $[-R,R]\cross M$,}\\
e^{2R}\hat\lambda^{(1)}&\text{on $X^{(1)}$,}
\end{cases} \]
where here $s$ denotes the $[-R,R]$ coordinate and $\alpha_0$ is the contact form at the negative end of $X^{(1)}$ (and the positive end of $X^{(2)}$). The factor of 2 here is directly related to the factor of 1/2 in \cref{g_alpha_J_equation_SW_theory} (c.f.~\cite[Remark 4.2]{HT_Arnold13}). Then
\[ (X_0,\lambda_0) = (X^{(1,2)},\hat{\lambda}^{(1,2)}) \]
by definition and $(X_0,\lambda_0)$ is homotopic through exact symplectic cobordisms to $(X_R,\lambda_R)$ for any $R>0$.  Denote the induced chain homotopy by $H_R$---by \cref{cobordism_maps_behave_nicely_on_V} (which we can apply to this composition by the discussion above) $H_R$ respects the bi-filtration.  Thus it is sufficient to show that, for large enough $R$, the chain maps
\[ \widehat{CM}_L(X_R,\lambda_R)\quad\text{and}\quad \widehat{CM}_L(X^{(2)},\hat\lambda^{(2)})\circ \widehat{CM}_L(X^{(1)},\hat\lambda^{(1)}) \]
are equal.  But this, as stated in the proof by Hutchings and Taubes, follows from Kronheimer and Mrowka~\cite[Proposition 26.1.6]{KM_monopoles07}.
\end{proof}

\section{Invariance results in the ECK setting}\label{invariance_results_in_the_ECK_setting}

\subsection{\texorpdfstring{$J$}{J}-invariance of ECK\texorpdfstring{$^L$}{\textasciicircum L}}

Provided that we retain the action bound $L$, proving $J$-invariance of the (bi-filtered homotopy type of) embedded contact knot homology is fairly straightforward.  When we remove the action bound things become more complex and we will need to make some compromises---this will be discussed in \cref{invariance_results_without_an_action_bound}.

Proving $J$-invariance of ECK$^L$ takes a slightly different approach to the proof of $J$-invariance of ECH as given by Hutchings and Taubes~\cite[Section 3.4]{HT_Arnold13}.  The reason for this is that $J$-invariance of ECH is proved via $J$-invariance of Seiberg-Witten Floer cohomology, which in turn is proved by constructing chain homotopies from variations in the data chosen to define the chain complex $\widehat{CM}$.  However such variations in data are not necessarily supported on holomorphic curves (in fact, the notion of holomorphic curves does not even make sense since the variations do not even give rise to an almost complex manifold) so we cannot use this technique to prove that the chain homotopies respect the bi-filtrations.

Instead, we will proof $J$-invariance of ECK$^L$ by constructing exact symplectic cobordisms and apply the results of \cref{cobordism_maps_on_ECK} to construct chain homotopies which do respect the bi-filtrations.

To start, let $L>0$ and $(N, \alpha)$ be a manifold with torus boundary with $\alpha$ extendable to a rational open book, as in \cref{contact_forms_and_rational_open_book_decompositions}.  Suppose that $J_+$ and $J_-$ are two regular adapted almost complex structures for $\alpha$.

The idea of the proof of $J$-invariance of ECK$^L$ is to construct exact symplectic cobordisms which ``interpolate'' between (scaled copies of) the two almost complex structures.

\begin{lemma}\label{chain_map_between_acs}
Suppose that $(M,\alpha)$ is $L$-non-degenerate, where here $M$ is either a closed manifold or has torus boundary. Let $J_\pm$ be two almost complex structures adapted to $\alpha$ and $\epsilon>0$.  Then we can construct an exact symplectic cobordism (and a cobordism-admissible almost complex structure) between $(M,\alpha, J_+)$ and $(M,e^{-\epsilon}\alpha, J_-^{-\epsilon})$.
\end{lemma}
Here recall that $J_-^{-\epsilon}$ is the re-scaled almost complex structure from \cref{rescale_alpha_J_iso}.
\begin{proof}
Consider the one-form
\[ \lambda_s := e^{-(1-s)\epsilon}\alpha \]
on $X=[0,1]\cross M$, which interpolates between $\alpha$ (at the positive end) and $e^{-\epsilon}\alpha$ (at the negative end).  This is a product region on the whole of $M$ so is a valid exact symplectic cobordism.

The space of adapted almost complex structures for $\alpha$ is contractible, since all such $J$ are identical on $\langle \del_s,R\rangle \subset T(\R\cross M)$ and the space of almost complex structures on $\xi$ compatible with $d\alpha$ is contractible.  Hence we can find a path $J_t$ between $J_-$ and $J_+$.  Each $J_t$ is an almost complex structure on $\R\cross M$ adapted to $\alpha$ so by \cref{rescale_alpha_J_iso} we can define $J_t^{-(1-s)\epsilon}$ which, for all $t$, is an adapted almost complex structure for $e^{-(1-s)\epsilon}\alpha$.  Define an almost complex structure $J$ on $[0,1]\cross M$ by
\[ \restr{J}{\set{s}\cross M} := J_s^{-(1-s)\epsilon}. \]
Then 
\begin{align*}
\restr{J}{\set{1}\cross M}&=J_+,\quad\text{and}\\
\restr{J}{\set{0}\cross M}&=J_-^{-\epsilon},
\end{align*}
an hence we can extend the definition of $J$ to the completion $\bar{X}$ by $J_+$ on the positive end and $J_-^{-\epsilon}$ on the negative end to form the required cobordism-admissible almost complex structure.
\end{proof}

\begin{thm}\label{J_invariance_bounded_case_thm}
Let $L>0$ and suppose that $(N,\alpha)$ is a contact manifold with torus boundary where $\alpha$ is extendable to a rational open book and is $L$-non-degenerate on $\mathrm{int}(N)$.  Let $J_i$ for $i=0,1$ be two choices of regular adapted almost complex structure for $\alpha$.  Then there exists bi-filtered chain homotopy equivalences
\begin{align*}
ECC\orbits{0he}{eh}(\mathrm{int}(N),\alpha,J_0;L) &\to ECC\orbits{0he}{eh}(\mathrm{int}(N),\alpha,J_1;L) \\
ECC\orbits{0he}{eh}(\mathrm{int}(N),\alpha,J_1;L) &\to ECC\orbits{0he}{eh}(\mathrm{int}(N),\alpha,J_0;L)
\end{align*}
such that their compositions are bi-filtered chain homotopic to the identity maps on both spaces and hence the two complexes have the same bi-filtered homotopy type.
\end{thm}
\begin{proof}
Pick $\epsilon>0$ such that $\alpha$ has no orbit sets with action in the interval $[L,e^{3\epsilon}L]$.  Following the procedure from \cref{chain_map_between_acs} yields an exact symplectic cobordism between $(N,\alpha, J_0)$ and $(N,e^{-\epsilon}\alpha, J_1^{-\epsilon})$ which in turn gives rise to a bi-filtered chain map 
\[ \Phi_0: ECC\orbits{0he}{eh}(\mathrm{int}(N),\alpha, J_0;L) \to ECC\orbits{0he}{eh}(\mathrm{int}(N),e^{-\epsilon}\alpha, J_1^{-\epsilon};L) \]
by \cref{maps_on_ECK}.  By the same reasoning we also obtain chain maps induced by exact symplectic cobordisms:
\begin{align*}
\Phi_1 &: ECC\orbits{0he}{eh}(\mathrm{int}(N),e^{-\epsilon}\alpha, J_1^{-\epsilon};L) \to ECC\orbits{0he}{eh}(\mathrm{int}(N),e^{-2\epsilon}\alpha, J_0^{-2\epsilon};L),\text{ and}\\
\Phi_2 &: ECC\orbits{0he}{eh}(\mathrm{int}(N),e^{-2\epsilon}\alpha, J_0^{-2\epsilon};L) \to ECC\orbits{0he}{eh}(\mathrm{int}(N),e^{-3\epsilon}\alpha, J_1^{-3\epsilon};L).
\end{align*}
Denote the cobordism inducing $\Phi_i$ by $([0,1]\cross N,\lambda^{(i)})$ for $i=0,1,2$. The composition of the first two cobordisms ($i=0,1$) is given explicitly by
\[ ([0,2]\cross N, e^{-(2-s)\epsilon}\alpha), \]
a cobordism between $(\alpha,J_0)$ and $(e^{-2\epsilon}\alpha, J_0^{-2\epsilon})$; we claim that the almost complex structure $J$ on the completion of this cobordism is homotopic to that of a product region.  

Recall by following the proof of \cref{chain_map_between_acs} that $J$ is constructed via a path $J_t$ of almost complex structures from $J_0$ to $J_1$ and then back to $J_0$; on each slice $\set{s}\cross M$, $J$ is a rescaled copy of $J_s$. Since the space of adapted almost complex structures for $\alpha$ is contractible (c.f.~ the proof of \cref{chain_map_between_acs}), the path $J_t$ is homotopic to the constant path at $J_0$.  Hence the composition cobordism is homotopic through exact symplectic cobordisms to the product cobordism
\[ ([0,2]\cross N, e^{-(2-s)\epsilon}\alpha, J'_0 ), \]
where $J'_0$ is the cobordism-admissible almost complex structure given by
\begin{itemize}
\item $J_0$ on the positive end of the completion,
\item $J_0^{-2\epsilon}$ on the negative end of the completion, and 
\item $J_0^{-(2-s)\epsilon}$ on the cobordism region of the completion ($s\in[0,2]$).
\end{itemize}
By \cref{prod_cob_for_torus_bdry_prop}, the chain map induced by this cobordism sends every orbit set to itself, and \cref{maps_on_ECK} and \cref{composition_homotopic} together imply that it is bi-filtered chain homotopic to the composition $\Phi_1\circ\Phi_0$.  Similarly, the composition $\Phi_2\circ\Phi_1$ is also bi-filtered chain homotopic to the map sending every orbit set to itself.

Next, note that, for $0\le k\le 3$ and $i=0,1$, we have isomorphisms
\[ \Theta^{(k)}_i : ECC\orbits{0he}{eh}(\mathrm{int}(N),\alpha,J_i;L) \xrightarrow{\iso} ECC\orbits{0he}{eh}(\mathrm{int}(N),e^{-k\epsilon}\alpha, J_i^{-k\epsilon};L),\]
defined by the composition
\[\begin{tikzcd}
ECC\orbits{0he}{eh}(\mathrm{int}(N),\alpha,J_i;L)\arrow[d,"{i_{L,e^{k\epsilon}L}}"] \\
ECC\orbits{0he}{eh}(\mathrm{int}(N),\alpha,J_i;e^{k\epsilon}L) \arrow[d,"{s^{-k\epsilon}}"]\\
ECC\orbits{0he}{eh}(\mathrm{int}(N),e^{-k\epsilon}\alpha, J_i^{-k\epsilon};L)
\end{tikzcd}\]
where the inclusion maps $i_{L,e^{k\epsilon}L}$ are isomorphisms since $\alpha$ has no orbit sets with action in the interval $[L,e^{3\epsilon}L]$ and $s^{-k\epsilon}$ are the rescaling isomorphisms from \cref{rescale_alpha_J_iso}.  Furthermore, these maps, and hence $\Theta^{(k)}_i$, map each orbit set to itself.

Now define chain maps, as in the statement of the theorem, by
\begin{align*}
g_1 &:= \left(\Theta^{(1)}_1\right)^{-1}\circ\Phi_0\circ\Theta^{(0)}_0: C(\alpha,J_0) \to C(\alpha,J_1),\\
f &:= \left(\Theta^{(2)}_0\right)^{-1}\circ\Phi_1\circ\Theta^{(1)}_1:C(\alpha,J_1) \to C(\alpha,J_0),\\
g_2 &:= \left(\Theta^{(3)}_1\right)^{-1}\circ\Phi_2\circ\Theta^{(2)}_0: C(\alpha,J_0) \to C(\alpha,J_1),
\end{align*}
where here we are using the notation 
\[  C(\alpha,J_i) := ECC\orbits{0he}{eh}(\mathrm{int}(N),\alpha,J_i;L). \]
The diagram below explains the construction of these maps.
\[\begin{tikzcd}
C(\alpha,J_0) \arrow[r,"g_1"]\arrow[dr,"\Phi_0"']\arrow[loop left,"\Theta_0^{(0)}","\iso"']& C(\alpha,J_1)\arrow[r,"f"]\arrow[d,"\Theta_1^{(1)}","\iso"'] & C(\alpha,J_0)\arrow[r,"g_2"]\arrow[dd,"\Theta_0^{(2)}","\iso"'] & C(\alpha,J_1)\arrow[ddd,"\Theta_1^{(3)}","\iso"'] \\
& C(e^{-\epsilon}\alpha,J_1^{-\epsilon})\arrow[dr,"\Phi_1"'] \\
&& C(e^{-2\epsilon}\alpha,J_0^{-2\epsilon})\arrow[dr,"\Phi_2"'] \\
&&& C(e^{-3\epsilon}\alpha,J_1^{-3\epsilon})
\end{tikzcd}\]
Then both $f\circ g_1$ and $g_2\circ f$ are bi-filtered chain homotopic to the identity map, and hence
\[ f\quad\text{and}\quad g_2\circ f\circ g_1 \]
are the desired bi-filtered chain homotopy equivalences.
\end{proof}

\subsection{Invariance results in the unbounded case}\label{invariance_results_without_an_action_bound}

In this section we aim to prove $J$ and $\alpha$ invariance of ECK.  By $\alpha$-invariance we mean invariance under homotopies of $\alpha$---such an invariance result in the bounded case would not be feasible since, for example, a simple rescaling of $\alpha$ can increase the action of Reeb orbits arbitrarily.

Our approach to proving invariance results is to construct exact symplectic cobordisms which are homotopic.  Then we obtain chain maps and chain homotopies which respect the bi-filtrations on the complexes as required, and we can proceed in a similar manner to the proof of \cref{J_invariance_bounded_case_thm}

When removing the action bound, we run into difficulties since our procedure of extending cobordisms to cobordisms between rational open books in \cref{cobordism_maps_on_ECK}  involves choosing an action bound $L>0$.  One would hope that a direct limit argument might work: \cref{maps_on_ECK} constructs maps
\[ \Phi^L([0,1]\cross N,\lambda) : ECC\orbits{0he}{eh}(\mathrm{int}(N), \alpha_+, J_+; L) \to ECC\orbits{0he}{eh}(\mathrm{int}(N), \alpha_-, J_-; L) \]
and a direct limit as $L$ tends to $\infty$, if it were possible, would yield a map
\[ ECC\orbits{0he}{eh}(\mathrm{int}(N), \alpha_+, J_+) \to ECC\orbits{0he}{eh}(\mathrm{int}(N), \alpha_-, J_-). \]
Unfortunately this approach seems difficult, if possible at all, since in order to take a direct limit we require a commutative diagram of the form
\[\begin{tikzcd}
ECC\orbits{0he}{eh}(\mathrm{int}(N), \alpha_+, J_+; L)\arrow[d,"{i_{L,L'}}"]\arrow[rr,"{\Phi^L([0,1]\cross N,\lambda)}"] && ECC\orbits{0he}{eh}(\mathrm{int}(N), \alpha_-, J_-; L)\arrow[d,"{i_{L,L'}}"] \\
ECC\orbits{0he}{eh}(\mathrm{int}(N), \alpha_+, J_+; L')\arrow[rr,"{\Phi^{L'}([0,1]\cross N,\lambda)}"] && ECC\orbits{0he}{eh}(\mathrm{int}(N), \alpha_-, J_-; L')
\end{tikzcd}\]
for $L<L'$.  But the chain maps $\Phi^L([0,1]\cross N,\lambda)$ are highly non-canonical and hard to control, since they are defined via Seiberg-Witten theory.  In fact, the only method we have of obtaining commutativity results is via homotopies of exact symplectic cobordisms---this is how commutativity is proved in the case of ECH by Hutchings and Taubes~\cite[Theorem 1.9]{HT_Arnold13}.  But at the chain level such homotopies do not give rise to commutative diagrams, only diagrams which commute up to bi-filtered chain homotopy. 

For this reason we will start by making a slight compromise when proving invariance results in this section, which allows us to avoid infinite complexes and hence direct limits.

The goal of this section is to prove the following:

\begin{thm}\label{ECC_independent_of_alpha}
Suppose that $\alpha_0$ and $\alpha_1$ are homotopic extendable contact forms on $N$ and that $J_0$ and $J_1$ are corresponding regular adapted almost complex structures on $\R\cross N$.  Then the complexes
\[ ECC\orbits{he}{eh}(\mathrm{int}(N),\alpha_i, J_i) \]
for $i=0,1$ have the same $e_+$-filtered homotopy type within each Alexander grading. As a result, we recover the result of Colin, Ghiggini and Honda~\cite[Theorem 10.2.2]{CGH_ECH_OBD}, namely that the groups 
\[ \widehat{ECK}(K,\alpha_i, J_i) \]
for $i=0,1$ are isomorphic within each Alexander grading.
\end{thm}

$J$-invariance follows as an immediate corollary by setting $\alpha_0=\alpha_1$. Notice that in the statement of the theorem we have dropped the differential from $h_+$ to $\emptyset$.  The resulting complexes split, with respect to Alexander grading, as direct sums of finite complexes:
\[ ECC\orbits{he}{eh}(\mathrm{int}(N),\alpha, J) = \bigoplus_{j\in\N} ECC\orbits{he}{eh}(\mathrm{int}(N),\alpha, J;j) \]
where the modifier $j$ means we restrict to solely orbit sets with Alexander grading $j$.  We will see in the proof below that it is relatively straightforward to prove $J$ and $\alpha$ invariance of these finite complexes, by essentially taking for each $j$ an $L_j$ larger than the action of every orbit set of Alexander grading $j$.  We will however make the following more general conjecture:

\begin{conj}\label{full_ECK_J_and_alpha_invariant}
The bi-filtered homotopy type of the full ECK complex
\[ ECC\orbits{0he}{eh}(\mathrm{int}(N),\alpha, J) \]
is invariant of the choice of $J$ and under homotopies of $\alpha$.
\end{conj}

Later, in \cref{invariance_results_for_full_ECK}, we will prove this conjecture in the case of \emph{integral} open  book decompositions.  This will follow from \cref{ECC_independent_of_alpha} and the results of \cref{ECK_hat_supported_in_low_Alexander_degrees}.

\begin{proof}[Proof of \cref{ECC_independent_of_alpha}]
The proof which follows is very similar conceptually to that of \cref{J_invariance_bounded_case_thm}.

By \cref{interpolating_cobordism_between_homotopic_forms}, there exist positive constants $A_0$, $A_1$ and interpolating cobordisms $([0,1]\cross N,\lambda_0)$ from $\alpha_1$ to $A_0\alpha_0$ and $([0,1]\cross N,\lambda_1)$ from $A_0\alpha_0$ to $A_1\alpha_1$.  We can assume that both $A_i$ are less than or equal to 1.  Furthermore, we can equip these cobordisms with cobordism-admissible almost complex structures which interpolate between $J_1$ and $J_0^{c_0}$ in the case of $\lambda_0$ and between $J_0^{c_0}$ and $J_1^{c_1}$ in the case of $\lambda_1$ (here $A_i=e^{c_i}$).  The reason for this follows from the contractability of the space of adapted almost complex structures for each form, via a very similar argument to that used in the proof of \cref{J_invariance_bounded_case_thm}.

Fix $j\ge 0$ and let $L>0$ be large enough so that all orbits of $\alpha_i$ with Alexander grading $j$ have $\alpha_i$-action less than $(A_0A_1)^{-1}L>L$, for $i=0,1$. Then
\[ ECC\orbits{he}{eh}(\mathrm{int}(N),A\alpha_i, J_i^{c};j,L) = ECC\orbits{he}{eh}(\mathrm{int}(N),A\alpha_i, J_i^{c};j) \]
for $i\in\set{0,1}$, $A\in\set{1,A_0, A_1, A_0A_1}$ and $c=\log(A)$.

Denote the induced chain maps by $\Phi_0$ and $\Phi_1$ respecticely.  We also have a cobordism $([0,1]\cross N, A_1\lambda_0)$ between $(A_1\alpha_1, J_1^{c_1})$ and $(A_0A_1\alpha_0, J_0^{c_0+c_1})$; denote its induced chain map by $\Phi_2$.

Recall from \cref{direct_limits_through_cobordism_maps} that interpolating cobordisms are defined via an isotopy of $M$.  By Colin, Ghiggini and Honda~\cite[Lemma 3.1.6]{CGH_ECH_OBD}, a homotopy of isotopies relative to its endpoints induces a homotopy of exact symplectic cobordisms.  In this case, if $\lambda_0$ is defined using some isotopy $\phi_t$ then we can assume $\lambda_1$ is defined via the inverse isotopy $\phi_{1-t}\phi_1^{-1}$.  The composition of the two cobordisms is therefore defined by the concatenation of the isotopies, which is in turn homotopic to the identity isotopy.

Therefore the composition of the two cobordisms above is homotopic, through exact symplectic cobordisms, to the product cobordism on $[0,2]\cross N$ between $\alpha_1$ and $A_1\alpha_1$, and hence, by \cref{maps_on_ECK} and \cref{prod_cob_for_torus_bdry_prop,composition_homotopic}, $\Phi_1\circ\Phi_0$ is bi-filtered chain homotopic to the map
\[ ECC\orbits{he}{eh}(\mathrm{int}(N),\alpha_1, J_1;j,L) \to ECC\orbits{he}{eh}(\mathrm{int}(N),A_1\alpha_1, J_1^{c_1};j,L) \]
which maps every orbit set to itself. Similarly the composition $\Phi_2\circ\Phi_1$ is also bi-filtered chain homotopic to the map sending every orbit set to itself.

Now, analogously to the proof of \cref{J_invariance_bounded_case_thm}, define a collection of isomorphisms
\[ \Theta_i^{(A)} : ECC\orbits{he}{eh}(\mathrm{int}(N),\alpha_i, J_i;j) \to ECC\orbits{he}{eh}(\mathrm{int}(N),A\alpha_i, J_i^c ; j,L), \]
for $i\in\set{0,1}$, $A\in\set{1,A_0, A_1, A_0A_1}$ and $c=\log(A)$, by the composition
\[\begin{tikzcd}
ECC\orbits{he}{eh}(\mathrm{int}(N),\alpha_i,J_i;j)\arrow[d,"="] \\
ECC\orbits{he}{eh}(\mathrm{int}(N),\alpha_i,J_i;j,L)\arrow[d,"{i_{L,A^{-1}L}}"] \\
ECC\orbits{he}{eh}(\mathrm{int}(N),\alpha,J_i;j,A^{-1}L) \arrow[d,"{s^{c}}"]\\
ECC\orbits{he}{eh}(\mathrm{int}(N),A\alpha, J_i^{c};j,L)
\end{tikzcd}\]
where the inclusion maps $i_{L,A^{-1}L}$ are isomorphisms since $\alpha$ has no orbit sets with action in the interval $[L,A^{-1}L]$ and $s^c$ are the rescaling isomorphisms from \cref{rescale_alpha_J_iso}.  Furthermore, these maps, and hence $\Theta^{(A)}_i$, map every orbit set to itself.

Now define chain maps by
\begin{align*}
g_1 &:= \left(\Theta^{(A_0)}_0\right)^{-1}\circ\Phi_0\circ\Theta^{(1)}_1: C_1 \to C_0,\\
f &:= \left(\Theta^{(A_1)}_1\right)^{-1}\circ\Phi_1\circ\Theta^{(A_0)}_0:C_0 \to C_1,\\
g_2 &:= \left(\Theta^{(A_0A_1)}_0\right)^{-1}\circ\Phi_2\circ\Theta^{(A_1)}_1: C_1 \to C_0,
\end{align*}
where here we are using the notation 
\[  C_i := ECC\orbits{he}{eh}(\mathrm{int}(N),\alpha_i,J_i;j). \]
Then both $f\circ g_1$ and $g_2\circ f$ are bi-filtered chain homotopic to the identity map, and hence
\[ f\quad\text{and}\quad g_2\circ f\circ g_1 \]
are the desired bi-filtered chain homotopy equivalences.
\end{proof}

\section{\texorpdfstring{$\widehat{ECK}$}{ECK hat} is supported in low Alexander degrees}\label{ECK_hat_supported_in_low_Alexander_degrees}

\begin{defn}
$\widehat{ECK}$ splits canonically as a direct sum with respect to the Alexander grading.  Write
\[ \widehat{ECK}(K,\alpha;j) \]
to denote the summand of $\widehat{ECK}$ with Alexander grading $j$, so that
\[ \widehat{ECK}(K,\alpha) = \bigoplus_{j\ge0}\widehat{ECK}(K,\alpha;j). \]
$\widehat{ECK}(K,\alpha;j)$ is the homology of the part of the full complex with coordinates $(0,j)$.
\end{defn}

In this section, assume that we are defining ECK from an integral open book decomposition, so that the binding $K$ is homologous to the boundary of the page $\Sigma$ and hence is null-homologous.  The goal of this section is to prove the following:

\begin{thm}\label{ECK_supported_genus}
Let $g$ be the genus of the surface $\Sigma$. Then
\[ \widehat{ECK}(K,\alpha;j) = 0 \]
for $j>2g$.  As a result, when considered up to bi-filtered chain homotopy type, the full complex
\[ ECC\orbits{0he}{eh}(\mathrm{int}(N), \alpha, J) \]
is supported entirely within the diagonal region of height $2g$,
\[ \set{ (i,j)\in \Z^2 \gappy{|} i\ge0, i\le j \le i+2g } \]
(c.f.~the discussion in \cref{symmetries_of_the_complex}).
\end{thm}

We will prove this theorem by building upon results of Colin, Ghiggini and Honda.  The first step is to switch from the embedded contact homology groups to the similarly defined \emph{periodic Floer homology} groups.

Let $\phi:(\Sigma,\omega)\to(\Sigma,\omega)$ be a symplectomorphism and define the mapping torus $N := M(\phi)$.  Periodic Floer homology is defined completely analogously to embedded contact homology, except we replace the Reeb vector field with the vector field $\del/\del t$ (where $t$ denotes the $[0,1]$ coordinate) and holomorphic curves are counted with respect to the almost complex structure $J$ on $\R\cross N$ satisfying 
\begin{itemize}
\item on each $\set{s}\cross\Sigma\cross\set{t}$, $J$ maps $T\Sigma$ to itself and is compatible with $\omega$ in the sense that $\omega(\cdot,J\cdot)$ is a Euclidean metric on $T\Sigma$; and
\item $J(\del_t) = \del_s$.
\end{itemize}
In fact, since we have a Morse-Bott torus at $\del N$, the definition of periodic Floer homology in this setting is more analogous to Morse-Bott contact homology and the differential is computed by counting very nice index 1 $J$-holomorphic Morse-Bott buildings.

The setup used in periodic Floer homology is called a stable Hamiltonian structure and is a generalization of a contact structure.  
\begin{defn}
A \emph{stable Hamiltonian structure} is a triple $(N,\omega,\alpha)$, where $\omega$ is a two-form on $N$ and $\alpha$ is a one-form, such that
\begin{itemize}
\item $\alpha\wedge\omega>0$ everywhere and
\item $\mathrm{ker}(\omega) \subset \mathrm{ker}(d\alpha)$.
\end{itemize}
A stable Hamiltonian structure then gives rise to a Reeb vector field $R$ on $N$ defined by
\[ \iota_R(\omega)=0 \quad\text{and}\quad \alpha(R)=1. \]
The associated symplectization is $(\R\cross N, J)$, where $J$ is defined on the plane distribution $\mathrm{ker}(\alpha)$ to be compatible with $\omega$ and otherwise by the equation $J(R)=\del_s$.
\end{defn}

The stable Hamiltonian structure used in periodic Floer homology is the triple
\[ (N, \omega, dt), \]
but note that any contact structure gives rise to a stable Hamiltonian structure by setting $\omega=d\alpha$.

We will follow the argument by Colin, Ghiggini and Honda~\cite[Section 3]{CGH_HF_ECH_I} to construct a sequence of contact forms $\alpha_\zeta$ on $N$ which limit towards a stable Hamiltonian structure, and obtain the result that for small enough $\zeta>0$ the resulting ECC and PFC chain complexes are isomorphic.

\subsection{Reparametrizing the contact form on \texorpdfstring{$N$}{N}}\label{reparametrizing_the_contact_form_on_N}

Recall the contact form $\alpha$ defined on $N$ in \cref{a_contact_form_on_N}, or more precisely the contact form $\alpha_3$ from which we obtained $\alpha$ by taking a small perturbation to achieve non-degeneracy on $\mathrm{int}(N)$.  We will recall a few important properties of $\alpha_3$ here, dropping the subscript for simplicity of notation.
\begin{itemize}
\item $\alpha = f_t\dd t + \beta_t$, where $f_t$ is a positive function and $\beta_t$ is a one-form on $\Sigma$.
\item On a neighbourhood $U:=\del N\cross[2,2+\epsilon)$ of $\del N$, $f_t$ and $\beta_t$ depend only on $y$, the $[2,2+\epsilon)$ coordinate.
\item On $N \sminus U$, the Reeb vector field of $\alpha$ is parallel to $\del/\del t$.
\item On $U$, the Reeb vector field is parallel to $-g'(y)\dd t + f'(y)\dd \theta$ (c.f.~\cref{contact_form_on_extension_of_N_fig}, recalling that the form as shown in the figure is pulled back to $N$ when defining $\alpha_3$).
\end{itemize}
We will now make a slight alteration to our coordinate system for $N$ on $U$ so that the Reeb vector field for $\alpha$ is parallel to $\del/\del t$ everywhere, and the monodromy $\phi$ for $N$, rather than restricting to the identity map on $\del\Sigma\cross[2,2+\epsilon)$, will turn to the \emph{right} away from $\del \Sigma$.  Let
\[ (t,\theta',y) \]
be a new coordinate system for $U$, where $\theta' = \theta - ts(y)$ for some function $s$ which is to be determined. Then on $U$,
\[\begin{split}
\alpha &= f(y)\dd t + g(y)\dd \theta \\
         &= f(y)\dd t + g(y)(d\theta' + s(y)\dd t + ts'(y)\dd y) \\
				 &= (f(y) + g(y)s(y))\dd t + g(y)\dd\theta' + tg(y)s'(y)dy
\end{split}\]
and hence
\[ d\alpha = (f'(y) + g'(y)s(y))\dd y\wedge dt + g'(y)\dd y\wedge d\theta'. \]
Now define
\[ s(y) := -\frac{f'(y)}{g'(y)}\ge 0. \]
Then
\[ d\alpha(\del_t,\cdot) = 0 \]
and $s(y) = 0$ for $y=2,2+\epsilon$, so the new coordinate system is compatible with the old coordinate system on the rest of $N$.  Also, in this new coordinate system, the monodromy near $\del\Sigma$ is given by the identification
\[  (1,\theta',y) \sim (0,\theta'+s(y),y), \]
which turns to the right away from $\del\Sigma$.  Of course, when defining the contact form $\alpha$ on $N$, we took a small perturbation to make all Reeb orbits non-degenerate.  In this setting, we make a small perturbation to the \emph{monodromy} to achieve non-degeneracy.

As a result of this discussion, we can assume that
\begin{itemize}
\item $\alpha = f_t\dd t + \beta_t$, where $f_t$ is a positive function and $\beta_t$ is a one-form on $\Sigma$.
\item The Reeb vector field of $\alpha$ is parallel to $\del/\del t$ everywhere.
\end{itemize}
Then we can compute that
\begin{align*}
 d\alpha &= d_\Sigma f_t\wedge dt + d_\Sigma \beta_t + dt\wedge \dot\beta_t \\
         &= d_\Sigma \beta_t + (d_\Sigma f_t - \dot\beta_t)\wedge dt,
\end{align*}
and hence, since $d\alpha(\del_t,\cdot) = 0$, we must have $d_\Sigma f_t = \dot\beta_t$.  Then 
\[ \frac{\del}{\del t}(d\alpha)=\frac{\del}{\del t}(d_\Sigma \beta_t) = d_\Sigma (\dot\beta_t) = d_\Sigma (d_\Sigma f_t) = 0, \]
so $d\alpha$ is independent of $t$ and hence can be denoted by $\omega$.

\subsection{Limiting towards a stable Hamiltonian structure}\label{limiting_towards_a_stable_hamiltonian_structure}

Now consider the one form
\[ \alpha_\zeta = C\dd t + \zeta\alpha \]
for $\zeta\in[0,1]$, which is a contact form for $\zeta\neq 0$ and defines a stable Hamiltonian structure $C\dd t$ when $\zeta=0$.  The derivative of $\alpha_\zeta$ is
\[ d\alpha_\zeta = \zeta d\alpha = \zeta\omega \]
and hence the data $(\alpha_\zeta,\omega)$ is a stable Hamiltonian structure with Reeb vector field parallel to $\del/\del t$ for all $\zeta\in[0,1]$.  Moreover we can choose compatible almost complex structures $J_\zeta$ such that 
\begin{itemize}
\item $J_\zeta$ is compatible with $(\alpha_\zeta,\omega)$,
\item $J_\zeta$ is smooth in $\zeta$,
\item $J_0$ is regular, in the sense that the moduli spaces between orbit sets with Alexander grading at most $j$ are transversely cut out~\cite[Lemma 3.5.2]{CGH_HF_ECH_I}.
\end{itemize}
Since $d\alpha_\zeta = \zeta\omega$, the almost complex structure $J_\zeta$ is also a suitable choice for defining the ECC complex, so we obtain the complexes
\[  ECC\orbits{he}{eh}(\mathrm{int}(N),\alpha_\zeta, J_\zeta;j)\quad\text{and}\quad PFC\orbits{he}{eh}(\mathrm{int}(N), C\dd t, J_0; j). \]
It then follows from a compactness argument~\cite[Theorem 3.6.1]{CGH_HF_ECH_I} that for $\zeta>0$ sufficiently small, the above chain complexes are isomorphic under the canonical identification of orbits.

Since $\alpha$ and $\alpha_\zeta$ are homotopic contact forms for all $\zeta>0$, their ECC complexes are filtered chain homotopy equivalent by \cref{ECC_independent_of_alpha}.  We therefore obtain the following:
\begin{thm}\label{ECC_is_PFC}
For a fixed $j\ge 0$, there exists $J_0$, a regular compatible almost complex structure for the stable Hamiltonian structure $(C\dd t, \omega)$, such that the complexes
\[  ECC\orbits{he}{eh}(\mathrm{int}(N),\alpha ;j)\quad\text{and}\quad PFC\orbits{he}{eh}(\mathrm{int}(N), \alpha_0, J_0; j) \]
have the same $(e_+)$-filtered homotopy type.
\end{thm}

When $N$ is a mapping torus formed via $\phi:\Sigma\to\Sigma$, and when $\phi$ is understood, we will denote the periodic Floer groups by
\[ PFC\orbits{he}{eh}(\Sigma;j), \]
dropping the almost complex structure $J_0$ from the notation.

\subsection{Stabilization of periodic Floer groups}

By \cref{ECC_is_PFC}, the proof of \cref{ECK_supported_genus} is reduced to proving that
\[ PFH\orbits{he}{h}(\Sigma; j) = 0 \]
for $j>2g$. We will extend the work of Colin, Ghiggini and Honda~\cite[Section 5]{CGH_HF_ECH_II} to prove the following lemma.
\begin{lemma}\label{e_map_is_quasi_iso}
The map
\begin{align*}
\mathfrak{J}_j:PFC\orbits{e}{h}(\Sigma;j) &\to PFC\orbits{e}{h}(\Sigma;j+1) \\
\Gamma &\mapsto e_-\Gamma
\end{align*}
is a quasi-isomorphism for $j\ge 2g$.
\end{lemma}
\begin{proof}[Proof of \cref{ECK_supported_genus} from \cref{e_map_is_quasi_iso}]
First denote by $C_j$ the complex $PFC\orbits{}{h}(\Sigma;j)$ and notice that, by the \refNamedThm{Cancellation Lemma}{cancellation_lemma}, $C_j$ is chain homotopic to the complex $PFC\orbits{he}{h}(\Sigma;j)$.  We have a short exact sequence
\[\begin{tikzcd}
0\arrow[r] & PFC\orbits{e}{h}(\Sigma;j)\arrow[r,"\mathfrak{J}_j"] & PFC\orbits{e}{h}(\Sigma;j+1) \arrow[r]& C_{j+1} \arrow[r] &0,
\end{tikzcd}\]
where the second map is the quotient map sending $e_-$ to 0.  The resulting exact triangle on homology completes the argument since, if $\mathfrak{J}_j$ is a quasi-isomorphism, then
\[ PFH\orbits{he}{h}(\Sigma; j) \iso H_*(C_{j+1})=0\]
for $j\ge 2g$ as required.
\end{proof}

The remainder of this section will be spent proving \cref{e_map_is_quasi_iso}, which is rather involved and borrows many tools from the work of Colin, Ghiggini and Honda~\cite[Section 5]{CGH_HF_ECH_II}, who employ a stabilization technique which we will outline here briefly.

Let $\Sigma' := \Sigma \#_\del T$, where $T$ is a punctured torus and $\#_\del$ denotes the boundary-connected sum.  Define a monodromy $\phi':\Sigma'\to\Sigma'$ by
\begin{itemize}
\item $\restr{\phi'}{\Sigma}=\phi$,
\item $\restr{\phi'}{T}=\phi_T$ where $\phi_T$ is a non-degenerate monodromy corresponding to the right-handed trefoil knot which turns to the right away from $\del T$, and
\item on the connected sum region, $\phi'$ is given by a small perturbation of the identity map, illustrated in \cref{connect_sum_region_perturbation}.
\end{itemize}
There exists an almost complex structure $J'$ on the symplectization of $M(\phi')$ such that in the connect sum region the differentials are as shown in \cref{connect_sum_region_differentials}~\cite[Lemma 5.2.2]{CGH_HF_ECH_II}.

\begin{figure}\centering
	\begin{subfigure}{0.49\textwidth}\centering
		\begin{tikzpicture}[scale=0.78]%
			\draw [fill=lightgray, even odd rule] (0,0) ellipse (4 and 3) (-1.5,0) circle (0.7)  node[yshift=2,scale=1.5] {$\Sigma$} (1.5,0) circle (0.7) node[yshift=2,scale=1.5] {$T$};
			\draw [fill] (-1.5,0.7) circle (0.05) node[above] {$e_0$};
			\draw [fill] (-1.5,-0.7) circle (0.05) node[above,inner sep=2pt] {$h_0$};
			\draw [fill] (1.5,0.7) circle (0.05) node[above] {$e_T$};
			\draw [fill] (1.5,-0.7) circle (0.05) node[above,inner sep=2pt] {$h_T$};
			\draw [fill] (0,3) circle (0.05) node[above] {$e_-$};
			\draw [fill] (0,-3) circle (0.05) node[below] {$h_-$};
			\draw [fill] (-1,1.7) circle (0.05) node[above,inner sep=5pt] {$\phantom{h}h_{1P}$};
			\draw [fill] (1,1.7) circle (0.05) node[above,inner sep=5pt] {$h_{2P}\phantom{h}$};
			\draw [fill] (0,-1.7) circle (0.05) node[above] {$e_P$};

			\draw (-1,1.7) to[out=-150,in=90] (-2.5,0) to[out=-90,in=180] (-1.5,-1) to[out=0,in=-90] (-0.5,0) to[out=90,in=-60] (-1,1.7);
			\draw (1,1.7) to[out=-30,in=90] (2.5,0) to[out=-90,in=0] (1.5,-1) to[out=180,in=-90] (0.5,0) to[out=90,in=-120] (1,1.7);
			\draw (-1,1.7) to[out=30,in=180] (0,2) to[out=0,in=150] (1,1.7);
			\draw (-1,1.7) to[out=120,in=90] (-3.6,0) to[out=-90,in=180] (0,-2.75) to[out=0,in=-90] (3.6,0) to[out=90,in=60] (1,1.7);
			\draw (0,-0.7) to[out=180,in=45] (-0.4,-1.1) to[out=-135,in=90] (-1.4,-1.7) to[out=-90,in=180] (0,-2.25) to[out=0,in=-90] (1.4,-1.7) to[out=90,in=-45] (0.4,-1.1) to[out=135, in=0] cycle;
			\draw (0,-1.6) circle (0.4);
			\draw (0,1) to[out=180,in=90] (-0.25,0) to[out=-90,in=0] (-1.5,-1.2) to[out=180,in=-45] (-2.4,-0.9) to[out=135,in=0] (-3,0) to[out=180,in=135] (-2.5,-1.6) to[out=-45,in=180] (0,-2.5) to[out=0,in=-135] (2.5,-1.6) to[out=45,in=0] (3,0) to[out=180,in=45] (2.4,-0.9) to[out=-135,in=0] (1.5,-1.2) to[out=180,in=-90] (0.25,0) to[out=90,in=0] cycle;

			\draw [-latex,shorten >=-3pt] (-2.5,0)--(-2.5,0.01);
			\draw [-latex,shorten >=-3pt] (-0.5,0)--(-0.5,-0.01);
			\draw [-latex,shorten >=-3pt] (2.5,0)--(2.5,-0.01);
			\draw [-latex,shorten >=-3pt] (0.5,0)--(0.5,0.01);
			\draw [-latex,shorten >=-3pt] (0,-0.7)--(-0.01,-0.7);
			\draw [-latex,shorten >=-3pt] (-1.4,-1.7)--(-1.4,-1.71);
			\draw [-latex,shorten >=-3pt] (0,-2.25)--(0.01,-2.25);
			\draw [-latex,shorten >=-3pt] (1.4,-1.7)--(1.4,-1.69);
			\draw [-latex,shorten >=-3pt] (0.4,-1.6)--(0.4,-1.59);
			\draw [-latex,shorten >=-3pt] (-0.4,-1.6)--(-0.4,-1.61);
			\draw [-latex,shorten >=-3pt] (0,2)--(-0.01,2);
			\draw [-latex,shorten >=-3pt] (3.6,0)--(3.6,0.01);
			\draw [-latex,shorten >=-3pt] (0,-2.75)--(0.01,-2.75);
			\draw [-latex,shorten >=-3pt] (-3.6,0)--(-3.6,-0.01);
			\draw [-latex,shorten >=-3pt] (0,1)--(-0.01,1);
			\draw [-latex,shorten >=-3pt] (-1.5,-1.2)--(-1.51,-1.2);
			\draw [-latex,shorten >=-3pt] (-3,0)--(-3.01,0);
			\draw [-latex,shorten >=-3pt] (0,-2.5)--(0.01,-2.5);
			\draw [-latex,shorten >=-3pt] (3,0)--(2.99,0);
			\draw [-latex,shorten >=-3pt] (1.5,-1.2)--(1.49,-1.2);
			
		\end{tikzpicture}			
		\caption{}
		\label{connect_sum_region_perturbation}
	\end{subfigure}
	\begin{subfigure}{0.49\textwidth}\centering
		\begin{tikzpicture}[scale=0.78]%
			\draw [fill=lightgray, even odd rule] (0,0) ellipse (4 and 3) (-1.5,0) circle (0.7)  node[scale=1.5] {$\Sigma$} (1.5,0) circle (0.7) node[scale=1.5] {$T$};
			\draw [fill] (-1.5,0.7) circle (0.05) node[above left, inner sep=1pt] {$e_0$};
			\draw [fill] (-1.5,-0.7) circle (0.05) node[below left, inner sep=0pt] {$h_0$};
			\draw [fill] (1.5,0.7) circle (0.05) node[above right, inner sep=1pt] {$e_T$};
			\draw [fill] (1.5,-0.7) circle (0.05) node[below right, inner sep=1pt] {$h_T$};
			\draw [fill] (0,3) circle (0.05) node[above] {$e_-$};
			\draw [fill] (0,-3) circle (0.05) node[below] {$h_-$};
			\draw [fill] (-1,1.7) circle (0.05) node[above, inner sep=3pt] {$h_{1P}\phantom{h}$};
			\draw [fill] (1,1.7) circle (0.05) node[above, inner sep=3pt] {$\phantom{h}h_{2P}$};;
			\draw [fill] (0,-1.7) circle (0.05) node[right,inner sep=5pt] {$e_P$};

			\draw (0,-1.7)--(0,-3);
			\draw (0,-1.7) to[out=150,in=-20] (-1,-1.3) to[out=160,in=-90] (-1.5,-0.7);
			\draw (0,-1.7) to[out=30,in=-160] (1,-1.3) to[out=20,in=-90] (1.5,-0.7);
			\draw (0,-1.7) to[out=-150,in=-90] (-3,0) to[out=90,in=150] (-1,1.7);
			\draw (0,-1.7) to[out=-30,in=-90] (3,0) to[out=90,in=30] (1,1.7);
			\draw (0,-1.7) to[out=120,in=-90] (-0.3,0) to[out=90,in=-30] (-1,1.7);
			\draw (0,-1.7) to[out=60,in=-90] (0.3,0) to[out=90,in=-150] (1,1.7);
			\draw (-1,1.7) to[out=-150,in=60] (-1.4,1.3) to[out=-120,in=90] (-1.5,0.7);
			\draw (-1,1.7) to[out=30,in=-120] (-0.4,2.3) to[out=60,in=-120] (0,3);
			\draw (1,1.7) to[out=150,in=-60] (0.4,2.3) to[out=120,in=-60] (0,3);
			\draw (1,1.7) to[out=-30,in=120] (1.4,1.3) to [out=-60,in=90] (1.5,0.7);

			\draw [-latex,shorten >=-3pt] (0,-2.3)--(0,-2.31);
			\draw [-latex,shorten >=-3pt] (-1,-1.3)-- +(160:0.01);
			\draw [-latex,shorten >=-3pt] (1,-1.3)-- +(20:0.01);
			\draw [-latex,shorten >=-3pt] (-3,0)-- +(90:0.01);
			\draw [-latex,shorten >=-3pt] (3,0)-- +(90:0.01);
			\draw [-latex,shorten >=-3pt] (-0.3,0)-- +(90:0.01);
			\draw [-latex,shorten >=-3pt] (0.3,0)-- +(90:0.01);
			\draw [-latex,shorten >=-3pt] (-1.4,1.3)-- +(-120:0.01);
			\draw [-latex,shorten >=-3pt] (-0.4,2.3)-- +(60:0.01);
			\draw [-latex,shorten >=-3pt] (0.4,2.3)-- +(120:0.01);
			\draw [-latex,shorten >=-3pt] (1.4,1.3)-- +(-60:0.01);
			\draw [-latex,shorten >=-3pt] (-4,0)-- +(90:0.01);
			\draw [-latex,shorten >=-3pt] (-2.2,0)-- +(90:0.01);
			\draw [-latex,shorten >=-3pt] (-0.8,0)-- +(90:0.01);
			\draw [-latex,shorten >=-3pt] (0.8,0)-- +(90:0.01);
			\draw [-latex,shorten >=-3pt] (2.2,0)-- +(90:0.01);
			\draw [-latex,shorten >=-3pt] (4,0)-- +(90:0.01);

		\end{tikzpicture}
		\caption{}
		\label{connect_sum_region_differentials}
	\end{subfigure}
  \caption{In (a) we see the small perturbation of the identity map on the connect sum region which defines $\phi'$.  We obtain three non-degenerate orbits in the interior of the region and the three tori $\del\Sigma$, $\del T$ and $\del\Sigma'$ are negative Morse-Bott.  The differentials in (b) arise by employing \cref{foliation_perp_to_R_xi} in the PFH setting.  To see this, first note that, by a change of coordinates, (a) can also be interpreted as a diagram for the connect sum region in which the monodromy is given by the identity map and the arrows indicate the behaviour of the ($S^1$-invariant) vector field $R$.  Then $X$ is obtained by rotating the vector field a quarter-turn to the left, and its integral curves gives rise to the $S^1$-invariant holomorphic cylinders seen in (b).}
	\label{connect_sum_region}
\end{figure}
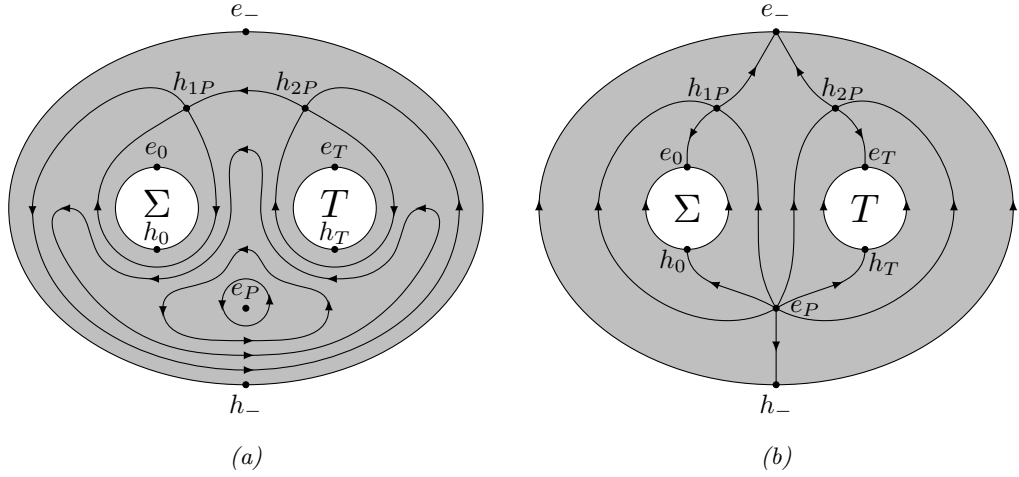

Next, consider the following diagram
\begin{equation}\label{stabilization_diag}\begin{tikzcd}
 \widehat{HF}(\Sigma, \phi) \arrow[r,"\Theta"]\arrow[d,"\Phi_{(S,\phi)}"] & \widehat{HF}(\Sigma',\phi') \arrow[d,"\Phi_{(S',\phi')}"] \\
 PFH\orbits{e}{h}(\Sigma;2g) \arrow[r,"i_*\circ \mathfrak{J}^2_*"] & PFH\orbits{e}{h}(\Sigma';2g+2).
\end{tikzcd}\end{equation}
Here the map $\Theta$ is an isomorphism induced by a chain map on the Heegaard Floer complexes given by including two intersection points between arcs in $T$, and the vertical maps are the isomorphisms between Heegaard Floer homology and periodic Floer homology induced by chain maps constructed by Colin, Ghiggini and Honda~\cite[Section 6.2]{CGH_HF_ECH_I} as part of their proof that HF=ECH.

The bottom map is the most interesting to us---it is induced by the chain map $i\circ \mathfrak{J}^2$, where $\mathfrak{J}^2 = \mathfrak{J}_{2g+1}\circ \mathfrak{J}_{2g}$
 and $i$ is the inclusion map
\[ PFC\orbits{e}{h}(\Sigma;2g+2) \to PFC\orbits{e}{h}(\Sigma';2g+2). \]

\refDiagram{stabilization_diag} commutes~\cite[Lemma 5.3.1]{CGH_HF_ECH_II} and hence, since three sides are isomorphisms, $i_*\circ \mathfrak{J}^2_*$ is also.

Let us introduce some notation which is used in the lemma below: 
\begin{align*}
V_i &:= PFH^{e_0}_{h_0}(\Sigma;i) \\
W_i &:= PFH^{e_T}_{h_T}(T;i).
\end{align*}
Here, to avoid confusion, we are denoting the orbits $e_-$ and $h_-$ appearing at $\del \Sigma$ by $e_0$ and $h_0$ and the orbits $e_-$ and $h_-$ appearing at $\del T$ by $e_T$ and $h_T$ (see \cref{connect_sum_region}). We will continue to refer to the orbits at $\del \Sigma'$ by $e_-$ and $h_-$.

\begin{lemma}[{\nogapcite[Lemma 5.4.2]{CGH_HF_ECH_II}}]\label{stabilization_iso}
There is a natural isomorphism
\[ PFH\orbits{e}{h}(\Sigma';j) \iso \left(\bigoplus_{i+k=j}V_i \tensor W_k\right)/\sim, \]
where $\sim$ is the equivalence relation generated by
\[ e_0\Gamma_0 \tensor \Gamma_1 \sim \Gamma_0 \tensor e_T\Gamma_1. \]
\end{lemma}

This is only stated by Colin, Ghiggini and Honda for the case when $j=2g+2$ but the proof applies to all $j$, as stated here.  The lemma is proven by a spectral sequence argument, which is easy to understand intuitively as follows.  Refer to \cref{connect_sum_region_differentials} and notice that we can cancel the differential from $e_P$ to $h_-$ without introducing any new differentials.  Next cancel the differential from $h_{2P}$ to $e_-$, which introduces a new differential from $h_{1P}$ to $e_T$.  We then have $d(h_{1P}\Gamma) = (e_0 + e_T)\Gamma$ so, on homology, the orbits $e_0$ and $e_T$ are identified as seen in the statement of the lemma.

The next step is to use \cref{stabilization_iso} to show that $\mathfrak{J}^2_*:V_{2g}\to V_{2g+2}$ is an isomorphism; we will briefly outline the method of the proof~\cite[Section 5.4]{CGH_HF_ECH_II} here.   Injectivity follows easily since $i_*\circ\mathfrak{J}^2_*$ is an isomorphism. To show surjectivity, start with an element
\[ v_{2g+2} \in V_{2g+2}, \]
whose image under $i_*$ in $PFH\orbits{e}{h}(\Sigma';2g+2)$ is $v_{2g+2}\tensor 1$ by \cref{stabilization_iso}. Here we are writing 1 to represent the non-zero element $\emptyset\in W_0$.  However since $i_*\circ\mathfrak{J}^2_*$ is an isomorphism, there exists $v_{2g}\in V_{2g}$ such that 
\[ v_{2g+2}\tensor 1 = i_*(\mathfrak{J}^2_*(v_{2g})) = e_0^2v_{2g} \tensor 1. \]
By \cref{stabilization_iso} we can compare terms to obtain
\[ v_{2g+2} = e_0^2v_{2g} = \mathfrak{J}_*^2(v_{2g})\]
and hence $\mathfrak{J}^2_*$ is surjective and therefore an isomorphism.

Colin, Ghiggini and Honda found that this result was sufficient to continue with their proof that ECH=HF. However, for our purposes we need to go a step further and prove that each induced map 
\[ (\mathfrak{J}_j)_*: V_{j}\to V_{j+1}  \]
for $j\ge 2g$ is an isomorphism.  For ease of notation, for the remainder of this section we will drop the $*$ and write $\mathfrak{J}_{j}$ for the map on homology.  We will first prove the case $j=2g,2g+1$, which follows from a slight generalization of the proof that $\mathfrak{J}^2_*$ is an isomorphism above, before proving the more general result by induction later.

\begin{proof}[Proof of \cref{e_map_is_quasi_iso} in the case $j=2g,2g+1$]
Since $\mathfrak{J}_{2g+1}\circ\mathfrak{J}_{2g} = \mathfrak{J}^2_*$ is an isomorphism, it suffices to prove that $\mathfrak{J}_{2g}$ is surjective.  

Let $v_{2g+1}\in V_{2g+1}$.  Then
\[\begin{split}
i_*(\mathfrak{J}_{2g+1}(v_{2g+1})) &= e_0v_{2g+1} \tensor 1 \quad\in PFH\orbits{e}{h}(\Sigma';2g+2)  \\
 &=v_{2g+1}\tensor e_T \\
&= i_*(\mathfrak{J}^2_*(v_{2g}))\quad\text{(for some $v_{2g}\in V_{2g}$)} \\
& = e_0^2 v_{2g}\tensor 1 \\
& = e_0 v_{2g} \tensor e_T.
\end{split}\]
Here the third line follows from the fact that $i_*\circ\mathfrak{J}_*^2$ is an isomorphism and the second and final lines follow from \cref{stabilization_iso}.

As discussed by Colin, Ghiggini and Honda~\cite[Proof of Proposition 5.4.1]{CGH_HF_ECH_II}, the element $e_T\in W_1$ is non-zero. Hence we can compare terms in lines 2 and 5 above to obtain
\[ v_{2g+1} = e_0 v_{2g} = \mathfrak{J}_{2g}(v_{2g}) \]
and we are done.
\end{proof}

\begin{corr}\label{stabilization_inclusion_is_iso}
The inclusion map
\[ i_*: PFH^{e_0}_{h_0}(\Sigma;2g+2) \to PFH\orbits{e}{h}(\Sigma';2g+2) \]
is an isomorphism.
\end{corr}

We now turn to the general case of \cref{e_map_is_quasi_iso}, when $j\ge 2g+3$.  The idea of the proof is to make repeated stabilizations and then use induction by employing a similar trick to the above proof.  We introduce the following notation:
\begin{itemize}
\item Write $\Sigma^{(m)}$ for the $m$-th stabilization of $\Sigma$, so that $\Sigma^{(0)}=\Sigma$ and $\Sigma^{(1)}=\Sigma'$.
\item Write $e_{m}$ and $h_{m}$ for the two orbits which appear at $\del \Sigma^{(m)}$, so that $e_1=e_-$ and $h_1=h_-$.
\item Write $V_j^{(m)} := PFH^{e_{m}}_{h_{m}}(\Sigma^{(m)};j)$, so that $V^{(0)}_j=V_j$.
\item Write $\mathfrak{J}^{(m)}_j:V^{(m)}_j \to V^{(m)}_{j+1}$ for the map induced on homology by the chain map $\Gamma\mapsto e_m\Gamma$, so that $\mathfrak{J}^{(0)}_j=\mathfrak{J}_j$.
\item Write $T^{(m)}$ for the copy of $T$ added when performing the $m$-th stabilization; $e_{T^{(m)}}$ and $h_{T^{(m)}}$ for the two orbits appearing at $\del T^{(m)}$; and $W^{(m)}_i$ for $PFH^{e_{T^{(m)}}}_{h_{T^{(m)}}}(T^{(m)};i)$.
\end{itemize}
With this new notation, \cref{stabilization_iso} can be rewritten as
\[ V^{(m)}_j \iso \left(\bigoplus_{i+k=j} V^{(m-1)}_i \tensor W^{(m)}_k \right ) / \sim, \]
and the map $\mathfrak{J}^{(m)}_j : V^{(m)}_j \to V^{(m)}_{j+1}$ induced by the chain map $\Gamma\mapsto e_m\Gamma$, under this isomorphism, takes the form
\[ \Gamma_0\tensor \Gamma_1 \mapsto e_{m-1}\Gamma_0 \tensor \Gamma_1 = \Gamma_0\tensor e_{T^{(m)}}\Gamma_1  \]
since the orbits $e_m$, $e_{m-1}$ and $e_{T^{(m)}}$ are homologous in the $m$-th stabilization.  By induction we obtain the equation
\begin{equation}\label{stabilization_iso_general}
V^{(m)}_j \iso \left(\bigoplus_{i+k^{(1)}+\cdots+k^{(m)}=j}V^{(0)}_i\tensor W^{(1)}_{k^{(1)}}\tensor\cdots\tensor W^{(m)}_{k^{(m)}}\right) / \sim,
\end{equation}
where the equivalence relation is generated by
\[\begin{split}
e_{0}\Gamma_0\tensor \Gamma_1 \tensor \cdots\tensor\Gamma_m &\sim  \Gamma_0\tensor e_{T^{(1)}}\Gamma_1\tensor\cdots\tensor\Gamma_m \\
&\vdots \\
&\sim \Gamma_0\tensor\Gamma_1\tensor\cdots\tensor e_{T^{(m)}}\Gamma_m.
\end{split}\]
\begin{proof}[Proof of \cref{e_map_is_quasi_iso} for $j\ge 2g+2$]
Fix $m\ge 1$ and by induction assume that the map
\[  i_*: V^{(0)}_{2(g+m)} \to V^{(m-1)}_{2(g+m)} \]
is an isomorphism.  Our aim is to prove that the two maps
\[ \mathfrak{J}_{2(g+m)}\quad\text{and}\quad\mathfrak{J}_{2(g+m)+1} \]
are isomorphisms; for each $m\ge 1$ this proves the result for $j=2g+2m$ and $2g+2m+1$.  For the purposes of induction, we also wish to prove that
\[ i_*: V^{(0)}_{2(g+m+1)} \to V^{(m)}_{2(g+m+1)}  \]
is also an isomorphism.

Consider the following diagram, in which the Alexander grading varies in the horizontal direction and the number of stabilizations varies in the vertical direction:
\[\begin{tikzcd}
V^{(0)}_{2(g+m)} \arrow[rr,"\mathfrak{J}_{2(g+m)}"]\arrow[d,"i_*","\iso"'] && V^{(0)}_{2(g+m)+1}\arrow[rr,"\mathfrak{J}_{2(g+m)+1}"] && V^{(0)}_{2(g+m)+2}\arrow[dd,"i_*"]  \\
V^{(m-1)}_{2(g+m)} \arrow[d,"i_*","\iso"'] \\
V^{(m)}_{2(g+m)}\arrow[rr,"\mathfrak{J}^{(m)}_{2(g+m)}","\iso"'] && V^{(m)}_{2(g+m)+1}\arrow[rr,"\mathfrak{J}^{(m)}_{2(g+m)+1}","\iso"'] && V^{(m)}_{2(g+m)+2}.
\end{tikzcd}\]
The diagram commutes by the discussion above; the orbits $e_i$ and $e_{i-1}$ are homologous in the $i$-th stabilization for $i=1,\dots,m$.  The second inclusion map on the left is an isomorphism by \cref{stabilization_inclusion_is_iso} since $\Sigma^{(m-1)}$ is a surface of genus $g+m-1$. The bottom two maps are isomorphisms by applying the proof of the case when $j=2g,2g+1$ to the surface $\Sigma^{(m)}$ of genus $g+m$.

Denote by $\mathfrak{J}^2$ the composition $\mathfrak{J}_{2(g+m)+1}\circ\mathfrak{J}_{2(g+m)}$. Then $\mathfrak{J}^2$ is injective, and hence $\mathfrak{J}_{2(g+m)}$ is also.  Hence it remains to show that the two maps are surjective---we will prove this by an argument completely analogous to the case when $j=2g,2g+1$ earlier.

Let $v_{2(g+m)+2}\in V^{(0)}_{2(g+m)+2}$ and consider 
\[ i_*(v_{2(g+m)+2}) =v_{2(g+m)+2}\tensor 1\tensor\cdots\tensor 1\in V^{(m)}_{2(g+m)+2}. \]
Then since $i_*\circ\mathfrak{J}^2$ is surjective, there exists some $v_{2(g+m)}\in V^{(0)}_{2(g+m)}$ for which
\[\begin{split}
i_*(v_{2(g+m)+2}) &= i_*\circ\mathfrak{J}^2(v_{2(g+m)}) \\
&= e_0^2 v_{2(g+m)}\tensor 1\tensor\cdots\tensor 1,
\end{split}\]
and hence by applying \cref{stabilization_iso_general} and comparing terms we obtain
\[ v_{2(g+m)+2}=e_0^2 v_{2(g+m)} = \mathfrak{J}^2(v_{2(g+m)}). \]
Hence we obtain surjectivity of $\mathfrak{J}^2$, and by extension surjectivity of $\mathfrak{J}_{2(g+m)+1}$.

The proof that $\mathfrak{J}_{2(g+m)}$ is surjective follows by generalising the argument from the case when $j=2g,2g+1$ in exactly the same way.  Note that therefore
\[ i_*: V^{(0)}_{2(g+m+1)} \to V^{(m)}_{2(g+m+1)} \]
is an isomorphism, which is required for the inductive step.
\end{proof}

\section{Invariance results for full ECK}\label{invariance_results_for_full_ECK}

In this section we are finally able to deduce $\alpha$ and $J$ invariance of the full ECK complex, at least for integral open book decompositions, by combining the results of \cref{invariance_results_without_an_action_bound,ECK_hat_supported_in_low_Alexander_degrees}.

We will briefly recap the chain of implications which have allowed us to reach this moment, since it is a little convoluted.
\begin{enumerate}
\item We first proved $\alpha$ and $J$ invariance of $\widehat{ECK}$ in \cref{ECC_independent_of_alpha}, exploiting the fact that the complex, although infinite, splits as a direct sum of finite complexes.
\item This invariance was used in the proof of \cref{ECK_supported_genus}, when taking a sequence of homotopic contact structures limiting towards a stable Hamiltonian structure.
\item This is turn allows us to apply \cref{full_ECK_computable_from_ECK_hat}, which states that, for an integral  open book decomposition, the full ECK complex is computable by considering only the subcomplex lying in the zeroth column and in Alexander gradings less than or equal to $2g$.
\end{enumerate}

\begin{thm}\label{ECK_invariant_ZOBD}
In the case of an integral open book decomposition, embedded contact knot homology, defined as the bi-filtered homotopy type of the complex
\[ ECC\orbits{0he}{eh}(\mathrm{int}(N),\alpha,J), \]
is invariant under homotopies of $\alpha$ and choice of $J$.
\end{thm}
\begin{proof}
By \cref{full_ECK_computable_from_ECK_hat}, it suffices to prove the result for the (finite) complex
\[ ECC\orbits{0he}{h}(\mathrm{int}(N),\alpha,J;A\le 2g). \]
Let $(\alpha_0,J_0)$ and $(\alpha_1,J_1)$ be two choices of data, with $\alpha_0$ and $\alpha_1$ homotopic.  Then by an argument completely analogous to that in the proof of \cref{ECC_independent_of_alpha} we obtain bi-filtered chain homotopy equivalences between the corresponding complexes.

In fact, the proof follows by simply replacing every occurrence of ``$j$'' in the proof of \cref{ECC_independent_of_alpha} with ``$A\le 2g$''---since we still have finiteness all the arguments apply as before.
\end{proof}

\chapter{A surgery formula}\label{a_surgery_formula_chapter}

\section{Large negative \texorpdfstring{$n$}{n}-surgery}\label{large_negative_n_surgery}

In the following section we will start with an integral open book decomposition, assuming that $N$ is a mapping torus formed by via a monodromy $\phi:\Sigma\to\Sigma$ which restricts to the identity map on $\del\Sigma$.  Recall from \cref{contact_forms_and_rational_open_book_decompositions} the curves $m$ and $l$; $m$ is the curve $\set{x} \cross [0,1] \in \del\Sigma\cross[0,1]$ oriented in the increasing $t$ direction and $l$ is the curve $\del\Sigma\cross\set{0}$ oriented in the increasing $\theta$ direction.  Recall that the closed manifold $M$ is obtained by attaching a copy of $S^1\cross D^2$ in a way which identifies the meridian $\set{1}\cross S^1$ with the degeneracy slope in $\del N$, which in the case of an integral open book decomposition is equal to $m$. The knot $K\in M$ is then defined as the core $S^1\cross\set{0}$ of $S^1\cross D^2$.

Our method of computing $n$-surgery on $K$ is as follows.  We will attach a collar region $S^1\cross[1,2]$ to $\del\Sigma$ to form a new surface 
\[ \Sigma_T := \Sigma\union \big(S^1\cross[1,2]\big), \]
with associated mapping torus denoted by $N\union T$ (here $T$ denotes the $S^1\cross[1,2]\cross[0,1]/\sim$ part of the mapping torus).  We equip $N\union T$ with a contact form $\alpha$, which is extended from the contact form on $N$ in such a way that the Reeb vector field rotates negatively in the $[1,2]$-coordinate, from vertical (at $\del N$) to slope $-n$ (at $\del (N\union T)$).  This means that $T$ is foliated by \emph{negative} Morse-Bott tori $S^1\cross\set{y}\cross[0,1]$.  Under a simple change of coordinates at the new boundary we see that the new mapping torus has a negative $1/n$ twist at the boundary and Reeb vector field parallel to the $t$ coordinate, and hence is extendable to a $(-1/n)$-rational open book as in \cref{extendable_contact_form}.  We will refer to $T$ as the \emph{twist region}.

The degeneracy slope of $(N\union T, \alpha)$, which recall is the curve obtained by flowing along $\del_t$ on the boundary, is
\[ n[m] - [l] \]
and when we attach $S^1\cross D^2$ by identifying a meridian $\set{1}\cross S^1$ to the degeneracy slope, we obtain ($-n$)-surgery.  Denote the resulting closed manifold by $M(-n)$ and the core of the attached $S^1\cross D^2$ by $K(-n)$.

\begin{absolutelynopagebreak} %
The aim of the remainder of this thesis is to compute the embedded knot contact homology of $K(-n)$ in terms of that of $K$.
\end{absolutelynopagebreak} %

\subsection{A contact form on \texorpdfstring{$N\union T$}{N∪T}}

Assume that we have a contact form $\alpha_N$ on $N$ which is extendable to an \emph{integral} open book decomposition.  Then there exists a neighbourhood of $\del N$ described by coordinates in $S^1\cross[2,2+\epsilon)\cross [0,1] / \sim$ on which the contact form $\alpha_N$ is a small non-degenerate perturbation of the form
\[ f(y)\dd t + g(y)\dd\theta , \]
where, by setting $q=c=1$ and $m=0$ in \cref{contact_form_on_nu_del_N},
\begin{equation}\label{form_near_del_sigma}
f(y) = a_2+(y-2)^2\quad\text{and}\quad g(y) = b_2-(y-2),
\end{equation}
for some $a_2,b_2\in\R$ with $a_2>0$.  Recall that $a_2$ can be taken arbitrarily large.  Furthermore, the equivalence relation $\sim$ is given on this small neighbourhood by the identity map.  

The collar $S^1\cross[1,2]$ is attached to $\Sigma$ by naturally identifying the points where $y=2$.  Now define a one-form on the twist region $T=S^1\cross[1,2]\cross[0,1]/\sim$ such that the following points hold.
\begin{enumerate}
\item The form looks like $f(y)\dd t + g(y)\dd \theta$ on the whole region and the path $(f,g)$ winds around the origin clockwise, so that it is a valid contact form.
\item On a small neighbourhood of $y=2$, the functions $f$ and $g$ are given by the same formula as in \cref{form_near_del_sigma}, so that the form is a smooth extension of $\alpha_N$ on $N$.
\item\label{functions_near_y_1} On a small neighbourhood of $y=1$, the functions take the form
\begin{align*}
f(y) &= f(1) - (y-1),\quad\text{and} \\
g(y) &= g(1) - n(y-1) - (y-1)^2.
\end{align*}
\item The ratio $f'/g'$ is strictly decreasing on the interval $[1,2]$ from $1/n$ (at $y=1$) to $0$ (at $y=2$).
\end{enumerate}
The final point ensures that the slope of the Reeb vector field, which is parallel to
\[-g'(y)\dd t + f'(y) \dd \theta,\]
 rotates gradually from $\infty$ to $-n$ as $y$ decreases from 2 to 1.  Hence at every $y\in[1,2]$ we obtain a negative Morse-Bott torus of slope $-g'(y)/f'(y)$. Provided that $a>0$ is large enough, it is possible to choose $f$ and $g$ so that all these requirements are met.  See \cref{MB_form_on_N_T}. 

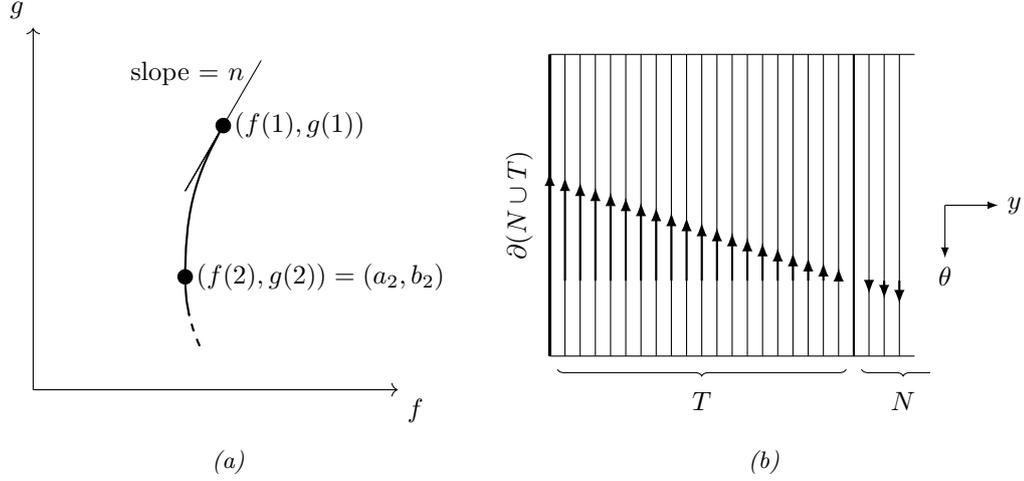
\begin{figure}\centering
	\begin{subfigure}[b]{0.45\textwidth}\centering
		\begin{tikzpicture}%
			\draw [->] (0,0) -- (0,4.8) node [above left] {$g$};
			\draw [->] (0,0) -- (4.8,0) node [below right] {$f$};
			\draw [thick,*-*,shorten >=-3pt,shorten <=-3pt](2,1.5) to [out=90,in=-120] (2.5,3.5);
			\draw [thin] (2.5,3.5) -- +(-120:1) -- +(60:1);
			\node [left] at (2.9,4.2) {slope = $n$};
			\node [right,inner sep=4pt] at (2,1.5) {$(f(2),g(2))=(a_2,b_2)$};
			\node [right,inner sep=4pt] at (2.5,3.5) {$(f(1),g(1))$};
			\begin{scope}
			  \path [clip] (2,1) rectangle (3,1.5);
				\draw [thick] (2,1.5) to[out=-90,in=120] (2.5,0);
			\end{scope}
			\begin{scope}
			  \path [clip] (2,0.5) rectangle (3,1);
				\draw [thick,dashed] (2,1.5) to[out=-90,in=120] (2.5,0);
			\end{scope}
		\end{tikzpicture}
		\caption{}
	\end{subfigure}
	\begin{subfigure}[b]{0.54\textwidth}\centering
		\begin{tikzpicture} %
			\draw  (0,0) --(4.8,0);
			\draw (0,4) --(4.8,4);
			\draw [very thick] (0,0)--(0,4); %
			\node [rotate=90,anchor=center] at (-0.4,2) {$\del(N\union T)$};

			\foreach \x in {0.2, 0.4,...,4.6} {
				 \draw [thin](\x,0)--(\x,4);
			}
			\draw [thick](4,0)--(4,4);
			\foreach \x in {-4,-3.8,...,-0.2,0.2,0.4,0.6}
				\draw [-latex,thick,shorten >=-3pt] (4+\x,1) -- +(0,-\x/3);

			\draw [decorate,decoration={brace,mirror,raise=5pt}] (0.1,0) --node[below=10pt] {$T$} (3.9,0) ;
			\begin{scope}
			\path [clip] (2,-1) rectangle (5,0);
			\draw [decorate,decoration={brace,mirror,raise=5pt}] (4.1,0) --node[below=10pt] {$N$} (5.2,0) ;
			\end{scope}
			\figureaxes{5.2}{2}{y}{}{}{\theta}
		\end{tikzpicture}	
		\caption{}
		\label{MB_form_on_N_T_b}
	\end{subfigure}
  \caption{In (a) we see the curve $(f,g)$ which defines the extension of $\alpha_N$ to $T$.  The curve has infinite slope at $y=2$ and slope $n$ at $y=1$.  In (b) we see a diagram representing a single page of the mapping torus $N\union T$; the flow represents the first return map of the Reeb vector field.}
	\label{MB_form_on_N_T}
\end{figure}

Denote the resulting form defined on the entirety of $N\union T$ by $\alpha$.  We will now check that point \cref{functions_near_y_1} ensures that $\alpha$ is extendable to a $(-1/n)$-rational open book.  We will verify this by performing a change of coordinates near $\del (N\union T)$; let
\[ (t,\theta',y') \]
be a new coordinate system for $\nu(\del N)$, where $ \theta' = \theta+\frac{t}{n}$ and $y'=y+1$.
Then near $y=1$,
\[\begin{split}
\alpha &= f(y)\dd t + g(y)\dd \theta \\
&= \bigg(f(1) - (y-1)\bigg)\dd t + \bigg(  g(1) - n(y-1) - (y-1)^2 \bigg)\dd \theta\\
&= \bigg(f(1) - (y'-2)\bigg)\dd t + \bigg(  g(1) - n(y'-2) - (y'-2)^2 \bigg)\bigg(d\theta' - \frac{1}{n}\dd t\bigg) \\
&= \bigg(f(1)-\frac{g(1)}{n} + \frac{1}{n}(y'-2)^2\bigg)\dd t + \bigg(  g(1) - n(y'-2) - (y'-2)^2 \bigg)\dd\theta'  
\end{split}\]
as required (c.f.~\cref{extendable_contact_form,contact_form_on_nu_del_N}).  Furthermore, in these new coordinates, the equivalence relation is given by
\[ (1,\theta', y') \sim (0,\theta' - \frac{1}{n}, y'), \]
and hence the monodromy at the boundary is a negative $1/n$ twist.

\subsection{Structure of proof}

In this section we will make some perturbations to $\alpha$ for non-degeneracy and computability purposes, and follow this with an outline of how these perturbed forms will be used to establish the surgery formula.  First note that $\alpha$ is a Morse-Bott contact form, with foliated tori at every $y\in[1,2]$ and Morse-Bott families of rational slope on a dense subset.
\begin{defn}\label{perturbed_contact_forms_on_N_union_T}
Define the following contact forms on $N\union T$.
\begin{enumerate}
\item Let $\alpha'$ denote the form obtained by making $\alpha$ completely non-degenerate on $\mathrm{int}(N\union T)$, leaving a single Morse-Bott family of orbits at the boundary ($y=1$).
\item For each $j\in\N$, let $\alpha_j$ denote the form obtained from $\alpha$ by perturbing every Morse-Bott family in $\mathrm{int}(T)$ of slope $p[l]+q[m]$ with $q\leq j$ into two non-degenerate orbits $e_{p/q}$ and $h_{p/q}$.  For each $j$ we leave a small neighbourhood of $\del N$ untouched, so that the Morse-Bott torus at $y=2$ remains.
\item For each $j\in\N$, let $\alpha_j'$ denote the form obtained by making a small perturbation of $\alpha_j$ such that the Morse-Bott torus at $y=2$ is perturbed into two orbits $e_{0/1}$ and $h_{0/1}$.
\end{enumerate}
See \cref{perturbed_forms_on_N_T}.
\end{defn}

\begin{figure}\centering
	\begin{subfigure}{0.49\textwidth}\centering
		\begin{tikzpicture} %
		\draw  (0,0) --(4.8,0);
		\draw (0,4) --(4.8,4);
		\draw [very thick] (0,0)--(0,4); %
		\node [rotate=90,anchor=center] at (-0.4,2) {$\del(N\union T)$};

		\foreach \x in {-4,-3.8,-3.6,-2.4,-2.2,-1.2,-1,-0.8,-0.6,-0.4,-0.2,0.2,0.4,0.6} {
			 \draw [thin](4+\x,0)--(4+\x,4);
		}
		\draw [thick](4,0)--(4,4);
		\foreach \x in {-4,-3.8,-3.6,-2.4,-2.2,-1.2,-1,-0.8,-0.6,-0.4,-0.2,0.2,0.4,0.6}
			\draw [-latex,thick,shorten >=-3pt] (4+\x,1) -- +(0,-\x/3);
		\foreach \x in {0.6,1.4,2,2.6} 
			\draw [dashed] (\x,0)--(\x,4);

		\foreach \y in {0,1,2,3} {
			\draw [fill] (1,\y+0.25) circle (0.05);
			\draw (1,\y+0.75) circle (0.05);
			\draw [-latex,shorten >=-2.5pt] (0.8,\y+0.75) -- +(90:0.01);
			\draw [-latex,shorten >=-2.5pt] (1.2,\y+0.75) -- +(-90:0.01);
		}
		\foreach \y in {0,1,2,3,4} {
			\draw [fill] (2.3,\y*4/5+0.2) circle (0.05);
			\draw (2.3,\y*4/5+0.6) circle (0.05);
			\draw [-latex,shorten >=-2.5pt] (2.15,\y*4/5+0.6) -- +(90:0.01);
			\draw [-latex,shorten >=-2.5pt] (2.45,\y*4/5+0.6) -- +(-90:0.01);
		}

		\begin{scope}
			\path[clip] (0,0) rectangle (4,4);

			\draw [thin] (0.8,-0.25) to[out=90,in=-90] (1.2,0.75) to[out=90,in=-90] (0.8,1.75) to[out=90,in=-90] (1.2,2.75) to[out=90,in=-90] (0.8,3.75) to[out=90,in=-90] (1.2,4.75);
			\draw [thin] (1.2,-0.25) to[out=90,in=-90] (0.8,0.75) to[out=90,in=-90] (1.2,1.75) to[out=90,in=-90] (0.8,2.75) to[out=90,in=-90] (1.2,3.75) to[out=90,in=-90] (0.8,4.75);	

			\draw [thin] (2.15,-0.2) to[out=90,in=-90] (2.45,0.6) to[out=90,in=-90](2.15,1.4) to[out=90,in=-90](2.45,2.2) to[out=90,in=-90](2.15,3) to[out=90,in=-90](2.45,3.8) to[out=90,in=-90] (2.15,4.6);
			\draw [thin] (2.45,-0.2) to[out=90,in=-90] (2.15,0.6) to[out=90,in=-90](2.45,1.4) to[out=90,in=-90](2.15,2.2) to[out=90,in=-90](2.45,3) to[out=90,in=-90](2.15,3.8) to[out=90,in=-90] (2.45,4.6);
		\end{scope}

		\draw [decorate,decoration={brace,mirror,raise=5pt}] (0.1,0) --node[below=10pt] {$T$} (3.9,0) ;
		\begin{scope}
		\path [clip] (2,-1) rectangle (5,0);
		\draw [decorate,decoration={brace,mirror,raise=5pt}] (4.1,0) --node[below=10pt] {$N$} (5.2,0) ;
		\end{scope}
	\end{tikzpicture}
		\caption{}
	\end{subfigure}
	\begin{subfigure}{0.49\textwidth}\centering
		\begin{tikzpicture} %
		\draw  (0,0) --(4.8,0);
		\draw (0,4) --(4.8,4);
		\draw [very thick] (0,0)--(0,4); %
		\node [rotate=90,anchor=center] at (-0.4,2) {$\del(N\union T)$};

		\foreach \x in {-4,-3.8,-3.6,-2.4,-2.2,-1.2,-1,-0.8} {
			 \draw [thin](4+\x,0)--(4+\x,4);
		}
		\foreach \x in {-4,-3.8,-3.6,-2.4,-2.2,-1.2,-1,-0.8,-0.6,-0.4,-0.2,0.2,0.4,0.6}
			\draw [-latex,thick,shorten >=-3pt] (4+\x,1) -- +(0,-\x/3);
		\foreach \x in {0.6,1.4,2,2.6} 
			\draw [dashed] (\x,0)--(\x,4);

		\foreach \y in {0,1,2,3} {
			\draw [fill] (1,\y+0.25) circle (0.05);
			\draw (1,\y+0.75) circle (0.05);
			\draw [-latex,shorten >=-2.5pt] (0.8,\y+0.75) -- +(90:0.01);
			\draw [-latex,shorten >=-2.5pt] (1.2,\y+0.75) -- +(-90:0.01);
		}
		\foreach \y in {0,1,2,3,4} {
			\draw [fill] (2.3,\y*4/5+0.2) circle (0.05);
			\draw (2.3,\y*4/5+0.6) circle (0.05);
			\draw [-latex,shorten >=-2.5pt] (2.15,\y*4/5+0.6) -- +(90:0.01);
			\draw [-latex,shorten >=-2.5pt] (2.45,\y*4/5+0.6) -- +(-90:0.01);
		}

		\begin{scope}
			\path[clip] (0,0) rectangle (5,4);

			\draw [thin] (0.8,-0.25) to[out=90,in=-90] (1.2,0.75) to[out=90,in=-90] (0.8,1.75) to[out=90,in=-90] (1.2,2.75) to[out=90,in=-90] (0.8,3.75) to[out=90,in=-90] (1.2,4.75);
			\draw [thin] (1.2,-0.25) to[out=90,in=-90] (0.8,0.75) to[out=90,in=-90] (1.2,1.75) to[out=90,in=-90] (0.8,2.75) to[out=90,in=-90] (1.2,3.75) to[out=90,in=-90] (0.8,4.75);	

			\draw [thin] (2.15,-0.2) to[out=90,in=-90] (2.45,0.6) to[out=90,in=-90](2.15,1.4) to[out=90,in=-90](2.45,2.2) to[out=90,in=-90](2.15,3) to[out=90,in=-90](2.45,3.8) to[out=90,in=-90] (2.15,4.6);
			\draw [thin] (2.45,-0.2) to[out=90,in=-90] (2.15,0.6) to[out=90,in=-90](2.45,1.4) to[out=90,in=-90](2.15,2.2) to[out=90,in=-90](2.45,3) to[out=90,in=-90](2.15,3.8) to[out=90,in=-90] (2.45,4.6);

			\draw [thin] (3.5,-1) to[out=90,in=-90] (3.4,1) to[out=90,in=-90] (3.5,3) to[out=90,in=-90] (3.4,5);
			\draw [thin] (4.5,-1) to[out=90,in=-90] (4.6,1) to[out=90,in=-90] (4.5,3) to[out=90,in=-90] (4.6,5);
			\foreach \y in {0,4} {
				\draw [thin] (4,\y-1) to [out=110, in=-90] (3.6,\y+1) to [out=90,in=-110] (4,\y+3);
				\draw [thin] (4,\y-1) to [out=70, in=-90] (4.4,\y+1) to [out=90,in=-70] (4,\y+3);
			}
			\draw [thin] (4,1) ellipse (0.2 and 0.8);
			\draw [fill] (4,3) circle (0.05);
			\draw (4,1) circle (0.05);
			
		\end{scope}

		\draw [decorate,decoration={brace,mirror,raise=5pt}] (0.1,0) --node[below=10pt] {$T$} (3.9,0) ;
		\begin{scope}
		\path [clip] (2,-1) rectangle (5,0);
		\draw [decorate,decoration={brace,mirror,raise=5pt}] (4.1,0) --node[below=10pt] {$N$} (5.2,0) ;
		\end{scope}
	\end{tikzpicture}
		\caption{}
	\end{subfigure}
  \caption{The figures show the return map of the Reeb vector fields corresponding to $\alpha_j$ and $\alpha_j'$ respectively in the case $n=3$ and $j=5$.  The form from \cref{MB_form_on_N_T_b} has been perturbed in a neighbourhood of the Morse-Bott torus of slope $4[m]-[l]$ to produce non-degenerate orbits $h_{1/4}$ and $e_{1/4}$, both of which are seen intersecting the page 4 times.  Similarly the torus of slope $5[m]-l$ has been perturbed to produce orbits $h_{1/5}$ and $e_{1/5}$.  In (a) the Morse-Bott torus at $\del N$, where $y=2$, has been left unperturbed whereas in (b) it has also been perturbed to yield the orbits $h_{0/1}$ and $e_{0/1}$.}
	\label{perturbed_forms_on_N_T}
\end{figure}
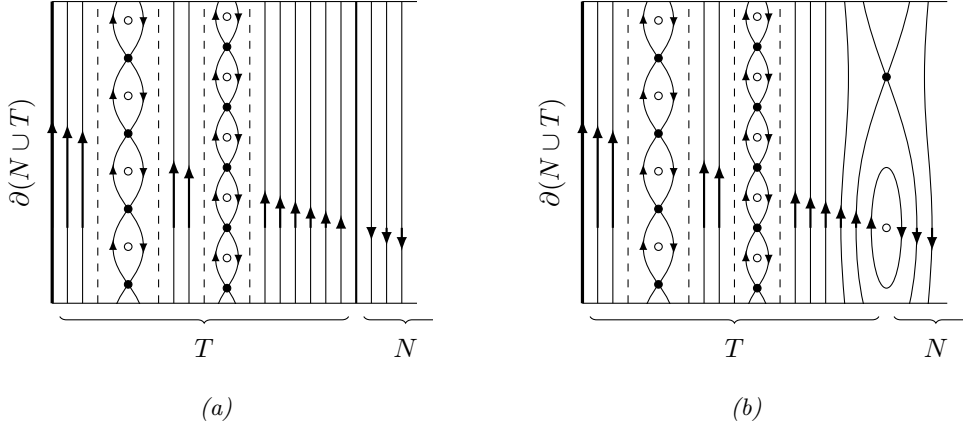

All the forms defined above are homotopic and we will now outline briefly how they will be used to compute the surgery formula.

The form $\alpha'$, as the only form which is completely non-degenerate on the interior of $N\union T$, is used simply for formal purposes, and will be used in statements of theorems summarising the surgery formula.  In other words, the result of this thesis is to compute
\[ \widehat{ECK}(K,\alpha') := H_*\left( ECC\orbits{he}{h}(\mathrm{int}(N\union T),\alpha')\right). \]

\begin{convention}
Here, once again, we have dropped the choice of (regular, adapted) almost complex structure from the notation, and we will continue to make this omission for the remainder of the chapter.
\end{convention}

When it comes to computing the groups above, however, we will actually use the forms $\alpha'_j$ to calculate ECK in each Alexander grading in turn.  This is legitimate since $\alpha_j'$ is non-degenerate in Alexander gradings $j$ and below and
\[ \widehat{ECK}(K,\alpha'_j;j) \iso \widehat{ECK}(K,\alpha';j) \]
for each $j>0$ by \cref{ECC_independent_of_alpha}.

The reason we move away from the form $\alpha'$ is that we do not have an explicit understanding of the Reeb orbits of $\alpha'$ in the twist region $T$.  However, for $\alpha_j'$ we do have an explicit description, namely the orbits $e_{p/q}$ and $h_{p/q}$ for every $p/q\in[0,1/n]$ with $q\le j$.  Furthermore, it is also possible to understand the behaviour of differentials contained within $T$ by applying results of Hutchings and Sullivan~\cite{HS05}, who use such perturbations to explicitly compute the periodic Floer homology of a Dehn twist.

Finally, the need for the forms $\alpha_j$ is so that we can apply Morse-Bott techniques, namely positivity of intersection and the Blocking Lemma, to help us compute the groups $\widehat{ECK}(K,\alpha'_j;j)$.  We have a slight problem since $\alpha_j$ has a Morse-Bott family of orbits at $y=2$ in the interior of $N\union T$ and therefore is not necessarily nice (recall \cref{niceness_defn}), and hence the complexes
\[ ECC\orbits{he}{h}(\mathrm{int}(N\union T),\alpha'_j; j)\quad\text{and}\quad ECC\orbits{he}{h}(\mathrm{int}(N\union T),\alpha_j; j) \]
do not necessarily agree.

We circumvent this problem by constructing filtrations on both these complexes such that $\alpha_j$ is nice with respect to the filtration-preserving part of the differential.  Then the differentials on the zero-th pages agree and we can therefore apply Morse-Bott techniques to $\alpha_j$ to allow us to compute the first page of the spectral sequence corresponding to $\alpha'_j$.

The first page has a ``two tower'' structure as detailed in \cref{computing_the_first_page_of_F} and illustrated in \cref{two_towers_example}.  In Alexander gradings 0 through $n-1$, the first page is supported entirely in the lowest filtration level, and the spectral sequence convergences at this point; this yields the groups $H_*(A_j)$ as seen in the statement of the surgery formula (\cref{surgery_formula}).

In Alexander gradings greater than or equal to $n$, the first page is supported in filtration levels 0 and 1; in \cref{the_differential_on_the_first_page} we use an algebraic argument and employ induction to compute the second page of the spectral sequence.  This computation yields the groups $H_*(B_i)$ as seen in the surgery formula.

\section{A filtration on the surgery complex}\label{a_filtration_on_the_surgery_complex}

For the remainder of the chapter we will always work with a fixed $j\ge 0$.  The complexes
\[ ECC\orbits{he}{h}(\mathrm{int}(N\union T), \alpha'_j;j)\quad\text{and}\quad ECC\orbits{he}{h}(\mathrm{int}(N\union T), \alpha_j;j) \]
are canonically isomorphic as vector spaces but their differentials, denoted by $d'$ and $d$ respectively, are not necessarily equal, since $\alpha_j$ is not necessarily a nice Morse-Bott contact form.

Define a descending filtration $\mathcal{F}$ on these chain complexes as follows.  Write any orbit set as
\[ \gamma\tensor\Gamma, \]
where $\gamma$ is constructed only from the orbits $e_-$, $h_\pm$ and orbits in the twist region $T$ and $\Gamma$ is constructed from orbits in $\mathrm{int}(N)$.  Every $\gamma$ has homology class equal to $q[m] - p[l]$ for some integers $p,q\ge 0$.  Here $m$ and $l$ are the meridian and longitude curves associated to $N$ as discussed at the start of \cref{large_negative_n_surgery}.  For example, the homology class of the orbits at $\del N$ ($y=2$) is $[m]$ and the homology class of the orbits at $\del (N\union T)$ ($y=1$) is $n[m]-[l]$. Define the filtration $\mathcal{F}$ on the two chain complexes above by
\[ \mathcal{F}(\gamma\tensor\Gamma) := p. \]
In particular,
\begin{align*}
&\mathcal{F}(h_{0/1}) = \mathcal{F}(e_{0/1}) = 0,\quad\text{and}\\
&\mathcal{F}(h_\pm) = \mathcal{F}(e_-) = 1.
\end{align*}

\begin{prop}
$\mathcal{F}$ defines a descending filtration on both complexes
\[ ECC\orbits{he}{h}(\mathrm{int}(N\union T), \alpha'_j;j)\quad\text{and}\quad ECC\orbits{he}{h}(\mathrm{int}(N\union T), \alpha_j;j). \]
\end{prop}
\begin{proof}
The proof proceeds in manner similar to the argument in \cref{a_sequence_of_filtrations_on_the_ECH_complexes} and the proof of \cref{cobordism_respects_eta_filtration}.  Let $J$ (resp.\ $J'$) denote the regular adapted almost complex structure for $\alpha_j$ (resp.\ $\alpha_j'$).

First note that in the case of the complex on the left we can apply \cref{MB_building_exists_epsilon_small_enough} to any $J'$-holomorphic curve between two orbit sets to obtain a nearby $J$-holomorphic Morse-Bott building.  Thus it suffices to prove the proposition for the complex on the right.

Suppose that there is a $J$-holomorphic Morse-Bott building $u$ between $\gamma_+\tensor\Gamma_+$ and $\gamma_-\tensor\Gamma_-$, and apply slope calculus (\cref{slope_calculus}) to the torus $T_\rho:=T^2\cross\set{\rho}$ for some $\rho\in(1,2)$.  Provided $\rho$ is large enough, the following is true:
\begin{itemize}
\item all orbits comprising $\gamma_+$ and $\gamma_-$ lie to one side $T_\rho$, i.e.~in the region parametrized by $1\le y<\rho$, and
\item $T_\rho$ is foliated by Reeb trajectories of slope $s=[m]-\epsilon_\rho[l]$ for some $\epsilon_\rho>0$ such that $\epsilon_\rho\to0$ as $\rho\to 2$.
\end{itemize}
Then, by slope calculus,
\[ [u_{T_{\rho}}] = [\gamma_-]-[\gamma_+] = k[m] - \big(\mathcal{F}(\gamma_-) - \mathcal{F}(\gamma_+)\big)[l] \]
for some $k\in\Z$, and hence, by positivity of intersection (\cref{positivity_of_intersections_3}),
\[\begin{split}
0<[u_{T_{\rho}}]\cdot s &= \big(k[m] - \left(\mathcal{F}(\gamma_-) - \mathcal{F}(\gamma_+)\right)[l]\big)\cdot ([m]-\epsilon_\rho[l]) \\
&= \left(\mathcal{F}(\gamma_-) - \mathcal{F}(\gamma_+)\right) -k\epsilon_\rho. 
\end{split}\]
Taking $\rho$ close to $2$ and hence $\epsilon_\rho$ sufficiently small forces $\mathcal{F}(\gamma_-)\ge\mathcal{F}(\gamma_+)$ as required.
\end{proof}

We shall call this filtration the \emph{twist filtration} and denote by $d_0$ the part of $d$ which does not decrease the filtration level.  Then apply slope calculus again to see that any Morse-Bott building contributing to $d_0$ must intersect the Morse-Bott torus at $\del N$ ($y=2$) with infinite slope.  Therefore the \refNamedThm{Blocking Lemma}{blocking_lemma} implies that every differential is one-sided at $y=2$ and, as a result, $\alpha_j$ is nice with respect to $d_0$ by \cref{one-sided_implies_nice} (recall \cref{niceness_defn}).  Hence \cref{non-degen_close_to_MB} implies the following:

\begin{thm}\label{alpha_j_and_alpha_j_prime_iso_thm}
Fix $j\ge 0$.  Provided that the perturbation in \cref{perturbed_contact_forms_on_N_union_T} is small enough, the complexes
\[ \left(ECC\orbits{he}{h}(\mathrm{int}(N\union T), \alpha'_j;j), d'_0\right) \quad\text{and}\quad 
   \left(ECC\orbits{he}{h}(\mathrm{int}(N\union T), \alpha_j;j), d_0\right) \]
are isomorphic for some appropriate choices of regular adapted almost complex structures. Here $d'_0$ denotes the filtration preserving part of $d'$.
\end{thm}

This theorem allows us to use Morse-Bott techniques to compute the first page of the spectral sequence corresponding to $d'$.  More precisely, we can state the following facts about the complex corresponding to $\alpha'_j$:

\begin{prop}\label{facts_about_differentials_on_zeroth_page}~
\begin{enumerate}
\item\label{all_curves_in_N_or_T} All holomorphic curves contributing towards $d'_0$ have non-connector parts contained either entirely in $\mathrm{int}(T)$ or entirely in $\mathrm{int}(N)$.
\item\label{only_negative_ends_at_del_N} All holomorphic curves contributing towards $d'_0$ can only have negative ends at the orbits $e_{0/1}$ and $h_{0/1}$.
\item\label{complex_on_N_part} The complex obtained by just considering orbits and differentials in $N$ is isomorphic to the Morse-Bott complex
\[ ECC\orbits{e}{h}(\mathrm{int}(N), \alpha_N). \]
\end{enumerate}
\end{prop}
\begin{proof}
Part \cref{all_curves_in_N_or_T} follows from the \refNamedThm{Blocking Lemma}{blocking_lemma} applied to the torus $\del N$, and part \cref{only_negative_ends_at_del_N} follows from the \refNamedThm{Trapping Lemma}{trapping_lemma}.  Part \cref{complex_on_N_part} follows since $\restr{\alpha_j}{N}=\alpha_N$ so we can assume that the almost complex structure for $\alpha_j$ is equal, on $\R\cross N$, to the almost complex structure for $\alpha_N$, and hence the moduli spaces of Morse-Bott buildings are equal.
\end{proof}
The next step is to compute the first page of the spectral sequence corresponding to $\mathcal{F}$ by considering further filtrations on the zeroth page.  Write
\[  (E^0(\mathcal{F}), d'_0) := \left(ECC\orbits{he}{h}(\mathrm{int}(N\union T), \alpha'_j;j), d'_0\right) \]
and
\[ E^0_p(\mathcal{F}) \]
for the part of $E^0(\mathcal{F})$ lying in $\mathcal{F}$-grading $p\ge 0$.

The idea is to define a filtration $\mathcal{G}$ on the zeroth page $E^0(\mathcal{F})$ whose spectral sequence converges to the first page $E^1(\mathcal{F})$.  In fact, we will treat the case $p=0$ slightly differently since it is simpler.

\subsubsection{The case when \texorpdfstring{$p=0$}{p=0}}\label{the_case_when_p_equals_0}

The complex $E_0^0(\mathcal{F})$ is much easier to understand, since the constraint that $\mathcal{F}=0$ means we never see any orbits in the twist region $T$ apart from $e_{0/1}$ and $h_{0/1}$.

That is,
\begin{equation}\label{zeroth_page_no_twist_isomorphism}
E_0^0(\mathcal{F}) \iso ECC\orbits{e}{h}(\mathrm{int}(N),\alpha_N; j),
\end{equation}
where here we are using \refListInThm{facts_about_differentials_on_zeroth_page}{complex_on_N_part} to replace $\alpha'_j$ with $\alpha_N$.

We will return to this formula later, but first will consider the case $p>0$.

\subsubsection{A second filtration when \texorpdfstring{$p>0$}{p>0}}\label{a_second_filtration_for_p_greater_0}

When $p>0$ the complex $E^0_p(\mathcal{F})$ is more complex, so we will define a second descending filtration $\mathcal{G}$ to break it down further.

As in the start of \cref{a_filtration_on_the_surgery_complex}, write every orbit set in $E^0_p(\mathcal{F})$ as $\gamma\tensor\Gamma$ where $\gamma$ consists of the orbits in the twist region (including $e_-$, $h_\pm$, $e_{0/1}$ and $h_{0,1}$), and $\Gamma$ consists of the orbits in $\mathrm{int}(N)$.

Define $\mathcal{G}(\gamma\tensor\Gamma)$ to be the Alexander grading of $\gamma$, that is the number of times that $\gamma$ intersects each page $\Sigma_T$.  For example,
\begin{align*}
&\mathcal{G}(e_{p/q})=\mathcal{G}(h_{p/q})=q\quad\text{and}\\
&\mathcal{G}(e_-) = \mathcal{G}(h_\pm) = n.
\end{align*}
By \cref{facts_about_differentials_on_zeroth_page}, the only differentials which can change $\mathcal{G}$ are those in $N$ with one-sided ends at $e_{0/1}$ and $h_{0/1}$.  Furthermore all such ends must be negative by the \refNamedThm{Trapping Lemma}{trapping_lemma} and hence $\mathcal{G}$ defines a valid descending filtration.

The zeroth page of the spectral sequence splits as a tensor product:
\begin{equation}\label{zeroth_page_with_twist_isomorphism}
E^0_p(\mathcal{G}) \iso \sum_{j_T+j_N=j} ECC\orbits{he}{h}(T, \restr{\alpha'_j}{T};j_T,p) \tensor ECC(\mathrm{int}(N),\alpha_N; j_N)
\end{equation}
where on the right we are using \refListInThm{facts_about_differentials_on_zeroth_page}{complex_on_N_part} to replace $\alpha'_j$ with $\alpha_N$.
The complex on the left is defined to be generated by those orbit sets which
\begin{itemize}
\item are constructed from the orbits $h_\pm$, $e_-$ and the orbits in $T$, including $e_{0/1}$ and $h_{0/1}$,
\item have Alexander grading $j_T$, and
\item have $\mathcal{F}$-grading $p$.
\end{itemize}
Thus is order to proceed it is necessary to understand the complexes
\[ ECC\orbits{he}{h}(T, \restr{\alpha'_j}{T};j_T,p)\]
for $j_T\le j$ and $p\in\N$. Note that by \cref{alpha_j_and_alpha_j_prime_iso_thm} and \cref{facts_about_differentials_on_zeroth_page} we have isomorphisms
\[ ECC\orbits{he}{h}(T, \restr{\alpha'_j}{T};j_T,p)\iso ECC\orbits{he}{h}(T, \restr{\alpha_j}{T};j_T,p) \]
for every $j_T$ and $p$.

\section{Orbits and differentials in the twist region}\label{orbits_and_differentials_in_the_twist_region}

In this section we will appeal to the work of Hutchings and Sullivan~\cite{HS05}, who compute the periodic Floer homology of a Dehn twist.  We will use similar techniques to those in \cref{ECK_hat_supported_in_low_Alexander_degrees} to ensure that our contact structure is close enough to the stable Hamiltonian structure used to define periodic Floer homology, and then the results carry over directly.

\subsection{A change of coordinates (again)}

Recall that in \cref{reparametrizing_the_contact_form_on_N} we performed a change of coordinates so that the Reeb vector field was parallel to $\del_t$ everywhere on $N$.  We will perform the same procedure here. Currently we have the following setup: there is a neighbourhood of $T$, which we identify with
\[ S^1\cross[1,2+\epsilon)\cross [0,1]/\sim,\]
on which the contact form $\alpha'_j$ is a small non-degenerate perturbation of the form
\[ f(y)\dd t + g(y)\dd\theta; \]
we will forget the small perturbation for the moment.  This form has Reeb vector field parallel to
\[ f'(y)\del_\theta - g'(y)\del_t \]
and on the remainder of our manifold, the Reeb vector field is already parallel to $\del_t$ (c.f.~\cref{reparametrizing_the_contact_form_on_N}).  We will perform a reparametrization on $S^1\cross[1,2+\epsilon)\cross [0,1]/\sim$, using the same formula as that used earlier, to obtain a coordinate system
\[ (t, \theta', y) \]
where $\theta' = \theta + t\frac{f'(y)}{g'(y)}$.

Then it is easy to check (c.f.~\cref{reparametrizing_the_contact_form_on_N}) that in this coordinate system the Reeb vector field is parallel to $\del_{t}$ everywhere, and the monodromy on the twist region is given by the formula
\[  (y,\theta') \mapsto (y, \theta'-\frac{f'(y)}{g'(y)}), \]
i.e.~a positive Dehn twist which is equal to the identity when $y=2$ and a rotation by $-1/n$ when $y=1$.

The addition of the small perturbation to make $\alpha'_j$ non-degenerate corresponds in this setting to a small perturbation of the monodromy to make all those orbits whose intersection number with $\Sigma_T$ is less than or equal to $j$ non-degenerate.  Denote this perturbed monodromy by $\phi'_j$---this is simply the return map of the Reeb vector field $R_{\alpha'_j}$.

Recall that in \cref{limiting_towards_a_stable_hamiltonian_structure} we took a family of contact forms parametrized by $\zeta$ which approached some stable Hamiltonian structure $C\dd t$ and here we can perform the exact same procedure. As a result we may assume that $\alpha'_j$ is taken close enough to $C\dd t$ (for some large $C>0$) so that we have an isomorphism of complexes
\begin{equation}\label{ECC_is_PFC_in_surgery_case}
 ECC\orbits{he}{h}(\mathrm{int}(N\union T), \alpha'_j; j)\iso PFC\orbits{he}{h}(\mathrm{int}(N\union T), C\dd t; j)
\end{equation}
for some appropriate choices of regular adapted almost complex structures.

In particular, the complexes obtained by restricting only to orbit sets and differentials contained in the twist region are also isomorphic.  At this point it is tempting to write ``the periodic Floer homology of $T$'' but note that some care is required here since the perturbed monodromy $\phi'_j$ does not restrict to a monodromy on the collar region $S^1\cross[1,2]$ as the Morse-Bott torus at $y=2$ has been perturbed into two orbits $e_{0/1}$ and $h_{0/1}$.

The solution is to replace the contact form $\alpha'_j$ with the form $\alpha_j$ which has two Morse-Bott tori at $y=1,2$, and the monodromy $\phi'_j$ with the analogous monodromy $\phi_j$ which is the return map of $R_{\alpha_j}$.  Then $\phi_j$ honestly restricts to the collar region $S^1\cross[1,2]$ and it makes sense to define the periodic Floer complex, which we will denote by
\begin{equation}\label{PFC_twist_region_def} PFC(S^1\cross[1,2], \phi_j; j_T), \end{equation}
for all $j_T\le j$, as the Morse-Bott periodic Floer complex where we include all four boundary elliptic and hyperbolic orbits arising from the Morse-Bott tori at $y=1,2$:
\[ e_{0/1}\text{, }h_{0/1}\text{, }e_{1/n}\text{ and }h_{1/n}. \]
Here we are denoting the orbits $e_-$ and $h_-$ by $e_{1/n}$ and $h_{1/n}$ respectively in accordance with the standard $p/q$-notation in the twist region.  We are also writing $\phi_j$ in \cref{PFC_twist_region_def} rather than the more cumbersome $\restr{\phi_j}{S^1\cross[1,2]}$.

The complex $PFC(S^1\cross[1,2], \phi_j; j_T)$ splits naturally as a direct sum by homology class, which in this case is given by the winding number of the orbit sets in the $S^1$ direction, or equivalently by the $\mathcal{F}$-filtration level.  In analogy to the notation used in \cref{zeroth_page_with_twist_isomorphism} we will write
\[ PFC(S^1\cross[1,2], \phi_j; j_T,p) \]
for the subcomplex generated by orbit sets with homology class
\[ j_T[m] - p[l], \]
that is, with Alexander grading $j_T$ and $\mathcal{F}$-grading $p$.
 
We also define the following augmented complexes by ``including'' the differential from $h_+$ to $e_{1/n}$:
\begin{defn}~
\begin{enumerate}
\item The complex $PFC^{(h_+)}(S^1\cross[1,2], \phi_j; j_T)$ is defined as a vector space by
\[\mathcal{R}[h_+]\tensor PFC(S^1\cross[1,2], \phi_j; j_T)\]
and is equipped with differential
\[ d(\gamma\tensor\Gamma) = \gamma \tensor d\Gamma + (\gamma/h_+) \tensor e_{1/n}\Gamma. \]
\item The complex in part (1) splits as a direct sum according to homology class, or $\mathcal{F}$-grading, and we define the complex $PFC^{(h_+)}(S^1\cross[1,2], \phi_j; j_T,p)$, for each $p\in\N$, to be the subcomplex generated by orbit sets with $\mathcal{F}$-grading $p$.
\end{enumerate}
\end{defn}

Given these definitions we are now able to state the following proposition which follows directly from \cref{ECC_is_PFC_in_surgery_case} and the above discussion.
\begin{prop}\label{ECC_is_PFC_on_twist_region}
We have an isomorphism
\[ ECC\orbits{he}{h}(T, \restr{\alpha_j}{T};j_T,p) \iso PFC^{(h_+)}(S^1\cross[1,2], \phi_j; j_T,p) \]
for all $j_T\le j$ and all $p\in\N$.
\end{prop}

Thus, in order to progress further with the description of $E_p^0(\mathcal{G})$ given in \cref{zeroth_page_with_twist_isomorphism}, it is necessary to compute the periodic Floer homology of the monodromy $\restr{\phi_j}{S^1\cross[1,2]}$.

\subsection{The periodic Floer homology of a Dehn twist}\label{the_periodic_floer_homology_of_a_dehn_twist}

In this section we will briefly review the results of Hutchings and Sullivan~\cite{HS05}.  The setup used is as follows: the authors start with an interval $[X_1, X_2]$, where for simplicity we will assume that $X_1$ and $X_2$ are irrational, and consider the right-handed Dehn twist on the annulus,
\begin{align*}
\phi_{[X_1,X_2]} : [X_1,X_2]\cross S^1 &\to [X_1,X_2]\cross S^1 \\
(x,\vartheta) &\mapsto (x,\vartheta-x),
\end{align*}
where $S^1$ is identified here with $\R/\Z$. The authors fix a homology class
\[ P [\pt] \cross [S^1] \cross [\pt] + Q [\pt] \cross [\pt] \cross [S^1] \]
in the corresponding mapping torus, $[X_1,X_2]\cross S^1\cross[0,1]/\sim$.  This homology class corresponds to the loop which winds around $P$ times in the $\vartheta$ direction and $Q$ times in the $t$ direction.  They then perform a small perturbation to the monodromy away from the boundary which ensures that every orbit set in this homology class is non-degenerate; for every rational number $p/q \in [X_1, X_2]$ with $q\le Q$ the Morse-Bott torus at $x=p/q$ is perturbed into two orbits which we label $e_{p/q}$ and $h_{p/q}$.  Note that this perturbation is locally exactly the same as that which we made in \cref{perturbed_contact_forms_on_N_union_T} when defining $\alpha_j$.

The resulting periodic Floer complex in this homology class is denoted by
\[ CP(X_1,X_2; P,Q) \]
and the generators are easy to understand explicitly; any orbit set can be written as an (ordered) product
\[ \Gamma = \gamma_1\gamma_2\cdots\gamma_k \]
where every $\gamma_i$ is equal to some $e_{p_i/q_i}$ or $h_{p_i/q_i}$ and $p_{i+1}/q_{i+1} \ge p_i/q_i$ for every $i$.  We associate a left-turning polygonal path in the plane by defining points
\[ w_j := \sum_{i=0}^{j} (q_i, p_i) \]
for $j=0,\dots,k$ and drawing line segments between consecutive $w_j$ and $w_{j+1}$ in turn\footnote{Note that we are using a slightly different convention to Hutchings and Sullivan but the constructions are equivalent.}.  Each line segment is given by the vector $(q_i, p_i)$ and has slope $p_i/q_i$, hence the path only turns to the left.  Note that the path starts at the origin and ends at the point $(Q,P)$.  We denote this path by $\mathcal{P}(\Gamma)$. See \cref{example_polygonal_path_dehn_twist} for an example.

\begin{figure}\centering
  \begin{tikzpicture}  %
		\foreach \x in {0,...,11}
			\foreach \y in {0,...,3}
				\fill [gray] (\x,\y) circle (0.05);
		\draw[thick] (0,0) node[below]{$(0,0)$} -- ++(2,0) node[below] {$(2,0)$} -- ++(4,1) node[below right] {$(6,1)$}-- ++(5,2) node[below right] {$(11,3)$};

	\end{tikzpicture}
  \caption{The left turning polygonal path associated to the orbit set $e^2_{0/1}e_{1/4}e_{2/5}$, which is a generator of the complex $CP(-\epsilon, \frac{2}{5}+\epsilon;3,11)$.}
	\label{example_polygonal_path_dehn_twist}
\end{figure}
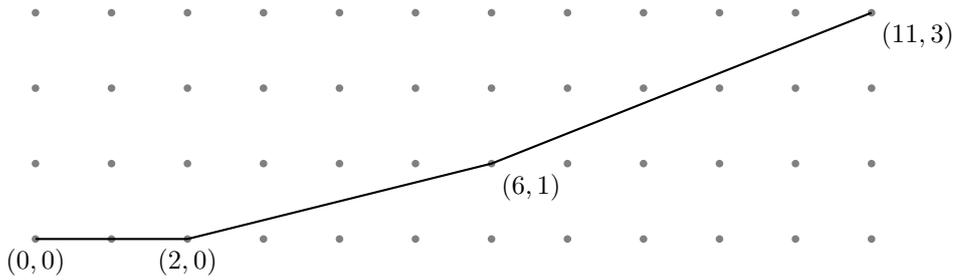

Hutchings and Sullivan use a variety of techniques to show that the differential is essentially\footnote{In fact, it is actually the \emph{dual differential} which is given by the corner rounding procedure. Furthermore, it is also possible that a differential exists in the case of so called ``double rounding'', but this phenomenon will never occur in our situation.} given by ``rounding a corner'' of the polygonal path above.  For a careful introduction to this concept refer to Hutchings and Sullivan~\cite[Definition 3.4]{HS05} but we will give the intuitive idea here.  Imagine that the plane is a board with pegs inserted at every integer lattice point.  Then the polygonal path $\mathcal{P}(\Gamma)$ associated to an orbit set can be imagined as a piece of elastic stretched around the pegs placed at the points $w_j$.  The process of ``rounding a corner'' in this scenario can be thought of as simply removing one of the pegs at a point where the path makes a turn---the elastic will snap onto the convex hull of the pegs immediately to the left of the path.  We also have to ``locally remove an $h$'' during this process.  See \cref{rounding_a_corner_examples} for an example.

The authors then use this explicit combinatorial description of the chain complex to compute the homology. Namely, they obtain the following:

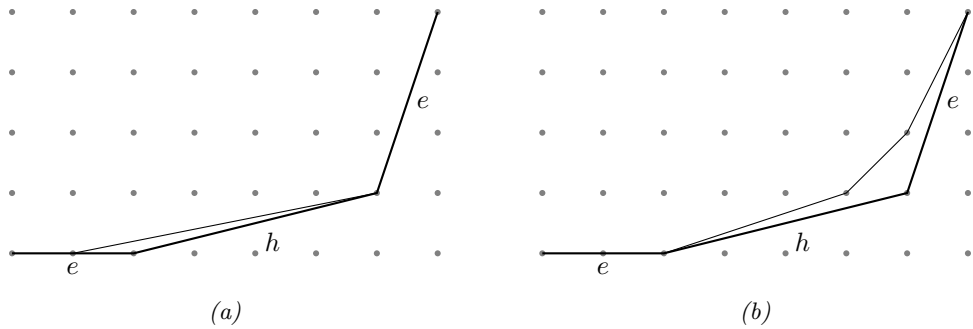
\begin{figure}\centering
	\begin{subfigure}{0.49\textwidth}\centering
		\begin{tikzpicture}[scale=0.8]  %
		\foreach \x in {0,...,7}
			\foreach \y in {0,...,4}
				\fill [gray] (\x,\y) circle (0.05);
		\draw[thick] (0,0) -- ++(2,0) -- ++(4,1) -- ++(1,3) ;
		\draw (1,0) -- ++(5,1);
		\node [below] at (1,0) {$e$};
		\node [below right] at (4,0.5) {$h$};
		\node [right] at (6.5,2.5) {$e$};

	\end{tikzpicture}
		\caption{}
	\end{subfigure}
	\begin{subfigure}{0.49\textwidth}\centering
		\begin{tikzpicture}[scale=0.8]  %
		\foreach \x in {0,...,7}
			\foreach \y in {0,...,4}
				\fill [gray] (\x,\y) circle (0.05);
		\draw[thick] (0,0) -- ++(2,0) -- ++(4,1) -- ++(1,3) ;
		\draw (2,0) -- ++(3,1)-- ++(1,1) -- ++(1,2);
		\node [below] at (1,0) {$e$};
		\node [below right] at (4,0.5) {$h$};
		\node [right] at (6.5,2.5) {$e$};

	\end{tikzpicture}
		\caption{}
	\end{subfigure}
  \caption{The bold paths in the figures above represent the orbit set $\Gamma=e_{0/1}^2h_{1/4}e_{3/1}\in CP(-\epsilon,3+\epsilon;4,7)$.  In (a) we round the corner at $(2,0)$ to obtain the orbit set $\Gamma_a=e_{0/1}e_{1/5}e_{3/1}$.  In (b) we round the corner at $(6,1)$ to obtain $\Gamma_b=e_{0/1}^2e_{1/3}e_{1/1}e_{2/1}$.  Notice that we have also locally removed an $h$ to obtain these orbit sets.  As a result we can say that the coefficients of $\Gamma$ in $d\Gamma_a$ and $d\Gamma_b$ are both 1 and that, since no other roundings are possible, $\Gamma$ does not appear in the differential of any other orbit set.}
	\label{rounding_a_corner_examples}
\end{figure}

\begin{thm}[{\nogapcite[Theorem 3.1]{HS05}}]\label{dehn_twist_homology}
Fix an interval $[X_1,X_2]$ with irrational endpoints and a homology class defined by integers $(P,Q)$ as above.  Then if $P/Q\in[X_1,X_2]$, the periodic Floer homology $HP(X_1,X_2; P,Q)$ is two-dimensional and generated by homology classes denoted by
\[ E(X_1,X_2;P,Q)\quad\text{and}\quad H(X_1,X_2;P,Q). \]
Here $E(X_1,X_2;P,Q)$ is defined as follows: draw a line of slope $X_1$ through the origin and a line of slope $X_2$ through the point $(Q,P)$, and denote their intersection point by $X$.  Imagine a piece of elastic which starts at the origin, ends at $(Q,P)$, and is stretched around a peg placed at the point $X$; let $\mathcal{P}$ denote the convex polygonal path obtained when we release the peg at $X$.  See \cref{generators_of_homology_of_1_n_dehn_twist}.

$E(X_1,X_2;P,Q)$ is a homology class containing the single orbit set which consists only of elliptic orbits and whose associated left-turning polygonal path is equal to $\mathcal{P}$. The homology class $H(X_1,X_2;P,Q)$ consists of all those orbit sets with precisely one hyperbolic orbit and whose associated left-turning polygonal path is also equal to $\mathcal{P}$.

If $P/Q\notin[X_1,X_2]$ then $HP(X_1,X_2; P,Q)=0$.
\end{thm}

A more rigourous definition of the path $\mathcal{P}$ in the above theorem is given by Hutchings and Sullivan in terms of convex hulls~\cite[Section 3.1]{HS05}, but for our purposes the intuitive definition above will suffice.

We are almost ready to apply the results of the above discussion to our surgery set-up, but first we have a small problem which must be addressed.  Namely, recall that in the setup above the Dehn twist is taken on an interval with irrational endpoints---the reason for this is so that the perturbation can be performed entirely on the interior of the resulting mapping torus and the chain complex is constructed solely from non-degenerate orbits in the interior.

In our situation, the boundary is not ignored: the monodromy $\restr{\phi_j}{S^1\cross[1,2]}$ corresponds to a small perturbation of the positive Dehn twist on the interval $[0,1/n]$ under a simple transformation between the coordinates $y\in[1,2]$ and $x\in[0,1/n]$ and, of course, $0$ and $1/n$ are rational. Furthermore the complex 
\[  PFC(S^1\cross[1,2], \phi_j; j_T, p) \]
is defined in a Morse-Bott-theoretic manner, which includes the four boundary orbits $e_{0/1}$, $h_{0/1}$, $e_{1/n}$ and $h_{1/n}$.

Fortunately this slight disparity can be resolved by extending the monodromy $\restr{\phi_j}{S^1\cross[1,2]}$ to a slightly larger interval
\[ S^1\cross[1-\epsilon,2+\epsilon], \]
where $\epsilon>0$ is sufficiently small so that no new orbits of length less than or equal to $j$ are introduced.  Then, under the reparametrization in the $y$ direction, this corresponds to a small perturbation of the positive Dehn twist on some interval:
\[ \phi_{\left[-\epsilon', \frac{1}{n}+\epsilon''\right]}:\left[-\epsilon', \frac{1}{n}+\epsilon''\right]\cross S^1 \to \left[-\epsilon', \frac{1}{n}+\epsilon''\right]\cross S^1 \]
where we can assume that $\epsilon'$ and $\epsilon''$ are irrational as required.  Of course, the perturbation we obtain still has Morse-Bott tori at $x=0,1/n$, so one final perturbation is required to remove these before we can honestly apply \cref{dehn_twist_homology}.

\begin{prop}\label{1_over_n_dehn_twist_homology}
We have an isomorphism
\[  PFC(S^1\cross[1,2], \phi_j; j_T, p) \iso CP(-\epsilon',1/n+\epsilon''; p,j_T), \]
and, in the case when $0<j_T\le j$ and $p/j_T\in[0,1/n]$, the homology groups are generated by the classes
\begin{align*}
E(-\epsilon',1/n+\epsilon''; p,j_T) &= [e_{1/n}^{p}e_{0/1}^{j_T-np}]\quad\text{and} \\
H(-\epsilon',1/n+\epsilon''; p,j_T) &=
\begin{cases*}
[h^{\phantom{0}}_{1/n}e_{1/n}^{p-1}] &\text{if $j_T-np=0$,}\\
[h^{\phantom{0}}_{0/1}e_{0/1}^{j_T-1}] &\text{if $p=0$,}\\
[h^{\phantom{0}}_{1/n}e_{1/n}^{p-1}e_{0/1}^{j_T-np}]=[e_{1/n}^{p}h_{0/1}^{\phantom{0}}e_{0/1}^{j_T-np-1}] &\text{otherwise.}
\end{cases*}
\end{align*}
\end{prop}

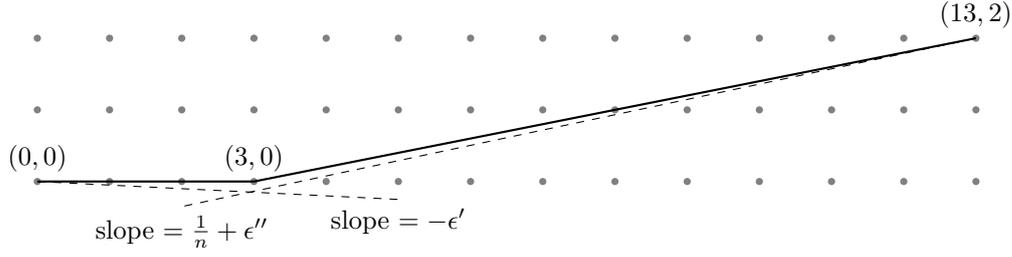
\begin{figure}\centering
  \begin{tikzpicture}[scale=0.95]  %
		\foreach \x in {0,...,13}
			\foreach \y in {0,1,2}
				\fill [gray] (\x,\y) circle (0.05);
		\draw[thick] (0,0) node[above]{$(0,0)$} -- ++(3,0) node[above] {$(3,0)$} -- ++(10,2) node[above] {$(13,2)$};
		\draw[dashed] (0,0) -- (5,-0.25) node[below] {slope $=-\epsilon'$};
		\draw[dashed] (13,2) -- (2,-0.35) node[below] {slope $=\frac{1}{n}+\epsilon''$};

	\end{tikzpicture}
  \caption{In this figure $n=5$, $p=2$ and $j_T=13$.  The dashed lines are of slope $-\epsilon'$ and $\frac{1}{n}+\epsilon''$ respectively.  Their (unmarked) intersection point is the point $X$ from \cref{dehn_twist_homology}; when we release the peg at $X$ the elastic ``snaps'' onto the path shown. Hence $HP(-\epsilon',1/5;2,13)$ is generated by the homology classes $[e_{0/1}^3e_{1/5}^2]$ and $[e_{0/1}^2h_{0/1}e_{1/5}^2]=[e_{0/1}^3h_{1/5}e_{1/5}]$.}
	\label{generators_of_homology_of_1_n_dehn_twist}
\end{figure}

\begin{proof}
We have the following isomorphisms:
\[\begin{split}
PFC(S^1\cross[1,2], \phi_j; j_T, p) &\iso PFC(S^1\cross[1-\epsilon,2+\epsilon], \tilde{\phi}_j; j_T, p) \\
 &\iso PFC(S^1\cross[1-\epsilon,2+\epsilon], \tilde{\phi}'_j; j_T, p) \\
 &\iso CP(-\epsilon',1/n+\epsilon''; p,j_T),
\end{split}\]
where $\tilde{\phi}_j$ is the monodromy obtained by extending $\restr{\phi_j}{S^1\cross[1,2]}$ and $\tilde{\phi}'_j$ is the monodromy obtained by taking a sufficiently small perturbation of $\tilde{\phi}_j$, making the two Morse-Bott tori at $y=1,2$ non-degenerate.  The first isomorphism follows since no new generators are introduced when we make the extension, and this in turn implies that all differentials must be one-sided at the tori at $y=1,2$.  Therefore we can apply \cref{non-degen_close_to_MB}\footnote{Although \cref{non-degen_close_to_MB} was stated in the embedded contact setting it applies just as well here in the periodic Floer setting since the proof relies only on the definition of $I$ and by exploiting Gromov compactness in the presence of a stable Hamiltonian structure.} to obtain the second isomorphism.  The third follows by a straightforward reparametrization as described in the above discussion.

The generators are computing by reading off the description of $E(-\epsilon',1/n+\epsilon''; p,j_T)$ and $H(-\epsilon',1/n+\epsilon''; p,j_T)$ from \cref{dehn_twist_homology} (see \cref{generators_of_homology_of_1_n_dehn_twist}).
\end{proof}

\subsection{The ``two tower'' structure of the first page \texorpdfstring{$E^1(\mathcal{F})$}{E\textasciicircum 1(F)}}\label{computing_the_first_page_of_F}

Recall that for a fixed $j\ge 0$, \cref{zeroth_page_with_twist_isomorphism} formulates $E^0_p(\mathcal{G})$ for $p>0$ as a tensor product, and hence we obtain
\begin{equation}\label{first_page_with_twist_isomorphism}
E^1_p(\mathcal{G}) \iso \sum_{j_T+j_N=j} ECH\orbits{he}{h}(T, \restr{\alpha'_j}{T};j_T,p) \tensor ECH(\mathrm{int}(N),\alpha_N; j_N).
\end{equation}
Recall that 
\[ ECC\orbits{he}{h}(T, \restr{\alpha'_j}{T};j_T,p) \iso PFC^{(h_+)}(S^1\cross[1,2], \phi_j; j_T,p) \]
by \cref{ECC_is_PFC_on_twist_region}.  We will compute the groups $PFH^{(h_+)}(S^1\cross[1,2], \phi_j; j_T,p)$ using yet another spectral sequence argument which is analogous to that used in the \refNamedThm{Cancellation Lemma}{cancellation_lemma}.

\begin{prop}\label{homology_of_augmented_dehn_twist}
For $0<j_T\le j$ and $p/j_T\in[0,1/n]$ the group
\[ PFH^{(h_+)}(S^1\cross[1,2], \phi_j; j_T,p)\]
is generated by the homology classes
\begin{itemize}
\item $[e_{0/1}^{j_T}]$ and $[h_{0/1}e_{0/1}^{j_T-1}]$ if $p=0$,
\item $[h_{1/n}]$ if $p=1$ and $j_T=n$, and
\end{itemize}
is 0 otherwise.
\end{prop}
\pagebreak %
\begin{proof}
First note that if $p=0$ or if $j_T<n$ then $h_+$ does not even appear in the complex and we have an isomorphism
\[ PFC^{(h_+)}(S^1\cross[1,2], \phi_j; j_T,p) \iso PFC(S^1\cross[1,2], \phi_j; j_T,p) \]
and the result follows immediately from \cref{1_over_n_dehn_twist_homology}.

Hence assume that $p\ge 1$ and $j_T\ge n$ and define a two-level ascending filtration $\mathcal{H}$ on the chain complex
\[  PFC^{(h_+)}(S^1\cross[1,2], \phi_j; j_T,p) \]
by counting the multiplicity of $h_+$. This exhibits the complex as the following diagram:
\[\begin{tikzcd}
h_+\tensor PFC(S^1\cross[1,2],\phi_j; j_T-n, p-1) \arrow[rr,"\cdot/h_+\tensor e_{1/n}"] && 1\tensor PFC(S^1\cross[1,2],\phi_j; j_T, p).
\end{tikzcd}\]
Therefore the first page of the spectral sequence is given by
\[ E^1(\mathcal{H}) \iso  PFH(S^1\cross[1,2], \phi_j; j_T, p) \gappy{\oplus} h_+PFH(S^1\cross[1,2], \phi_j; j_T-n, p-1) \]
with differential $d_{\mathcal{H},1}$ given by
\[ [\Gamma_1]\oplus h_+[\Gamma_2] \mapsto [e_{1/n}\Gamma_1]\oplus [\Gamma_2]. \]
Then, referring to \cref{1_over_n_dehn_twist_homology} for the possibilities of the orbit sets $\Gamma_1$ and $\Gamma_2$, we see that $\mathrm{Ker}(d_{\mathcal{H},1})=0$ and $\mathrm{Im}(d_{\mathcal{H},1})$ is precisely those generators containing a copy of $e_{1/n}$.

Therefore if $p=1$ and $j_T=n$ the only surviving generator on the second page $E^2(\mathcal{H})$ is $h_{1/n}$ and for all remaining values of $p$ and $j_T$ all generators are killed.  As a result the spectral sequence converges on the second page and we are done.
\end{proof}

\begin{thm}\label{first_page_of_G_spectral_sequence}
For a fixed $j\ge 0$ the first page of the spectral sequence $E^1_p(\mathcal{G})$ in $\mathcal{F}$-grading $p>0$ is given by
\[ E^1_p(\mathcal{G}) \iso
\begin{cases}
h_{1/n} ECH(\mathrm{int}(N),\alpha_N; j-n) &\text{if $p=1$ and $j\ge n$, and}\\
0 &\text{otherwise.}
\end{cases}
\]
\end{thm}
\begin{proof}
Refer to \cref{first_page_with_twist_isomorphism} and note that if $p>1$ or $j<n$ (in which case certainly $j_T<n$) then
\[ ECH\orbits{he}{h}(T, \restr{\alpha'_j}{T};j_T,p) = 0 \]
by \cref{ECC_is_PFC_on_twist_region,homology_of_augmented_dehn_twist}.

When $p=1$ and $j\ge n$ then by \cref{homology_of_augmented_dehn_twist} only the $j_T=n$ summand in \cref{first_page_with_twist_isomorphism} survives and the result follows.
\end{proof}
Notice that the complex $h_{1/n} ECH(\mathrm{int}(N),\alpha_N; j-n)$ lies entirely within $\mathcal{G}$-grading $n$, and hence the spectral sequence converges at this point.  Hence we obtain the following theorem, recalling \cref{zeroth_page_no_twist_isomorphism} from \cref{the_case_when_p_equals_0}.

\begin{thm}\label{first_page_of_F_spectral_sequence}
For each fixed $j\ge 0$, the first pages of the spectral sequences corresponding to the filtration $\mathcal{F}$ defined on the complexes
\[ ECC\orbits{he}{h}(\mathrm{int}(N\union T), \alpha'_j;j)\quad\text{and}\quad ECC\orbits{he}{h}(\mathrm{int}(N\union T), \alpha_j;j) \] are isomorphic as vector spaces to
\[ E^1(\mathcal{F}) \iso
\begin{cases}
ECH\orbits{e}{h}(\mathrm{int}(N),\alpha_N; j) &\text{if $j<n$, and}\\
ECH\orbits{e}{h}(\mathrm{int}(N),\alpha_N; j) \oplus h_{1/n} ECH(\mathrm{int}(N),\alpha_N; j-n) &\text{if $j\ge n$.}
\end{cases}
\]
\end{thm}
It is important to remember that since $\alpha_j$ is not necessarily nice, the differentials on the first pages of the two spectral sequences corresponding to $\alpha'_j$ and $\alpha_j$ may differ.

Recall that we were required to fix $j$ in order to define the filtration $\mathcal{F}$ on 
\[ ECC\orbits{he}{h}(\mathrm{int}(N\union T), \alpha'_j;j).\]
However it is helpful for intuitive reasons to picture the complex in full.  In other words, if we make the $j$ explicit in the notation, writing $E_{(j)}^1(\mathcal{F})$ for each $j$ in turn, we can write
\[\begin{split}
\bigoplus_{j\in\N}E_{(j)}^1(\mathcal{F}) &\iso ECH\orbits{e}{h}(\mathrm{int}(N),\alpha_N) \oplus h_{1/n} ECH(\mathrm{int}(N),\alpha_N)\\
&\iso ECH\orbits{e}{h}(\mathrm{int}(N),\alpha_N) \oplus h_{1/n} ECH\orbits{}{eh}(\mathrm{int}(N),\alpha_N),
\end{split}\]
where the second line follows from the first by the \refNamedThm{Cancellation Lemma}{cancellation_lemma}.

We refer to this direct sum as \emph{the two towers} of the surgery formula.  Furthermore, we will refer to the left-hand summand as the \emph{$e_-$ tower} and to the right-hand summand as the \emph{$h_{1/n}$ (or $e_+$) tower}.  Refer to \cref{two_towers_example} for an example of how these two summands appear in the case of ($-8$)-surgery on the torus knot $T(2,5)$.

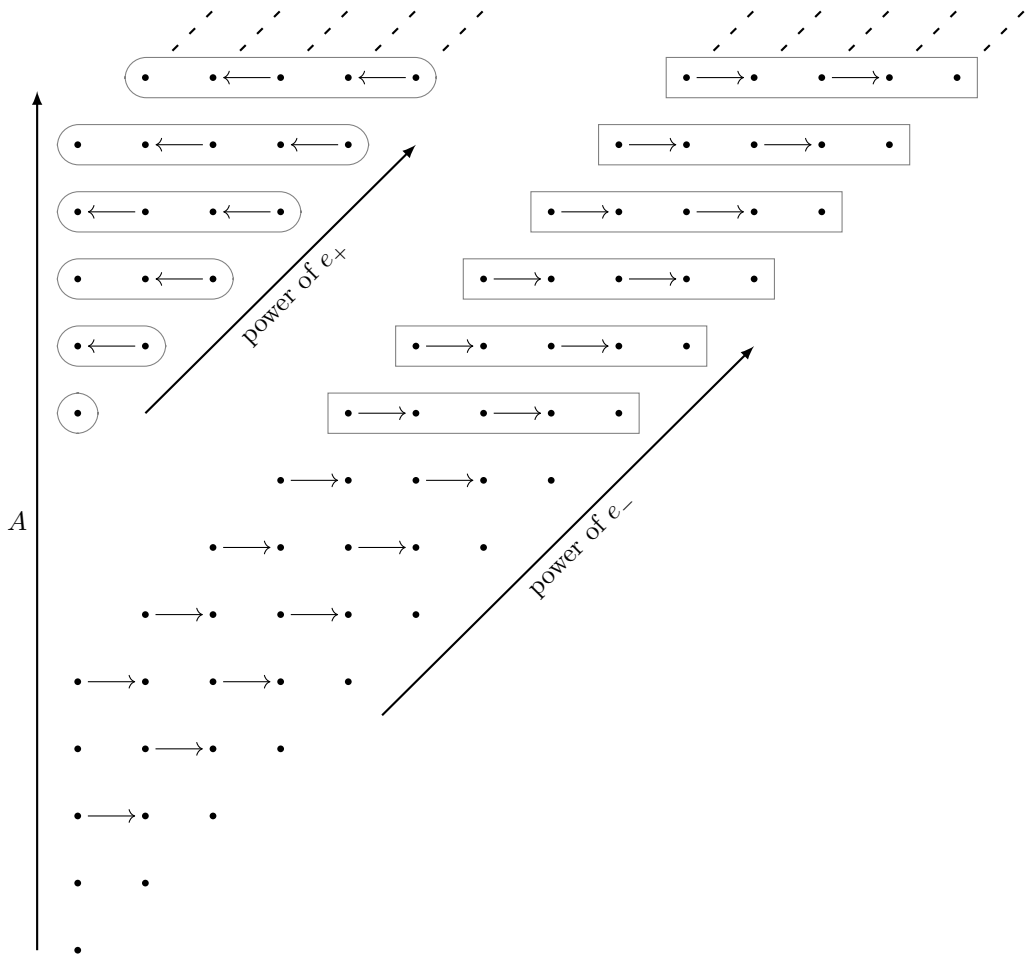
\begin{figure}\centering
  \begin{tikzpicture}[scale=0.89]  %
		\draw [-latex,thick] (-0.6,0) -- (-0.6,12.8) node [midway, left] {$A$};
		\draw [-latex,thick] (4.5,3.5) -- +(5.5,5.5);
		\node [rotate=45] at (7.25+0.2,6.25-0.2) {power of $e_-$};
		\draw [-latex,thick] (1,8) --  +(4,4);
		\node [rotate=45,anchor=center] at (3+0.2,10-0.2) {power of $e_+$};
		
		\begin{scope}
			\path[clip] (-0.5,-0.5) rectangle (13.5,13.5);
			\foreach \x in {0,...,13} {
				\foreach \y in {0,...,4}
					\fill (\x,\x+\y) circle (0.05);
				\draw [->] (\x+0.15,2+\x) -- +(0.7,0);
				\draw [->] (\x+0.15,4+\x) -- +(0.7,0);
				}
			\foreach \x in {0,...,5} {
				\foreach \y in {0,...,4}
					\fill (\x,8+\x+\y) circle (0.05);
				\draw [->] (\x+0.85,9+\x) -- +(-0.7,0);
				\draw [->] (\x+0.85,11+\x) -- +(-0.7,0);
			}
		\end{scope}

		\foreach \x in {1,...,5,9,10,...,13}
			\draw [loosely dashed,thick] (\x+0.4,13.4) -- (\x+1,14);

		\foreach \y in {8,...,13} {
			\draw [rounded corners=8pt,gray]({max(0,\y-12)-0.3}, \y-0.3) rectangle (\y-8+0.3, \y+0.3);
			\draw [gray](\y-4-0.3,\y-0.3) rectangle (\y+0.3,\y+0.3);
		}	

	\end{tikzpicture}
  \caption{The two-tower structure of the first page $E^1(\mathcal{F})$ for ($-n$)-surgery on the torus knot $T(2,5)$ in the case $n=8$.  The Alexander grading is represented by the vertical coordinate. The upper left tower is the $e_+$ tower, $h_{1/n}ECC\protect\orbits{}{eh}(\mathrm{int}(N),\alpha_N)$, and in this tower the horizontal coordinate represents the power of $e_+$; in particular the lowest diagonal is generated by the orbit sets $h_{1/n},h_{1/n}e_+,h_{1/n}e_+^2,\dots$  The bottom right tower is the $e_-$ tower, $ECC\protect\orbits{e}{h}(\mathrm{int}(N),\alpha_N)$, and in this tower the horizontal coordinate represents the power of $e_-$; in particular the lower diagonal is generated by the orbit sets $\emptyset,e_-,e_-^2,\dots$ Of course, in the $e_+$ tower differentials can only preserve or decrease the horizontal coordinate, and in the $e_-$ tower differentials can only preserve or increase the horizontal coordinate.  In this example we have also exploited the \refNamedThm{Cancellation Lemma}{cancellation_lemma} and \cref{A_j_iso} which allows us to appeal to the diagram in \cref{ECK_full_complex_diagram}; the $e_+$ tower looks like the full ECK complex for $T(2,5)$ but without any vertical arrows, and the $j$-th row of the $e_-$ tower looks like the subcomplex of the zeroth column of $ECK(T(2,5))$ which is generated by orbit sets with Alexander grading less than or equal to $j$.  Finally, note that the $e_-$ tower lies in $\mathcal{F}$-filtration level 0 and the $e_+$ tower lies in level 1.  Therefore, within each $j$-th row (for $j\ge n$), we expect the differential $d_{\mathcal{F},1}$ to exhibit itself as a map between the homology of the complex contained in the rectangular box and that of the complex in the oval box.  Understanding the behaviour of this map is the subject of \cref{the_differential_on_the_first_page}.}
	\label{two_towers_example}
\end{figure}

\section{The differential on the first page}\label{the_differential_on_the_first_page}

In this section we will compute the differential $d_{\mathcal{F},1}$ on the first page of the spectral sequence corresponding to the filtration $\mathcal{F}$ defined on the complex 
\[ ECC\orbits{he}{h}(\mathrm{int}(N\union T), \alpha'_j;j). \]
By \cref{first_page_of_F_spectral_sequence}, the first page is supported entirely within $\mathcal{F}$-gradings 0 and 1. Furthermore if $j<n$ then it is supported entirely in $\mathcal{F}$-grading 0 and hence we have convergence already.

Now assume $j\ge n$.  The differential $d_{\mathcal{F},1}$ will exhibit itself as a differential from the $e_+$ tower to the $e_-$ tower and hence is computed by counting holomorphic curves which intersect the torus at $y=2$ with non-infinite slope.  In fact, by slope calculus (\cref{slope_calculus}) we see that the slope of any such curve at $y=2$ must equal $n[m]-[l]$.  This is a problem since such holomorphic curves are hard to understand; for instance it is not clear how one would construct finite energy foliations in this case.

We will proceed by induction on $n$, and exploit \cref{ECK_supported_genus} via an algebraic argument to understand how the differential $d_{\mathcal{F},1}$ behaves.  For this we will need to introduce $n$ into the notation so denote by 
\[ d^{(n)}_{\mathcal{F},1} \]
the differential on the first page in the case of ($-n$)-surgery.

\subsection{A commutative diagram}

The goal of this section is to prove the following:
\begin{prop}\label{commutative_diagram_L_filtration}
Suppose that $j\ge n+1$. Then following diagram commutes:
\begin{equation}\label{commutative_diagram_L_filtration_eqn}
\begin{tikzcd}
ECH\orbits{e}{h}(\mathrm{int}(N),\alpha_N;j) \arrow[rr,"d^{(n+1)}_{\mathcal{F},1}"]\arrow[d,"d^{(n)}_{\mathcal{F},1}"'] && h_{\nicefrac{1}{n+1}}ECH(\mathrm{int}(N),\alpha_N;j-(n+1)) \arrow[d,"\iso"] \\
h_{\nicefrac{1}{n}}ECH(\mathrm{int}(N),\alpha_N;j-n) \arrow[rr,"h_{\nicefrac{0}{1}}(d'_N)_*"] && h_{\nicefrac{1}{n}}h_{\nicefrac{0}{1}}ECH(\mathrm{int}(N),\alpha_N;j-(n+1)).
\end{tikzcd}
\end{equation}
In this diagram the isomorphism on the right hand side is the canonical isomorphism
\[ h_{\nicefrac{1}{n+1}}\Gamma \mapsto h_{\nicefrac{1}{n}}h_{\nicefrac{0}{1}}\Gamma \]
and the map on the bottom is induced by
\[ h_{\nicefrac{1}{n}}\Gamma \mapsto h_{\nicefrac{1}{n}}h_{\nicefrac{0}{1}}d'_N(\Gamma). \]
\end{prop}
Here are using the notation from \cref{computing_ech_via_a_filtration_argument}, writing
\[  d'_N(\Gamma) = d_N(\Gamma) / h_{\nicefrac{0}{1}} \]
where $d_N$ is the differential on $ECC\orbits{e}{h}(\mathrm{int}(N),\alpha_N;j)$.

The proof of this proposition will proceed by exhibiting the above diagram as the first page of a spectral sequence corresponding to a descending filtration $\mathcal{L}$ defined on the complex $ECC\orbits{he}{h}(\mathrm{int}(N\union T),\alpha'_j;j)$. Then the result will follow from $d^2=0$ and the fact that the notions of commutativity and anticommutativity are equivalent since we are working with coefficients in $\F_2$. The idea is to choose the filtration so that we isolate parts of the complex giving the four terms above.  We will define this filtration now.

\begin{defn}
Suppose that we have an orbit set $\Gamma \in ECC\orbits{he}{h}(\mathrm{int}(N\union T),\alpha'_j;j)$.  Define $\mathcal{L}$ as follows, recalling the filtrations $\mathcal{F}$ and $\mathcal{G}$ from \cref{a_filtration_on_the_surgery_complex,a_second_filtration_for_p_greater_0}:
\begin{itemize}
\item If $\mathcal{F}(\Gamma)=0$, define $\mathcal{L}(\Gamma):=0$.
\item If $\mathcal{F}(\Gamma)=1$, then note that we must have $\mathcal{G}(\Gamma)\ge n$.
\begin{itemize}
\item If $\mathcal{G}(\Gamma)=n$, define $\mathcal{L}(\Gamma):=1$.
\item If $\mathcal{G}(\Gamma)=n+1$ and $\Gamma$ contains the orbit $h_{\nicefrac{1}{n+1}}$, define $\mathcal{L}(\Gamma):=1$.
\item If $\mathcal{G}(\Gamma)=n+1$ and $\Gamma$ does not contain the orbit $h_{\nicefrac{1}{n+1}}$, define $\mathcal{L}(\Gamma):=2$.
\item If $\mathcal{G}(\Gamma)>n+1$, define $\mathcal{L}(\Gamma):=3$.
\end{itemize}
\item If $\mathcal{F}(\Gamma)>1$, define $\mathcal{L}(\Gamma):=\mathcal{F}(\Gamma)+2 >3$.
\end{itemize}
\end{defn}

In order to prove that $\mathcal{L}$ defines a descending filtration will we need the lemma below. For that we introduce the notation
\[ C_{\nicefrac{1}{k}} := ECC\orbits{he}{h}(T, \restr{\alpha'_j}{T};k,1). \]
\begin{lemma}\label{complex_in_twist_region_1_n+1}~
\begin{enumerate}
\item The complex $C_{\nicefrac{1}{n}}$ is generated by the orbit sets $h_{\nicefrac{1}{n}}$, $e_{\nicefrac{1}{n}}$ and $h_+$ and has a single differential
\[ h_+ \longrightarrow e_{\nicefrac{1}{n}}. \]
\item The complex $C_{\nicefrac{1}{n+1}}$ is given by the following diagram:
\[\begin{tikzcd}
h_{\nicefrac{1}{n+1}}\arrow[d] & e_{\nicefrac{1}{n+1}}\arrow[d]\arrow[dr] & h_{\nicefrac{0}{1}}h_+\arrow[d] & e_{\nicefrac{0}{1}}h_+\arrow[d] \\
h_{\nicefrac{0}{1}}h_{\nicefrac{1}{n}} & e_{\nicefrac{0}{1}}h_{\nicefrac{1}{n}} & h_{\nicefrac{0}{1}}e_{\nicefrac{1}{n}} & e_{\nicefrac{0}{1}}h_{\nicefrac{1}{n}}
\end{tikzcd}\]
\end{enumerate}
\end{lemma}
\begin{proof}
By \cref{ECC_is_PFC_on_twist_region,1_over_n_dehn_twist_homology} the complex $C_{\nicefrac{1}{k}}$ is isomorphic to
\[  \mathcal{R}[h_+] \tensor CP(-\epsilon', 1/n+\epsilon'', 1, k), \]
equipped with differential
\[ d(\gamma\tensor\Gamma) = \gamma/h_+\tensor e_{\nicefrac{1}{n}}\Gamma + \gamma\tensor d_\mathrm{T}(\Gamma), \]
where $d_\mathrm{T}$ is the differential on the Dehn twist region described briefly in terms of rounding corners in \cref{the_periodic_floer_homology_of_a_dehn_twist} and more concretely by Hutchings and Sullivan~\cite[Theorem 3.5]{HS05}.  By examining the possible left-turning polygonal paths for $k=n,n+1$ we see that the generators of the complexes $C_{\nicefrac{1}{k}}$ are given precisely by those in the statement of the lemma, and by rounding the corners of these paths we see that the differential is also as claimed in both cases.
\end{proof}

\begin{prop}\label{L_is_filtration}
$\mathcal{L}$ defines a descending filtration on $ECC\orbits{he}{h}(\mathrm{int}(N\union T),\alpha'_j;j)$.
\end{prop}
\begin{proof}
Denote by $d$ the differential on the complex in question and suppose for contradiction that $\mathcal{L}(\Gamma_1)>\mathcal{L}(\Gamma_2)$ and $\langle d\Gamma_1,\Gamma_2\rangle\ne 0$.

If $\mathcal{L}(\Gamma_2)=0$ or $\mathcal{L}(\Gamma_1)>3$ then we must have $\mathcal{F}(\Gamma_1)>\mathcal{F}(\Gamma_2)$ which contradicts the fact that $\mathcal{F}$ is a descending filtration.

In all other cases we have $\mathcal{F}(\Gamma_1)=\mathcal{F}(\Gamma_2)=1$ and hence, since $\mathcal{G}$ is a descending filtration on $E^0(\mathcal{F})$ we cannot have $\mathcal{G}(\Gamma_1)>\mathcal{G}(\Gamma_2)$.  As a result, we immediately see that we cannot have $\mathcal{L}(\Gamma_1)=3$.

We are left with the case $\mathcal{L}(\Gamma_2)=1$ and $\mathcal{L}(\Gamma_1)=2$. In this situation the only possibility is that $\mathcal{G}(\Gamma_1)=\mathcal{G}(\Gamma_2)=n+1$ and $\Gamma_2$ contains the orbit $h_{\nicefrac{1}{n+1}}$ whereas $\Gamma_1$ does not.

If we write $\Gamma_i$ in the form $\gamma_i\tensor\Gamma'_i$ for $i=1,2$ as in \cref{a_filtration_on_the_surgery_complex}, then $\mathcal{F}(\Gamma_i)=1$ and $\mathcal{G}(\Gamma_i)=n+1$ means that both $\gamma_i$ are elements of the complex $C_{\nicefrac{1}{n+1}}$.

Then we have a differential with $\gamma_1\tensor\Gamma'_1$ at the positive end and $h_{\nicefrac{1}{n+1}}\tensor\Gamma'_2$ at the negative end.  However we know that the non-connector part of this differential cannot intersect the torus at $y=2$, and by \cref{index_0_1_and_2_holo_curves} there is only one non-connector component.  Hence $\Gamma'_1=\Gamma'_2$ and we have a differential
\[  \gamma_1 \longrightarrow h_{\nicefrac{1}{n+1}} \]
lying entirely within the twist region. But by referring to the description of $C_{\nicefrac{1}{n+1}}$ in \cref{complex_in_twist_region_1_n+1} we see that this is not possible, so again we have a contradiction.
\end{proof}

We now compute the first page of the spectral sequence $E^1(\mathcal{L})$ by considering each $\mathcal{L}$-grading in turn.

\begin{prop}[$\mathcal{L}=0$]\label{L_0_part_prop}
The $\mathcal{L}=0$ part of the complex $ECC\orbits{he}{h}(\mathrm{int}(N\union T),\alpha'_j;j)$ is equal to $ECC^{e_{\nicefrac{0}{1}}}_{h_{\nicefrac{0}{1}}}(\mathrm{int}(N),\alpha_N;j)$ and hence on the first page of the spectral sequence we obtain
\[ E^1_0(\mathcal{L})\iso ECH^{e_{\nicefrac{0}{1}}}_{h_{\nicefrac{0}{1}}}(\mathrm{int}(N),\alpha_N;j). \]
\end{prop}
\begin{proof}
This follows immediately from the fact that $\mathcal{F}=0$ as in \cref{the_case_when_p_equals_0} and by replacing $\restr{\alpha_j'}{N}$ with $\alpha_N$ by \cref{facts_about_differentials_on_zeroth_page}.
\end{proof}

\begin{prop}[$\mathcal{L}=1$]\label{L_1_part_prop}
The $\mathcal{L}=1$ part of the complex $ECC\orbits{he}{h}(\mathrm{int}(N\union T),\alpha'_j;j)$ is given by
\begin{equation}\label{L=1_part_of_complex} \bigg(C_{\nicefrac{1}{n}} \tensor ECC(\mathrm{int}(N),\alpha_N;j-n)\bigg) \oplus 
 h_{\nicefrac{1}{n+1}} ECC(\mathrm{int}(N),\alpha_N;j-(n+1)) \end{equation}
and on the first page of the spectral sequence we obtain
\[ E^1_1(\mathcal{L})\iso h_{\nicefrac{1}{n}} ECH(\mathrm{int}(N),\alpha_N;j-n) \oplus h_{\nicefrac{1}{n+1}} ECH(\mathrm{int}(N),\alpha_N;j-(n+1)). \]
\end{prop}
\begin{proof}
If $\mathcal{L}(\Gamma)=1$ then $\mathcal{F}(\Gamma)=1$ and either 
\begin{itemize}
\item $\mathcal{G}(\Gamma)=n$ and hence $\Gamma=\gamma\Gamma'$ for some $\Gamma'\in ECC(\mathrm{int}(N),\alpha'_j;j-(n+1))$ and $\gamma\in\set{h_{\nicefrac{1}{n}},e_{\nicefrac{1}{n}},h_+}$.
\item $\mathcal{G}(\Gamma)=n+1$ and $\Gamma=h_{\nicefrac{1}{n+1}}\Gamma'$ for some $\Gamma'\in ECC(\mathrm{int}(N),\alpha'_j;j-(n+1))$.
\end{itemize}
Hence $E_1^0(\mathcal{L})$ is isomorphic to \cref{L=1_part_of_complex} as a vector space.  (Again, here we are applying \cref{facts_about_differentials_on_zeroth_page} to replace $\restr{\alpha_j'}{N}$ with $\alpha_N$.)  We claim that as a chain complex we also have this splitting, i.e.~there are no differentials between the two summands.  To see this, note that since we are in a fixed $\mathcal{F}$-grading, $\langle d\Gamma_1, \Gamma_2\rangle \ne 0$ implies that $\mathcal{G}(\Gamma_2)\geq\mathcal{G}(\Gamma_1)$; to obtain a contradiction suppose that we do not have equality---then $\mathcal{G}(\Gamma_1)=n$ and $\mathcal{G}(\Gamma_2)=n+1$.  But then there is a holomorphic curve whose positive end contains one of $h_{\nicefrac{1}{n}}$, $e_{\nicefrac{1}{n}}$ or $h_+$ and whose negative end does not contain any of these orbits. This gives a contradiction since only negative ends can appear at $h_{\nicefrac{1}{n}}$ and $e_{\nicefrac{1}{n}}$, and the only differentials containing $h_+$ at a positive end are those of the form
\[  h_+\Gamma \longrightarrow e_{\nicefrac{1}{n}}\Gamma. \]
When we pass to the first page, we use the fact that
\[ H_*(C_{1/n})=\langle[h_{1/n}]\rangle \]
from \cref{complex_in_twist_region_1_n+1} to obtain the desired result.
\end{proof}

\begin{prop}[$\mathcal{L}=2$]
The $\mathcal{L}=2$ part of the complex $ECC\orbits{he}{h}(\mathrm{int}(N\union T),\alpha'_j;j)$ is equal to
\begin{equation}\label{L=2_part_diagram} 
C'_{\nicefrac{1}{n+1}} \tensor ECC(\mathrm{int}(N),\alpha_N;j-(n+1)) 
\end{equation}
where $C'_{\nicefrac{1}{n+1}}$ is defined to be the subcomplex of $C_{\nicefrac{1}{n+1}}$ obtained by omitting the generator $h_{\nicefrac{1}{n+1}}$ (c.f.~\cref{complex_in_twist_region_1_n+1}), and on on the first page of the spectral sequence we obtain
\[ E_2^1(\mathcal{L}) \iso h_{\nicefrac{1}{n}}h_{\nicefrac{0}{1}} ECH(\mathrm{int}(N),\alpha_N;j-(n+1)). \]
\end{prop}
\begin{proof}
Since $\mathcal{L}(\Gamma)=2$, we have $\mathcal{F}(\Gamma)=1$, $\mathcal{G}(\Gamma)=n+1$ and $\Gamma$ does not contain $h_{\nicefrac{1}{n+1}}$.  This immediately gives us the complex in \cref{L=2_part_diagram} (once again, we are using \cref{facts_about_differentials_on_zeroth_page} to replace $\restr{\alpha_j'}{N}$ with $\alpha_N$).

The complex $C'_{\nicefrac{1}{n+1}}$ is given by the diagram
\[\begin{tikzcd}
  & e_{\nicefrac{1}{n+1}}\arrow[d]\arrow[dr] & h_{\nicefrac{0}{1}}h_+\arrow[d] & e_{\nicefrac{0}{1}}h_+\arrow[d] \\
h_{\nicefrac{0}{1}}h_{\nicefrac{1}{n}} & e_{\nicefrac{0}{1}}h_{\nicefrac{1}{n}} & h_{\nicefrac{0}{1}}e_{\nicefrac{1}{n}} & e_{\nicefrac{0}{1}}h_{\nicefrac{1}{n}}
\end{tikzcd}\]
by \cref{complex_in_twist_region_1_n+1}, and hence its homology is generated by the class $[h_{\nicefrac{0}{1}}h_{\nicefrac{1}{n}}]$.
Taking the homology of \cref{L=2_part_diagram} yields the desired result.
\end{proof}

\begin{prop}[$\mathcal{L}>2$]\label{L_3_part_prop}
The first page of the spectral sequence vanishes in $\mathcal{L}$-gradings greater than 2.
\end{prop}
\begin{proof}
First consider $E_3^0(\mathcal{L})$, where $\mathcal{F}=1$ and $\mathcal{G}>n+1$.  Then we can proceed by the same argument as used in \cref{computing_the_first_page_of_F}, namely by using the $\mathcal{G}$ filtration to define a spectral sequence $E_1^r(\mathcal{G})$ which converges to $E^1_1(\mathcal{F})$.  Then by \cref{first_page_of_G_spectral_sequence} we see that $E^1_1(\mathcal{G})$ vanishes in $\mathcal{G}$-gradings greater than $n$.

Now consider $E^0_{p+2}(\mathcal{L})$, for $p+2>3$, where $\mathcal{F}=p>1$.  Then the exact same argument, here noting that the spectral sequence corresponding to $\mathcal{G}$ vanishes in $\mathcal{F}$-gradings greater than 1, yields the result.
\end{proof}

We can summarize Propositions~\ref{L_0_part_prop} to~\ref{L_3_part_prop} with the following diagram of $E^1(\mathcal{L})$:
\[\hskip \textwidth minus \textwidth \begin{tikzpicture}[commutative diagrams/every diagram, scale=1]
\node (P0) at (90:2cm) {$ECH^{e_{\nicefrac{0}{1}}}_{h_{\nicefrac{0}{1}}}(\mathrm{int}(N),\alpha_N;j)$};
\node (P1) at (180:3cm) {$h_{\nicefrac{1}{n}} ECH(\mathrm{int}(N),\alpha_N;j-n)$} ;
\node (P2) at (0:3cm) {$h_{\nicefrac{1}{n+1}} ECH(\mathrm{int}(N),\alpha_N;j-(n+1))$};
\node (P3) at (270:2cm) {$h_{\nicefrac{1}{n}}h_{\nicefrac{0}{1}} ECH(\mathrm{int}(N),\alpha_N;j-(n+1)).$};
\node at (-6.2,0) {};
\node at (7,2) {($\mathcal{L}=0$)};
\node at (7,0) {($\mathcal{L}=1$)};
\node at (7,-2) {($\mathcal{L}=2$)};
\path[commutative diagrams/.cd, every arrow, every label]
(P0) edge node[swap] {$d_{\mathcal{L},1}$} (P1)
(P0) edge node {$d_{\mathcal{L},1}$} (P2)
(P1) edge node[swap] {$d_{\mathcal{L},1}$} (P3)
(P2) edge node {$d_{\mathcal{L},1}$} (P3);
\end{tikzpicture}\]
Hence to prove \cref{commutative_diagram_L_filtration} it suffices to show that the differential $d_{\mathcal{L},1}$ yields the four maps in \refDiagram{commutative_diagram_L_filtration_eqn}.

\begin{proof}[Proof of \cref{commutative_diagram_L_filtration}]
We will consider each component of $d_{\mathcal{L},1}$ in turn.  Let $d$ denote the original differential on the complex
\[ ECC\orbits{he}{h}(\mathrm{int}(N\union T),\alpha'_j;j). \]
\paragraph{1.}The component of $d_{\mathcal{L},1}$ from
\[ ECH^{e_{\nicefrac{0}{1}}}_{h_{\nicefrac{0}{1}}}(\mathrm{int}(N),\alpha_N;j)\quad\text{to}\quad h_{\nicefrac{1}{n}} ECH(\mathrm{int}(N),\alpha_N;j-n) \]
is induced by the part of $d$ obtained by counting holomorphic curves between orbit sets in
\[  ECC^{e_{\nicefrac{0}{1}}}_{h_{\nicefrac{0}{1}}}(\mathrm{int}(N),\alpha_N;j)\quad\text{and}\quad  h_{\nicefrac{1}{n}}  ECC(\mathrm{int}(N),\alpha_N;j-n), \]
which is precisely the definition of $d_{\mathcal{F},1}^{(n)}$.

\paragraph{2.}Similarly the component of $d_{\mathcal{L},1}$ from
\[ ECH^{e_{\nicefrac{0}{1}}}_{h_{\nicefrac{0}{1}}}(\mathrm{int}(N),\alpha_N;j)\quad\text{to}\quad h_{\nicefrac{1}{n+1}} ECH(\mathrm{int}(N),\alpha_N;j-(n+1)) \]
is induced by the part of $d$ obtained by counting holomorphic curves between orbit sets in
\[  ECC^{e_{\nicefrac{0}{1}}}_{h_{\nicefrac{0}{1}}}(\mathrm{int}(N),\alpha_N;j)\quad\text{and}\quad h_{\nicefrac{1}{n+1}}  ECC(\mathrm{int}(N),\alpha_N;j-(n+1)). \]
By slope calculus (\cref{slope_calculus}) and the \refNamedThm{Blocking Lemma}{blocking_lemma}, all such holomorphic curves are contained entirely in the region 
\[ T^2\cross[y_{\nicefrac{1}{n+1}}, 2] \union N, \]
where $y_{\nicefrac{1}{n+1}}$ is the value of $y$ at which $\alpha$ has Reeb slope $n[m]-l$.  Then the count of such curves is precisely equal to that when defining the differential in the case of $-(n+1)$-surgery, hence we obtain the map $d_{\mathcal{F},1}^{(n+1)}$.

\paragraph{3.}The component of $d_{\mathcal{L},1}$ from
\[ h_{\nicefrac{1}{n}} ECH(\mathrm{int}(N),\alpha_N;j-n)\quad\text{to}\quad h_{\nicefrac{1}{n}}h_{\nicefrac{0}{1}} ECH(\mathrm{int}(N),\alpha_N;j-(n+1)) \]
is induced by the part of $d$ obtained by counting holomorphic curves between orbit sets in
\[ h_{\nicefrac{1}{n}} ECC(\mathrm{int}(N),\alpha_N;j-n)\quad\text{and}\quad h_{\nicefrac{1}{n}}h_{\nicefrac{0}{1}}  ECC(\mathrm{int}(N),\alpha_N;j-(n+1)). \]
By slope calculus and the Blocking Lemma, all such holomorphic curves are supported entirely within $N$, and furthermore have a single negative end at $h_{\nicefrac{0}{1}}$ and no positive ends at $\del N$.  This count of curves is precisely the definition of the map
\[ \Gamma \mapsto h_{\nicefrac{0}{1}} d'_N(\Gamma), \]
hence this component of $d_{\mathcal{L},1}$ is induced by the map
\[ h_{\nicefrac{1}{n}}\Gamma \mapsto h_{\nicefrac{1}{n}}h_{\nicefrac{0}{1}} d'_N(\Gamma). \]

\paragraph{4.}Finally, the component of $d_{\mathcal{L},1}$ from
\[ h_{\nicefrac{1}{n+1}} ECH(\mathrm{int}(N),\alpha_N;j-(n+1))\quad\text{to}\quad h_{\nicefrac{1}{n}}h_{\nicefrac{0}{1}} ECH(\mathrm{int}(N),\alpha_N;j-(n+1)) \]
is induced by the part of $d$ obtained by counting holomorphic curves between orbit sets in
\[ h_{\nicefrac{1}{n+1}} ECC(\mathrm{int}(N),\alpha_N;j-(n+1))\quad\text{and}\quad h_{\nicefrac{1}{n}}h_{\nicefrac{0}{1}} ECC(\mathrm{int}(N),\alpha_N;j-(n+1)). \]
Applying slope calculus, the Blocking Lemma and \cref{index_0_1_and_2_holo_curves} here gives the result that the non-connector part of all such holomorphic curves are contained entirely entirely within the twist region $T$, and hence we only see trivial cylinders in $\mathrm{int}(N)$. We can then appeal to the description of $C_{\nicefrac{1}{n}}$ in \cref{complex_in_twist_region_1_n+1}, noting that there is precisely one contributing differential,
\[ h_{\nicefrac{1}{n+1}} \longrightarrow h_{\nicefrac{0}{1}}h_{\nicefrac{1}{n}}, \]
and hence this component of $d_{\mathcal{L},1}$ is induced by the map
\[ h_{\nicefrac{1}{n+1}}\Gamma \mapsto h_{\nicefrac{0}{1}}h_{\nicefrac{1}{n}}\Gamma \]
as required.
\end{proof}

\subsection{The mapping cone of \texorpdfstring{$d_{\mathcal{F},1}$}{d\_\{F,1\}}}

The goal of this section is to prove the following theorem
\begin{thm}\label{surgery_mapping_cones}
Fix $j>2g$, where $g$ is the genus of the surface $\Sigma$, and suppose that $1\le n\le j$.  The mapping cone of the map
\[ d_{\mathcal{F},1}^{n} : ECH\orbits{e}{h}(\mathrm{int}(N), \alpha_N; j) \to h_{\nicefrac{1}{n}}ECH(\mathrm{int}(N),\alpha_N; j-n) \]
has the same homology as the mapping cone of the quotient map
\begin{align*}
ECC\orbits{}{eh}(\mathrm{int}(N), \alpha_N; j-1) &\to ECC\orbits{}{eh}(\mathrm{int}(N), \alpha_N; j-n) \\
\Gamma &\mapsto \Gamma/e_+^{n-1},
\end{align*}
which in turn is given by the homology of the subcomplex
\[ ECC\orbits{}{eh}(\mathrm{int}(N), \alpha_N; j-1, n-2) \]
generated only by orbit sets with Alexander grading $j-1$ and ($e_+$)-multiplicity less than or equal to $n-2$.
\end{thm}

We start with the following proposition.

\begin{prop}\label{U_map_commutes_prop}
There are quasi-isomorphisms
\[ \hat\Phi_k : ECC(\mathrm{int}(N),\alpha_N; k)\to ECC\orbits{}{eh}(\mathrm{int}(N),\alpha_N;k) \]
for all $k\ge 1$ such that the following diagram commutes:
\begin{equation}\label{U_map_commutes}\begin{tikzcd}
ECC(\mathrm{int}(N),\alpha_N; k) \arrow[rr,"d'_N"]\arrow[d,"\hat\Phi_k"] && ECC(\mathrm{int}(N),\alpha_N; k-1)\arrow[d,"\hat\Phi_{k-1}"] \\
ECC\orbits{}{eh}(\mathrm{int}(N),\alpha_N;k) \arrow[rr,"\cdot /e_+"] && ECC\orbits{}{eh}(\mathrm{int}(N),\alpha_N;k-1)
\end{tikzcd}\end{equation}
for all $k$. 
\end{prop}
\begin{proof}
Recall the quasi-isomorphism 
\[ \sigma : ECC\orbits{0he}{}(\mathrm{int}(N),\alpha_N) \to ECC\orbits{0he}{eh}(\mathrm{int}(N),\alpha_N) \]
 from the proof of \cref{U_natural_and_U_sharp_iso_prop}, which commutes with the $U$-maps $U^\natural$ and $U^\sharp$.

We define $\hat\Phi_k$ completely analogously, by the formula
\[ \Gamma \mapsto \sum_{i=0}^\infty e_+^i \tensor (d'_N)^i(\Gamma). \]
Then commutativity of \refDiagram{U_map_commutes} follows easily:
\[ \hat\Phi_k(\Gamma)/e_+ = \sum_{i=0}^\infty e_+^{i-1} \tensor (d'_N)^i(\Gamma) = \sum_{i=1}^\infty e_+^{i} \tensor (d'_N)^{i+1}(\Gamma) = \hat\Phi_{k-1}(d'_N(\Gamma)). \]
We see that the maps $\hat\Phi_k$ are quasi-isomorphisms by examining the proof of the \refNamedThm{Cancellation Lemma}{cancellation_lemma} in exactly the same manner as that used when showing that $\sigma$ was a quasi-isomorphism in \cref{U_natural_and_U_sharp_iso_prop}.
\end{proof}

We can join \hyperref[commutative_diagram_L_filtration_eqn]{Diagrams \ref*{commutative_diagram_L_filtration_eqn}} and \ref{U_map_commutes} to form the following:
\[\begin{tikzcd}
ECH\orbits{e}{h}(\mathrm{int}(N),\alpha_N;j) \arrow[rr,"d^{(n+1)}_{\mathcal{F},1}"]\arrow[d,"d^{(n)}_{\mathcal{F},1}"'] && h_{\nicefrac{1}{n+1}}ECH(\mathrm{int}(N),\alpha_N;j-(n+1)) \arrow[d,"\iso"'] \\
h_{\nicefrac{1}{n}}ECH(\mathrm{int}(N),\alpha_N;j-n) \arrow[d,"\cdot/h_{1/n}"', "\iso"]\arrow[rr,"h_{\nicefrac{0}{1}}(d'_N)_*"] && h_{\nicefrac{1}{n}}h_{\nicefrac{0}{1}}ECH(\mathrm{int}(N),\alpha_N;j-(n+1))\arrow[d,"\cdot/(h_{1/n}h_{0/1})", "\iso"'] \\
ECH(\mathrm{int}(N),\alpha_N; j-n) \arrow[rr,"(d'_N)_*"]\arrow[d,"(\hat\Phi_{j-n})_*"',"\iso"] && ECH(\mathrm{int}(N),\alpha_N; j-(n+1))\arrow[d,"\Phi_{j-(n+1)}", "\iso"'] \\
ECH\orbits{}{eh}(\mathrm{int}(N),\alpha_N;j-n) \arrow[rr,"(\cdot/e_+)_*"] && ECH\orbits{}{eh}(\mathrm{int}(N),\alpha_N;j-(n+1))
\end{tikzcd}\]
where the middle square commutes by definition.

We therefore obtain the following equation:
\begin{equation}\label{induction_equation}
d_{\mathcal{F},1}^{(n+1)} =  \Phi^{-1}_{n+1}\circ U_*\circ \Phi_{n}\circ d_{\mathcal{F},1}^{(n)},
\end{equation}
for all $n\le j-1$, where
\begin{itemize}
\item $\Phi_k: h_{\nicefrac{1}{k}}ECH(\mathrm{int}(N),\alpha_N;j-k)\to ECH\orbits{}{eh}(\mathrm{int}(N),\alpha_N;j-k)$ is the isomorphism induced by the chain map $h_{\nicefrac{1}{k}}\Gamma \mapsto \hat\Phi_{j-k}(\Gamma)$, and
\item $U_*$ is induced by the chain map $U$ defined by $\Gamma \mapsto \Gamma/e_+$.
\end{itemize}

The equation in \cref{induction_equation} is used in the inductive step when proving \cref{surgery_mapping_cones} below.  The ingredient making up the ``base case'' of the inductive argument involves considering ($-1$)-surgery:

\begin{prop}\label{1_surgery}
Suppose that $j > 2g$. Then the map
\begin{equation}\label{d_F_1_map} d_{\mathcal{F},1}^{(1)} : ECH\orbits{e}{h}(\mathrm{int}(N), \alpha_N; j) \to h_{\nicefrac{1}{1}}ECH(\mathrm{int}(N),\alpha_N; j-1) \end{equation}
is an isomorphism.
\end{prop}
\begin{proof}
Note that ($-1$)-surgery yields an integral open book decomposition, since we perform a full twist in the twist region $S^1\cross[1,2]$ and hence the resulting monodromy restricts to the identity map on the new boundary at $y=1$.

We can therefore apply \cref{ECK_supported_genus}, which states that $\widehat{ECK}$ is supported in Alexander gradings $2g$ and below, to see that
\[ \widehat{ECK}(K(-1),\alpha'_j;j) := ECH\orbits{eh}{h}(\mathrm{int}(N\union T), \alpha'_j;j) = 0. \]

But \cref{d_F_1_map} describes the entirety of the first page of the spectral sequence corresponding to the filtration $\mathcal{F}$, and hence it must be the case that the spectral sequence converges at the second page and hence $d_{\mathcal{F},1}^{(1)}$ is an isomorphism.
\end{proof}

We are now ready to prove \cref{surgery_mapping_cones}.

\begin{proof}[Proof of \cref{surgery_mapping_cones}]
Let $j>2g$ be fixed, and suppose that $1\le n \le j$. Then we can apply \cref{induction_equation} repeatedly to write 
\[\begin{split}
\Phi_n\circ d_{\mathcal{F},1}^{(n)} &= U_*\circ \Phi_{n-1}\circ d_{\mathcal{F},1}^{(n-1)} \\
&= U_*\circ \Phi_{n-1}\circ \left(\Phi^{-1}_{n-1}\circ U_*\circ \Phi_{n-2}\circ d_{\mathcal{F},1}^{(n-2)}\right) \\
&= U^2_*\circ \Phi_{n-2}\circ d_{\mathcal{F},1}^{(n-2)} \\
&\cdots \\
&= U_*^{n-1} \circ \Phi_{1}\circ d_{\mathcal{F},1}^{(1)}.
\end{split}\]
This yields the following commutative diagram:
\[ \begin{tikzcd}
ECH\orbits{e}{h}(\mathrm{int}(N),\alpha_N;j)\arrow[rr,"d_{\mathcal{F},1}^{(n)}"]\arrow[d,"\Phi_{1}\circ d_{\mathcal{F},1}^{(1)}"] && h_{\nicefrac{1}{n}}ECH(\mathrm{int}(N),\alpha_N;j-n)\arrow[d,"\Phi_n"] \\
ECH\orbits{}{eh}(\mathrm{int}(N),\alpha_N;j-1)\arrow[rr,"U_*^{n-1}"] && ECH\orbits{}{eh}(\mathrm{int}(N),\alpha_N;j-n)
\end{tikzcd}\]
where the vertical maps are isomorphisms by \cref{U_map_commutes_prop,1_surgery}.
and hence the mapping cones associated to the top and bottom maps respectively are isomorphic.

The main result of the theorem follows since the homology of the mapping cones associated to $U^{n-1}$ and $U^{n-1}_*$ are equal. The fact that the mapping cone of $U^{n-1}$ is chain homotopy equivalent to the subcomplex
\[ ECC\orbits{}{eh}(\mathrm{int}(N), \alpha_N; j-1, n-2). \]
follows by the symmetry of the full ECK complex as discussed in \cref{symmetries_of_the_complex}; $U^{n-1}$ is surjective and its kernel is precisely those orbit sets with ($e_+$)-multiplicity less than or equal to $n-2$.
\end{proof}

\section{Proof of the surgery formula}

In this section we will prove \cref{surgery_formula}.  Recall that embedded contact homology splits as a direct sum with respect to the first homology group of the manifold, and the hat version of embedded knot contact homology splits as a direct sum with respect to the first homology group of the knot complement.

The correspondence between these splittings is given by the map on homology induced by the inclusion map
\[ M\sminus K \hookrightarrow M. \]

In our setting, the homology of the knot complement is equal to $\Z\oplus G$ where the $\Z$ factor, generated by the meridian curve $[m]$, we of course know as the Alexander grading.  When we perform ($-n$)-surgery, the $\Z$ factor is reduced modulo $n$, since $n[m]$ is identified with $[l]$, the boundary of $\Sigma$.  We therefore have that
\[ H_1(M(-n)) \iso \Z/n \oplus G \]
and hence we can write
\[ \widehat{ECH}(M(-n), \alpha') = \bigoplus_{[j]\in \Z/n}\widehat{ECH}(M(-n), \alpha'; [j]) \]
where $\widehat{ECH}(M(-n), \alpha'; [j])$ is obtained by considering all orbit sets with Alexander grading equal to $j$ modulo $n$.  This motivates the notation
\[ \widehat{ECK}(K(-n); \alpha'; [j]) := \bigoplus_{j' \equiv j \ (n)}\widehat{ECK}(K(-n),\alpha';j') \]
which we saw in the introduction.

The proof of \cref{surgery_formula} follows almost immediately from the results of the last section, but we must first relate the complexes
\[  ECC\orbits{e}{h}(\mathrm{int}(N), \alpha_N; j)\quad\text{and}\quad ECC\orbits{}{eh}(\mathrm{int}(N), \alpha_N; (n+j)-1, n-2) \]
to the complexes $A_j$ and $B_{2g-j}$ from \cref{complexes_A_j_and_B_i}.

\begin{prop}\label{B_j_isos}
Suppose that $j \ge 0$.  We have a series of isomorphisms
\[\begin{split}
ECH\orbits{}{eh}(\mathrm{int}(N), \alpha_N; (n+j)-1, n-2) &\iso ECH\orbits{he}{eh}(\mathrm{int}(N), \alpha_N; (n+j)-1, n-2) \\
 &\iso ECH\orbits{he}{eh}(\mathrm{int}(N), \alpha_N; 2g, 2g-j-1) \\
 &\iso H_*(B_{2g-j}).
\end{split}\]
In particular, the above groups are 0 for $j\ge 2g$.
\end{prop}

\begin{proof}
The first isomorphism follows by the \refNamedThm{Cancellation Lemma}{cancellation_lemma} applied to the differential from $h_+$ to $e_-$.  The second isomorphism follows from the diagonal translational symmetry of the full ECK complex as discussed in \cref{symmetries_of_the_complex} and the fact that the ($e_+$)-filtered homotopy type of each row stabilizes for rows $2g$ and higher by \cref{ECK_supported_genus}. (Note that $n+j-1 \ge 2g$ since $n>2g$ and $j\ge 0$.) The third isomorphism follows by definition.
\end{proof}

\begin{prop}\label{A_j_iso}
Suppose that $j\ge 0$. We have an isomorphism
\[ ECH\orbits{e}{h}(\mathrm{int}(N), \alpha_N; j) \iso H_*(A_j). \]
In particular, for $j\ge 2g$, the groups above are simply equal to $\widehat{ECH}(M)$.
\end{prop}

\begin{proof}
Consider the chain map
\[ \sigma: A_j  \to  ECC\orbits{e}{h}(\mathrm{int}(N); \alpha_N; j) \]
defined by killing $h_+$ and padding all other orbit sets with the appropriate power of $e_-$:
\begin{align*}
h_+\Gamma &\mapsto 0,\\
\Gamma &\mapsto e_-^{j-A(\Gamma)}\Gamma
\end{align*}
where $A(\Gamma)$ is the Alexander grading of the orbit set $\Gamma$.  It is easy to see that this is a chain map since
\[ \sigma(d(h_+\Gamma)) = \sigma(\Gamma + e_-\Gamma) = e_-^{j-A(\Gamma)}\Gamma + e_-^{j-A(e_-\Gamma)}e_-\Gamma = 0 \]
(as we are working over $\F_2$) and, as the orbit $e_-$ can only appear at negative ends of differentials,
\[ e_- d\Gamma = d(e_- \Gamma) \]
for all orbit sets $\Gamma$ contained in $N$.

Write $C_{\le j}$ for the subcomplex of $ECC\orbits{e}{h}(\mathrm{int}(N), \alpha_N)$ generated by orbits sets with Alexander grading less than or equal to $j$.  Then, as in the proof of \cref{relative_ECH_iso_augmented_complex_prop}, the complex $A_j$ can be exhibited as the mapping cone of the function
\begin{align*}
f: C_{\le j-1} &\to C_{\le j} \\
\Gamma &\mapsto (1+e_-)\Gamma.
\end{align*}
Hence we have a triangle
\[  C_{\le j} \hookrightarrow A_j \to C_{\le j-1} \xrightarrow{f} \]
(where the second map is given by $\Gamma\mapsto \Gamma / h_+$), which induces a long exact sequence on homology.

We also have a short exact sequence
\begin{equation}\label{short_exact_sequence_final_isomorphism}
C_{\le j-1} \xrightarrow{f} C_{\le j} \xrightarrow{\sigma'} ECC\orbits{e}{h}(\mathrm{int}(N); \alpha_N; j)
\end{equation}
(where $\sigma'$ is the map $\Gamma\mapsto e_-^{j-A(\Gamma)}\Gamma$), since $\mathrm{Im}(f)$ and $\mathrm{ker}(\sigma')$ are both equal to the ideal generated by $1+e_-$.

The induced map $\sigma_*$ fits into the diagram of long exact sequences below. 
\[\begin{tikzcd}
\arrow[r] & H_*(C_{\le j-1})\arrow[r,"f_*"]\arrow[d,"\id"] & H_*(C_{\le j})\arrow[r,"i_*"]\arrow[d,"\id"] & H_*(A_j)\arrow[r,]\arrow[d,"\sigma_*"] & H_*(C_{\le j-1})\arrow[d,"\id"]\arrow[r] & ~\\
\arrow[r] & H_*(C_{\le j-1})\arrow[r,"f_*"] & H_*(C_{\le j})\arrow[r,"\sigma'_*"] & H_*(C_j)\arrow[r,] & H_*(C_{\le j-1}) \arrow[r] &~
\end{tikzcd}\]
Here we are writing $C_j$ instead of $ECC\orbits{e}{h}(\mathrm{int}(N); \alpha_N; j)$.

The first square commutes trivially, and the second since we already have commutativity at the chain level.  Commutativity of the third square follows since
\begin{align*}
 f_*: H_*(C_{\le j-1}) &\to H_*(C_{\le j}) \\
[\Gamma] &\mapsto [\Gamma] + [e_-\Gamma]
\end{align*}
is injective and hence both the horizontal maps in the third square are in fact zero.  Applying the five lemma yields the result that $\sigma_*$ is an isomorphism.

The fact that $H_*(A_j)\iso \widehat{ECH}(M)$ for $j\ge 2g$ follows from \cref{ECK_supported_genus}.
\end{proof}

\begin{proof}[Proof of \cref{surgery_formula}]
The proof follows from \cref{B_j_isos,A_j_iso}.  We simply compute the homology in each Alexander grading $j$.

\paragraph{1. ($j<n$)}In these Alexander gradings,
\[ \widehat{ECK}(K(-n),\alpha';j) \iso ECC\orbits{e}{h}(\mathrm{int}(N),\alpha_N; j) \iso H_*(A_j), \]
where the first isomorphism follows from \cref{first_page_of_F_spectral_sequence} and the second by \cref{A_j_iso}.  In particular, for $2g\le j < n$ we obtain $\widehat{ECH}(M)$ as required.

\paragraph{2. ($j \ge n$)}In these Alexander gradings, we have
\[ \widehat{ECK}(K(-n),\alpha';j) \iso ECH\orbits{}{eh}(\mathrm{int}(N), \alpha_N, j-1, n-2) \iso H_*(B_{2g-(j-n)}), \]
where the first isomorphism follows from \cref{surgery_mapping_cones} and the second from \cref{B_j_isos}.  In particular, for $j\ge n+2g$ the groups vanish and we are done.
\end{proof}

It is helpful here to refer back to \cref{two_towers_example}.  Firstly, the $A_j$ summands can be seen as the $j$-th row of the $e_-$ tower for $j<n$.  In rows $j\ge n$, recall that $d_{\mathcal{F},1}$ is a map in each row between (the homology of) the rectangular box and the oval box.  We can now say that the mapping cone in these rows has homology equal to that of $B_{2g-(j-n)}$, and that homology vanishes in rows $n+2g$ and above.

\section{Further questions}

\paragraph{A surgery formula for arbitrary knots.}The surgery formula in \cref{surgery_formula} only applies to knots which are the binding of an integral open book decomposition, but it should be possible to generalize this result to arbitrary knots.  Indeed, given any knot $K\subset M$, it is possible to realize $K$ as a single component of the binding of some integral open book (proved independently by Baker, Etnyre and Van Horn-Morris~\cite[Section 6]{BEV} and Guyard~\cite[Theorem 0.12]{Guyard}).  We can then proceed by the methods of \cref{ECH_via_rational_open_book_decompositions_chapter} to express $ECH(M)$ in terms of an appropriately defined relative ECH of the mapping torus $N$ with, say, $k\ge 1$ torus boundary components (c.f.~\cref{relative_ECH_equals_ECH}).  It should be possible to also generalize the work of Colin, Ghiggini and Honda (c.f.~the discussion in \cref{ECH_equals_HF_section}) to find an isomorphism between these relative ECH groups and the Heegaard Floer groups defined via a Heegaard diagram corresponding to the open book.  If $g$ is the genus of the page of the open book and $k$ is the number of boundary components then the relevant Heegaard surface is obtained by doubling this surface along its boundary (c.f.~the construction in \cite[Section 5]{CGHsummary}) and therefore has genus $2g+k-1$.  As a result, we expect that by an argument similar to that of \cref{ECK_hat_supported_in_low_Alexander_degrees}, it should be possible to show that that the resulting $\widehat{ECK}$ groups vanish in Alexander gradings greater than $2g+k-1$.  Given this, the arguments of \cref{a_surgery_formula_chapter} carry over directly and hence would result in a surgery formula for all $n>2g+k-1$.

\paragraph{A surgery formula for $\widehat{ECH}$.} The surgery formula in \cref{surgery_formula} computes $\widehat{ECK}(K(-n))$ in terms of the full complex, $ECK(K)$, of the original knot.  This agrees with the large-$n$ formula in the Heegaard Floer setting, but does not provide the full picture. Indeed, in the Heegaard Floer setting, we are able to compute $\widehat{HF}(M(n))$ from the full complex $CFK^+(K)$.  To achieve an analogous result in the embedded contact setting, it is necessary to understand how the differential
\[ h_+\longrightarrow \emptyset, \]
which lowers Alexander grading by $n$, manifests itself in the surgery formula.  More concretely, we know by \cref{relative_ECH_equals_ECH} and \cref{ECH_0he_h_iso_relative_ECH} that we have an isomorphism
\[ \widehat{ECH}(M(-n)) \iso ECH\orbits{0he}{h}(\mathrm{int}(N\union T),\alpha,J), \]
and the complex on the right admits a filtration by Alexander grading whose associated graded complex computes
\[ \widehat{ECK}(K(-n)). \]
In other words, the surgery formula in \cref{surgery_formula} computes the first page of the spectral sequence corresponding to this Alexander filtration.  It should be possible then to understand how the differential between $h_+$ and the empty set induces a chain map on $\widehat{ECK}(K(-n))$, the homology of which yields $\widehat{ECH}(M(-n))$.  (C.f.~the vertical differential between the oval and rectangle complexes in \cref{surgery_formula_figure}.)

Unfortunately, despite the author's best efforts, little progress has been made towards this goal.  It seems that such a formula may require understanding the behaviour of ``complex'' holomorphic curves with slope $n[m]-[l]$ at $y=2$, similar to those which make up the differential $d_{\mathcal{F},1}$.  Whereas the curves forming $d_{\mathcal{F},1}$ had negative ends at $h_{\nicefrac{1}{n}}$, the curves we seek in this instance must have negative ends at the neighbouring orbit $e_{\nicefrac{1}{n}}$. More research is necessary in this regard.

\paragraph{A rational surgery formula.}There is a more general surgery formula in the Heegaard Floer setting, which computes the Heegaard Floer homology of any ($p/q$)-Dehn surgery on $K$ in $M$ via a mapping cone construction~\cite[Theorem 1.1]{OS11knot}.  Furthermore, despite not being written down anywhere as far as the author is aware, a version of this formula can compute $\widehat{HFK}$ of the surgered knot $K_{p/q}$.  It would be an interesting to see if this mapping cone formula could be reconstructed in the case of ECK.

\paragraph{A surgery formula with $\Z$ coefficients.}It is possible to define ECH with $\Z$ coefficients~\cite[Section 9]{HT09} rather than $\F_2$ as in this thesis, and it would be interesting to investigate a surgery formula in this setting. As remarked by Colin, Ghiggini and Honda~\cite[Remark 9.9.5]{CGH_ECH_OBD}, it is expected that all results from \crefrange{ECH_chapter}{ECH_via_rational_open_book_decompositions_chapter} will hold over $\Z$. For this reason we also expect all results in \crefrange{embedded_contact_knot_homology}{invariance_results_in_the_ECK_setting} of \cref{a_knot_version_of_ECH_chapter} to hold, only running into trouble when proving that $\widehat{ECK}$ is supported in low Alexander degrees in \cref{ECK_hat_supported_in_low_Alexander_degrees} (the problem arises here since we exploited the isomorphism between $\widehat{HF}$ and PFH with $\F$ coefficients in \cref{stabilization_diag}).

In \cref{a_surgery_formula_chapter} most results carry over to the $\Z$-setting---in particular we note that the periodic Floer homology of a Dehn twist can be computed with $\Z$ coefficients, using an orientation system described by Hutchings and Sullivan in 2006~\cite{HS06}.  A less trivial point is the use of $1=-1\in\F_2$ in \cref{commutative_diagram_L_filtration}; when considered over $\Z$ the equation $d^2=0$ will instead produce an \emph{anticommutative} diagram on the first page $E^1(\mathcal{L})$.  However the orientation system described by Hutchings and Sullivan~\cite{HS06} seems to suggest that, over $\Z$, the differential $d_{\mathcal{L},1}$ will induce the map
\begin{align*} 
h_{\nicefrac{1}{n+1}}ECH(\mathrm{int}(N),\alpha_N;j-(n+1)) &\to h_{\nicefrac{1}{n}}h_{\nicefrac{0}{1}} ECH(\mathrm{int}(N),\alpha_N;j-(n+1)) \\
h_{\nicefrac{1}{n+1}}\Gamma &\mapsto -h_{\nicefrac{1}{n}}h_{\nicefrac{0}{1}}\Gamma, 
\end{align*} 
and hence \refDiagram{commutative_diagram_L_filtration_eqn}, in which the map on the right hand side is given by
\[h_{\nicefrac{1}{n+1}}\Gamma \mapsto +h_{\nicefrac{1}{n}}h_{\nicefrac{0}{1}}\Gamma,\]
 will commute as desired.

\paragraph{Remaining conjectures.} Finally, there are the two conjectures in the thesis which remain unproven.  The first, \cref{full_ECK_J_and_alpha_invariant}, concerning $\alpha$ and $J$ invariance of the bi-filtered homotopy type of the full ECK complex
\[ ECC\orbits{0he}{eh}(\mathrm{int}(N),\alpha,J) \]
in the case of a rational open book decomposition, seems the more approachable of the two, and would follow immediately from a proof that $\widehat{ECK}$ is supported in low Alexander degrees in this setting. The second, \cref{ECK_HFK_conjecture}, concerning the equivalence between HFK and ECK, is perhaps a long way off. However, a proof of this conjecture would provide a substantial and important bridge within knot theory, between the two fields of contact topology and Heegaard Floer theory.

\printbibliography

\end{document}